\documentclass[11pt]{article}
\usepackage[a4paper,margin=0.9in]{geometry}
\usepackage[T1]{fontenc}
\usepackage{lmodern}
\usepackage{amsmath,amssymb,amsthm,mathtools,esint,stmaryrd}
\usepackage{booktabs,array,enumitem,microtype,xcolor}

\usepackage[markup=underlined]{changes} 
\definechangesauthor[name={Z. C.}, color=orange]{Z.C}

\usepackage[colorlinks=true,allcolors=black]{hyperref}
\allowdisplaybreaks[2]
\numberwithin{equation}{section}
\newtheorem{theorem}{Theorem}[section]
\newtheorem{proposition}[theorem]{Proposition}
\newtheorem{lemma}[theorem]{Lemma}

\theoremstyle{remark}
\newtheorem{remark}[theorem]{Remark}
\newcommand{\R}{\mathbb R}
\newcommand{\Sph}{S}

\newcommand{\cP}{\mathcal P}

\newcommand{\pa}{\partial}
\newcommand{\eps}{\varepsilon}
\newcommand{\norm}[1]{\lVert #1\rVert}
\newcommand{\dd}{\mathop{}\!\mathrm{d}}
\newcommand{\arxiv}[1]{\href{https://arxiv.org/abs/#1}{arXiv:\nolinkurl{#1}}}
\DeclareMathOperator{\tr}{tr}
\DeclareMathOperator{\diver}{div}

\title{\bfseries Isolated singularities of solutions to the Yamabe equation in higher dimensions I}
\author{Zheng-Chao Han,\quad Qinfeng Jiang,\quad Hua-Yang Wang\thanks{H.-Y. Wang was partially supported by NSFC grant 12601183.}\\[2mm]
Jingang Xiong\thanks{J. Xiong was partially supported by NSFC grant 12325104.},\quad Lei Zhang\thanks{Lei Zhang was partially supported by Simons Foundation Grant SFI-MPS-TSM-00013752 }}
\date{}

\begin{document}
\maketitle
\begin{abstract}
We complete, in the general Riemannian setting, the classification of isolated
singularities of positive solutions to the Yamabe equation for dimensions
\(3\le n\le29\), initiated in the flat case by Caffarelli, Gidas, and
Spruck in 1989. Precisely, every isolated singularity is either removable
or the solution is asymptotic to a Fowler solution. The sharpness of this range is
established by counterexamples for every dimension \(n\ge30\) in our companion paper. 
\end{abstract}
\setcounter{tocdepth}{1}
\tableofcontents
\clearpage

\section{Introduction}

Let \(g\) be a smooth Riemannian metric on the unit ball
\(B_1\subset\R^n\), where \(n\ge3\), and let
\(u\in C^2(B_1\setminus\{0\})\) be positive. Set
\[
 p=\frac{n+2}{n-2},\qquad
 c(n)=\frac{n-2}{4(n-1)},\qquad
 L_g=\Delta_g-c(n)\operatorname{Scal}_g,
\]
where \(\Delta_g\) and \(\operatorname{Scal}_g\) are the Laplace--Beltrami operator and
the scalar curvature. Write \(d_g\) for the Riemannian distance.
We consider smooth solutions $u$ to the Yamabe equation
\begin{equation}\label{eq:yamabe}
 -L_gu=n(n-2)u^p, \quad u>0,
 \quad\text{in }B_1\setminus\{0\}.
\end{equation}
The singularity is said to be removable
if \(u\) extends to a positive smooth solution of \eqref{eq:yamabe}
across the origin.

A positive Fowler solution is a radial solution of the flat equation
\[
 -\Delta u=n(n-2)u^p
 \quad\text{in }\R^n\setminus\{0\},
\]
with a nonremovable singularity at the origin. It has the form
\[
 u_F(r)=r^{-(n-2)/2}v_F(-\log r),
\]
where \(v_F\) is a positive constant or a positive nonconstant periodic
solution of
\begin{equation}\label{eq:fowler-ode}
 v_F''-\frac{(n-2)^2}{4}v_F+n(n-2)v_F^p=0.
\end{equation}
This description goes back to Fowler  \cite{F}. 
A classical theorem of  Caffarelli--Gidas--Spruck~\cite{CGS} asserts that every positive solution of the flat equation in $\R^n\setminus\{0\}$ with a nonremovable singularity at the origin is radial, and hence is a Fowler solution.  Moreover, for the local equation \eqref{eq:yamabe} with flat background metric, they proved that every solution with a nonremovable singularity at the origin is asymptotically radial, and in fact, converges to a Fowler solution as $x\to 0$. A key component of their proof is measure theoretic moving planes. Later, Chen--Lin \cite{ChenLin1995} and C. Li \cite{Lc} provided complementary partial generalizations of the Caffarelli--Gidas--Spruck asymptotic-symmetry results. A different proof using spectrum analysis and blow up iteration, together with a refinement of the results in \cite{CGS}, was given by Korevaar--Mazzeo--Pacard--Schoen \cite{KMPS}. 

In this paper, we consider the case where the background metric is not
necessarily flat. The convergence to a Fowler solution has been proved by
Marques~\cite{f-mar} in dimensions $3\le n\le 5$, by Xiong--Zhang~\cite{x-z-1}
in dimension $n=6$, and by Han--Xiong--Zhang~\cite{HanXiongZhang} for metrics
whose flatness order at $0$ is at least $(n-2)/2$ or for solutions satisfying $\limsup_{x\to 0} |x|^{n-2}u(x)<\infty$. In the current paper, we prove that the convergence holds for all metrics in the remaining dimensions
$7\le n\le 29$. This range is sharp: we construct counterexamples showing that the general convergence statement fails in dimensions $n\ge30$ in \cite{YamabePartII}.

\begin{theorem}\label{thm:main}
	Let $7\le n\le29$, and let $u$ be a positive solution of \eqref{eq:yamabe}.
	Then either $u$ extends smoothly across the origin, or, in any geodesic normal coordinate $x$ of $g$
	centered at the origin,
	\begin{equation}\label{eq:main-fowler-asymptotic}
		 u(x)=u_F(|x|)\bigl(1+O(|x|^\alpha)\bigr)
		\quad\text{as } x\to0,
	\end{equation}
	for a Fowler solution $u_F$ and some $\alpha>0$.
    Furthermore, there exist $a\in \mathbb{R}^n$ and
some $\beta>1$ such that
\begin{equation}\label{eq:main-fowler-asymptotic-order1}
 u(x)=u_F(|x|)
+
(a\cdot x)\Bigl[|x|\,u_F'(|x|)+(n-2)u_F(|x|)\Bigr]
+
O(|x|^\beta)u_F(|x|)
 \quad\text{as }x\to0.
\end{equation}
\end{theorem} 

  The proof of Theorem~\ref{thm:main} relies crucially on the following
critical upper bound on the growth of solutions near the singularity.

\begin{theorem}\label{thm:upper-bound}
	Let $7\le n\le29$, and let $u$ be a positive solution of \eqref{eq:yamabe}.
	There exist $C_0>0$ and $r_0\in(0,1)$, depending only on $g$ and $u$, such
	that
	\begin{equation}\label{eq:critical-upper}
		u(x)\le C_0\,d_g(x,0)^{-(n-2)/2}
		\quad\text{whenever }0<d_g(x,0)<r_0.
	\end{equation}
\end{theorem}

As we show in the companion paper \cite{YamabePartII}, the critical upper bound
\eqref{eq:critical-upper} may fail in dimensions $n\ge30$, which explains why the dimensional range in Theorem~\ref{thm:main} cannot be extended. More precisely, for each $n\ge 30$, we construct a smooth metric on $\mathbb R\mathbb P^n$ and a
positive solution of the corresponding Yamabe equation having an isolated singularity, near
which the solution blows up along a sequence at a rate much faster than that of Fowler solutions.
Moreover, the singular solution defines a complete conformal metric of constant positive scalar
curvature and the metric agrees with the round metric to infinite order at the singular point. For the existence of singular solutions at prescribed points, we first construct local singular solutions satisfying \eqref{eq:main-fowler-asymptotic} in all dimensions in Theorem A.1 of  \cite{YamabePartII}; then  in the forthcoming paper \cite{YamabePart3}, we  prove that if $(M,g)$ is a smooth,
connected, closed Riemannian manifold of dimension $n\ge3$,
has positive Yamabe constant, and is not conformally diffeomorphic
to the round sphere, then, for every nonempty finite set
$\{P_1,\ldots,P_m\}\subset M$, there exists a positive smooth
function $u$ on $M\setminus\{P_1,\ldots,P_m\}$ such that
$u^{4/(n-2)}g$ is complete and has constant positive scalar
curvature, and $u$ has Fowler asymptotics at each $P_i$. 

All the upper-bound arguments in the works mentioned above rely on the
method of moving planes or moving spheres. When the metric is not flat,
$u-u_{\mathrm{reflection}}$ satisfies an inhomogeneous linear elliptic
equation, and thus a key ingredient of the reflection argument---the
maximum principle---is no longer directly applicable. Chen--Lin~\cite{ChenLin1997} introduced an idea of constructing suitable auxiliary
functions to adapt the moving planes/spheres method. This idea was adapted and extended in later works. The auxiliary functions in Marques~\cite{f-mar} and
Xiong--Zhang~\cite{x-z-1} were radially symmetric, while
Han--Xiong--Zhang~\cite{HanXiongZhang}, inspired by the proofs of the
compactness of the Yamabe equation in higher dimensions (see
Li--Zhang~\cite{LZII, LZIII} and Khuri--Marques--Schoen~\cite{KMS}), constructed
non-radial auxiliary functions and proved the quantitative estimates needed for their argument; their construction, however, requires a flatness condition and is restricted to dimensions $n\le24$.

In this paper, we develop a new method that does not use moving planes or
moving spheres. In particular, our argument 
exploits the balance of terms in  the adjoint integral identity
\eqref{eq:global-four-term-identity} satisfied by the error between the rescaled solution and a standard bubble solution with corrections as given in \cite{KMS}, but only the positivity
of the Khuri--Marques--Schoen  quadratic form in \cite{KMS} on the space of metric jets of degree up to six is required; this allows us to cover the
additional dimensions $25\le n\le29$. It is worth noting that the
compactness of the Yamabe equation fails in dimensions $n\ge25$
(Brendle~\cite{Bre} for $n\ge 52$, and Brendle--Marques~\cite{Bre-Mar} for
$25\le n\le51$); in contrast, the upper bound \eqref{eq:critical-upper}
remains valid up to dimension $n=29$. An overview of our method is given at the end of this section.

Under the critical upper bound assumption, the following theorem provides a matching lower bound with no restriction on the dimension.

\begin{theorem}\label{thm:lower-bound}
Let \(n\ge7\), $u$ be a solution of \eqref{eq:yamabe}, and suppose that the origin is nonremovable and that
\eqref{eq:critical-upper} holds for some \(C_0,r_0>0\). Then there exist
\(c>0\) and \(r_1\in(0,r_0)\) such that
\begin{equation}\label{eq:lower-bound}
 u(x)\ge c\,d_g(x,0)^{-(n-2)/2}
 \quad\text{whenever }0<d_g(x,0)<r_1.
\end{equation}
\end{theorem}

Our proof of the lower bound is based on the Pohozaev identity on
successive annuli, continuing arguments developed by Chen--Lin
\cite{ChenLin1999} for scalar curvature equations with variable
coefficients, used by Marques~\cite{f-mar} for metrics that need not be
conformally flat, and refined by Han--Xiong--Zhang~\cite[Proposition~5.3 and
Section~6]{HanXiongZhang} through a sharper analysis of the spherical
average. These arguments estimate the Pohozaev bulk term by means of upper
bounds near a local minimum of the weighted spherical average.  In the present proof, we first derive a refined expansion of the solution
by adding the corrected profile, and then perform the comparison on the
full annulus. This leads to a quadratic annular form, whose analysis again
relies on the linearized Yamabe operator at a bubble~\cite{KMS}. See  the overview of our method at the end of this section for more details.

Theorems~\ref{thm:main} and \ref{thm:upper-bound} are formulated for $7\le n\le 29$ and  Theorem~\ref{thm:lower-bound} is formulated for $n\ge 7$, but their conclusions hold for $3\le n\le 6$ based on earlier results of Marques \cite{f-mar} and 
Xiong--Zhang~\cite{x-z-1}.

Whenever the critical upper bound \eqref{eq:critical-upper} and the matching lower bound  \eqref{eq:lower-bound} hold for a solution, its convergence to a Fowler solution reduces to the problem of phase
selection, namely, a unique choice of $u_F$ from the two-parameter family of Fowler solutions for \eqref{eq:main-fowler-asymptotic} to hold. This is an important issue but has been established in the
previous works of Caffarelli-Gidas-Spruck \cite{CGS}, Korevaar-Mazzeo-Pacard-Schoen \cite{KMPS},  and Taliaferro--Zhang \cite{TZ} for the flat metric cases,  and Marques \cite{f-mar} for the general metric cases; the arguments in \cite{TZ}, supplemented by \mbox{\cite[Theorem  3]{HanLiTeixeira}}, apply to the general metric cases as well. Thus the conclusion of Theorem~\ref{thm:main} holds for such a solution without the dimensional restrictions.

We also remark that a solution of \eqref{eq:yamabe} satisfying the critical upper bound  and the matching lower bound  gives rise to a metric $u^{\frac{4}{n-2}}g$ with an end that is complete toward \(0\) and has bounded curvature. Thus, all solutions of \eqref{eq:yamabe} with a non-removable singularity at $0$ in dimensions $3\le n \le 29$ give rise to a metric with a complete bounded curvature end at $0$, as do solutions in dimensions $\ge 30$ but satisfying the critical upper bound \eqref{eq:critical-upper}.

The results of \cite{CGS} and \cite{KMPS} have been extended to some
fully nonlinear Yamabe equations and higher order conformally invariant
equations; see, for example, Li--Li~\cite{LiLi2003},
Li~\cite{Li2006}, Chang--Han--Yang~\cite{CHY2005,CHY2020},  
Han--Li--Teixeira~\cite{HanLiTeixeira}, 
Caffarelli--Jin--Sire--Xiong~\cite{CJSX2014}, Frank-König~\cite{FK2019},
Han--Li--Li \cite{HLL}, and Jin--Xiong~\cite{JX2021}.

\subsection*{Overview of the proofs}

\medskip\noindent\textbf{Notation.}
We write \(B_r=\{y\in\R^n:|y|<r\}\) and
\(A_{r,R}=B_R\setminus\overline{B_r}\), where \(0<r<R\).
The Euclidean metric is \(\delta\), and differential operators without a
metric subscript are Euclidean. We use \(\dd y\), \(\dd S\), and
\(\dd\theta\) for Euclidean volume, Euclidean surface measure, and round
measure on \(\Sph^{n-1}\), respectively. Unless stated otherwise,
\(C>0\) may change from line to line and may depend on
the dimension, the fixed metric and solution, and parameters fixed before
the estimate, but not on the rescaling index. For nonnegative quantities
\(A\) and \(B\), we write \(A\asymp B\) if
\(C^{-1}B\le A\le CB\), with the stated dependence of \(C\).
In coordinate estimates, \(D^a\) denotes the collection of Euclidean
partial derivatives of order \(a\). Derivatives of metric functionals,
such as \(D\!\operatorname{Scal}_\gamma[k]\) and
\(D^2\mathcal E_\gamma[k,\ell]\), denote variations with respect to the
metric. In conformal normal coordinates, we write

\[
 h=\log g,\qquad g=e^h,\qquad g^{-1}=e^{-h}.
\]

 For \(m\ge2\), let \(\mathcal V_m\) be the space of homogeneous symmetric
matrix polynomials \(H\) of degree \(m\) satisfying
\begin{equation}\label{eq:jet-gauge}
 \sum_{i=1}^nH_{ii}(y)=0,\qquad
 \sum_{j=1}^nH_{ij}(y)y^j=0\quad(1\le i\le n).
\end{equation}
For each integer \(\ell\ge2\), let \(\mathcal V_{\le\ell}:=\bigoplus_{m=2}^{\ell}\mathcal V_m\).
For\(H=(H_{ij})\in\mathcal V_m\) with
\(H_{ij}(y)=\sum_{|\alpha|=m}H_{ij,\alpha}y^\alpha\), define the coefficient norm by
\begin{equation}\label{eq:coefficient-norm}
|H|^2=\sum_{i,j=1}^n\sum_{|\alpha|=m}|H_{ij,\alpha}|^2.
\end{equation}
Since \(\mathcal V_m\) is finite dimensional, this norm is equivalent to
every fixed norm on \(\mathcal V_m\), including the
\(L^2(\Sph^{n-1})\) norm.

Set \(d=\left\lfloor\frac{n-2}{2}\right\rfloor\)
throughout the paper.
For the standard bubbles, we use the scale parameter:
\begin{equation}\label{eq:standard-bubble}
 U_{\lambda,b}(y)
 =\left(\frac{\lambda}{\lambda^2+|y-b|^2}\right)^{(n-2)/2},
 \qquad U_\lambda=U_{\lambda,0},\qquad U=U_1.
\end{equation}
In the fixed conformal normal coordinates \(x\), for a differentiable function \(f\), define its dilation derivative by
\begin{equation}\label{eq:dilation-operator}
 D_{\mathrm{di}}f=\frac{n-2}{2}f
 +\sum_{i=1}^nx^i\partial_i f,
\end{equation}
and define its
Pohozaev boundary integral by
\begin{equation}\label{eq:pohozaev}
 \cP(r,f)=\int_{\partial B_r}
 \left(\frac{n-2}{2}f\partial_r f-\frac r2|\nabla f|^2
       +r(\partial_r f)^2+\frac{(n-2)^2}{2}r|f|^{p+1}\right)\,\dd S.
\end{equation}
Here \(\partial_r\) is the outward Euclidean radial derivative, \(\nabla=\nabla_\delta\) and \(dS=dS_\delta\)
are Euclidean quantities in these coordinates. If the functional is subsequently applied after $x=\rho_j y$, we use the same Euclidean boundary functional in the rescaled
coordinates \(y\).

\subsection*{Part I: Overview of the upper bound}

Let \((g_0,u_0)\) denote the original metric and solution of \eqref{eq:yamabe}.
Suppose that the critical upper bound fails. The argument in
Section  \ref{sec:automatic-bound} produces local maxima \(x_k\to0\) and
scales
\[
 \varepsilon_k=u_0(x_k)^{-2/(n-2)},\qquad
 \frac{\varepsilon_k}{d_{g_0}(x_k,0)}\to0.
\]
After conformal normalization at \(x_k\) and rescaling by
\(\varepsilon_k\), the metric \(g_k\) and solution \(v_k\) are defined
on \(B_{R_k}\), where \(R_k\to\infty\), and
\[
 v_k(0)=1,\qquad \nabla v_k(0)=0,\qquad
 v_k\to U
 \quad\text{in }C^2_{\mathrm{loc}}(\R^n).
\]
Let \(H_k^{(m)}\) be the homogeneous term of degree \(m\) in the Taylor
expansion of \(\log\widehat g^{(k)}\) in the unrescaled normal coordinates
centered at \(x_k\). Thus the corresponding term in \(\log g_k(y)\) is
\(\varepsilon_k^mH_k^{(m)}(y)\). The corrections of
Khuri--Marques--Schoen  \cite{KMS} satisfy
\[
 \bigl(\Delta+n(n+2)U^{4/(n-2)}\bigr)
 \psi_{m,1}(H_k^{(m)})
 =c(n)\sum_{i,j=1}^n\partial_i\partial_jH_{k,ij}^{(m)}U,
\]
with value and gradient zero at the origin. Their scaled versions define
the family
\[
 \varphi_{k,\lambda}
 =U_\lambda+\sum_{m=4}^{n-4}\varepsilon_k^m
   \psi_{m,\lambda}(H_k^{(m)}),\qquad
 Z_k=-\lambda\partial_\lambda\varphi_{k,\lambda}
       \big|_{\lambda=1}.
\]
Thus \(Z_k\) is the negative logarithmic scale derivative of
\(\varphi_{k,\lambda}\) at \(\lambda=1\), with \(y\),
\(\varepsilon_k\), and the tensors \(H_k^{(m)}\) held fixed.
For \(\varphi_k=\varphi_{k,1}\), define 
\(E_k\) by

\[
 E_k=-L_{g_k}\varphi_k-n(n-2)\varphi_k^p,\quad w_k=v_k-\varphi_k.
\]
From this the error equation for $w_k$ \eqref{eq:short-error-equation} is written as
\[
 \mathcal L_k: =-L_{g_k}-pn(n-2)\varphi_k^{p-1}, \quad L_kw_k=N_k-E_k,
 \]
 where \(N_k=n(n-2)\bigl(v_k^p-\varphi_k^p
        -p\varphi_k^{p-1}w_k\bigr)\ge0\)  is the nonlinear remainder.
        
 Extend
\(Z_k\) from a fixed ball by a solution of the homogeneous linearized
equation on the exterior annulus. The extension \(\widetilde Z_k\), defined in \eqref{eq:short-global-adjoint}, is
negative there. Testing the error equation against \(\widetilde Z_k\) and integrating by parts gives the four-term identity
\eqref{eq:global-four-term-identity}. Its outer-boundary term supplies a positive contribution of order
\(\varepsilon_k^{(n-2)/2}\)(Lemma~\ref{lem:exterior-corrected-adjoint}); its principal geometric volume term \(\int_{B_{R_k}} \widetilde Z_kE_k\)
  is first compared with
\(\int_{B_{R_k}} Z_kE_k\) (Lemma~\ref{lem:adjoint-replacement-interface}), and the latter is compared with the finite-radius quadratic form
\(I_{\varepsilon_k,R_k}^{(n)}(H_k,H_k)\) defined in
\eqref{eq:truncated-pohozaev-form}--see Lemma~\ref{lem:signed-residual}, and hence with the whole-space Khuri--Marques--Schoen form
\(I_{\varepsilon_k}^{(n)}\) in
\eqref{eq:kms-pohozaev-form}.

Let
\[
 \mathcal I_k^\pm
 =\int_{B_{R_k}}(\widetilde Z_k)_\pm N_k\,\dd y,
\]
then four-term identity
\eqref{eq:global-four-term-identity} and the outer boundary flux (Lemmas~\ref{lem:exterior-corrected-adjoint} and  \ref{lem:global-balance}) give
\[
 \mathcal I_k^+-\mathcal I_k^-
 \ge c\varepsilon_k^{(n-2)/2}.
\]
This is the signed lower bound used in the contradiction. On the fixed ball
that contains the positive part of \(\widetilde Z_k\), Green
representation (Lemma  \ref{lem:fixed-core}) gives \(\|w_k\|_{L^1}\le C\mathcal I_k^+\). Since
\(N_k\le C|w_k|^2\) there and \(w_k\to0\) uniformly on the fixed ball,
\[
 0<\mathcal I_k^+\le C\int_{B_{R_0}}|w_k|^2\dd y
 \le o(1)\,\|w_k\|_{L^1}
 \le o(\mathcal I_k^+),
\]
which is impossible. This proves Theorem  \ref{thm:upper-bound}. The threshold dimension $29$ comes from estimating \(I_{\varepsilon_k,R_k}^{(n)}(H_k,H_k)\)
by splitting \(H_k\) into its part of degree $\le 6$ and the higher degree part and showing that, precisely in the range $n\le 29$, the resulting remainder terms are of higher order than  \(\varepsilon_k^{(n-2)/2}\) of the outer boundary flux.

\subsection*{Part II: Overview of the lower bound}

Use Günther's conformal normal coordinates at the fixed singular point \cite[Section  3]{Gunther}. Suppose
that the singularity is nonremovable, the critical upper bound holds,
and the lower bound fails. Define
\[
 \bar u(r)=\fint_{\Sph^{n-1}}u(r\theta)\,\dd\theta.
\]
By Lemma  \ref{lem:peak-minimum-radii}, the weighted spherical average
\(r^{(n-2)/2}\bar u(r)\) has alternating local maxima and minima. The
estimate from above and below in \eqref{eq:multiplicative-two-sided} and the geometric
mean estimate \eqref{eq:multiplicative-midpoint} locate each minimum
between consecutive maxima. They lead to radii
\[
 \bar\rho_j<\rho_j<\bar\rho_{j-1},\qquad
 \sigma_j=\frac{\bar\rho_j}{\rho_j},\qquad
 R_j=\frac{\bar\rho_{j-1}}{\rho_j}.
\]

On \(A_{\sigma_j,R_j}\), set
\[
 v_j(y)=\rho_j^{(n-2)/2}u(\rho_jy),\qquad
 g_j=\rho_j^{-2}(y\mapsto\rho_jy)^*g,
 \qquad \mu_j=-\cP(\bar\rho_j,u).
\]
The exact Pohozaev identity becomes
\begin{equation}\label{eq:introduction-pohozaev-increment}
 Q_j:=\mu_j-\mu_{j-1}
 =\int_{A_{\sigma_j,R_j}}
 (D_{\mathrm{di}}v_j)(\Delta-L_{g_j})v_j\,\dd y.
\end{equation}
Proposition  \ref{prop:radial-scales} also identifies the two boundary
scales:
\[
 C^{-1}\sigma_j^{n-2}\le\mu_j\le C\sigma_j^{n-2},\qquad
 C^{-1}R_j^{2-n}\le\mu_{j-1}\le CR_j^{2-n}.
\]
It remains to estimate the single volume integral in
\eqref{eq:introduction-pohozaev-increment}.

Compactness first produces a bubble \(U_{\lambda,b}\) whose center may
differ from the puncture. For each homogeneous term \(H^{(m)}\) in
\(h=\log g\), where \(2\le m<(n-2)/2\), define
\(Z_{\lambda,b}[H^{(m)}]\) by

\begin{equation}\label{eq:introduction-correction-equation}
 \bigl(\Delta+n(n+2)U_{\lambda,b}^{4/(n-2)}\bigr)
 Z_{\lambda,b}[H^{(m)}]
 =\sum_{i,j=1}^n
 \partial_i\bigl(H_{ij}^{(m)}\partial_jU_{\lambda,b}\bigr)
 +c(n)\left(\sum_{i,j=1}^n
 \partial_i\partial_jH_{ij}^{(m)}\right)U_{\lambda,b}.
\end{equation}
The linearized operator has the dilation and translation Jacobi fields
in its kernel. We therefore impose the normalization

\[
 \int_{\R^n}U_{\lambda,b}^{4/(n-2)}
 Z_{\lambda,b}[H^{(m)}]J_{\lambda,b,l}\,\dd y=0,
 \qquad 0\le l\le n,
\]
where \(J_{\lambda,b,0}=\lambda\partial_\lambda U_{\lambda,b}\) and
\(J_{\lambda,b,l}=\partial_{b^l}U_{\lambda,b}\) for \(1\le l\le n\).
The correction satisfying \eqref{eq:introduction-correction-equation} cancels the error that is linear in the metric term
\(H^{(m)}\). Section  \ref{sec:euclidean-response} constructs this
normalized solution and proves the estimates needed below.

After the corrections have been constructed, the parameters
\((\lambda_j,b_j)\) are chosen by
Proposition  \ref{prop:final-modulation}. Set

\[
 \Phi_j=U_{\lambda_j,b_j}
 +\sum_{2\le m<(n-2)/2}\rho_j^m
 Z_{\lambda_j,b_j}[H^{(m)}],\qquad
 e_j=v_j-\Phi_j.
\]
The function \(\Phi_j\) satisfies
\[
 -L_{g_j}\Phi_j=n(n-2)\Phi_j^p+\mathcal R_j^\Phi,
\]
where \(\mathcal R_j^\Phi\) is defined in
\eqref{eq:reference-equation-error} and estimated in
\eqref{eq:reference-equation-error-estimate}. The equation for \(\Phi_j\) therefore records
precisely the metric and nonlinear errors that remain after the
corrections have removed the leading jet contributions.

Section  \ref{sec:reference-lower-bound} estimates
\(\widetilde Q_j=B_j(\Phi_j,\Phi_j)\). For every
\(m<(n-2)/2\), Proposition  \ref{prop:finite-annulus-self-pairing}
compares the homogeneous quadratic coefficient with the spherical form
\(\mathcal Q_{\lambda_j,b_j}\). If the first nonzero degree \(m\) is
below \((n-2)/2\), its coefficient gives a positive term of order
\(\rho_j^{2m}\).
Section  \ref{sec:exact-reference-comparison} estimates
\[
 Q_j-\widetilde Q_j
 =B_j(v_j,v_j)-B_j(\Phi_j,\Phi_j)
\]
using \(e_j=v_j-\Phi_j\) and the two boundary corrections of \(e_j\). Combining
these estimates with \eqref{eq:exact-Qj-increment} gives
\[
 \mu_{j-1}\le(1+C\eta_j)\mu_j
 +C\rho_j^{n-2}(1+|\log\rho_j|),
 \qquad \eta_j=\bar\rho_{j-1}^2.
\]
The summability of \(\eta_j\) permits backward iteration. The resulting
upper bound for \(\mu_j\), together with the lower bound supplied by the
next peak, contradicts the existence of infinitely many successive
minima. This proves Theorem  \ref{thm:lower-bound}.

\subsection*{Organization of the rest of the paper}

Section  \ref{sec:preliminaries} records the common normalization and
analytic preliminaries, and Section  \ref{sec:automatic-bound} proves the
critical upper bound. Section  \ref{sec:pohozaev-dichotomy} constructs the
successive maximum and minimum radii, while
Section  \ref{sec:euclidean-annuli} derives the exact Pohozaev increment on
the rescaled annuli. Section  \ref{sec:euclidean-response} constructs the
linear corrections and proves positivity of the associated quadratic
form. Sections  \ref{sec:reference-lower-bound} and
\ref{sec:exact-reference-comparison} estimate the Pohozaev contribution
of \(\Phi_j\) and compare it with that of the actual solution.
Section  \ref{sec:recurrence} proves the conditional lower bound and the
Fowler asymptotics. Appendix  \ref{app:quadratic-estimate} contains the
calculation on a finite ball used for the upper bound, and
Appendix  \ref{app:proofs} contains the coefficient calculations for the
annular estimates.

\section{Preliminaries}\label{sec:preliminaries}
\label{sec:common-preliminaries}

{
This section records the exact volume coordinate normalization and some of its properties, Euclidean bubbles, and distributional extension used in the two
estimates.

For a positive function \(\kappa\), conformal covariance takes the form
\[
 L_{\kappa^{4/(n-2)}g}(\kappa^{-1}v)=\kappa^{-p}L_gv.
\]
}

 Let \(K\) be a compact subset of the coordinate neighborhood,
 let \(0<\omega<1\). By G\"unther's
conformal normal coordinate construction \cite[Section  3]{Gunther}, there
exist \(R,C>0\) such that the following holds for every \(a\in K\). There
are a positive factor \(\kappa_a\) and a normal coordinate map \(\Theta_a\),
based on a frame at \(a\) that is orthonormal with respect to \(g_0\), for which
\(g_a=\Theta_a^*(\kappa_a^{4/(n-2)}g_0)\) is defined on \(B_R\) and
\begin{equation}\label{eq:conformal-normal-gauge}
\begin{aligned}
 &\kappa_a(a)=1,\qquad d\kappa_a(a)=0,
   \qquad C^{-1}\le\kappa_a\le C,\\
 &\det [(g_a)_{ij}] =1,\qquad g_a(0)=I,\qquad
   \sum_{i=1}^n(g_a)_{ij}(y)y^i=y^j\quad(1\le j\le n),\\
 &\|g_a\|_{C^{n+3,\omega}(B_R)}
   +\|g_a^{-1}\|_{C^{n+3,\omega}(B_R)}\le C,\qquad
 C^{-1}|y|\le d_{g_0}(a,\Theta_a(y))\le C|y|.
\end{aligned}
\end{equation}
The factors and coordinate maps have uniform \(C^{n+4,\omega}\) bounds in
the original charts. The construction uses the method of
Lee--Parker \cite[Theorem  5.1 and its proof]{LeeParker}, followed by
G\"unther's correction. This construction involves only finitely many derivatives of
\(g_0\), so the constants and bounds are uniform in \(a\), with dependence
only on \(K,\omega,n\) and the indicated metric bounds.

We next record some consequences of this conformal normalization.
In a normalized chart, write \(h=\log g\). The Gauss lemma and the normalization \(\eqref{eq:conformal-normal-gauge}\) give
\begin{equation}\label{eq:normal-gauge}
 g=e^h,\qquad 
 \sum_{i=1}^nh_{ii}(x)=0,\qquad
 \sum_{j=1}^nh_{ij}(x)x^j=0, \quad 1\le i\le n.
\end{equation}
The expansion in normal coordinates and volume normalization give
\begin{equation}\label{eq:normal-orders}
 h(0)=0,\qquad \partial h(0)=0,\qquad
 |h(x)|\le C|x|^2,\qquad |\operatorname{Scal}_g(x)|\le C|x|^2;
\end{equation}
see Lee--Parker \cite[Theorem  5.1]{LeeParker}. Since \(g^{-1}=e^{-h}\),
we also have
\[
 \sum_{i=1}^ng^{ij}(x)x_i=x^j,\qquad
 \Delta_g f(|x|)=f''(|x|)+\frac{n-1}{|x|}f'(|x|).
\]
For any $H\in \mathcal V_m$, differentiating the radial
identity in \eqref{eq:jet-gauge} and integrating by parts on \(B_1\) give
the following moment cancellations:
{
\[
 \int_{\Sph^{n-1}}\sum_{i,j=1}^n\partial_i\partial_jH_{ij}\,\dd\theta=0,
 \qquad
 \int_{\Sph^{n-1}}\theta_a
       \sum_{i,j=1}^n\partial_i\partial_jH_{ij}\,\dd\theta=0,
 \quad 1\le a\le n.
\]
}

\begin{lemma}\label{lem:radial-cancellation}
Let \(0<r<R\), and suppose that \(\overline{B_R}\) is contained in the
normalized coordinate neighborhood. For every
\(f\in C^2(\overline{A_{r,R}})\) and \(\zeta\in C^1([r,R])\),
\[
 \int_{\pa B_s}(\Delta_g-\Delta)f\,\dd S=0
 \quad\text{for every }s\in(r,R),\qquad
 \int_{A_{r,R}}\zeta(|x|)(\Delta_g-\Delta)f\,\dd x=0.
\]
\end{lemma}
\begin{proof}
Because \((g^{-1}-I)x=0\), the normal flux vanishes on both boundary
spheres. Integration by parts gives
\[
 \int_{A_{r,R}}\zeta(|x|)(\Delta_g-\Delta)f\,\dd x
 =-\sum_{i,j=1}^n\int_{A_{r,R}}
 \zeta'(|x|)\frac{x_i}{|x|}(g^{ij}-\delta^{ij})\partial_jf\,\dd x=0.
\]
Applying the coarea formula to the second identity with
\(\zeta\in C_c^1((r,R))\) gives the first identity by continuity.
\end{proof}

The standard bubbles in \eqref{eq:standard-bubble} satisfy
\(-\Delta U_{\lambda,b}=n(n-2)U_{\lambda,b}^p\).
The dilation kernel function is
{
\[
 Z_0=\frac{n-2}{2}U+\sum_{i=1}^ny^i\partial_iU
 =\frac{n-2}{2}\frac{1-|y|^2}{1+|y|^2}U.
\]
}
Thus \(Z_0>0\) in \(B_1\), \(Z_0=0\) on \(\partial B_1\), and
\(Z_0<0\) outside \(\overline{B_1}\).
For the kernel statement,  let \(\dot H^1(\mathbb R^n)\) be the homogeneous Sobolev completion of \(C_c^\infty(\mathbb R^n)\) under \(L^2\) convergence of gradients.
By Bianchi--Egnell \cite[Lemma  A1]{BianchiEgnell},
{
\[
 \ker_{\dot H^1(\R^n)}
 \bigl(-\Delta-n(n+2)U^{4/(n-2)}\bigr)
 =\operatorname{span}\{Z_0,\partial_1U,\ldots,\partial_nU\}.
\]
}
Every function in this kernel is uniquely determined by its value and
gradient at the origin, since \(Z_0(0)=(n-2)/2\),
\(\nabla Z_0(0)=0\), and
\(\partial_i\partial_jU(0)=-(n-2)\delta_{ij}\).

{We close with the distributional extension used for Green
representation on a rescaled ball containing the original puncture.}
\begin{lemma}\label{lem:no-dirac}
Let \(u>0\) solve \eqref{eq:yamabe} on \(B_1\setminus\{0\}\).
There exists \(\rho_0=\rho_0(g,u)\in(0,1)\) such that, for every
\(0<\rho<\rho_0\),
{
\[
 u,u^p\in L^1_{\mathrm{loc}}(B_\rho),\qquad
 -L_gu=n(n-2)u^p\quad\text{in }\mathcal D'(B_\rho).
\]
}
\end{lemma}
The classical argument of Caffarelli--Gidas--Spruck
\cite[Lemma  2.1]{CGS} applies to the present metric through
Dupaigne--Ponce \cite[Theorem  4]{DupaignePonce}. For
\(A>\|c(n)\operatorname{Scal}_g\|_{L^\infty(B_\rho)}\), the supersolution inequality
\((-\Delta_g+A)u\ge0\) extends across \(\{0\}\), a set of zero
\(H^1\) capacity. The resulting nonnegative measure decomposition gives
\(u,u^p\in L^1_{\mathrm{loc}}(B_\rho)\) and
\(-L_gu=n(n-2)u^p+\alpha\delta_0\) with \(\alpha\ge0\).
If \(\alpha>0\), the local Green lower estimate of Gr\"uter--Widman
\cite{GruterWidman} gives
\(u(x)\ge c\alpha d_g(x,0)^{2-n}\), contrary to
\(u^p\in L^1_{\mathrm{loc}}(B_\rho)\). Thus \(\alpha=0\).

\section{The critical upper bound}
\label{sec:automatic-bound}
\label{sec:upper-green-identity}
\label{sec:upper-local-estimate}

{
We prove the critical upper bound in
Theorem  \ref{thm:upper-bound} by contradiction. The corrections of
Khuri--Marques--Schoen give the lower bound for
\(\int_{B_{R_k}}Z_kE_k\,\dd y\) in
Lemma  \ref{lem:signed-residual}. Green's identity and the exterior
extension of \(Z_k\) then give the positive lower bound in
Lemma  \ref{lem:global-balance}. Iteration of the Green representation
gives an estimate in \(L^1\) for the \(w_k\) on a fixed ball, which is
incompatible with the quadratic estimate for the nonlinear remainder.

Suppose that Theorem  \ref{thm:upper-bound} fails.  By
Han--Xiong--Zhang \cite[Lemma  2.2]{HanXiongZhang} and the usual recentring
argument, there are local maximum points \(x_k\to0\) such that
\[
 \varepsilon_k:=u_0(x_k)^{-2/(n-2)}\to0,
 \qquad
 \frac{\varepsilon_k}{d_{g_0}(x_k,0)}\to0.
\]
At \(x_k\), apply the conformal normal coordinate construction from
Section  \ref{sec:preliminaries}.  Let \(C_{\mathrm g}\)
denote the resulting uniform metric bound.  Choose \(\kappa_k\) so that
\[
 \kappa_k(x_k)=1,\qquad \nabla\kappa_k(x_k)=0,
\]
and write
\[
 \widehat g^{(k)}=\kappa_k^{4/(n-2)}g_0,\qquad
 \widehat u_k=\kappa_k^{-1}u_0.
\]
Use conformal normal coordinates $y$ centered at \(x_k\) for
\(\widehat g^{(k)}\), so that \(\det(\widehat g^{(k)}_{ij})=1\).
For now, let \(0<\delta<1\), and set
\[
 R_k=\frac{\delta}{\varepsilon_k},\qquad
 (g_k)_{ij}(y)=(\widehat g^{(k)})_{ij}(\varepsilon_ky),\qquad
 v_k(y)=\varepsilon_k^{(n-2)/2}\widehat u_k(\varepsilon_ky).
\]
The value of \(\delta\) will be fixed in
Lemma  \ref{lem:adjoint-replacement-interface}.
Throughout this section, every \(o(\,\cdot\,)\) term is
taken as \(k\to\infty\), with the auxiliary parameters fixed in the order
specified below.
Then
\begin{equation}\label{eq:short-bubble-limit}
\begin{gathered}
 v_k(0)=1,\qquad \nabla v_k(0)=0,\\
 -L_{g_k}v_k=n(n-2)v_k^p
 \quad\text{in }\mathcal D'(B_{R_k}),\\
 v_k\to U=(1+|y|^2)^{-(n-2)/2}
 \quad\text{in }C^2_{\mathrm{loc}}(\R^n).
\end{gathered}
\end{equation}

In the normal coordinates $y$ centered at \(x_k\), put
\((h_k(y))_{ij}=\log[(\widehat g^{(k)})_{ij}(y)]\) and define
\[
 H_{k,ij}^{(m)}(y)
 =\sum_{|\alpha|=m}
 \frac{\partial^\alpha h_{k,ij}(0)}{\alpha!}y^\alpha.
\]
Then \(\log [(g_k)_{ij}(y)]=(h_k)_{ij}(\varepsilon_k y)\). We use the two truncations
\[
 H_k=\sum_{m=2}^dH_k^{(m)},\qquad
 \widehat H_k=\sum_{m=2}^{n-4}H_k^{(m)}.
\]
$H_k$ will be used to construct 
the Khuri--Marques--Schoen quadratic form \eqref{eq:truncated-pohozaev-form} and its weighted norm, while
$\widehat H_k$ will be used to construct the correction $z_{k,\lambda}$ below.

For \(4\le m\le n-4\), let
\(\psi_{m,1}(H_k^{(m)})\) be the normalized corrector in
\cite[Proposition  4.1]{KMS}. It satisfies
\begin{equation}\label{eq:upper-polynomial-correction}
 \bigl(\Delta+n(n+2)U^{4/(n-2)}\bigr)\psi_{m,1}(H_k^{(m)})
 =c(n)\sum_{i,j=1}^n\partial_i\partial_jH_{k,ij}^{(m)}U,
\end{equation}
with value and gradient zero at the origin. 
Define
\[
 \psi_{m,\lambda}(H_k^{(m)})(y)
 =\lambda^{m-(n-2)/2}\psi_{m,1}(H_k^{(m)})(y/\lambda),
 \qquad
 z_{k,\lambda}=\sum_{m=4}^{n-4}\varepsilon_k^m
 \psi_{m,\lambda}(H_k^{(m)}).
\]
For the normalized metric at \(x_k\), the function \(z_{k,1}\)
corresponds to the rescaled correction
\(\widetilde z_{\varepsilon_k}\) in \cite[Section  4]{KMS}.
Thus
\[
 \bigl(\Delta+n(n+2)U_\lambda^{4/(n-2)}\bigr)z_{k,\lambda}
 =c(n)\sum_{m=4}^{n-4}\sum_{i,j=1}^n
 \varepsilon_k^m\partial_i\partial_jH_{k,ij}^{(m)}U_\lambda.
\]

Set
\[
 \varphi_{k,\lambda}=U_\lambda+z_{k,\lambda},\qquad
 Z_k=-\lambda\partial_\lambda\varphi_{k,\lambda}\big|_{\lambda=1}
 =Z_0-\lambda\partial_\lambda z_{k,\lambda}\big|_{\lambda=1}.
\]
In the definition of \(Z_k\), the derivative with respect to
\(\lambda\) is taken with \(y\), \(\varepsilon_k\), and the tensors
\(H_k^{(m)}\) fixed. Define \(E_{k,\lambda}\) by
\begin{equation}\label{eq:short-profile-residual}
 E_{k,\lambda}=-L_{g_k}\varphi_{k,\lambda}
       -n(n-2)\varphi_{k,\lambda}^{p}.
\end{equation}

Write \(\varphi_k=\varphi_{k,1}\), \(E_k=E_{k,1}\), and set
\[
 V_k=p\,n(n-2)\varphi_k^{p-1},\qquad
 \mathcal L_k=-L_{g_k}-V_k,\qquad
 w_k=v_k-\varphi_k.
\]
Then
\begin{equation}\label{eq:short-error-equation}
 \mathcal L_kw_k=N_k-E_k,
\end{equation}
where
\begin{equation}\label{eq:short-nonlinear-remainder}
 N_k=n(n-2)\bigl(v_k^p-\varphi_k^p
        -p\varphi_k^{p-1}w_k\bigr).
\end{equation}
}

 \begin{lemma}\label{lem:profile-boundary-estimates}
There exists $\delta_1=\delta_1(n,C_{\mathrm g})>0$ such that, for every
$0<\delta\le\delta_1$ and all sufficiently large $k$,
{
\[
 \frac12U\le\varphi_k\le2U\quad\text{in }B_{R_k},
 \qquad N_k\ge0.
\]
}
 On $\partial B_{R_k}$,
{
\[
 c\varepsilon_k^{(n-2)/2}\le v_k\le
 C\varepsilon_k^{(n-2)/2},
 \qquad
 |\varphi_k|+|Z_k|\le C\varepsilon_k^{n-2}.
\]
}
 Here $c$ and $C$ depend only on $n$, $C_{\mathrm g}$, $\delta$,
$\|u_0\|_{C^0(A_{\delta/2,2\delta})}$, and
 $(\inf_{A_{\delta/2,2\delta}}u_0)^{-1}$.
\end{lemma}

\begin{proof}
 Fix $y\in B_{R_k}$ and write $r=|y|$. By
\cite[Proposition  4.1 and (4.4)]{KMS}, we have
{
\[
\begin{aligned}
 |\varphi_k(y)-U(y)|
 +\bigg|Z_k(y)-\frac{n-2}{2}\frac{1-r^2}{1+r^2}U(y)\bigg|
 \le{}C\sum_{m=4}^{n-4}\varepsilon_k^m
 |H_k^{(m)}|(1+r)^{m+2-n}
 \le C\delta^4U(y).
\end{aligned}
\]
}
 Decrease $\delta_1$, if necessary, so that
$C\delta_1^4\le1/2$.  Since $\delta\le\delta_1$, the first estimate in the
lemma follows.  The convexity of $t\mapsto t^p$ then gives $N_k\ge0$.
On $\partial B_{R_k}$, we have
$U\le C\varepsilon_k^{n-2}$; hence the displayed corrector estimate
gives $|\varphi_k|+|Z_k|\le C\varepsilon_k^{n-2}$.  On the corresponding
fixed physical annulus, the Harnack inequality gives
{
\[
 c\le\widehat u_k(\varepsilon_k y)\le C.
\]
}
 Multiplying this inequality by $\varepsilon_k^{(n-2)/2}$ proves the
boundary estimate for $v_k$.
\end{proof}

For $H=\sum_{m=2}^dH^{(m)}\in\mathcal V_{\le d}$, set
{
\[
 h_{\varepsilon,t}
 =t\sum_{m=2}^d\varepsilon^mH^{(m)},
 \qquad
 g_{\varepsilon,t}=\exp h_{\varepsilon,t},
 \qquad
 \varphi_{\lambda,\varepsilon,t}
 =U_\lambda+t\sum_{m=4}^d
 \varepsilon^m\psi_{m,\lambda}(H^{(m)}).
\]
}
 Define
{
\[
 Z_{\varepsilon,t}
 =-\lambda\partial_\lambda
 \varphi_{\lambda,\varepsilon,t}\big|_{\lambda=1},
 \qquad
 E_{\varepsilon,t}
 =-L_{g_{\varepsilon,t}}\varphi_{1,\varepsilon,t}
 -n(n-2)\varphi_{1,\varepsilon,t}^p.
\]
}
 The \emph{truncated Pohozaev quadratic form} is
{
\begin{equation}\label{eq:truncated-pohozaev-form}
 I_{\varepsilon,R}^{(n)}(H,H)
 :=\frac{1}{2c(n)}\frac{\dd^2}{\dd t^2}
 \bigg(\int_{B_R}Z_{\varepsilon,t}E_{\varepsilon,t}\dd y\bigg)
 \bigg|_{t=0}.
\end{equation}
}
Khuri--Marques--Schoen define the Pohozaev quadratic form on Euclidean space
$I_\varepsilon^{(n)}$ in \cite[Appendix  A]{KMS}.  The truncated form
\eqref{eq:truncated-pohozaev-form} is defined directly by the second
variation on $B_R$, with the boundary contributions retained.

 For every integer $m$, let $\theta_m=1$ if $m=(n-2)/2$, and let
$\theta_m=0$ otherwise.  For $\varepsilon>0$, $R\ge1$, and a finite
sum $K=\sum_mK^{(m)}$, write
$K_{\le r}=\sum_{m\le r}K^{(m)}$,
 where absent homogeneous components are understood to be zero,
and $K_{>r}=K-K_{\le r}$, and set
{
\begin{equation}\label{eq:weighted-jet-norm}
 \|K\|_{\varepsilon,R}^2
 =\sum_m\varepsilon^{2m}(1+\log R)^{\theta_m}|K^{(m)}|^2,
\end{equation}
}
 where the sum is over the degrees present in $K$.

\begin{lemma} \label{lem:signed-residual}
Let $7\le n\le29$. For the sufficiently small $\delta>0$ fixed in
Lemma  \ref{lem:adjoint-replacement-interface}, there is
$C_1=C_1(n)>0$ such that, for all sufficiently large $k$,
{
\[
 \int_{B_{R_k}}Z_kE_k\dd y
 \ge C_1\|H_{k,\le6}\|_{\varepsilon_k,R_k}^2
 -o\bigl(\varepsilon_k^{(n-2)/2}\bigr).
\]
}
\end{lemma}

 The proof will be given in
Appendix  \ref{app:quadratic-estimate} using
three estimates established there.  Proposition  \ref{prop:kms-algebraic-positivity}
gives positivity of $I_\varepsilon^{(n)}$.  For $7\le n\le24$, this is the positivity theorem of
Khuri--Marques--Schoen.  When $25\le n\le29$, their full form is
indefinite, and we use the restriction to degrees at most six
established by the explicit calculation in
Proposition  \ref{prop:kms-algebraic-positivity}.
Lemma  \ref{lem:finite-kms-bridge} compares
$I_{\varepsilon,R}^{(n)}$ with $I_\varepsilon^{(n)}$ when
$R=\delta/\varepsilon$, and Lemma  \ref{lem:actual-finite-comparison}
compares \(c(n)^{-1}\int_{B_{R_k}}Z_kE_k\dd y\) with
$I_{\varepsilon_k,R_k}^{(n)}(H_k,H_k)$.

We next establish the Green function estimates that will be applied to an extension of \(Z_k\)
outside a fixed ball and to control its Green representation on that ball.
For a domain $\Omega\subset B_{R_k}$, let $G_\Omega(x,y)$ be the
Dirichlet Green  function of the conformal Laplacian $-L_{g_k}$  on $\Omega$, and write
{
\[
 G_\Omega[f](x)=\int_\Omega G_\Omega(x,y)f(y)\dd y.
\]
}
\begin{lemma} \label{lem:global-green}
There exist $\delta_2=\delta_2(n,C_{\mathrm g})>0$ and $C=C(n,C_{\mathrm g})>1$ such that, for every
$0<\delta\le\delta_2$, all sufficiently large $k$,  and every $1\le R_0<R_k$,

\begin{enumerate}[label=\emph{(\roman*)}]
\item
$0<G_\Omega(x,y)\le C|x-y|^{2-n}$ for every
$\Omega\in\{B_{R_k},A_{R_0,R_k}\}$ and $x,y\in\Omega$ with $x\ne y$;
\item
$G_{B_{R_k}}(0,y)\ge C^{-1}\bigl(|y|^{2-n}-R_k^{2-n}\bigr)$ for every
$y\in B_{R_k}\setminus\{0\}$;
\item The boundary fluxes satisfy
\[
C^{-1}\le-\int_{\partial B_{R_k}}\partial_{r_\xi}G_{B_{R_k}}(0,\xi)\dd S(\xi)
\le\sup_{y\in B_{R_k}}\int_{\partial B_{R_k}}-\partial_{r_\xi}G_{B_{R_k}}(y,\xi)\dd S(\xi)\le C;
\]
\item For every distinct $x,y,z\in B_{R_k}$,
\[
\frac{G_{B_{R_k}}(z,x)G_{B_{R_k}}(x,y)}{G_{B_{R_k}}(z,y)}
\le C\bigl(|z-x|^{2-n}+|x-y|^{2-n}\bigr).
\]
\end{enumerate}
 \end{lemma}

\begin{proof}
 Let $\Gamma_k$ be the Dirichlet Green function of $-\Delta_{g_k}$ on
 $B_{R_k}$.   After rescaling $B_{R_k}$ to the unit ball, the coefficient
matrices remain uniformly elliptic.  By the Green function estimate of
Littman--Stampacchia--Weinberger \cite{LittmanStampacchiaWeinberger} and
Cranston--Fabes--Zhao \cite[Theorem  3.1]{CranstonFabesZhao}, we have
{
\[
\begin{aligned}
 0<\Gamma_k(x,y)&\le C|x-y|^{2-n},\qquad x\ne y,\\
 \frac{\Gamma_k(z,x)\Gamma_k(x,y)}{\Gamma_k(z,y)}
 &\le C\bigl(|z-x|^{2-n}+|x-y|^{2-n}\bigr),
 \qquad x,y,z\ \text{distinct}.
\end{aligned}
\]
}
 The constants are independent of $k$ because the rescaled domain is the unit
ball and the operators remain uniformly elliptic.

The conformal normal  expansion and
$R_k=\delta/\varepsilon_k$ give
{
\[
 \bigl|c(n)\operatorname{Scal}_{g_k}(y)\bigr|
 \le C\varepsilon_k^4|y|^2,
 \qquad
 R_k^2\bigl\|c(n)\operatorname{Scal}_{g_k}
 \bigr\|_{L^\infty(B_{R_k})}\le C\delta^4.
\]
}
Using these Green function estimates and the preceding
scalar curvature bound, we obtain
{
\[
\begin{aligned}
 \sup_{x\in B_{R_k}}\int_{B_{R_k}}
 \Gamma_k(x,z)\bigl|c(n)\operatorname{Scal}_{g_k}(z)\bigr|\dd z
 &\le C\delta^4,\\
 \sup_{x\ne y}\frac{1}{\Gamma_k(x,y)}
 \int_{B_{R_k}}\Gamma_k(x,z)
 \bigl|c(n)\operatorname{Scal}_{g_k}(z)\bigr|
 \Gamma_k(z,y)\dd z
 &\le C\delta^4.
\end{aligned}
\]
}
For  $m\ge0$, define
{
\[
 \Gamma_{0,k}=\Gamma_k,
 \qquad \Gamma_{m+1,k}(x,y)=\int_{B_{R_k}}\Gamma_{m,k}(x,z)
 c(n)\operatorname{Scal}_{g_k}(z)
 \Gamma_k(z,y)\dd z.
\]
}
 The second estimate gives, by induction,
{
\[
 |\Gamma_{m,k}(x,y)|\le(C\delta^4)^m\Gamma_k(x,y).
\]
}
 After decreasing $\delta_2$, assume that $C\delta_2^4<1/3$.  The series
{
\[
 \sum_{m=0}^{\infty}(-1)^m\Gamma_{m,k}(x,y)
\]
}
 then converges absolutely off the diagonal.  Moreover, the scaled
Poincar\'e inequality gives
{
\[
 \bigg|\int_{B_{R_k}}c(n)
 \operatorname{Scal}_{g_k}\eta^2\dd y\bigg|
 \le C\delta^4\int_{B_{R_k}}|\nabla\eta|_{g_k}^2\dd y,
 \qquad \eta\in H_0^1(B_{R_k}).
\]
}
Thus the Dirichlet realization of
$-L_{g_k}=-\Delta_{g_k}+c(n)\operatorname{Scal}_{g_k}$ is
coercive.  The recursion for $\Gamma_{m,k}$ gives
{
\[
{\left\{\begin{aligned}
 (-L_{g_k})_x\sum_{m=0}^{\infty}(-1)^m\Gamma_{m,k}(x,y)
 &=\delta_y&&\quad\text{in }B_{R_k},\\
 \sum_{m=0}^{\infty}(-1)^m\Gamma_{m,k}(x,y)
 &=0&&\quad\text{on }\partial B_{R_k}.
\end{aligned}\right.}
\]
}
Uniqueness identifies the series with $G_{B_{R_k}}$, and the geometric
bound gives
{
\[
 C^{-1}\Gamma_k(x,y)
 \le G_{B_{R_k}}(x,y)
 \le C\Gamma_k(x,y).
\]
}
This proves positivity and the Newtonian upper bound on $B_{R_k}$.
The same coercivity holds on every $A_{R_0,R_k}$; domain monotonicity
therefore gives
{
\[
 0<G_{A_{R_0,R_k}}(x,y)\le G_{B_{R_k}}(x,y),
\]
}
which completes \emph{(i)}.  The upper and lower kernel comparisons,
together with the second estimate above for $\Gamma_k$, prove \emph{(iv)}.

 By the Gauss lemma and $\det g_k=1$, the function
$\Gamma_k(0,\cdot)$ is radial and satisfies
{
\[
 \Gamma_k(0,y)=\frac{|y|^{2-n}-R_k^{2-n}}
 {(n-2)|\Sph^{n-1}|}.
\]
}
The lower kernel comparison proves \emph{(ii)}.  Both kernels vanish on
$\partial B_{R_k}$ and are $C^1$ up to the boundary away from the pole.
Divide the upper and lower comparison estimates by the distance to the boundary and take
the inward normal limit.  This gives
{
\[
 c\bigl(-\partial_{r_\xi}\Gamma_k(y,\xi)\bigr)
 \le-\partial_{r_\xi}G_{B_{R_k}}(y,\xi)
 \le C\bigl(-\partial_{r_\xi}\Gamma_k(y,\xi)\bigr)
\]
}
for $y\in B_{R_k}$ and $\xi\in\partial B_{R_k}$.  Green's identity gives
{
\[
 -\int_{\partial B_{R_k}}\partial_{r_\xi}
 \Gamma_k(y,\xi)\dd S(\xi)=1.
\]
}
Integration of the preceding comparison proves \emph{(iii)}.
\end{proof}

For $R_0\ge2$, define the exterior extension
$\widetilde Z_k$ of $Z_k$ by setting
{
\begin{equation}\label{eq:short-global-adjoint}
 {\left\{\begin{aligned}
  \widetilde Z_k&=Z_k&&\text{in }B_{R_0},\\
  \mathcal L_k\widetilde Z_k&=0&&\text{in }A_{R_0,R_k},\\
  \widetilde Z_k&=0&&\text{on }\partial B_{R_k}.
 \end{aligned}\right.}
\end{equation}
}
 \begin{lemma} \label{lem:exterior-corrected-adjoint}
There exists $R_*=R_*(n,C_{\mathrm g})>1$ such that,
for every $R_0\ge R_*$, there are
$\delta_3=\delta_3(n,R_0,C_{\mathrm g})>0$ and
$C=C(n,R_0,C_{\mathrm g})>1$ such that, for every
$0<\delta\le\delta_3$ and all sufficiently large $k$, the boundary value
problem  \eqref{eq:short-global-adjoint} has a unique solution satisfying
\begin{enumerate}
\item[\emph{(i)}]
$\widetilde Z_k<0$ in $A_{R_0,R_k}$, and
{
\[
 C^{-1}G_{B_{R_k}}(0,y)\le-\widetilde Z_k(y)
 \le CG_{B_{R_k}}(0,y),\qquad 2R_0\le |y|<R_k.
\]
}
 \item[\emph{(ii)}]
The outer flux satisfies
{
\[
 C^{-1}\le\int_{\partial B_{R_k}}
 \partial_r\widetilde Z_k\dd S\le C.
\]
}
\end{enumerate}
\end{lemma}

\begin{proof}
 By the corrector estimate in the proof of
Lemma  \ref{lem:profile-boundary-estimates}, the correction in $Z_k$
converges uniformly to zero on $\partial B_{R_0}$.  Hence
{
\[
 -Z_k(y)\to
 \frac{n-2}{2}\frac{|y|^2-1}{1+|y|^2}U(y),
 \qquad y\in\partial B_{R_0}.
\]
}
 For every fixed $R_0\ge2$, the limit is comparable to $R_0^{2-n}$.
Thus, for all sufficiently large $k$,
{
\[
 cR_0^{2-n}\le -Z_k\le CR_0^{2-n}
 \qquad\text{on }\partial B_{R_0}.
\]
}
 The conformal normal expansion and $R_k=\delta/\varepsilon_k$ give
{
\[
 R_k^2\Big\|c(n)\varepsilon_k^2
 \operatorname{Scal}_{\widehat g^{(k)}}(\varepsilon_k\,\cdot)
 \Big\|_{L^\infty(B_{R_k})}
 \le C\delta^4.
\]
}
For any $\eta\in H_0^1(A_{R_0,R_k})$, the Hardy inequality and
Lemma  \ref{lem:profile-boundary-estimates} give
{
\[
\begin{aligned}
 \int_{A_{R_0,R_k}}\eta\mathcal L_k\eta\dd y
 &\ge \int_{A_{R_0,R_k}}|\nabla\eta|_{g_k}^2\dd y
 -\int_{A_{R_0,R_k}}V_k\eta^2\dd y
 -\Big|\int_{A_{R_0,R_k}}c(n)\varepsilon_k^2
 \operatorname{Scal}_{\widehat g^{(k)}}(\varepsilon_k y)
 \eta^2\dd y\Big|\\
 &\ge\bigl\{1-C(R_0^{-2}+\delta^4)\bigr\}
 \int_{A_{R_0,R_k}}|\nabla\eta|_{g_k}^2\dd y\\
 &\ge\frac12\int_{A_{R_0,R_k}}|\nabla\eta|_{g_k}^2\dd y,
\end{aligned}
\]
}
Choose $R_*$ so large that the $R_0^{-2}$ contributions
in the coercivity estimate above and the estimate for the Green potential below are at
most $1/4$ whenever $R_0\ge R_*$.  Then choose $\delta_3$ so that
$C\delta_3^4\le1/4$.
{Coercivity gives unique solvability of
\eqref{eq:short-global-adjoint}, and the comparison principle for
\(\mathcal L_k\) holds on \(A_{R_0,R_k}\).}  Hence
{
\[
 \widetilde Z_k<0\qquad\text{in }A_{R_0,R_k}.
\]
}

 For all sufficiently large $k$, Lemma  \ref{lem:global-green}\emph{(i),
(ii)} and the inner boundary estimate give
{
\[
 cG_{B_{R_k}}(0,\cdot)
 \le -Z_k
 \le CG_{B_{R_k}}(0,\cdot)
 \qquad\text{on }\partial B_{R_0}.
\]
}
Since $\widetilde Z_k=Z_k$ on $\partial B_{R_0}$ and
both $\widetilde Z_k$ and $G_{B_{R_k}}(0,\cdot)$ vanish on
$\partial B_{R_k}$, we have
{
\[
\begin{aligned}
  \mathcal L_k\bigl(-\widetilde Z_k-cG_{B_{R_k}}(0,\cdot)\bigr)
 &=cV_kG_{B_{R_k}}(0,\cdot)\ge0
 &&\text{in }A_{R_0,R_k}, \\
 -\widetilde Z_k-cG_{B_{R_k}}(0,\cdot)&\ge0
 &&\text{on }\partial A_{R_0,R_k}.
\end{aligned}
\]
}
The comparison principle gives
{
\[
 -\widetilde Z_k(y)\ge cG_{B_{R_k}}(0,y)
 \qquad\text{for }y\in A_{R_0,R_k}.
\]
}
For the upper bound, domain monotonicity, Lemma  \ref{lem:global-green}\emph{(iv)},
and $V_k(z)\le C(1+|z|)^{-4}$ imply
{
\[
 \frac{G_{A_{R_0,R_k}}[V_kG_{B_{R_k}}(0,\cdot)](y)}
 {G_{B_{R_k}}(0,y)}
 \le C\int_{\R^n\setminus B_{R_0}}
 \bigl(|y-z|^{2-n}+|z|^{2-n}\bigr)(1+|z|)^{-4}\dd z
 \le CR_0^{-2}\le\frac12.
\]
}
It follows that
{
\[
 \mathcal L_k\Bigl(G_{B_{R_k}}(0,\cdot)
 +2G_{A_{R_0,R_k}}[V_kG_{B_{R_k}}(0,\cdot)]\Bigr)
 =V_k\Bigl(G_{B_{R_k}}(0,\cdot)
 -2G_{A_{R_0,R_k}}[V_kG_{B_{R_k}}(0,\cdot)]\Bigr)\ge0.
\]
}
After increasing $C$ if necessary, the boundary comparison gives
{
\[
 C\Bigl(G_{B_{R_k}}(0,\cdot)
 +2G_{A_{R_0,R_k}}[V_kG_{B_{R_k}}(0,\cdot)]\Bigr)
 +\widetilde Z_k\ge0
 \qquad\text{on }\partial A_{R_0,R_k}.
\]
}
The comparison principle gives
{
\[
 -\widetilde Z_k(y)
 \le C\Bigl(G_{B_{R_k}}(0,y)
 +2G_{A_{R_0,R_k}}[V_kG_{B_{R_k}}(0,\cdot)](y)\Bigr)
 \le CG_{B_{R_k}}(0,y).
\]
}
This proves \emph{(i)}.
The functions $-\widetilde Z_k$ and $G_{B_{R_k}}(0,\cdot)$ vanish on
$\partial B_{R_k}$.  By the upper and lower comparison estimates and the  boundary point
lemma,
{
\[
 c\bigl(-\partial_rG_{B_{R_k}}(0,\xi)\bigr)
 \le\partial_r\widetilde Z_k(\xi)
 \le C\bigl(-\partial_rG_{B_{R_k}}(0,\xi)\bigr).
\]
}
 By Lemma  \ref{lem:global-green}\emph{(iii)}, integration over
$\partial B_{R_k}$ proves \emph{(ii)}.
\end{proof}

 We next estimate $E_k$ and $\mathcal L_kZ_k$.  These bounds will be used to
compare $Z_k$ with $\widetilde Z_k$ and in the argument on a fixed ball.
\begin{lemma} \label{lem:profile-error-estimate}
There are $\delta_4=\delta_4(n,C_{\mathrm g})>0$ and
$C=C(n,C_{\mathrm g})>0$ such that, for every $0<\delta\le\delta_4$ and all
sufficiently large $k$,
 {
\begin{equation}\label{eq:pointwise-profile-error}
\begin{aligned}
 |E_k(y)|+|\mathcal L_kZ_k(y)|
 &\le C\sum_{m=2}^d\varepsilon_k^{2m}|H_k^{(m)}|^2
       (1+|y|)^{2m-n}\\
 &\quad+C\varepsilon_k^{n-3}(1+|y|)^{-3}
  +C\varepsilon_k^{n-2}(1+|y|)^{-2},
 \qquad y\in B_{R_k}.
\end{aligned}
\end{equation}
}
 \end{lemma}

\begin{proof}
 Fix $\lambda$ in a compact neighborhood of $1$. With $h_k$ and
$\widehat H_k$ defined above, the Taylor expansion at $x=0$ has the form
{
\[
 h_{k,ij}=\widehat H_{k,ij}+H_{k,ij}^{(n-3)}+T_{k,ij},
 \qquad
 |D^aT_{k,ij}(x)|\le C|x|^{n-2-a},
 \quad 0\le a\le2.
\]
}
 Since $\det g_k=1$, we have
{
\[
 \Delta_{g_k}f
 =\sum_{i,j=1}^ng_k^{ij}\partial_i\partial_jf
  +\sum_{i,j=1}^n(\partial_i g_k^{ij})\partial_jf.
\]
}
 The radial identity $\Delta_{g_k}U_\lambda=\Delta U_\lambda$ and the
corrector equations therefore give the exact decomposition
{
\begin{equation*}
\begin{aligned}
 E_{k,\lambda}
={}&c(n)\Bigg\{\operatorname{Scal}_{g_k}
 -\sum_{m=4}^{n-4}\sum_{i,j=1}^n
   \varepsilon_k^m\partial_i\partial_jH_{k,ij}^{(m)}\Bigg\}U_\lambda\\
&+\sum_{i,j=1}^n(\delta_{ij}-g_k^{ij})
       \partial_i\partial_jz_{k,\lambda}
 -\sum_{i,j=1}^n(\partial_i g_k^{ij})
       \partial_jz_{k,\lambda}
 +c(n)\operatorname{Scal}_{g_k}z_{k,\lambda}\\
&-n(n-2)\Big\{(U_\lambda+z_{k,\lambda})^p-U_\lambda^p
       -pU_\lambda^{p-1}z_{k,\lambda}\Big\}.
\end{aligned}
\end{equation*}
}
By \eqref{eq:appendix-moments}, the terms of degrees two and three vanish:
{
\[
 \sum_{i,j=1}^n\partial_i\partial_jH_{k,ij}^{(m)}=0,
 \qquad m=2,3,
\]
}
 Thus the expression in braces in the
first line equals
{
\[
 \varepsilon_k^2\Bigg\{\operatorname{Scal}_{\widehat g^{(k)}}(\varepsilon_ky)
 -\sum_{m=2}^{n-4}\sum_{i,j=1}^n
  \partial_i\partial_jH_{k,ij}^{(m)}(\varepsilon_ky)\Bigg\}.
\]
}
 By Proposition  4.3 of Khuri--Marques--Schoen \cite{KMS} and Taylor's
formula,
{
\begin{equation*}
\begin{aligned}
&\Bigg|c(n)\Bigg\{\operatorname{Scal}_{g_k}
 -\sum_{m=4}^{n-4}\sum_{i,j=1}^n
 \varepsilon_k^m\partial_i\partial_jH_{k,ij}^{(m)}\Bigg\}
 U_\lambda\Bigg|\\
\le{}&C\sum_{s,t=2}^{n-4}
 \varepsilon_k^{s+t}|H_k^{(s)}|\,|H_k^{(t)}|
 (1+|y|)^{s+t-n}
 +C\varepsilon_k^{n-3}(1+|y|)^{-3}
 +C\varepsilon_k^{n-2}(1+|y|)^{-2}.
\end{aligned}
\end{equation*}
}
In this estimate, the first term is the quadratic part of the scalar
curvature expansion.  The last two terms come from
$H_k^{(n-3)}$ and $T_k$, respectively.  We have also used
$\varepsilon_k(1+|y|)\le1+\delta$ on $B_{R_k}$.
By the polynomial representation in Proposition  4.1 of
Khuri--Marques--Schoen \cite{KMS}  and its estimate  (4.4),
{
\[
 |D^az_{k,\lambda}(y)|
 \le C\sum_{t=4}^{n-4}\varepsilon_k^t|H_k^{(t)}|
 (1+|y|)^{t+2-n-a},
 \qquad 0\le a\le2.
\]
}
The Taylor expansion of $g_k^{-1}$ and the preceding curvature estimate
now give
{
\begin{equation*}
\begin{aligned}
&\Bigg|
 \sum_{i,j=1}^n(\delta_{ij}-g_k^{ij})
       \partial_i\partial_jz_{k,\lambda}
 -\sum_{i,j=1}^n(\partial_i g_k^{ij})
       \partial_jz_{k,\lambda}
 +c(n)\operatorname{Scal}_{g_k}z_{k,\lambda}\Bigg|\\
\le{}&C\sum_{s=2}^{n-4}\sum_{t=4}^{n-4}
 \varepsilon_k^{s+t}|H_k^{(s)}|\,|H_k^{(t)}|
 (1+|y|)^{s+t-n}
+C\varepsilon_k^{n-2}(1+|y|)^{-2}.
\end{aligned}
\end{equation*}
}
After decreasing $\delta_4$, if necessary,
$\frac{1}{2}U_\lambda\le U_\lambda+z_{k,\lambda}\le2U_\lambda$.
Taylor's formula for the power nonlinearity
and the corrector estimate yield
{
\begin{equation*}
\big|(U_\lambda+z_{k,\lambda})^p-U_\lambda^p
 -pU_\lambda^{p-1}z_{k,\lambda}\big|
\le{}C\sum_{s,t=4}^{n-4}
 \varepsilon_k^{s+t}|H_k^{(s)}|\,|H_k^{(t)}|
 (1+|y|)^{s+t-n}.
\end{equation*}
}
Combining the last three estimates, we obtain
{
\begin{equation*}
 |E_{k,\lambda}(y)|
\le{}C\sum_{s,t=2}^{n-4}
 \varepsilon_k^{s+t}|H_k^{(s)}|\,|H_k^{(t)}|
 (1+|y|)^{s+t-n}
 +C\varepsilon_k^{n-3}(1+|y|)^{-3}
 +C\varepsilon_k^{n-2}(1+|y|)^{-2}.
\end{equation*}
}
For  $s,t\le d$, Young's inequality gives
{
\[
\varepsilon_k^{s+t}|H_k^{(s)}|\,|H_k^{(t)}|
 (1+|y|)^{s+t-n}
\le{}\frac12\varepsilon_k^{2s}|H_k^{(s)}|^2
 (1+|y|)^{2s-n}
 +\frac12\varepsilon_k^{2t}|H_k^{(t)}|^2
 (1+|y|)^{2t-n}.
\]
}
 If $t>d$, then
{
\[
 \varepsilon_k^{2t}(1+|y|)^{2t-n}
 =\varepsilon_k^{n-2}(1+|y|)^{-2}
  \{\varepsilon_k(1+|y|)\}^{2t-n+2}
 \le C\varepsilon_k^{n-2}(1+|y|)^{-2}.
\]
}
If $\max\{s,t\}>d$, Young's inequality and the preceding estimate give the
required bound.  We have therefore proved the asserted bound for
$E_{k,\lambda}$, uniformly for $\lambda$ near $1$.

 Finally, $g_k$ is independent of $\lambda$, and direct differentiation of
\eqref{eq:short-profile-residual} gives
{
\[
 \mathcal L_kZ_k
 =-\lambda\partial_\lambda E_{k,\lambda}
  \big|_{\lambda=1}.
\]
}
 Differentiation with respect to $\lambda$ changes only the constants in the
preceding pointwise bounds; the powers $(1+|y|)^{m+2-n-a}$ remain unchanged.
Applying $-\lambda\partial_\lambda$ to the exact decomposition of
$E_{k,\lambda}$ proves the bound for $\mathcal L_kZ_k$.
\end{proof}

 We define the \emph{jump of the normal derivative} $\partial_\nu\widetilde Z_k$ across the interface $\partial B_{R_0}$ by
{
\[
 \llbracket\partial_\nu\widetilde Z_k\rrbracket_{\partial B_{R_0}}
 :=\partial_rZ_k-
 \partial_r\widetilde Z_k\big|_{A_{R_0,R_k}}.
\]
}
 We next compare the pairings defined by $Z_k$ and $\widetilde Z_k$ and
estimate the two terms supported on $B_{R_0}$ and $\partial B_{R_0}$.
\begin{lemma}\label{lem:adjoint-replacement-interface}
For every $s>n$ and $\vartheta>0$, there exist
$R_0=R_0(n,C_{\mathrm g})\ge R_*$,
$\delta_0=\delta_0(n,R_0,C_{\mathrm g})>0$, and
$C=C(n,s,R_0,C_{\mathrm g})>0$ such that, for every
$0<\delta\le\delta_0$ and all sufficiently large $k$,
{
\[
\begin{aligned}
 \Big|\int_{B_{R_k}}(\widetilde Z_k-Z_k)E_k\dd y\Big|
 &\le\vartheta\|H_{k,\le6}\|_{\varepsilon_k,R_k}^2
 +o\bigl(\varepsilon_k^{(n-2)/2}\bigr),\\
 \|\mathcal L_kZ_k\|_{L^s(B_{R_0})}
 +\Big\|\llbracket\partial_\nu\widetilde Z_k\rrbracket_{\partial B_{R_0}}
  \Big\|_{C^0(\partial B_{R_0})}
 &\le C\|H_{k,\le6}\|_{\varepsilon_k,R_k}^2
 +o\bigl(\varepsilon_k^{(n-2)/2}\bigr).
\end{aligned}
\]
}
 \end{lemma}

\begin{proof}
 Put $\Theta_k=\widetilde Z_k-Z_k$ on  $A_{R_0,R_k}$.  Since
$\widetilde Z_k=Z_k$ in $B_{R_0}$,
{
\[
 \int_{B_{R_k}}(\widetilde Z_k-Z_k)E_k\dd y
 =\int_{A_{R_0,R_k}}\Theta_kE_k\dd y.
\]
}
 Differentiating \eqref{eq:short-profile-residual} in the scale parameter
and subtracting the equations for $\widetilde Z_k$ and $Z_k$, we obtain
{
\[
{\left\{\begin{aligned}
 -L_{g_k}\Theta_k
 &=p\,n(n-2)\varphi_k^{p-1}\Theta_k-\mathcal L_kZ_k
 &&\quad\text{in }A_{R_0,R_k},\\
 \Theta_k&=0&&\quad\text{on }\partial B_{R_0},\\
 \Theta_k&=-Z_k&&\quad\text{on }\partial B_{R_k}.
\end{aligned}\right.}
\]
}
 For this proof, set
{
\[
 \mathcal T_kF:=G_{A_{R_0,R_k}}[V_kF].
\]
}
 Green representation for the preceding boundary value problem gives,
for $x\in A_{R_0,R_k}$,
{
\[
 \Theta_k(x)
 =\int_{\partial B_{R_k}}Z_k(\xi)
   \partial_{\nu_\xi}G_{A_{R_0,R_k}}(x,\xi)\dd S(\xi)
   -G_{A_{R_0,R_k}}[\mathcal L_kZ_k](x)
 +\mathcal T_k\Theta_k(x).
\]
}
 Iterating the exact Green representation $N+1$ times gives
{
\[
 \Theta_k
 =\sum_{j=0}^{N}
 \mathcal T_k^j
 \Bigl(
   \int_{\partial B_{R_k}}Z_k(\xi)
   \partial_{\nu_\xi}G_{A_{R_0,R_k}}(\,\cdot\,,\xi)\dd S(\xi)
   -G_{A_{R_0,R_k}}[\mathcal L_kZ_k]
 \Bigr)
 +
 \mathcal T_k^{N+1}\Theta_k.
\]
}
By Lemma  \ref{lem:profile-boundary-estimates},
$V_k(z)\le C(1+|z|)^{-4}$. Lemma  \ref{lem:global-green}\emph{(iv)}
therefore gives, for \(x,y\in A_{R_0,R_k}\),
{
\[
\begin{aligned}
 &\int_{A_{R_0,R_k}}
 G_{B_{R_k}}(x,z)V_k(z)G_{B_{R_k}}(z,y)\,\dd z\\
 \le&\,
 CG_{B_{R_k}}(x,y)
 \int_{\R^n\setminus B_{R_0}}
 \bigl(|x-z|^{2-n}+|z-y|^{2-n}\bigr)
 (1+|z|)^{-4}\,\dd z\\
 \le&\, CR_0^{-2}G_{B_{R_k}}(x,y).
\end{aligned}
\]
}
Choose \(R_0\ge R_*\) so that \(CR_0^{-2}\le1/2\). Then choose
\(\delta_0>0\) so that
\[
 \delta_0\le\min\{\delta_1,\delta_2,\delta_3,\delta_4\}
\]
and Lemma  \ref{lem:actual-finite-comparison} applies for every
\(0<\delta\le\delta_0\), with \(\vartheta=b_n/8\) when
\(7\le n\le24\) and with \(\vartheta=b_n^{(6)}/8\) when
\(25\le n\le29\). The constants \(b_n\) and \(b_n^{(6)}\) are given in
Proposition  \ref{prop:kms-algebraic-positivity}. Fix such a \(\delta\)
for the remainder of this section.

By domain monotonicity, for every \(F\ge0\),
\[
 \begin{aligned}
 G_{A_{R_0,R_k}}[F]
 &\le G_{B_{R_k}}[F\mathbf 1_{A_{R_0,R_k}}],\\
 \mathcal T_k\bigl(G_{B_{R_k}}
 [F\mathbf 1_{A_{R_0,R_k}}]\bigr)
 &\le\frac12G_{B_{R_k}}
 [F\mathbf 1_{A_{R_0,R_k}}].
 \end{aligned}
\]
The Newtonian bound gives
\[
 \mathcal T_k1(x)
 \le C\int_{\R^n\setminus B_{R_0}}
 |x-z|^{2-n}(1+|z|)^{-4}\,\dd z
 \le CR_0^{-2}\le\frac12.
\]
Hence \(\|\mathcal T_k\|_{L^\infty\to L^\infty}\le1/2\).

The first term in the Green representation solves the homogeneous
equation for \(-L_{g_k}\), with boundary values \(0\) on
$\partial B_{R_0}$ and  $-Z_k$ on $\partial B_{R_k}$.  Lemmas
\ref{lem:global-green}\emph{(iii)} and
\ref{lem:profile-boundary-estimates} give the bound
$C\varepsilon_k^{n-2}$. Positivity of the Green kernels and the preceding
two estimates give
{
\[
 \begin{aligned}
 |\Theta_k(x)|
 \le C\varepsilon_k^{n-2}
 +C G_{B_{R_k}}
 [|\mathcal L_kZ_k|\mathbf 1_{A_{R_0,R_k}}](x)+2^{-N-1}
 \|\Theta_k\|_{L^\infty(A_{R_0,R_k})}.
 \end{aligned}
\]
}
For fixed \(k\), let \(N\to\infty\). Multiplication by \(|E_k|\) and
integration give
{
\[
 \int_{A_{R_0,R_k}}|\Theta_kE_k|\dd x
 \le C\varepsilon_k^{n-2}\int_{A_{R_0,R_k}}|E_k|\dd x
 +C\int_{A_{R_0,R_k}}|E_k|
 G_{B_{R_k}}
 [|\mathcal L_kZ_k|\mathbf 1_{A_{R_0,R_k}}]\dd x.
\]
}
 By \eqref{eq:pointwise-profile-error}, the first term on the right is
$O(\varepsilon_k^{n-2})$.  To estimate the second term, the Newtonian
bound and a shell decomposition give, for  $0<a,b\le n-2$,
{
\[
\begin{aligned}
 \varepsilon_k^{a+b}\iint_{A_{R_0,R_k}^2}
 \frac{|x-y|^{2-n}\dd x\dd y}
 {(1+|x|)^{n-a}(1+|y|)^{n-b}}
 \le C\varepsilon_k^{a+b}
 \begin{cases}
  1,&a+b<n-2,\\
  1+\log R_k,&a+b=n-2,\\
  R_k^{a+b+2-n},&a+b>n-2.
 \end{cases}
\end{aligned}
\]
}
 Insert \eqref{eq:pointwise-profile-error}, split the  jet sum at degree
six, and apply Young's inequality.  It follows that
{
\[
\begin{aligned}
 \int_{A_{R_0,R_k}}|\Theta_kE_k|\dd x
 &\le C\|H_{k,\le6}\|_{\varepsilon_k,R_k}^4
 +\frac{\vartheta}{2}\|H_{k,\le6}\|_{\varepsilon_k,R_k}^2\\
 &\quad+C_\vartheta\|H_{k,>6}\|_{\varepsilon_k,R_k}^2
 +C\varepsilon_k^{\min\{n-3,14\}}
 (1+|\log\varepsilon_k|).
\end{aligned}
\]
}
The logarithmic weight occurs precisely when $a+b=n-2$.  The Taylor
coefficient bounds and Young's inequality give
{
\[
 \|H_{k,>6}\|_{\varepsilon_k,R_k}^2
 \le C\varepsilon_k^{14}(1+|\log\varepsilon_k|)
 =o\bigl(\varepsilon_k^{(n-2)/2}\bigr),
\]
}
where the left side is understood as zero when $d\le6$.
$\|H_{k,\le6}\|_{\varepsilon_k,R_k}\to0$ and
$\min\{n-3,14\}>(n-2)/2$ for  $7\le n\le29$.  Thus, for all sufficiently
large $k$, the preceding two displays give the first estimate.

It remains to estimate the interface terms.  By \eqref{eq:pointwise-profile-error},
{
\[
 \|\mathcal L_kZ_k\|_{L^s(B_{R_0+2})}
 \le C\|H_{k,\le6}\|_{\varepsilon_k,R_k}^2
 +o\bigl(\varepsilon_k^{(n-2)/2}\bigr).
\]
}
 For the exterior source, Lemma  \ref{lem:global-green}\emph{(i)} and
\eqref{eq:pointwise-profile-error} give
{
\[
\begin{aligned}
&G_{B_{R_k}}
 [|\mathcal L_kZ_k|\mathbf1_{A_{R_0+2,R_k}}](0)\\
 \le &\,C\sum_{m=2}^d\varepsilon_k^{2m}|H_k^{(m)}|^2
 \bigg(1+\int_{R_0+2}^{R_k}r^{2m+1-n}\dd r\bigg)
 +C\varepsilon_k^{n-3}
 +C\varepsilon_k^{n-2}(1+\log R_k)\\
 \le &\,C\|H_{k,\le6}\|_{\varepsilon_k,R_k}^2
 +C\varepsilon_k^{14}(1+|\log\varepsilon_k|)
 +C\varepsilon_k^{n-3}
 +C\varepsilon_k^{n-2}(1+|\log\varepsilon_k|)\\
 \le &\,C\|H_{k,\le6}\|_{\varepsilon_k,R_k}^2
 +o\bigl(\varepsilon_k^{(n-2)/2}\bigr).
\end{aligned}
\]
}
 Fix $y\in A_{R_0,R_0+1}$ and regard both kernels below as functions of
$z\in A_{R_0+2,R_k}$.  Symmetry, domain monotonicity, and the maximum
principle give
{
\[
 G_{A_{R_0,R_k}}(y,z)\le CG_{B_{R_k}}(0,z).
\]
}
 On $\partial B_{R_0+2}$, the required boundary inequality follows from
Lemma  \ref{lem:global-green}\emph{(i), (ii)}; both sides vanish on
$\partial B_{R_k}$.  The finite Green iteration, the preceding comparison,
and the contraction estimate therefore yield
{
\[
\begin{aligned}
 \|\Theta_k\|_{L^\infty(A_{R_0,R_0+1})}
 &\le C\varepsilon_k^{n-2}
 +C\|\mathcal L_kZ_k\|_{L^s(B_{R_0+2})}+C\,G_{B_{R_k}}
 [|\mathcal L_kZ_k|\mathbf1_{A_{R_0+2,R_k}}](0)\\
 &\le C\|H_{k,\le6}\|_{\varepsilon_k,R_k}^2
 +o\bigl(\varepsilon_k^{(n-2)/2}\bigr).
\end{aligned}
\]
}
 The constants in these boundary estimates are uniform
in $k$.  On the enlarged collar $A_{R_0,R_0+1}$, $\Theta_k$ satisfies
{
\[
 -L_{g_k}\Theta_k=V_k\Theta_k-\mathcal L_kZ_k
 \quad\text{in }A_{R_0,R_0+1},
 \qquad
 \Theta_k=0\quad\text{on }\partial B_{R_0},
\]
}
 The boundary $W^{2,s}$ estimate
on $A_{R_0,R_0+1/2}$ and Morrey's inequality for $s>n$ give
{
\[
\begin{aligned}
 \|\partial_r\Theta_k\|_{C^0(\partial B_{R_0})}
 &\le C\bigl(
 \|\Theta_k\|_{L^\infty(A_{R_0,R_0+1})}
 +\|\mathcal L_kZ_k\|_{L^s(B_{R_0+2})}\bigr)\\
 &\le C\|H_{k,\le6}\|_{\varepsilon_k,R_k}^2
 +o\bigl(\varepsilon_k^{(n-2)/2}\bigr).
\end{aligned}
\]
}
 Because
$\llbracket\partial_\nu\widetilde Z_k\rrbracket_{\partial B_{R_0}}
=-\partial_r\Theta_k$, the preceding boundary derivative and
$L^s$ estimates imply
{
\[
 \|\mathcal L_kZ_k\|_{L^s(B_{R_0})}
 +\big\|\llbracket\partial_\nu\widetilde Z_k\rrbracket_{\partial B_{R_0}}
 \big\|_{C^0(\partial B_{R_0})}
 \le C\|H_{k,\le6}\|_{\varepsilon_k,R_k}^2
 +o\bigl(\varepsilon_k^{(n-2)/2}\bigr).
\]
}
 \end{proof}

 With
$f_+=\max\{f,0\}$ and  $f_-=\max\{-f,0\}$, define
{
\[
 \mathcal I_k^+=\int_{B_{R_k}}(\widetilde Z_k)_+N_k\dd y,
 \qquad
 \mathcal I_k^-=\int_{B_{R_k}}(\widetilde Z_k)_-N_k\dd y,
\]
}
where $\widetilde Z_k$ is defined in \eqref{eq:short-global-adjoint} and $N_k$ is the nonlinear remainder in \eqref{eq:short-nonlinear-remainder}.
We now derive the signed lower bound needed for the argument on a fixed ball.
\begin{lemma}\label{lem:global-balance}
There is $C_2=C_2(n,g_0,u_0)>0$ such that, for all sufficiently large $k$,
{
\[
 \mathcal I_k^+-\mathcal I_k^-
 \ge C_2\Bigl(
 \varepsilon_k^{(n-2)/2}
 +\|H_{k,\le6}\|_{\varepsilon_k,R_k}^2\Bigr).
\]
}
\end{lemma}

 \begin{proof}
The function $\widetilde Z_k$ is continuous across
$\partial B_{R_0}$.  By integration by parts on the two sides of this
interface, we have
{
\[
 \mathcal L_k\widetilde Z_k
 =(\mathcal L_kZ_k)\mathbf1_{B_{R_0}}
 +\llbracket\partial_\nu\widetilde Z_k\rrbracket_{\partial B_{R_0}}
  \delta_{\partial B_{R_0}}
  \quad\text{in }\mathcal D'(B_{R_k}).
\]
}
Combining this identity with \eqref{eq:short-error-equation} and
Green's second identity, we obtain
{
\begin{equation}\label{eq:global-four-term-identity}
\begin{aligned}
 \mathcal I_k^+-\mathcal I_k^-&=
 \int_{\partial B_{R_k}}w_k\partial_r\widetilde Z_k\dd S
 +\int_{B_{R_k}}\widetilde Z_kE_k\dd y\\
 &\quad+\int_{B_{R_0}}(\mathcal L_kZ_k)w_k\dd y
 +\int_{\partial B_{R_0}}
 \llbracket\partial_\nu\widetilde Z_k\rrbracket_{\partial B_{R_0}}
 w_k\dd S.
\end{aligned}
\end{equation}
}
By Lemmas  \ref{lem:profile-boundary-estimates} and \ref{lem:exterior-corrected-adjoint}, we have
{
\[
 \int_{\partial B_{R_k}}w_k
 \partial_r\widetilde Z_k\dd S
 \ge c\varepsilon_k^{(n-2)/2}-C\varepsilon_k^{n-2}.
\]
}
Choose $\vartheta<C_1/4$.  By
Lemmas  \ref{lem:signed-residual} and
\ref{lem:adjoint-replacement-interface}, we have
{
\begin{align*}
 \int_{B_{R_k}}\widetilde Z_kE_k\dd y
 \ge\int_{B_{R_k}}Z_kE_k\dd y
 -\Big|\int_{B_{R_k}}(\widetilde Z_k-Z_k)E_k\dd y\Big|
 \ge\frac{C_1}{2}
 \|H_{k,\le6}\|_{\varepsilon_k,R_k}^2
 -o\Bigl(\varepsilon_k^{(n-2)/2}\Bigr).
\end{align*}
}
By \eqref{eq:short-bubble-limit} and the corrector
estimate in the proof of Lemma  \ref{lem:profile-boundary-estimates},
$w_k=v_k-U-z_{k,1}\to0$ uniformly on $B_{R_0}$ and in a fixed neighborhood
of $\partial B_{R_0}$.  By
Lemma  \ref{lem:adjoint-replacement-interface},
{
\[
 \Big|\int_{B_{R_0}}(\mathcal L_kZ_k)w_k\dd y\Big|
 +\Big|\int_{\partial B_{R_0}}
 \llbracket\partial_\nu\widetilde Z_k\rrbracket_{\partial B_{R_0}}w_k\dd S\Big|
 =o\Bigl(
 \|H_{k,\le6}\|_{\varepsilon_k,R_k}^2
 +\varepsilon_k^{(n-2)/2}\Bigr).
\]
}
Substituting these estimates into
\eqref{eq:global-four-term-identity}, we obtain
{
\[
 \mathcal I_k^+-\mathcal I_k^-
 \ge c\varepsilon_k^{(n-2)/2}
 +\frac{C_1}{2}\|H_{k,\le6}\|_{\varepsilon_k,R_k}^2
 -o\Bigl(\varepsilon_k^{(n-2)/2}
 +\|H_{k,\le6}\|_{\varepsilon_k,R_k}^2\Bigr).
\]
}
For all sufficiently large $k$,  absorption of the last term gives the
asserted lower bound.
\end{proof}

 To control the $L^1$ mass of $w_k$ on one fixed ball, we first remove
the exterior potential by repeated Green representation.

 \begin{lemma} \label{lem:green-iteration}
There exist $\rho=\rho(n,R_0,C_{\mathrm g})>4R_0$ and
$C_\rho=C(n,\rho,g_0,u_0)>0$ such that, for all sufficiently large $k$,
there are an operator $K_k:L^1(B_\rho)\to L^1(B_\rho)$ and a function
$f_k\in L^1(B_\rho)$ satisfying
{
\[
\begin{aligned}
 w_k&=K_kw_k+f_k \quad \text{in }L^1(B_\rho),\\
 \|f_k\|_{L^1(B_\rho)}
 &\le C_\rho\big\{
 \varepsilon_k^{(n-2)/2}+\mathcal I_k^-
 +\|H_{k,\le6}\|_{\varepsilon_k,R_k}^2\big\}.
\end{aligned}
\]
}
 \end{lemma}

\begin{proof}
By  Lemmas  \ref{lem:global-green}\emph{(i)} and
\ref{lem:profile-boundary-estimates}, for every $\rho>4R_0$ and all
sufficiently large $k$,
{
\[
 \sup_{y\in B_{R_k}}
 \int_{B_{R_k}\setminus B_\rho}
 G_{B_{R_k}}(y,z)V_k(z)\dd z
 \le C\rho^{-2}.
\]
}
The constants in the preceding tail estimate, the Green kernel estimate
below, and \eqref{eq:exterior-potential-contraction} are independent of
\(k\) and \(\rho\). Choose \(C_*\) larger than these constants and fix
$\rho=\rho(n,R_0,C_{\mathrm g})>4R_0$ so that
$C_*\rho^{-2}\le1/2$.
 On $B_\rho$, \eqref{eq:short-bubble-limit} and Taylor's formula give
{
\[
 N_k=a_kw_k,\qquad
 a_k=p\,n(n-2)\int_0^1
 \big\{(\varphi_k+tw_k)^{p-1}-\varphi_k^{p-1}\big\}\dd t,
 \qquad \|a_k\|_{L^\infty(B_\rho)}\to0.
\]
}
 By Lemma  \ref{lem:no-dirac}, the equation holds across the
rescaled puncture. Thus, in the distributional sense,
{
\[
 -L_{g_k}w_k=
 \begin{cases}
  (V_k+a_k)w_k-E_k,&\text{in }B_\rho,\\
  V_kw_k+N_k-E_k,&\text{in }B_{R_k}\setminus B_\rho,
 \end{cases}.
\]
}
For $y\in B_{R_k}$, Green representation gives
{
\begin{equation}\label{eq:basic-green-decomposition}
 w_k(y)=S_kw_k(y)+\widetilde S_kw_k(y)+b_k(y),
\end{equation}
}
where
{
\[
 \begin{aligned}
 S_kF(y)&:=\int_{B_{R_k}\setminus B_\rho}
 G_{B_{R_k}}(y,z)V_k(z)F(z)\dd z,\\
 \widetilde S_kF(y)&:=\int_{B_\rho}
 G_{B_{R_k}}(y,z)(V_k+a_k)(z)F(z)\dd z.
 \end{aligned}
\]
}
For \(y\in B_\rho\) and \(z\in B_{R_k}\setminus\{y\}\),
Lemma  \ref{lem:global-green}\emph{(iv)} gives
\[
\begin{aligned}
 S_k\bigl(G_{B_{R_k}}(\,\cdot\,,y)\bigr)(z)
 =\int_{B_{R_k}\setminus B_\rho}
 G_{B_{R_k}}(z,x)V_k(x)G_{B_{R_k}}(x,y)\,\dd x
 \le C\rho^{-2}G_{B_{R_k}}(z,y)
 \le\frac12G_{B_{R_k}}(z,y).
\end{aligned}
\]
Positivity and induction yield
\[
 S_k^j\bigl(G_{B_{R_k}}(\,\cdot\,,y)\bigr)(z)
 \le2^{-j}G_{B_{R_k}}(z,y),\qquad j\ge0.
\]
The remaining term is
{
\begin{equation}\label{eq:bk-green-forcing}
 \begin{aligned}
 b_k(y)&:=\int_{B_{R_k}\setminus B_\rho}
 G_{B_{R_k}}(y,z)(N_k-E_k)(z)\dd z
 -\int_{B_\rho}G_{B_{R_k}}(y,z)E_k(z)\dd z\\
 &\quad-\int_{\partial B_{R_k}}\partial_{r_\xi}G_{B_{R_k}}(y,\xi)
                 w_k(\xi)\dd S(\xi).
 \end{aligned}
\end{equation}
}
For every integer $\ell\ge0$, finite iteration gives
{
\begin{equation}\label{eq:finite-green-iteration}
 w_k=S_k^{\ell+1}w_k
     +\sum_{j=0}^{\ell}S_k^j\widetilde S_kw_k
     +\sum_{j=0}^{\ell}S_k^jb_k.
\end{equation}
}
\textbf{Claim.}  For every integer $j\ge0$ and every
$F\in L^1(B_{R_k})$,
{
\begin{equation}\label{eq:iterated-green-estimate}
 \|S_k^jG_{B_{R_k}}[F]\|_{L^1(B_\rho)}
 \le C_\rho 2^{-j}\bigg\{
   \int_{B_\rho}|F|\dd z
   +\int_{B_{R_k}\setminus B_\rho}(-\widetilde Z_k)|F|\dd z\bigg\}.
\end{equation}
}

 We first estimate the integral of the Green function over $B_\rho$:
{
\begin{equation}\label{eq:green-mass-estimate}
 \int_{B_\rho}G_{B_{R_k}}(z,y)\dd y
 \le C_\rho
 \begin{cases}
  1,&z\in B_\rho,\\
  -\widetilde Z_k(z),&z\in B_{R_k}\setminus B_\rho.
 \end{cases}
\end{equation}
}
 Suppose that $z\in B_\rho$.  By Lemma  \ref{lem:global-green}\emph{(i)},
{
\[
 \int_{B_\rho}G_{B_{R_k}}(z,y)\dd y
 \le C\int_{B_\rho}|z-y|^{2-n}\dd y\le C_\rho.
\]
}
 For $z\in B_{R_k}\setminus B_\rho$, we have
{
\[
 -L_{g_k}\bigg(
 \int_{B_\rho}G_{B_{R_k}}(z,y)\dd y
 +C_\rho\widetilde Z_k(z)\bigg)
 =C_\rho V_k(z)\widetilde Z_k(z)\le0.
\]
}
The term in braces vanishes on $\partial B_{R_k}$.  Choose $C_\rho$
sufficiently large.  By Lemma  \ref{lem:global-green}\emph{(i),(ii)} and
Lemma  \ref{lem:exterior-corrected-adjoint}\emph{(i)}, this term is then
nonpositive on $\partial B_\rho$.
The maximum principle gives
{
\[
 \int_{B_\rho}G_{B_{R_k}}(z,y)\dd y
 \le C_\rho(-\widetilde Z_k(z)),
 \qquad z\in B_{R_k}\setminus B_\rho.
\]
}
We next establish the estimate for the exterior potential
{
\begin{equation}\label{eq:exterior-potential-contraction}
 \int_{B_{R_k}\setminus B_\rho}
 G_{B_{R_k}}(z,x)V_k(x)(-\widetilde Z_k(x))\dd x
 \le C\rho^{-2}
 \begin{cases}
  1,&z\in B_\rho,\\
  -\widetilde Z_k(z),&z\in B_{R_k}\setminus B_\rho.
 \end{cases}
\end{equation}
}
Suppose that $z\in B_\rho$.  By Lemma  \ref{lem:global-green}\emph{(i)} and
Lemma  \ref{lem:exterior-corrected-adjoint}\emph{(i)},
$-\widetilde Z_k\le C_\rho$ and
{
\[
 \int_{B_{R_k}\setminus B_\rho}
 G_{B_{R_k}}(z,x)V_k(x)(-\widetilde Z_k(x))\dd x
 \le C\int_{B_{R_k}\setminus B_\rho}
 G_{B_{R_k}}(z,x)V_k(x)\dd x\le C\rho^{-2}.
\]
}
Suppose that $z\in B_{R_k}\setminus B_\rho$.  By
Lemma  \ref{lem:profile-boundary-estimates},
$V_k(x)\le C(1+|x|)^{-4}$.  By
Lemma  \ref{lem:global-green}\emph{(iv)} and
Lemma  \ref{lem:exterior-corrected-adjoint}\emph{(i)}, we have
{
\[
\begin{aligned}
 &\frac{1}{-\widetilde Z_k(z)}
 \int_{B_{R_k}\setminus B_\rho}
 G_{B_{R_k}}(z,x)V_k(x)(-\widetilde Z_k(x))\dd x \\
 {}\le{}&C\int_{B_{R_k}\setminus B_\rho}
 \frac{G_{B_{R_k}}(z,x)G_{B_{R_k}}(0,x)}
      {G_{B_{R_k}}(0,z)}V_k(x)\dd x\\
 {}\le{}&C\int_{\R^n\setminus B_\rho}
 \big\{|z-x|^{2-n}+|x|^{2-n}\big\}(1+|x|)^{-4}\dd x
 \\[-2mm]
 {}\le{}&C\sup_z
 \int_{\R^n\setminus B_\rho}|z-x|^{2-n}(1+|x|)^{-4}\dd x
 +C\int_\rho^\infty r^{-3}\dd r\\
 {}\le{}&C\rho^{-2}.
\end{aligned}
\]
}
From \eqref{eq:green-mass-estimate} and
\eqref{eq:exterior-potential-contraction}, induction gives
{
\[
 S_k^j\bigg(\int_{B_\rho}G_{B_{R_k}}(\,\cdot\,,y)\dd y\bigg)(z)
 \le C_\rho 2^{-j}
 \begin{cases}
  1,&z\in B_\rho,\\
 -\widetilde Z_k(z),&z\in B_{R_k}\setminus B_\rho.
 \end{cases}
\]
}
The case $j=0$ is \eqref{eq:green-mass-estimate}.  If the estimate holds
for $j$, the nonnegativity of the Green function and $V_k$ 
give
 {
\[
\begin{aligned}
 &S_k^{j+1}\bigg(\int_{B_\rho}
 G_{B_{R_k}}(\,\cdot\,,y)\dd y\bigg)(z)\\
 {}={}&\int_{B_{R_k}\setminus B_\rho}
 G_{B_{R_k}}(z,x)V_k(x)
 S_k^j\bigg(\int_{B_\rho}
 G_{B_{R_k}}(\,\cdot\,,y)\dd y\bigg)(x)\dd x\\
 {}\le{}&C_\rho 2^{-j}
 \int_{B_{R_k}\setminus B_\rho}
 G_{B_{R_k}}(z,x)V_k(x)(-\widetilde Z_k(x))\dd x\\
 {}\le{}&C_\rho 2^{-j-1}
 \begin{cases}
  1,&z\in B_\rho,\\
  -\widetilde Z_k(z),&z\in B_{R_k}\setminus B_\rho.
 \end{cases}
\end{aligned}
\]
}
For \(j\ge1\), the integral kernel of
\(S_k^jG_{B_{R_k}}\) is
\[
 \int_{(B_{R_k}\setminus B_\rho)^j}
 G_{B_{R_k}}(y,x_1)V_k(x_1)\cdots
 V_k(x_j)G_{B_{R_k}}(x_j,z)\,
 \dd x_1\cdots\dd x_j.
\]
Reversing \(x_1,\ldots,x_j\) and using the symmetry of
\(G_{B_{R_k}}\) shows that this kernel is symmetric in \(y,z\).
The same conclusion for \(j=0\) is the symmetry of \(G_{B_{R_k}}\).
Thus Fubini's theorem gives, for
\(F\in L^1(B_{R_k})\) and \(j\ge0\),
{
\[
\begin{aligned}
 \|S_k^jG_{B_{R_k}}[F]\|_{L^1(B_\rho)}
 &\le\int_{B_{R_k}}|F(z)|
 S_k^j\Bigl(\int_{B_\rho}G_{B_{R_k}}(\,\cdot\,,y)\dd y\Bigr)(z)\dd z\\
 &\le C_\rho 2^{-j}\Big(
   \int_{B_\rho}|F|\dd z
   +\int_{B_{R_k}\setminus B_\rho}(-\widetilde Z_k)|F|\dd z\Big).
\end{aligned}
\]
}
 Thus \eqref{eq:iterated-green-estimate} holds.  We now justify the passage to the limit
$\ell\to\infty$ in \eqref{eq:finite-green-iteration}.
Fix $k$.  Apply \eqref{eq:iterated-green-estimate} to the function that
equals $V_kw_k$ on $B_{R_k}\setminus B_\rho$ and vanishes on $B_\rho$.
Then
{
\[
 \|S_k^{\ell+1}w_k\|_{L^1(B_\rho)}
 \le C_\rho 2^{-\ell}
 \int_{B_{R_k}\setminus B_\rho}
 (-\widetilde Z_k)V_k|w_k|\dd z\to0
 \quad\text{as }\ell\to\infty.
\]
}
 Fix $\eta\in L^1(B_\rho)$.  By the
definition of $\widetilde S_k$ and
\eqref{eq:iterated-green-estimate}, we have
{
\[
\begin{aligned}
 \sum_{j=0}^{\ell}
 \|S_k^j\widetilde S_k\eta\|_{L^1(B_\rho)}
 &=\sum_{j=0}^{\ell}
 \Bigl\|S_k^jG_{B_{R_k}}
 \big[(V_k+a_k)\eta\,\mathbf 1_{B_\rho}\big]
 \Bigr\|_{L^1(B_\rho)}\\
 &\le C_\rho\sum_{j=0}^{\ell}2^{-j}
 \int_{B_\rho}|V_k+a_k|\,|\eta|\dd z\\
 &\le C_\rho\|V_k+a_k\|_{L^\infty(B_\rho)}
 \sum_{j=0}^{\ell}2^{-j}
 \|\eta\|_{L^1(B_\rho)}\\
 &\le C_\rho\|\eta\|_{L^1(B_\rho)}.
\end{aligned}
\]
}
{Letting \(\ell\to\infty\), we obtain absolute convergence in operator
norm on \(L^1(B_\rho)\). The same integral estimate gives absolute
convergence on every ball compactly contained in \(B_{R_k}\). Hence the
Green series on the full ball}
{
\begin{equation}\label{eq:fixed-core-green-series}
 K_k\eta:=\sum_{j=0}^{\infty}S_k^j\widetilde S_k\eta
 \qquad\text{in }L^1_{\mathrm{loc}}(B_{R_k})
\end{equation}
}
is well defined, and its restriction to $B_\rho$ converges in operator
norm on $L^1(B_\rho)$.

 By \eqref{eq:bk-green-forcing} and the triangle inequality, for every
integer $\ell\ge0$,
{
\[
 \Bigl\|\sum_{j=0}^{\ell}S_k^jb_k\Bigr\|_{L^1(B_\rho)}
 \le J_{1,k}^{(\ell)}+J_{2,k}^{(\ell)}+J_{3,k}^{(\ell)},
\]
}
where
{
\[
\begin{aligned}
 J_{1,k}^{(\ell)}
 &:=\sum_{j=0}^{\ell}
 \Bigl\|S_k^j\Bigl[
 \int_{\partial B_{R_k}}
 \partial_{r_\xi}G_{B_{R_k}}(\,\cdot\,,\xi)
 w_k(\xi)\dd S(\xi)\Bigr]\Bigr\|_{L^1(B_\rho)},\\
 J_{2,k}^{(\ell)}
 &:=\sum_{j=0}^{\ell}
 \Bigl\|S_k^jG_{B_{R_k}}
 [N_k\mathbf 1_{B_{R_k}\setminus B_\rho}]
 \Bigr\|_{L^1(B_\rho)},\\
 J_{3,k}^{(\ell)}
 &:=\sum_{j=0}^{\ell}
 \Bigl\|S_k^jG_{B_{R_k}}[E_k]
 \Bigr\|_{L^1(B_\rho)}.
\end{aligned}
\]
}
 We first estimate $J_{1,k}^{(\ell)}$.  By
Lemma  \ref{lem:global-green}\emph{(iii)} and
Lemma  \ref{lem:profile-boundary-estimates},
{
\[
\begin{aligned}
 &\bigg\|\int_{\partial B_{R_k}}
 \partial_{r_\xi}G_{B_{R_k}}(\,\cdot\,,\xi)w_k(\xi)\dd S(\xi)
 \bigg\|_{L^\infty(B_{R_k})}\\
 {}\le{}&
 \|w_k\|_{L^\infty(\partial B_{R_k})}
 \sup_{y\in B_{R_k}}
 \int_{\partial B_{R_k}}
 \bigl(-\partial_{r_\xi}G_{B_{R_k}}(y,\xi)\bigr)\dd S(\xi)\\
 {}\le{}&C\bigl(
 \|v_k\|_{L^\infty(\partial B_{R_k})}
 +\|\varphi_k\|_{L^\infty(\partial B_{R_k})}\bigr)
 \le C\varepsilon_k^{(n-2)/2}.
\end{aligned}
\]
}
 Moreover,
{
\[
 \|S_k\|_{L^\infty\to L^\infty}
 \le\sup_{y\in B_{R_k}}
 \int_{B_{R_k}\setminus B_\rho}
 G_{B_{R_k}}(y,z)V_k(z)\dd z
 \le\frac12.
\]
}
 For all sufficiently large $k$, we therefore have
{
\[
 J_{1,k}^{(\ell)}
 \le |B_\rho|\sum_{j=0}^{\ell}
 \|S_k\|_{L^\infty\to L^\infty}^j
 C\varepsilon_k^{(n-2)/2}
 \le C_\rho\varepsilon_k^{(n-2)/2}
 \sum_{j=0}^{\ell}2^{-j}\le C_\rho\varepsilon_k^{(n-2)/2}.
\]
}
For  $J_{2,k}^{(\ell)}$, use \eqref{eq:iterated-green-estimate},
$N_k\ge0$, and $\widetilde Z_k<0$ on
$B_{R_k}\setminus B_\rho$.  Then
{
\[
\begin{aligned}
 J_{2,k}^{(\ell)}
 &=\sum_{j=0}^{\ell}
 \Big\|S_k^jG_{B_{R_k}}
 [N_k\mathbf 1_{B_{R_k}\setminus B_\rho}]
 \Big\|_{L^1(B_\rho)}\\
 &\le C_\rho\sum_{j=0}^{\ell}2^{-j}
 \int_{B_{R_k}\setminus B_\rho}(-\widetilde Z_k)N_k\dd y\\
 &\le C_\rho
 \int_{B_{R_k}\setminus B_\rho}(-\widetilde Z_k)N_k\dd y
 \le C_\rho\,\mathcal I_k^-.
\end{aligned}
\]
}
 For $J_{3,k}^{(\ell)}$, fix $0\le j\le\ell$.  By
\eqref{eq:iterated-green-estimate},
\eqref{eq:pointwise-profile-error}, and
Lemma  \ref{lem:exterior-corrected-adjoint}\emph{(i)}, we have
{
\[
\begin{aligned}
 \Big\|S_k^jG_{B_{R_k}}[E_k]
 \Big\|_{L^1(B_\rho)}
 &\le C_\rho 2^{-j}
 \Big\{
 \int_{B_\rho}|E_k|\dd y
 +\int_{B_{R_k}\setminus B_\rho}(-\widetilde Z_k)|E_k|\dd y
 \Big\}\\
 &\le C_\rho 2^{-j}
 \Big\{
 \sum_{m=2}^d\varepsilon_k^{2m}|H_k^{(m)}|^2
 \bigg(1+\int_\rho^{R_k}r^{2m+1-n}\dd r\bigg)
 +\varepsilon_k^{n-3}
 +\varepsilon_k^{n-2}(1+\log R_k)
 \Big\}\\
 &\le C_\rho 2^{-j}\big\{
 \|H_{k,\le6}\|_{\varepsilon_k,R_k}^2
 +\varepsilon_k^{14}(1+|\log\varepsilon_k|)
 +o\bigl(\varepsilon_k^{(n-2)/2}\bigr)\big\}\\
 &\le C_\rho 2^{-j}\big\{
 \|H_{k,\le6}\|_{\varepsilon_k,R_k}^2
 +o\bigl(\varepsilon_k^{(n-2)/2}\bigr)\big\}.
\end{aligned}
\]
}
 Summing this estimate over $j$ gives
{
\[
 J_{3,k}^{(\ell)}
 \le C_\rho\sum_{j=0}^{\ell}2^{-j}
 \big\{\|H_{k,\le6}\|_{\varepsilon_k,R_k}^2
 +o\bigl(\varepsilon_k^{(n-2)/2}\bigr)\big\}
 \le C_\rho\big\{\|H_{k,\le6}\|_{\varepsilon_k,R_k}^2
 +o\bigl(\varepsilon_k^{(n-2)/2}\bigr)\big\}.
\]
}
For all sufficiently large $k$, adding the three estimates yields
{
\[
 \Bigl\|\sum_{j=0}^{\ell}S_k^jb_k\Bigr\|_{L^1(B_\rho)}
 \le C_\rho\big\{\varepsilon_k^{(n-2)/2}+\mathcal I_k^-
 +\|H_{k,\le6}\|_{\varepsilon_k,R_k}^2\big\},
\]
}
uniformly in $\ell$.  Hence the series converges absolutely in
$L^1(B_\rho)$, and we may define
{
\[
 f_k:=\sum_{j=0}^{\infty}S_k^jb_k.
\]
}
{The uniform estimate for the partial sums passes to the limit and gives the asserted
bound for \(f_k\).  Letting \(\ell\to\infty\) in
\eqref{eq:finite-green-iteration}, we obtain}
 $w_k=K_kw_k+f_k$ in $L^1(B_\rho)$.
\end{proof}

 Fix $\rho$ and $K_k$ as in
Lemma  \ref{lem:green-iteration}.  We next control $w_k$ on $B_\rho$ by the
positive part of the nonlinear pairing.

\begin{lemma}\label{lem:fixed-core}
There is $C_\rho=C(n,\rho,g_0,u_0)>0$ such that, for
all sufficiently large $k$,
{
\[
 \|w_k\|_{L^1(B_\rho)}\le C_\rho\mathcal I_k^+.
\]
}
 \end{lemma}

 \begin{proof}
We first prove
{
\[
 \|w_k\|_{L^1(B_\rho)}
 \le C_\rho\big\{
 \varepsilon_k^{(n-2)/2}+\mathcal I_k^-
 +\|H_{k,\le6}\|_{\varepsilon_k,R_k}^2\big\}.
\]
}
 Set $M_k=\|w_k\|_{L^1(B_\rho)}$ and $\eta_k=M_k^{-1}w_k$ on $B_\rho$.
If the asserted estimate fails, then, up to a subsequence,
 {
\[
 M_k^{-1}\big\{
 \varepsilon_k^{(n-2)/2}+\mathcal I_k^-
 +\|H_{k,\le6}\|_{\varepsilon_k,R_k}^2\big\}
 \to0.
\]
}
 Put $u_k=K_k\eta_k$, where $K_k$ is defined in \eqref{eq:fixed-core-green-series}.  
By Lemma  \ref{lem:green-iteration}, we have
{
\[
 \eta_k-u_k = f_k/M_k\to0\quad\text{in }L^1(B_\rho),
 \qquad \|\eta_k\|_{L^1(B_\rho)}=1.
\]
}

 We first obtain compactness on every fixed ball.  Fix $\rho_1>\rho$.
Since $R_k\to\infty$, we have $B_{2\rho_1}\subset B_{R_k}$ for all
sufficiently large $k$.  Repeating the preceding integral estimate with
$B_{2\rho_1}$ as the integration region gives

{
\[
 \int_{B_{2\rho_1}}G_{B_{R_k}}(z,y)\dd y
 \le C_{\rho_1}
 \begin{cases}
  1,&z\in B_{4\rho_1},\\
  -\widetilde Z_k(z),&z\in B_{R_k}\setminus B_{4\rho_1}.
\end{cases}
\]
}
\nopagebreak[4]
The Newtonian bound applies when $z\in B_{4\rho_1}$, and the
comparison with $-\widetilde Z_k$ applies when
$z\in B_{R_k}\setminus B_{4\rho_1}$.  The same induction therefore yields,
for all sufficiently large $k$,
{
\[
 \|S_k^j\widetilde S_k\eta_k\|_{L^1(B_{2\rho_1})}
 \le C_{\rho_1} 2^{-j}\|\eta_k\|_{L^1(B_\rho)}.
\]
}
 Consequently,
{
\[
 \|u_k\|_{L^1(B_{2\rho_1})}
 \le\sum_{j=0}^{\infty}
 \|S_k^j\widetilde S_k\eta_k\|_{L^1(B_{2\rho_1})}
 \le C_{\rho_1}\sum_{j=0}^{\infty}2^{-j}
 \le C_{\rho_1}.
\]
}
The operators $S_k$ and $\widetilde S_k$ were introduced in
\eqref{eq:basic-green-decomposition}.  For every integer $N\ge1$, their
definitions and the Green identity give
{
\[
 -L_{g_k}\sum_{j=0}^{N}S_k^j\widetilde S_k\eta_k
 =(V_k+a_k)\mathbf1_{B_\rho}\eta_k
 +V_k\mathbf1_{B_{R_k}\setminus B_\rho}
   \sum_{j=0}^{N-1}S_k^j\widetilde S_k\eta_k \qquad \text{in } \mathcal D'(B_{R_k}).
\]
}
Both sums converge in $L^1(B_{2\rho_1})$ by the preceding geometric
estimate.  Passing to the limit in the  distributional sense, we obtain
{
\begin{equation}\label{eq:normalized-core-equation}
 -L_{g_k}u_k
 =V_k\mathbf1_{B_{R_k}\setminus B_\rho}u_k
  +(V_k+a_k)\mathbf1_{B_\rho}\eta_k
 \qquad\text{in }\mathcal D'(B_{2\rho_1}).
\end{equation}
}
 The functions $V_k$ and $a_k$ are uniformly bounded on $B_{2\rho_1}$
and $B_\rho$, respectively.  By \eqref{eq:normalized-core-equation},
{
\[
 \|\!-\!L_{g_k}u_k\|_{L^1(B_{2\rho_1})}
 \le C_{\rho_1}\bigl(
 \|u_k\|_{L^1(B_{2\rho_1})}
 +\|\eta_k\|_{L^1(B_\rho)}\bigr)
 \le C_{\rho_1}.
\]
}
 Fix $s_0$ with $1<s_0<n/(n-1)$.
{Let $\zeta_k$ solve
\[
 -\Delta_{g_k}\zeta_k=-\Delta_{g_k}u_k
 \quad\text{in }B_{2\rho_1},
 \qquad \zeta_k=0\quad\text{on }\partial B_{2\rho_1}.
\]
Since
\[
 \|-\Delta_{g_k}u_k\|_{L^1(B_{2\rho_1})}
 \le \|-L_{g_k}u_k\|_{L^1(B_{2\rho_1})}
 {}+C\|u_k\|_{L^1(B_{2\rho_1})},
\]
the Brezis--Strauss estimate
\cite[Lemma  9]{BrezisStrauss}, applied to \(-\Delta_{g_k}\) and obtained
for \(L^1\) data by approximation, gives
\[
 \|\zeta_k\|_{W^{1,s_0}(B_{2\rho_1})}
 \le C_{\rho_1,s_0}
 \|{-\Delta_{g_k}u_k}\|_{L^1(B_{2\rho_1})}.
\]
The function \(u_k-\zeta_k\) satisfies
\(\Delta_{g_k}(u_k-\zeta_k)=0\) in \(B_{2\rho_1}\).
The interior gradient estimate applied on $B_{\rho_1}$ then gives}
{
\[
 \|u_k-\zeta_k\|_{W^{1,s_0}(B_{\rho_1})}
 \le C_{\rho_1,s_0}
 \|u_k-\zeta_k\|_{L^1(B_{2\rho_1})}.
\]
}
Combining these two estimates, we obtain
{
\[
 \|u_k\|_{W^{1,s_0}(B_{\rho_1})}
 \le C_{\rho_1,s_0}\bigl(
 \|u_k\|_{L^1(B_{2\rho_1})}
 +\|\!-\!L_{g_k}u_k\|_{L^1(B_{2\rho_1})}\bigr)
 \le C_{\rho_1,s_0}.
\]
}
 Apply this estimate with $\rho_1=\rho+m$ for $m \in \mathbb{N}$.
For every $m$, the embedding
$W^{1,s_0}(B_{\rho+m})\Subset L^1(B_{\rho+m})$ is compact.  Up to a
diagonal subsequence, there is
$\widetilde\eta\in L^1_{\mathrm{loc}}(\R^n)$ such that
{
\[
 u_k\to\widetilde\eta
 \quad\text{in }L^1_{\mathrm{loc}}(\R^n).
\]
}
 The convergence on $B_\rho$ and $\eta_k-u_k\to0$ in $L^1(B_\rho)$ imply
{
\[
 \eta_k\to\widetilde\eta
 \quad\text{in }L^1(B_\rho),
 \qquad
 \|\widetilde\eta\|_{L^1(B_\rho)}=1.
\]
}
Set $V=p\,n(n-2)U^{p-1}$.  The coefficient and potential convergences are
{
\begin{equation}\label{eq:coefficient-potential-convergence}
 g_{k,ij}\to\delta_{ij}\text{ in }C^2_{\mathrm{loc}}(\R^n),
 \quad V_k\to V\text{ in }C^0_{\mathrm{loc}}(\R^n),
 \quad a_k\to0\text{ in }L^\infty(B_\rho).
\end{equation}
}
Together with the preceding $L^1$ convergences, these limits give
{
\[
\begin{aligned}
 V_k\mathbf1_{B_{R_k}\setminus B_\rho}u_k
 &\to V\mathbf1_{\R^n\setminus B_\rho}\widetilde\eta
 &&\quad\text{in }L^1_{\mathrm{loc}}(\R^n),\\
 (V_k+a_k)\mathbf1_{B_\rho}\eta_k
 &\to V\mathbf1_{B_\rho}\widetilde\eta
 &&\quad\text{in }L^1(\R^n).
\end{aligned}
\]
}
Testing \eqref{eq:normalized-core-equation} and passing to the limit, we
obtain
{
\begin{equation}\label{eq:whole-space-limit-equation}
 -\Delta\widetilde\eta
 =V\mathbf1_{\R^n\setminus B_\rho}\widetilde\eta
   +V\mathbf1_{B_\rho}\widetilde\eta
  =V\widetilde\eta
 \qquad\text{in }\mathcal D'(\R^n).
\end{equation}
}
We next prove that $\widetilde\eta$ has finite Dirichlet energy, which permits
us to apply the kernel classification.
The pointwise Green kernel estimate in
Lemma  \ref{lem:green-iteration} and positivity give, for
\(y\in B_\rho\), \(z\in B_{R_k}\setminus\{y\}\), and \(j\ge0\),
\[
 S_k^j\bigl(G_{B_{R_k}}(\,\cdot\,,y)\bigr)(z)
 \le2^{-j}G_{B_{R_k}}(z,y).
\]
On $B_\rho$, the functions $V_k+a_k$ are uniformly bounded and
$\|\eta_k\|_{L^1(B_\rho)}=1$.  By Fubini's theorem, the Harnack inequality
in the second Green variable, and
Lemma  \ref{lem:exterior-corrected-adjoint}\emph{(i)},
{
\[
\begin{aligned}
 |u_k(z)|
 &\le\sum_{j=0}^\infty 2^{-j}
 \int_{B_\rho}G_{B_{R_k}}(z,y)
 |V_k(y)+a_k(y)|\,|\eta_k(y)|\dd y\\
 &\le C_\rho G_{B_{R_k}}(z,0)
 \sum_{j=0}^\infty 2^{-j}\le C_\rho(-\widetilde Z_k(z))
 \le C_\rho(1+|z|)^{2-n},
 \qquad 2\rho\le|z|\le R_k/2.
\end{aligned}
\]
}
For $4\rho\le|y|\le R_k/3$, the interior gradient estimate on
 $B_{|y|/4}(y)$ gives
{
\[
 |\nabla u_k(y)|\le C_\rho(1+|y|)^{1-n}.
\]
}
On compact subsets of $\R^n\setminus\overline{B_{4\rho}}$, local elliptic
compactness and the preceding estimates give
{
\[
 |\widetilde\eta(y)|+(1+|y|)|\nabla\widetilde\eta(y)|
 \le C_\rho(1+|y|)^{2-n},
 \qquad |y|\ge4\rho.
\]
}
By local elliptic regularity on $B_{4\rho}$ and the preceding decay,
{
\[
 \int_{\R^n}|\nabla\widetilde\eta|^2\dd y
 \le \int_{B_{4\rho}}|\nabla\widetilde\eta|^2\dd y
 +C_\rho\int_{4\rho}^{\infty}r^{1-n}\dd r<\infty.
\]
}
Thus $\widetilde\eta$ has finite Dirichlet energy.  By
Bianchi--Egnell \cite[Lemma  A1]{BianchiEgnell},
{
\[
 \widetilde\eta
 =c_0\frac{n-2}{2}\frac{1-|y|^2}{1+|y|^2}U
  +\sum_{i=1}^nc_i\partial_iU.
\]
}

It remains to eliminate these Jacobi fields.  Since
$\eta_k=M_k^{-1}w_k$ on $B_\rho$,
$\eta_k(0)=0$ and  $\nabla\eta_k(0)=0$.
Moreover, in $B_\rho$,
{
\[
 \big(-L_{g_k}-V_k-a_k\big)\eta_k=-M_k^{-1}E_k.
\]
}
 Fix $r<\rho/2$ and $s>n$.  By
\eqref{eq:whole-space-limit-equation} and local elliptic regularity,
$\widetilde\eta\in C^\infty_{\mathrm{loc}}(\R^n)$.  On $B_{2r}$,
{
\[
 \big(-L_{g_k}-V_k-a_k\big)(\eta_k-\widetilde\eta)
 =-M_k^{-1}E_k+(L_{g_k}-\Delta)\widetilde\eta
 +(V_k+a_k-V)\widetilde\eta.
\]
}
 By \eqref{eq:pointwise-profile-error},
\eqref{eq:coefficient-potential-convergence},
the contradiction normalization
$M_k^{-1}\{\varepsilon_k^{(n-2)/2}+\mathcal I_k^-
+\|H_{k,\le6}\|_{\varepsilon_k,R_k}^2\}\to0$, and the interior
Calder\'on--Zygmund estimate,
{
\[
 \|\eta_k-\widetilde\eta\|_{W^{2,s}(B_r)}
 \le C_r\Big(
 \|\eta_k-\widetilde\eta\|_{L^1(B_{2r})}
 +\|\big(-L_{g_k}-V_k-a_k\big)
       (\eta_k-\widetilde\eta)\|_{L^s(B_{2r})}\Big)
 \to0.
\]
}
Since $s>n$, Morrey's inequality gives convergence in $C^1(B_r)$.
Therefore
{
\[
 \widetilde\eta(0)=0,\qquad \nabla\widetilde\eta(0)=0.
\]
}
The  dilation field has value $(n-2)/2$ at the origin, while
$\nabla\partial_iU(0)=-(n-2)e_i$.  The expansion in Jacobi fields therefore
gives $c_0=c_1=\cdots=c_n=0$, contrary to
$\|\widetilde\eta\|_{L^1(B_\rho)}=1$.
Hence
{
\[
 \|w_k\|_{L^1(B_\rho)}
 \le C_\rho\big\{
 \varepsilon_k^{(n-2)/2}+\mathcal I_k^-
 +\|H_{k,\le6}\|_{\varepsilon_k,R_k}^2\big\}.
\]
}
 By Lemma  \ref{lem:global-balance},
{
\[
 \varepsilon_k^{(n-2)/2}+\mathcal I_k^-
 +\|H_{k,\le6}\|_{\varepsilon_k,R_k}^2
 \le C\mathcal I_k^+.
\]
}
The asserted estimate follows.
\end{proof}

 \begin{proof}[Proof of Theorem  \ref{thm:upper-bound}]
{Under the standing contradiction hypothesis and the construction above,
Lemmas  \ref{lem:signed-residual}  through  \ref{lem:fixed-core} apply. By
Lemma  \ref{lem:global-balance},}
{
\[
 \mathcal I_k^+\ge \mathcal I_k^-+C_2\varepsilon_k^{(n-2)/2}.
\]
}
 By Lemma  \ref{lem:fixed-core},
{
\[
 \|w_k\|_{L^1(B_\rho)}\le C\mathcal I_k^+.
\]
}
 Moreover, $\{\widetilde Z_k>0\}\subset B_{R_0}$, and
\eqref{eq:short-bubble-limit} and Taylor's formula  give
{
\[
 0\le N_k\le C|w_k|^2,
 \qquad
 \|w_k\|_{L^\infty(B_{R_0})}\to0.
\]
}
 Therefore,
{
\[
 \mathcal I_k^+
 \le C\int_{B_{R_0}}|w_k|^2\dd y
 \le C\|w_k\|_{L^\infty(B_{R_0})}
             \|w_k\|_{L^1(B_\rho)}
 =o(\mathcal I_k^+),
\]
}
 {which is impossible because \(\mathcal I_k^+>0\).  This contradiction
proves Theorem  \ref{thm:upper-bound}.}
\end{proof}
\section{The Pohozaev limit and annular scales}
\label{sec:pohozaev-dichotomy}

\emph{
Throughout Sections  \ref{sec:pohozaev-dichotomy}  through  \ref{sec:fowler}, we
assume the critical upper bound \eqref{eq:critical-upper}; the argument
applies in every dimension.}  This section first identifies the
limiting Pohozaev quantity and the removable decay regime.  In the remaining
nonremovable case with \(\cP_\infty=0\), it constructs the alternating peak
and minimum radii and proves the endpoint estimates used on the rescaled
annuli.

Starting from the original smooth metric and solution \((g_0,u_0)\),
choose the fixed conformal representative from
Section  \ref{sec:common-preliminaries} at the origin:
\[
 \widehat g=\kappa_0^{4/(n-2)}g_0,
 \qquad \widehat u=\kappa_0^{-1}u_0.
\]
Fix \(\omega\in(0,1)\) and choose the representative with coefficients
in \(C^{d+4,\omega}\). After reducing the normal coordinate
radius if necessary, apply one fixed dilation so that the coordinate domain
is \(B_1\). We henceforth write
\(g,u\) for this metric and solution. Conformal covariance and local
distance comparability preserve \eqref{eq:critical-upper} and
nonremovability, with changed constants; in these coordinates
\(d_g(x,0)=|x|\). Thus \eqref{eq:normal-gauge} and
\eqref{eq:normal-orders} hold, with \(h=\log g\), and the coefficients
are fixed independently of all later rescalings.

\subsection{The Pohozaev identity and radial estimates}

{
The Pohozaev identity expresses the change between two radii in terms of the
metric error.  Uniform estimates on fixed annuli then provide compactness
under rescaling and a convexity estimate for the weighted spherical average.
}

Recall that the normalized metric and solution satisfy
\eqref{eq:normal-gauge}, \eqref{eq:normal-orders}, and
\eqref{eq:critical-upper}. In Euclidean polar coordinates \(x=r\theta\),
define the spherical average for \(0<r<r_0\) by
\begin{equation}\label{eq:spherical-average}
 \bar u(r)=\fint_{\Sph^{n-1}}u(r\theta)\,\dd\theta.
\end{equation}
Use the Pohozaev boundary integral \(\cP(r,u)\) defined in
\eqref{eq:pohozaev}.
{
Equation  \eqref{eq:yamabe} is equivalently written as
\[
 -\Delta u=n(n-2)u^p-(\Delta-L_g)u.
\]
}

\par
\begin{lemma}\label{lem:exact-pohozaev}
For \(0<s<r<r_0\),
\begin{equation}\label{eq:exact-pohozaev}
 \cP(r,u)-\cP(s,u)
 =
 \int_{A_{s,r}}
 \bigl(\frac{n-2}{2}u+x\cdot\nabla_xu\bigr)(\Delta-L_g)u\,\dd x .
\end{equation}
\end{lemma}

\par
\begin{proof}
Consider the Euclidean vector field
\[
 X=(x\cdot\nabla_xu)\nabla_xu-\frac12|\nabla_xu|^2x
   +\frac{n-2}{2}u\nabla_xu+\frac{(n-2)^2}{2}u^{p+1}x.
\]
{Using the preceding identity, we compute}
\[
 \operatorname{div}X
 =
 \bigl(\frac{n-2}{2}u+x\cdot\nabla_xu\bigr)
 \bigl(\Delta u+n(n-2)u^p\bigr)
 =\bigl(\frac{n-2}{2}u+x\cdot\nabla_xu\bigr)(\Delta-L_g)u.
\]
{On \(\partial B_r\),}
\[
 X\cdot\frac{x}{r}
 =
 \frac{n-2}{2}u\,\pa_r u-\frac r2|\nabla_xu|^2+r(\pa_r u)^2
 +\frac{(n-2)^2}{2}r u^{p+1}.
\]
{Applying the divergence theorem on \(A_{s,r}\), we obtain
\eqref{eq:exact-pohozaev}.}
\end{proof}

\par
Let \(\nabla_\theta\) denote the Levi--Civita connection of the round
metric on \(\Sph^{n-1}\).

\par
\begin{lemma}\label{lem:radial-harnack}
{
Let $g$ be a Riemannian metric whose coefficients belong to
$C^3(B_1)$ in geodesic normal coordinates centered at $0$.
Let $u>0$ solve \eqref{eq:yamabe} on $B_1\setminus\{0\}$
and satisfy \eqref{eq:critical-upper}. Then there exist \(C_{\rm H}\ge1\) and \(r_{\rm ann}\in(0,r_0)\) such that,
for every integer \(0\le k\le2\), there exists \(C_k>0\) for which the
following estimates hold whenever \(0<r<r_{\rm ann}\):
}
\begin{equation}\label{eq:annular-harnack}
 \sup_{\theta\in\Sph^{n-1}}u(r\theta)
 \le C_{\rm H}\inf_{\theta\in\Sph^{n-1}}u(r\theta),
\end{equation}
\begin{equation}\label{eq:radial-derivatives}
 {\sum_{\ell+m\le k}
 \left|
 (r\pa_r)^\ell\nabla_\theta^m
 \bigl(r^{\frac{n-2}{2}}u(r\theta)\bigr)
 \right|
 \le C_k r^{\frac{n-2}{2}}\bar u(r).}
\end{equation}
The constants depend only on $n$, $C_0$, and the fixed metric
on $B_1$, with additional dependence on $k$ for $C_k$.
The radius $r_{\rm ann}$ also depend on $r_0$.
\end{lemma}

\begin{proof}
For \(0<r<r_0/2\), define
\[
 u_r(y)=r^{\frac{n-2}{2}}u(ry),
 \qquad
 g_r(y)=g(ry),
 \qquad y\in A_{1/2,2}.
\]
The rescaled equation is
\[
 -\Delta_{g_r}u_r
 =
 n(n-2)u_r^p-c(n)r^2\operatorname{Scal}_g(ry)u_r.
\]
By \eqref{eq:critical-upper} and the boundedness of
$\operatorname{Scal}_g$,
\[
 0<u_r\le C,
 \qquad
 \left|
 n(n-2)u_r^{p-1}-c(n)r^2\operatorname{Scal}_g(ry)
 \right|
 \le C
 \quad\text{on }A_{1/2,2}.
\]
{A finite Harnack chain in \(A_{1/2,2}\) compares all values
of \(u_r\) on \(A_{3/5,5/3}\). Restriction to \(\partial B_1\) and scaling
back give \eqref{eq:annular-harnack}.}

Let 
\[
\widehat u_r(y)
=\frac{u_r(y)}
{\fint_{\Sph^{n-1}}u_r(\theta)\,\dd\theta}
\]
be \(u_r\) normalized to have spherical average one on \(\partial B_1\).
The Harnack inequality gives
\[
 c\le\widehat u_r\le C
 \quad\text{on }A_{3/5,5/3}.
\]
Fix \(\gamma\in(0,1)\) and choose \(s>n/(1-\gamma)\). Interior
\(W^{2,s}\) estimates, Sobolev embedding, and Schauder estimates yield
\begin{equation}\label{eq:normalized-annular-schauder}
 \norm{\widehat u_r}_{C^{k,\gamma}(A_{3/4,4/3})}
 \le C_{k,\gamma},
 \qquad 0\le k\le2.
\end{equation}
Since
\[
 \fint_{\Sph^{n-1}}u_r(\theta)\,\dd\theta
 =r^{\frac{n-2}{2}}\bar u(r),
\]
scaling \eqref{eq:normalized-annular-schauder} back to \(\partial B_r\)
gives \eqref{eq:radial-derivatives}.
\end{proof}

{The next lemma proves the corresponding ODE consequence of
the annular estimates.  Its strict convexity conclusion will be used to
control the low sublevel intervals and to select the peak and minimum radii.}

\begin{lemma}\label{lem:radial-average-convexity}
For \(0<r<r_{\rm ann}\), the spherical average satisfies
\begin{subequations}
\begin{equation}\label{eq:averaged-radial-identity}
\begin{aligned}
 &(r\pa_r)^2\bigl(r^{\frac{n-2}{2}}\bar u(r)\bigr)
 -\frac{(n-2)^2}{4}r^{\frac{n-2}{2}}\bar u(r)\\
{}={}&-n(n-2)r^{\frac{n+2}{2}}
 \fint_{\Sph^{n-1}}u(r\theta)^p\,\dd\theta
 +c(n)r^{\frac{n+2}{2}}
 \fint_{\Sph^{n-1}}\operatorname{Scal}_g(r\theta)u(r\theta)\,\dd\theta,
\end{aligned}
\end{equation}
\begin{equation}\label{eq:averaged-curvature-error}
 \left|
 c(n)r^{\frac{n+2}{2}}
 \fint_{\Sph^{n-1}}\operatorname{Scal}_g(r\theta)u(r\theta)\,\dd\theta
 \right|
 \le Cr^4r^{\frac{n-2}{2}}\bar u(r).
\end{equation}
\end{subequations}
Moreover, there exist \(\eps_{\rm cv}>0\) and
\(r_{\rm cv}\in(0,r_{\rm ann})\) such that, if \(0<r<r_{\rm cv}\) and
\(r^{\frac{n-2}{2}}\bar u(r)\le\eps_{\rm cv}\), then
\begin{equation}\label{eq:radial-convexity}
 (r\pa_r)^2\bigl(r^{\frac{n-2}{2}}\bar u(r)\bigr)
 \ge\frac{(n-2)^2}{8}r^{\frac{n-2}{2}}\bar u(r)>0.
\end{equation}
\end{lemma}

\begin{proof}
By Lemma  \ref{lem:radial-cancellation},
\[
 \fint_{\Sph^{n-1}}\Delta_gu(r\theta)\,\dd\theta
 =
 \bar u''(r)+\frac{n-1}{r}\bar u'(r).
\]
{Averaging \eqref{eq:yamabe}, we find}
\[
 \bar u''+\frac{n-1}{r}\bar u'
 +n(n-2)\fint_{\Sph^{n-1}}u^p\,\dd\theta
 =
 c(n)\fint_{\Sph^{n-1}}\operatorname{Scal}_gu\,\dd\theta.
\]
{We use the identity}
\[
 (r\pa_r)^2(r^{\frac{n-2}{2}}\bar u)
 =
 r^{\frac{n+2}{2}}\left(\bar u''+\frac{n-1}{r}\bar u'\right)
 +\frac{(n-2)^2}{4}r^{\frac{n-2}{2}}\bar u
\]
and the estimate
\[
 r^{\frac{n+2}{2}}|\operatorname{Scal}_g(r\theta)|u(r\theta)
 \le Cr^4r^{\frac{n-2}{2}}u(r\theta)
\]
{to obtain \eqref{eq:averaged-radial-identity} and
\eqref{eq:averaged-curvature-error}.}

{
Finally, by \eqref{eq:annular-harnack},
\eqref{eq:averaged-radial-identity}, and
\eqref{eq:averaged-curvature-error}, we have
\[
 (r\pa_r)^2(r^{\frac{n-2}{2}}\bar u)
 \ge
 \frac{(n-2)^2}{4}r^{\frac{n-2}{2}}\bar u
 -C(r^{\frac{n-2}{2}}\bar u)^p
 -Cr^4r^{\frac{n-2}{2}}\bar u.
\]
Choose \(\eps_{\rm cv}\) first so that
\(C\eps_{\rm cv}^{p-1}\le(n-2)^2/16\). Then choose
\(r_{\rm cv}\in(0,r_{\rm ann})\) so that
\(Cr_{\rm cv}^4\le(n-2)^2/16\). This proves
\eqref{eq:radial-convexity}.
}
\end{proof}

\par
\subsection{The Pohozaev limit and the removable alternative}

{
We next identify the limiting Pohozaev quantity and characterize its zero
case. Under the nonremovable contradiction hypothesis, the weighted
spherical average satisfies the two limits in
\eqref{eq:radial-oscillation} below.
}

\begin{lemma}\label{lem:pohozaev-limit}
{
The limit
\(\cP_\infty:=\lim_{r\downarrow0}\cP(r,u)\) exists and satisfies
\(\cP_\infty\le0\). Moreover, \(\cP_\infty=0\) if and only if
\[
 \liminf_{r\downarrow0}r^{\frac{n-2}{2}}\bar u(r)=0.
\]
}
\end{lemma}

\par
\begin{proof}
\emph{Step 1. Existence of the limit.}
{
By \eqref{eq:normal-orders} and
\eqref{eq:radial-derivatives} with \(k=2\), for \(0<r<r_{\rm ann}\), we have
\[
 |u|+r|\nabla_xu|+r^2|\nabla_x^2u|
 \le Cr^{-\frac{n-2}{2}}
 \quad\text{on }A_{r/2,2r}.
\]
}
{
Since \(\det g=1\) in the fixed coordinates,
\[
 (\Delta-L_g)u
 =\sum_{i,j=1}^n(\delta^{ij}-g^{ij})\partial_i\partial_j u
  -\sum_{i,j=1}^n(\partial_i g^{ij})\partial_j u
  +c(n)\operatorname{Scal}_gu.
\]
By \eqref{eq:normal-orders},
}
\[
 |(\Delta-L_g)u|
 \le
 C\bigl(r^2|\nabla_x^2u|+r|\nabla_xu|+r^2u\bigr)
 \le Cr^{-\frac{n-2}{2}}.
\]
{By Lemma  \ref{lem:exact-pohozaev}, for
\(0<s<r<r_{\rm ann}\),}
\[
 |\cP(r,u)-\cP(s,u)|
 \le C\int_s^r \tau\,\dd\tau
 =\frac C2(r^2-s^2)
 \le Cr^2.
\]
Thus \(\cP_\infty\) exists.

\emph{Step 2. Compactness under rescaling.}
Let \(s_\nu\downarrow0\) as \(\nu\to\infty\), and define
\[
 u_\nu(y)=s_\nu^{\frac{n-2}{2}}u(s_\nu y),
 \qquad
 g_\nu(y)=g(s_\nu y).
\]
For every fixed \(R>1\) and \(\gamma\in(0,1)\),
{a finite covering of \(A_{R^{-1},R}\) by rescaled copies
of \(A_{3/4,4/3}\), together with \eqref{eq:critical-upper} and
\eqref{eq:normalized-annular-schauder}, gives, for all sufficiently
large \(\nu\),}
\[
 \begin{aligned}
 &(g_\nu)_{ij}
 \to\delta_{ij}
 \qquad \text{in }C^2(A_{R^{-1},R})
 \quad\text{as }\nu\to\infty,\\
 &\norm{u_\nu}_{C^{2,\gamma}(A_{R^{-1},R})}
 \le C_{R,\gamma}.
 \end{aligned}
\]
After passing to a diagonal subsequence,
\[
 u_\nu\to u_\infty
 \quad\text{in }C^2_{\rm loc}(\R^n\setminus\{0\})
 \quad\text{as }\nu\to\infty,
\]
where
\[
 -\Delta u_\infty=n(n-2)u_\infty^p,
 \qquad u_\infty\ge0.
\]
{Scaling \eqref{eq:pohozaev} and letting \(\nu\to\infty\) on
\(\pa B_1\), we obtain}
\[
 \cP_\infty=\cP(1,u_\infty).
\]
\emph{Step 3. Sign of the limiting Pohozaev quantity.}
If \(u_\infty\equiv0\), then
\(\cP_\infty=\cP(1,u_\infty)=0\).
Otherwise \(u_\infty>0\), and Caffarelli--Gidas--Spruck
\cite[Theorem  8.1 and Corollary  8.2]{CGS}
show that \(u_\infty\) is either an entire standard bubble or a positive
radial solution with a nonremovable singularity at the puncture.

{For an entire bubble, it follows from the Euclidean Pohozaev
identity and smoothness at the origin that}
\[
 \cP(r,u_\infty)=\lim_{s\downarrow0}\cP(s,u_\infty)=0,
 \qquad r>0.
\]
For a radial nonremovable solution,
\[
 \frac{\cP(r,u_\infty)}{|\Sph^{n-1}|}
 =\frac12\left[r\partial_r\bigl(r^{\frac{n-2}{2}}u_\infty(r)\bigr)\right]^2
 -\frac{(n-2)^2}{8}\bigl(r^{\frac{n-2}{2}}u_\infty(r)\bigr)^2
 +\frac{n(n-2)}{p+1}
 \bigl(r^{\frac{n-2}{2}}u_\infty(r)\bigr)^{p+1}.
\]
The profile \(r^{(n-2)/2}u_\infty(r)\) is constant or periodic in
\(-\log r\).
{By the radial ODE phase portrait in the cited
classification, the conserved energy displayed above is strictly negative.}
Hence \(\cP(r,u_\infty)<0\), \(r>0\).
We conclude in all three cases that
\(\cP_\infty=\cP(1,u_\infty)\le0\).

\emph{Step 4. Characterization of the zero case.}
Suppose first that
\(s_\nu^{\frac{n-2}{2}}\bar u(s_\nu)\to0\) as \(\nu\to\infty\) for some
sequence \(s_\nu\downarrow0\). By \eqref{eq:annular-harnack} and
\eqref{eq:radial-derivatives},
\[
 s_\nu^{\frac{n-2}{2}}u(s_\nu\,\cdot)\to0
 \quad\text{in }C^2(A_{1/2,2})
 \quad\text{as }\nu\to\infty.
\]
Hence \(\cP_\infty=0\).

Conversely, suppose that
\[
 \liminf_{r\downarrow0}r^{\frac{n-2}{2}}\bar u(r)>0.
\]
Every rescaling limit above is nonzero.  For each fixed \(s>0\),
\[
 s^{\frac{n-2}{2}}\fint_{\Sph^{n-1}}u_\infty(s\theta)\,\dd\theta
 =
 \lim_{\nu\to\infty}
 (s_\nu s)^{\frac{n-2}{2}}\bar u(s_\nu s)
 \ge
 \liminf_{r\downarrow0}r^{\frac{n-2}{2}}\bar u(r)>0.
\]
An entire standard bubble instead satisfies
\[
 \lim_{s\downarrow0}s^{(n-2)/2}
 \fint_{\Sph^{n-1}}u_\infty(s\theta)\,\dd\theta
 =\lim_{s\to\infty}s^{(n-2)/2}
 \fint_{\Sph^{n-1}}u_\infty(s\theta)\,\dd\theta=0.
\]
Thus every rescaling limit belongs to the singular radial class described
above, for which \(\cP_\infty<0\). This proves the equivalence.
\end{proof}

\begin{lemma}\label{lem:radial-decay-removable}
If
\[
 \lim_{r\downarrow0}r^{\frac{n-2}{2}}\bar u(r)=0,
\]
then the singularity at the origin is removable.
\end{lemma}

\par
\begin{proof}
\emph{Step 1. Radial decay.}
Fix \(\varepsilon_{\rm d}\in(0,\frac{n-2}{2})\) so small that
\(p\varepsilon_{\rm d}<2\). By
\eqref{eq:averaged-radial-identity} and
\eqref{eq:averaged-curvature-error}, there is
\(r_{\rm d}\in(0,r_{\rm cv})\) such that
\[
 (r\pa_r)^2\bigl(r^{\frac{n-2}{2}}\bar u(r)\bigr)
 \ge(\frac{n-2}{2}-\varepsilon_{\rm d})^2r^{\frac{n-2}{2}}\bar u(r)
 \quad\text{for }0<r<r_{\rm d}.
\]
Therefore
{
\[
 \begin{aligned}
 &r\pa_r\bigg[
 r^{\frac{n-2}{2}-\varepsilon_{\rm d}}\Bigl\{
 r\pa_r\bigl(r^{\frac{n-2}{2}}\bar u(r)\bigr)
 -(\frac{n-2}{2}-\varepsilon_{\rm d})r^{\frac{n-2}{2}}\bar u(r)
 \Bigr\}
 \bigg]\\
{}={}&
 r^{\frac{n-2}{2}-\varepsilon_{\rm d}}\bigg[
 (r\pa_r)^2\bigl(r^{\frac{n-2}{2}}\bar u(r)\bigr)
 -(\frac{n-2}{2}-\varepsilon_{\rm d})^2
 r^{\frac{n-2}{2}}\bar u(r)
 \bigg]\ge0.
 \end{aligned}
\]
}
By \eqref{eq:radial-derivatives}, the expression inside the square brackets
tends to zero as \(r\downarrow0\).  It follows that
\[
 r\pa_r\bigl(r^{\frac{n-2}{2}}\bar u(r)\bigr)
 \ge(\frac{n-2}{2}-\varepsilon_{\rm d})r^{\frac{n-2}{2}}\bar u(r)
 \quad\text{for }0<r<r_{\rm d}.
\]
After integration from \(r\) to \(r_{\rm d}\),
\[
 r^{\frac{n-2}{2}}\bar u(r)\le C_{\rm d}r^{\frac{n-2}{2}-\varepsilon_{\rm d}}.
\]
{By \eqref{eq:annular-harnack} and
\eqref{eq:radial-derivatives},
\[
 u(x)\le C_{\rm d}|x|^{-\varepsilon_{\rm d}},
 \qquad
 |\nabla_xu(x)|\le C_{\rm d}|x|^{-1-\varepsilon_{\rm d}}.
\]
}

{
\emph{Step 2. Regularity, positivity, and bootstrap.}
By Lemma  \ref{lem:no-dirac}, after decreasing \(r_{\rm d}\) if necessary,
we have \(u,u^p\in L^1_{\mathrm{loc}}(B_{r_{\rm d}})\), and
\eqref{eq:yamabe} holds in \(\mathcal D'(B_{r_{\rm d}})\). The preceding
pointwise bounds give
\[
 u^p=O(|x|^{-p\varepsilon_{\rm d}})
 \quad\text{as }x\to0.
\]
}
Choose
\[
 \frac n2<s<\frac{n}{p\varepsilon_{\rm d}}.
\]
{Applying local elliptic estimates and Sobolev embedding,
we obtain}
\[
 u,u^p\in L^s_{\rm loc}(B_{r_{\rm d}})
 \to
 u\in W^{2,s}_{\rm loc}(B_{r_{\rm d}})
 \hookrightarrow C^{0,\gamma}_{\rm loc}(B_{r_{\rm d}})
\]
for some \(\gamma\in(0,1)\).

The extension is nonnegative and nontrivial.
Using the Harnack inequality for
\(-L_gu=n(n-2)u^{p-1}u\), we have, on a smaller ball \(B_r\), \(0<c\le u\le C\).
Since \(z\mapsto z^p\) is smooth on \((0,\infty)\), we apply
Schauder estimates of finite order to obtain \(u\in C_{\rm loc}^{d+4,\omega}(B_r)\).
Undoing the fixed conformal and coordinate changes, we recover
a positive \(C^2\) solution of the original equation. Since the original
metric is smooth, a final Schauder bootstrap makes the extension smooth
across the origin.
\end{proof}

\par
To prove \eqref{eq:lower-bound}, it remains to exclude
\begin{equation}\label{eq:contradiction-case}
 \cP_\infty=0
 \quad\text{and}\quad
 0\text{ is nonremovable}.
\end{equation}
{Applying Lemmas  \ref{lem:pohozaev-limit} and
\ref{lem:radial-decay-removable} under
\eqref{eq:contradiction-case}, we obtain}
\begin{equation}\label{eq:radial-oscillation}
 \liminf_{r\downarrow0}r^{\frac{n-2}{2}}\bar u(r)=0,
 \qquad
 \limsup_{r\downarrow0}r^{\frac{n-2}{2}}\bar u(r)>0.
\end{equation}

\subsection{Low sublevel intervals}

{
Assume \eqref{eq:contradiction-case}.  We first compare the weighted
spherical average on every connected low sublevel interval with the
corresponding homogeneous radial equation.  The resulting
estimate from above and below will locate the minimum between two
successive peaks.
}

{Recall that \(\bar u\) is given by
\eqref{eq:spherical-average}.}

\par
\begin{lemma}\label{lem:multiplicative-interval}
There exist \(\widehat\eps\in(0,\eps_{\rm cv}]\) and
\(\widehat r\in(0,r_{\rm cv}]\) with the following property. Let
\(0<\eps\le\widehat\eps\), and let \(0<r_-<r_+<\widehat r\) be the endpoints
of a connected component of
\[
 \left\{0<r<\widehat r:
 r^{\frac{n-2}{2}}\bar u(r)<\eps\right\}.
\]
Then the endpoint values satisfy \(r_-^{\frac{n-2}{2}}\bar u(r_-)
 =r_+^{\frac{n-2}{2}}\bar u(r_+)=\eps\).
There is a unique \(r_*\in(r_-,r_+)\) such that \(r_*\pa_r\bigl(r^{\frac{n-2}{2}}\bar u(r)\bigr)\big|_{r=r_*}=0\).
For \(r\in[r_-,r_+]\),
\begin{equation}\label{eq:multiplicative-two-sided}
 c\eps\left[
 \left(\frac r{r_+}\right)^{\frac{n-2}{2}}
 +\left(\frac{r_-}{r}\right)^{\frac{n-2}{2}}
 \right]
 \le r^{\frac{n-2}{2}}\bar u(r)\le
 C\eps\left[
 \left(\frac r{r_+}\right)^{\frac{n-2}{2}}
 +\left(\frac{r_-}{r}\right)^{\frac{n-2}{2}}
 \right].
\end{equation}
The value at \(r_*\) satisfies both endpoint comparisons, and its radius
satisfies the geometric mean estimate:
\begin{equation}\label{eq:multiplicative-midpoint}
 \begin{aligned}
 c\eps\left(\frac{r_*}{r_+}\right)^{\frac{n-2}{2}}
 &\le r_*^{\frac{n-2}{2}}\bar u(r_*)
 \le C\eps\left(\frac{r_*}{r_+}\right)^{\frac{n-2}{2}},\\
 c\eps\left(\frac{r_-}{r_*}\right)^{\frac{n-2}{2}}
 &\le r_*^{\frac{n-2}{2}}\bar u(r_*)
 \le C\eps\left(\frac{r_-}{r_*}\right)^{\frac{n-2}{2}},\\
 c\,r_-r_+
 &\le r_*^2\le C\,r_-r_+.
 \end{aligned}
\end{equation}
All constants depend only on \(n\), the normalized metric, \(C_0\), and
\(r_0\).
\end{lemma}

\par
\begin{proof}
\emph{Step 1. Strict convexity and the unique minimum.}
By \eqref{eq:radial-convexity},
\[
 (r\pa_r)
 \left[
 r\pa_r\bigl(r^{\frac{n-2}{2}}\bar u(r)\bigr)
 \right]>0
 \quad\text{for }r_-<r<r_+.
\]
The function \(r^{\frac{n-2}{2}}\bar u(r)\) has equal endpoint values and is smaller in
the interior.  Hence its dilation derivative changes sign exactly once,
which proves the existence and uniqueness of \(r_*\).

{The comparison principle for
\((r\pa_r)^2w=(n-2)^2w/8\), with boundary value \(\eps\) at \(r_-\) and
\(r_+\), gives}
\[
 r^{\frac{n-2}{2}}\bar u(r)
 \le C\eps\left[
 \left(\frac r{r_+}\right)^{\frac{n-2}{2\sqrt2}}
 +\left(\frac{r_-}{r}\right)^{\frac{n-2}{2\sqrt2}}
 \right].
\]

{
\emph{Step 2. The coefficient in the radial equation.}
For \(r\in(r_-,r_+)\), the scalar coefficient in the radial equation is
\[
 \widehat V(r)
 =
 \frac{r^{\frac{n+2}{2}}}{r^{\frac{n-2}{2}}\bar u(r)}
 \fint_{\Sph^{n-1}}
 \bigl[n(n-2)u(r\theta)^p-c(n)\operatorname{Scal}_g(r\theta)u(r\theta)\bigr]
 \,\dd\theta.
\]
{Using \eqref{eq:annular-harnack},
\eqref{eq:averaged-radial-identity}, and
\eqref{eq:averaged-curvature-error}, we have}
\[
 \bigl[-(r\pa_r)^2+\frac{(n-2)^2}{4}\bigr]\bigl(r^{\frac{n-2}{2}}\bar u(r)\bigr)
 =
 \widehat V(r)r^{\frac{n-2}{2}}\bar u(r)
\]
and
\[
 |\widehat V(r)|
 \le C\bigl(r^{\frac{n-2}{2}}\bar u(r)\bigr)^{p-1}+Cr^4.
\]
By the preceding estimate,
\begin{equation}\label{eq:q-ode-L1}
 \int_{r_-}^{r_+}|\widehat V(s)|\,\frac{\dd s}{s}
 \le C\eps^{p-1}+Cr_+^4.
\end{equation}
}

\emph{Step 3. Green representation and absorption.}
For \(r,s\in(r_-,r_+)\), write
\[
 m_{r,s}:=\min\{r,s\},
 \qquad
 M_{r,s}:=\max\{r,s\}.
\]
The Dirichlet Green kernel of
\(-(r\pa_r)^2+\frac{(n-2)^2}{4}\), with integration measure \(\dd s/s\), is
\[
 G_{r_-,r_+}(r,s)
 =
 \frac{
 \left[
 \left(\dfrac{m_{r,s}}{r_-}\right)^{\frac{n-2}{2}}
 -\left(\dfrac{r_-}{m_{r,s}}\right)^{\frac{n-2}{2}}
 \right]
 \left[
 \left(\dfrac{r_+}{M_{r,s}}\right)^{\frac{n-2}{2}}
 -\left(\dfrac{M_{r,s}}{r_+}\right)^{\frac{n-2}{2}}
 \right]
 }{
 (n-2)\left[
 \left(\dfrac{r_+}{r_-}\right)^{\frac{n-2}{2}}
 -\left(\dfrac{r_-}{r_+}\right)^{\frac{n-2}{2}}
 \right]
 }.
\]
Direct differentiation on \(r<s\) and \(r>s\), together with the jump of
\(r\pa_rG_{r_-,r_+}\) at \(r=s\), verifies this formula.  Uniformly in the
ratio \(r_+/r_-\),
\begin{equation}\label{eq:multiplicative-green-bounds}
 \begin{aligned}
 0\le G_{r_-,r_+}(r,s)
 &\le C\left(\frac{m_{r,s}}{M_{r,s}}\right)^{\frac{n-2}{2}},\qquad 
 \left|r\pa_rG_{r_-,r_+}(r,s)\right|
 &\le C\left(\frac{m_{r,s}}{M_{r,s}}\right)^{\frac{n-2}{2}},\\
 \left(\frac{m_{r,s}}{M_{r,s}}\right)^{\frac{n-2}{2}}
 \left[
 \left(\frac s{r_+}\right)^{\frac{n-2}{2}}
 +\left(\frac{r_-}{s}\right)^{\frac{n-2}{2}}
 \right]
 &\le C\left[
 \left(\frac r{r_+}\right)^{\frac{n-2}{2}}
 +\left(\frac{r_-}{r}\right)^{\frac{n-2}{2}}
 \right].
 \end{aligned}
\end{equation}

The homogeneous solution with value \(\eps\) at both endpoints is
\[
 h_{r_-,r_+}(r)
 =
 \eps
 \frac{
 \left(\dfrac r{r_-}\right)^{\frac{n-2}{2}}-\left(\dfrac{r_-}r\right)^{\frac{n-2}{2}}
 +\left(\dfrac{r_+}r\right)^{\frac{n-2}{2}}-\left(\dfrac r{r_+}\right)^{\frac{n-2}{2}}
 }{
 \left(\dfrac{r_+}{r_-}\right)^{\frac{n-2}{2}}
 -\left(\dfrac{r_-}{r_+}\right)^{\frac{n-2}{2}}
 }.
\]
There are constants \(0<c\le C\), depending only on \(n\), such that
\begin{equation}\label{eq:multiplicative-homogeneous-bounds}
 c\eps\left[
 \left(\frac r{r_+}\right)^{\frac{n-2}{2}}
 +\left(\frac{r_-}{r}\right)^{\frac{n-2}{2}}
 \right]
 \le h_{r_-,r_+}(r)
 \le C\eps\left[
 \left(\frac r{r_+}\right)^{\frac{n-2}{2}}
 +\left(\frac{r_-}{r}\right)^{\frac{n-2}{2}}
 \right].
\end{equation}
Choose \(\widehat\eps\le\eps_{\rm cv}\), and then choose
\(\widehat r\le r_{\rm cv}\), so that
\eqref{eq:q-ode-L1} and \eqref{eq:multiplicative-green-bounds} imply
{
\[
 C\int_{r_-}^{r_+}|\widehat V(s)|\,\frac{\dd s}{s}
 \le
 \min\left\{\frac12,\frac{c}{4C},\frac{n-2}{16}\right\}.
\]
}
{The Green representation reads
\[
 r^{\frac{n-2}{2}}\bar u(r)
 =
 h_{r_-,r_+}(r)
 +
 \int_{r_-}^{r_+}
 G_{r_-,r_+}(r,s)\widehat V(s)s^{\frac{n-2}{2}}\bar u(s)
 \,\frac{\dd s}{s}.
\]
}
{Combining
\eqref{eq:multiplicative-green-bounds}--%
\eqref{eq:multiplicative-homogeneous-bounds} with the choice of
\(\widehat\eps,\widehat r\), we obtain}
 {
\[
 \sup_{r\in[r_-,r_+]}
 \frac{r^{\frac{n-2}{2}}\bar u(r)}
 {(r/r_+)^{\frac{n-2}{2}}+(r_-/r)^{\frac{n-2}{2}}}
 \le C\eps+\frac12
 \sup_{r\in[r_-,r_+]}
 \frac{r^{\frac{n-2}{2}}\bar u(r)}
 {(r/r_+)^{\frac{n-2}{2}}+(r_-/r)^{\frac{n-2}{2}}}.
\]
}
{After absorbing the last term, we obtain the upper bound in
\eqref{eq:multiplicative-two-sided}.}
{Substituting in the Green integral, we obtain}
\[
 \left|
 r^{\frac{n-2}{2}}\bar u(r)-h_{r_-,r_+}(r)
 \right|
 \le\frac{c}{2}\eps\left[
 \left(\frac r{r_+}\right)^{\frac{n-2}{2}}
 +\left(\frac{r_-}{r}\right)^{\frac{n-2}{2}}
 \right],
\]
which proves the lower bound.

\emph{Step 4. The minimum radius and minimum value.}
{Differentiating directly, we find the exact identity}
\[
 r\pa_rh_{r_-,r_+}(r)
 =
 \frac{n-2}{2}\eps
 \frac{(r_+/r_-)^{\frac{n-2}{2}}}{(r_+/r_-)^{\frac{n-2}{2}}+1}
 \left[
 \left(\frac r{r_+}\right)^{\frac{n-2}{2}}
 -\left(\frac{r_-}r\right)^{\frac{n-2}{2}}
 \right].
\]
{The prefactor
\(\bigl[1+(r_-/r_+)^{(n-2)/2}\bigr]^{-1}\) lies in
\([1/2,1)\).}
Differentiate the Green representation at \(r_*\).  Since
\(r\pa_r(r^{\frac{n-2}{2}}\bar u)=0\) there,
{we infer from the last two estimates in
\eqref{eq:multiplicative-green-bounds}, the upper bound already proved, and
\eqref{eq:q-ode-L1} that}
\[
 \left|
 \left(\frac{r_*}{r_+}\right)^{\frac{n-2}{2}}
 -\left(\frac{r_-}{r_*}\right)^{\frac{n-2}{2}}
 \right|
 \le
 \frac12\left[
 \left(\frac{r_*}{r_+}\right)^{\frac{n-2}{2}}
 +\left(\frac{r_-}{r_*}\right)^{\frac{n-2}{2}}
 \right].
\]
{Thus
\(\frac13\le\bigl(r_*^2/(r_-r_+)\bigr)^{(n-2)/2}\le3\), which is
equivalent to}
\[
 c\,r_-r_+\le r_*^2\le C\,r_-r_+.
\]
{Applying \eqref{eq:multiplicative-two-sided} at \(r_*\), we
obtain the first two lines of \eqref{eq:multiplicative-midpoint}.}
\end{proof}

\par
\subsection{Peak and minimum scales}

{
The oscillation in \eqref{eq:radial-oscillation} determines alternating
maximum and minimum radii. Rescaling at a maximum radius produces a standard
bubble. The radial average formula below identifies its critical radius and
confines the limiting parameters to a fixed compact set. For \(s>0\) and
\(\theta\in\Sph^{n-1}\), direct substitution gives
}
\begin{equation}\label{eq:bubble-radial-form}
 s^{\frac{n-2}{2}}U_{\lambda,b}(s\theta)
 =
 \left(
 \frac{\lambda}{2(\lambda^2+|b|^2)^{1/2}}
 \right)^{\frac{n-2}{2}} \cdot
 \left[
 \frac12\left(
 \frac{s}{(\lambda^2+|b|^2)^{1/2}}
 +\frac{(\lambda^2+|b|^2)^{1/2}}{s}
 \right)
 -\frac{b \cdot\theta }{(\lambda^2+|b|^2)^{1/2}}
 \right]^{-\frac{n-2}{2}}.
\end{equation}

The peak sequence is determined by a fixed level below both the convexity
threshold and the limiting upper oscillation. Fix
\begin{equation}\label{eq:epsilon-zero-choice}
 0<\eps_0<
 \min\left\{
 \widehat\eps,
 \frac14\limsup_{r\downarrow0}r^{\frac{n-2}{2}}\bar u(r)
 \right\}.
\end{equation}
The same value \(\eps_0\) is used throughout
Sections  \ref{sec:pohozaev-dichotomy}  through \ref{sec:recurrence}.

{In the next lemma, we use \(\rho_j\) for the decreasing
maximum radii and \(\bar\rho_j\) for the intervening minimum radii.}
For the ordered superlevel components, \(r_j^+\) denotes the lower endpoint
of the component indexed by \(j\) and \(r_j^-\) the upper endpoint of the
component indexed by \(j+1\). The upper endpoint of the first component is denoted by
\(r_0^-\). Set
\begin{equation}\label{eq:minimum-depth-definition}
 a_j=\bar\rho_j^{\frac{n-2}{2}}\bar u(\bar\rho_j).
\end{equation}
This weighted minimum value is compared with the Pohozaev quantity at
\(\bar\rho_j\) in Proposition  \ref{prop:minimum-scale-limit}.

\par
\begin{lemma}\label{lem:peak-minimum-radii}
After discarding finitely many connected components of
\[
 \left\{
 0<r<\widehat r:r^{\frac{n-2}{2}}\bar u(r)>\eps_0
 \right\},
\]
each remaining component contains exactly one critical radius.  These radii
can be enumerated so that
\[
 \rho_1>\rho_2>\cdots\downarrow0.
\]
Each \(\rho_j\) is a strict local maximum of \(r^{\frac{n-2}{2}}\bar u(r)\).  Between the
components containing \(\rho_j\) and \(\rho_{j+1}\), there is exactly one
critical radius
\[
 \rho_{j+1}<\bar\rho_j<\rho_j.
\]
It is a strict local minimum, and the quantities in
\eqref{eq:minimum-depth-definition} satisfy
\begin{equation}\label{eq:peak-minimum-separation}
 a_j\to0,
 \qquad
 \frac{\bar\rho_j}{\rho_j}\to0,
 \qquad
 \frac{\rho_{j+1}}{\bar\rho_j}\to0
 \quad\text{as }j\to\infty.
\end{equation}
\end{lemma}

\par
\begin{proof}
We begin with the standard bubble in \eqref{eq:standard-bubble}.
The derivative of
\[
 \frac12\left(
 \frac{s}{(\lambda^2+|b|^2)^{1/2}}
 +\frac{(\lambda^2+|b|^2)^{1/2}}{s}
 \right)
\]
is negative for
\(0<s<(\lambda^2+|b|^2)^{1/2}\), zero at
\(s=(\lambda^2+|b|^2)^{1/2}\), and positive for larger \(s\).
After differentiating \eqref{eq:bubble-radial-form} under the spherical
integral, we conclude that
\[
 s^{\frac{n-2}{2}}\fint_{\Sph^{n-1}}U_{\lambda,b}(s\theta)\,\dd\theta
\]
is strictly increasing on
\((0,(\lambda^2+|b|^2)^{1/2})\), strictly decreasing on
\(((\lambda^2+|b|^2)^{1/2},\infty)\), and has a nondegenerate maximum at
\(s=(\lambda^2+|b|^2)^{1/2}\).

\emph{Step 1. Critical points with value at least \(\eps_0\).}
After decreasing \(\widehat r\), every critical point of
\(r\mapsto r^{\frac{n-2}{2}}\bar u(r)\) in \((0,\widehat r)\) with value at
least \(\eps_0\) is a nondegenerate maximum.
Otherwise, there are radii \(s_k\downarrow0\) such that
\[
 s_k^{\frac{n-2}{2}}\bar u(s_k)\ge\eps_0,
 \qquad
 r\pa_r\bigl(r^{\frac{n-2}{2}}\bar u(r)\bigr)\big|_{r=s_k}=0,
 \qquad
 (r\pa_r)^2\bigl(r^{\frac{n-2}{2}}\bar u(r)\bigr)\big|_{r=s_k}\ge0.
\]
By Lemma  \ref{lem:radial-harnack}, after taking a further subsequence,
\[
 s_k^{\frac{n-2}{2}}u(s_k\,\cdot)
 \to U
 \quad\text{in }C^2_{\rm loc}(\R^n\setminus\{0\})
 \quad\text{as }k\to\infty,
\]
where \(U\not\equiv0\) and \(\cP(1,U)=\cP_\infty=0\). By the
Caffarelli--Gidas--Spruck classification cited in
Lemma  \ref{lem:pohozaev-limit}, \(U\) is a standard bubble.
{Passing to the limit in the critical point condition, we obtain}
\[
 \left.s\partial_s\left(s^{(n-2)/2}
 \fint_{\Sph^{n-1}}U(s\theta)\,\dd\theta\right)\right|_{s=1}=0.
\]
{At this point, combining \eqref{eq:bubble-radial-form} with
the limiting second derivative condition, we obtain the contradiction}
\[
 0>\left.(s\partial_s)^2\left(s^{(n-2)/2}
 \fint_{\Sph^{n-1}}U(s\theta)\,\dd\theta\right)\right|_{s=1}\ge0.
\]
At a radius \(s\in(0,\widehat r)\) with
\(s^{\frac{n-2}{2}}\bar u(s)=\eps_0\),
{by \eqref{eq:radial-convexity},}
\[
 (r\pa_r)^2\bigl(r^{\frac{n-2}{2}}\bar u(r)\bigr)\big|_{r=s}
 \ge\frac{(n-2)^2}{8}\eps_0>0.
\]
By Step  1, the derivative
\(r\pa_r(r^{\frac{n-2}{2}}\bar u(r))\) is nonzero at \(r=s\).

\emph{Step 2. Choice of the maxima and minima.}
By \eqref{eq:radial-oscillation}, decrease \(\widehat r\) further such that \(\widehat r^{\frac{n-2}{2}}\bar u(\widehat r)<\eps_0\).
The nonzero derivatives at level \(\eps_0\) make the crossings locally
finite. {It follows from \eqref{eq:radial-oscillation} and
\eqref{eq:epsilon-zero-choice} that the ordered components are}
\[
 \{0<r<\widehat r:r^{\frac{n-2}{2}}\bar u(r)>\eps_0\}
 =\bigcup_{j\ge1}(r_j^+,r_{j-1}^-),
 \qquad r_j^-\le r_j^+.
\]
Choose \(\rho_j\in[r_j^+,r_{j-1}^-]\) such that \(r^{(n-2)/2}\bar u(r)\) attains its maximum on this interval at \(r=\rho_j\).
The endpoint values are \(\eps_0\), so
\[
 r_j^+<\rho_j<r_{j-1}^-,
 \qquad
 \left.r\pa_r(r^{\frac{n-2}{2}}\bar u(r))\right|_{r=\rho_j}=0.
\]
{By Step  1, there cannot be a second critical point: two
nondegenerate maxima would have a critical minimum between them.} Hence
\(\rho_1>\rho_2>\cdots\downarrow0\).

{At the adjacent endpoints, by Step  1 and the superlevel
inequalities, we have}
\[
 \left.r\pa_r(r^{\frac{n-2}{2}}\bar u(r))\right|_{r=r_j^-}
 <0<
 \left.r\pa_r(r^{\frac{n-2}{2}}\bar u(r))\right|_{r=r_j^+}.
\]
Thus \(r_j^-<r_j^+\). An interior point at level \(\eps_0\) would have
zero derivative, contrary to Step  1. Therefore
\[
 r^{\frac{n-2}{2}}\bar u(r)<\eps_0,\qquad
 (r\pa_r)^2(r^{\frac{n-2}{2}}\bar u(r))>0,
 \qquad r_j^-<r<r_j^+,
\]
where the second inequality is \eqref{eq:radial-convexity}.
{By the endpoint derivative signs and strict convexity in
\(\log r\), we have}
\[
 \left\{r\in(r_j^-,r_j^+):
 r\pa_r(r^{\frac{n-2}{2}}\bar u(r))=0\right\}
 =\{\bar\rho_j\}.
\]
This point is a strict minimum, and \(\rho_{j+1}<r_j^-<\bar\rho_j<r_j^+<\rho_j\).

\emph{Step 3. Decay of the minimum values.}
If \(a_j\not\to0\) as \(j\to\infty\), there are \(j_k\to\infty\) and
\(c_0>0\) such that \(a_{j_k}\ge c_0\). After taking a further subsequence,
\[
 \bar\rho_{j_k}^{\frac{n-2}{2}}
 u(\bar\rho_{j_k}\,\cdot)
 \to U
 \quad\text{in }C^2_{\rm loc}(\R^n\setminus\{0\})
 \quad\text{as }k\to\infty,
\]
where \(U\) is again a standard bubble with vanishing Pohozaev quantity.
{Passing to the limit in the critical point condition, we obtain}
\[
 \left.s\partial_s\left(s^{(n-2)/2}
 \fint_{\Sph^{n-1}}U(s\theta)\,\dd\theta\right)\right|_{s=1}=0.
\]
{On the other hand, by \eqref{eq:radial-convexity},}
\[
 (r\pa_r)^2\bigl(r^{\frac{n-2}{2}}\bar u(r)\bigr)
 \big|_{r=\bar\rho_{j_k}}
 \ge\frac{(n-2)^2c_0}{8}.
\]
{Passing to the limit and using
\eqref{eq:bubble-radial-form} at its critical point, we obtain}
\[
 0>\left.(s\partial_s)^2\left(s^{(n-2)/2}
 \fint_{\Sph^{n-1}}U(s\theta)\,\dd\theta\right)\right|_{s=1}
 \ge\frac{(n-2)^2c_0}{8}>0.
\]
This contradiction proves \(a_j \to 0\) as \(j\to\infty\).

\emph{Step 4. Separation of the three radii.}
If \(\bar\rho_j/\rho_j\not\to0\), choose \(j_k\to\infty\) such that
\[
 \frac{\bar\rho_{j_k}}{\rho_{j_k}}
 \to\tau\in(0,1]
 \quad\text{as }k\to\infty.
\]
After taking a further subsequence,
\[
 \rho_{j_k}^{\frac{n-2}{2}}u(\rho_{j_k}\,\cdot)
 \to U
 \quad\text{in }C^2_{\rm loc}(\R^n\setminus\{0\})
 \quad\text{as }k\to\infty,
\]
where \(U\) is a positive standard bubble. The weighted spherical averages
at \(\bar\rho_{j_k}/\rho_{j_k}\) equal
\(a_{j_k}\to0\) as \(k\to\infty\). This contradicts
positivity of the limiting bubble at \(s=\tau\). Hence \(\frac{\bar\rho_j}{\rho_j}\to0\) as \(j\to\infty\).
If the second ratio did not tend to zero, a further subsequence would satisfy
\[
 \frac{\rho_{j_k+1}}{\bar\rho_{j_k}}
 \to\tau'\in(0,1]
 \quad\text{as }k\to\infty.
\]
{Rescale now by \(\rho_{j_k+1}\).  After passing to a
subsequence, the rescaled solutions converge to a positive standard bubble,
whereas their weighted spherical averages at
\(\bar\rho_{j_k}/\rho_{j_k+1}\to1/\tau'\) equal \(a_{j_k}\to0\).
This is impossible.} Therefore \(\frac{\rho_{j+1}}{\bar\rho_j}\to0\) as \(j\to\infty\).
This proves
\eqref{eq:peak-minimum-separation}.
\end{proof}

Using the Harnack ratio, we confine the translation and
dilation parameters of every limiting bubble to a fixed compact set.
Let \(C_{\rm H}\) be the Harnack
constant in
\eqref{eq:annular-harnack}, and set
\begin{equation}\label{eq:bubble-parameter-set}
 \begin{gathered}
 b_*=
 \frac{C_{\rm H}^{\frac{2}{n-2}}-1}{C_{\rm H}^{\frac{2}{n-2}}+1},
 \qquad
 \lambda_*=(1-b_*^2)^{1/2},\\
 \mathcal K=\bigl\{
 (\lambda,b)\in(0,\infty)\times\R^n:
 \tfrac12\lambda_*\le\lambda\le2,\;
 |b|\le\tfrac12(1+b_*)
 \bigr\}.
 \end{gathered}
\end{equation}

\par
\begin{lemma}\label{lem:peak-bubble-compactness}
The sequence \(\rho_j^{\frac{n-2}{2}}u(\rho_j\,\cdot)\) is precompact in
\(C^2_{\rm loc}(\R^n\setminus\{0\})\), and every limit is
\(U_{\lambda,b}\) for parameters satisfying
\begin{equation}\label{eq:peak-limit-parameter-bounds}
 \lambda^2+|b|^2=1,\qquad
 |b|\le b_*,\qquad
 \lambda_*\le\lambda\le1.
\end{equation}
In particular, every limit parameter lies in the interior of
\(\mathcal K\) specified in \eqref{eq:bubble-parameter-set}.
\end{lemma}

\par
\begin{proof}
At each peak radius,
\[
 \eps_0<\rho_j^{\frac{n-2}{2}}\bar u(\rho_j)\le C.
\]
For fixed \(0<a<A<\infty\) and \(\gamma\in(0,1)\),
{it follows from Lemma  \ref{lem:radial-harnack} that}
\[
 \|\rho_j^{(n-2)/2}u(\rho_j\,\cdot)\|_{C^{2,\gamma}(A_{a,A})}
 \le C_{a,A,\gamma}.
\]
Every subsequence therefore has a further subsequence, still indexed by
\(j\), such that
 {
\[
 \begin{gathered}
 \rho_j^{(n-2)/2}u(\rho_j\,\cdot)\to U
 \quad\text{in }C^2_{\rm loc}(\R^n\setminus\{0\}),\\
 \fint_{\Sph^{n-1}}U(\theta)\,\dd\theta\ge\eps_0>0,
 \qquad \cP(1,U)=\cP_\infty=0.
 \end{gathered}
\]
}
 {By the Caffarelli--Gidas--Spruck classification cited in
Lemma  \ref{lem:pohozaev-limit}, \(U=U_{\lambda,b}\).}

{Combining the maximum condition at \(\rho_j\) with
\eqref{eq:bubble-radial-form}, we obtain the first equality in}
\eqref{eq:peak-limit-parameter-bounds}. The Harnack inequality passes to the limit on
\(\pa B_1\), and
\[
 \frac{
 \sup_{\theta\in\Sph^{n-1}}U_{\lambda,b}(\theta)
 }{
 \inf_{\theta\in\Sph^{n-1}}U_{\lambda,b}(\theta)
 }
 =
 \left(\frac{1+|b|}{1-|b|}\right)^{\frac{n-2}{2}}
 \le C_{\rm H}.
\]
{The remaining bounds in
\eqref{eq:peak-limit-parameter-bounds} follow from the definitions in
\eqref{eq:bubble-parameter-set}.}
\end{proof}

\par
\subsection{Limits at minimum scales and endpoint capacities}

{
We first rescale at a minimum radius and identify the limiting harmonic
function and the corresponding Pohozaev asymptotic.  We then compare the
crossing radii with the adjacent peak radii and derive the capacity estimates
and summability needed in Section  \ref{sec:recurrence}.
}

{For the radii in Lemma  \ref{lem:peak-minimum-radii}, define
the negative Pohozaev value}
\begin{equation}\label{eq:mu-definition}
 \mu_j=-\cP(\bar\rho_j,u).
\end{equation}
Set \(e_1=(1,0,\ldots,0)\in\partial B_1\).

\begin{proposition}\label{prop:minimum-scale-limit}
{
For every \(\gamma\in(0,1)\),
\begin{equation}\label{eq:minimum-harmonic-limit}
 \frac{u(\bar\rho_jy)}{u(\bar\rho_je_1)}
 \to
 \frac12\bigl(1+|y|^{2-n}\bigr)
 \quad\text{in }C^{2,\gamma}(A_{1/2,2})
\quad\text{as }j\to\infty,
\end{equation}
and
\begin{equation}\label{eq:mu-minimum-depth}
 \frac{\mu_j}{a_j^2}
 \to
 \frac{(n-2)^2}{8}|\Sph^{n-1}|,
 \qquad
 \mu_j>0\quad\text{for all sufficiently large }j.
\end{equation}
}
\end{proposition}

\par
\begin{proof}
\emph{Step 1. Harmonic limit at the minimum scale.}
Fix \(\gamma\in(0,1)\).
{For this proof, define}
\[
 w_j(y)=\frac{u(\bar\rho_jy)}{u(\bar\rho_je_1)}.
\]
By \eqref{eq:annular-harnack},
\[
 C^{-1}a_j
 \le\bar\rho_j^{\frac{n-2}{2}}u(\bar\rho_je_1)
 \le Ca_j.
\]
The equation for \(w_j\) on every fixed annulus is
\[
 -\Delta_{g(\bar\rho_j\,\cdot)}w_j
 +c(n)\bar\rho_j^2\operatorname{Scal}_g(\bar\rho_jy)w_j
 =
 n(n-2)
 \bigl[\bar\rho_j^{\frac{n-2}{2}}u(\bar\rho_je_1)\bigr]^{p-1}w_j^p.
\]
{In view of the comparison in
\eqref{eq:annular-harnack}, the bounds in normal coordinates
\eqref{eq:normal-orders}, and \(a_j\to0\), we have the following
convergences on every fixed compact annulus as \(j\to\infty\):}
 {
\[
 \begin{aligned}
 g(\bar\rho_j\,\cdot)
 &\to(\delta_{ij})
 &&\text{in }C^{2,\gamma},\\
 \bar\rho_j^2\operatorname{Scal}_g(\bar\rho_j\,\cdot)
 &\to0
 &&\text{in }C^{0,\gamma},
 \end{aligned}
\]
}
Moreover, \(n(n-2)\bigl[\bar\rho_j^{\frac{n-2}{2}}u(\bar\rho_je_1)\bigr]^{p-1} \to 0\).

{Applying the Harnack inequality and interior estimates, we
see that every subsequence \(j_k\to\infty\) has a further subsequence
\(j_{k_\ell}\) such that}
\[
 w_{j_{k_\ell}}\to w_\infty
 \quad\text{in }C^2_{\rm loc}(\R^n\setminus\{0\})
 \quad\text{as }\ell\to\infty,
\]
where \(w_\infty\) is positive and harmonic.
{By B\^ocher's theorem at zero and after Kelvin inversion at
infinity,}
\[
 w_\infty(y)=c_1+c_2|y|^{2-n},\qquad c_1,c_2\ge0.
\]
{Passing to the limit in the normalization at \(e_1\) and the
minimum condition, we obtain}
\[
 c_1+c_2=w_\infty(e_1)=1, \qquad 
 0=\left.s\pa_s\left[
 s^{\frac{n-2}{2}}\fint_{\Sph^{n-1}}w_\infty(s\theta)\,\dd\theta
 \right]\right|_{s=1}
 =\frac{n-2}{2}(c_1-c_2).
\]
Hence \(c_1=c_2=1/2\), and \(w_\infty(y)=\frac12\bigl(1+|y|^{2-n}\bigr)\).
The limit is unique, so the entire sequence converges in
\(C^2_{\rm loc}(\R^n\setminus\{0\})\) as \(j\to\infty\).

{Fix \(\gamma'\in(\gamma,1)\). The Harnack inequality on
\(A_{1/3,3}\), interior \(W^{2,s}\) estimates with \(s>n\) on
\(A_{2/5,5/2}\), and Schauder estimates on \(A_{1/2,2}\) give}
\[
 \|w_j\|_{C^{2,\gamma'}(A_{1/2,2})}\le C_{\gamma'}.
\]
Interpolation and the \(C^2\) convergence imply
\[
 \|w_j-w_\infty\|_{C^{2,\gamma}(A_{1/2,2})}
 \le C\|w_j-w_\infty\|_{C^2(A_{1/2,2})}^{1-\gamma/\gamma'}
 \to0.
\]
This is \eqref{eq:minimum-harmonic-limit}.

\emph{Step 2. Pohozaev asymptotic.}
Since the limit equals \(1\) on \(\pa B_1\),
\[
 \frac{\bar\rho_j^{\frac{n-2}{2}}u(\bar\rho_je_1)}{a_j}
 =
 \left(
 \fint_{\Sph^{n-1}}w_j(\theta)\,\dd\theta
 \right)^{-1}
 \to1
 \quad\text{as }j\to\infty.
\]
The Pohozaev expression \eqref{eq:pohozaev} is invariant under the critical
rescaling \(u(x)\mapsto r^{\frac{n-2}{2}}u(rx)\).  Thus
\[
 \cP(\bar\rho_j,u)
 =
 \cP\bigl(1,\bar\rho_j^{\frac{n-2}{2}}u(\bar\rho_j\,\cdot)\bigr).
\]
{By \eqref{eq:minimum-harmonic-limit}, on \(\pa B_1\) we have}
\[
 w_j\to1,
 \qquad
 \pa_rw_j\to-\frac{n-2}{2},
 \qquad
 \nabla_\theta w_j\to0
 \quad\text{as }j\to\infty
\]
uniformly on \(\partial B_1\). In addition,
\[
 \bigl[\bar\rho_j^{\frac{n-2}{2}}u(\bar\rho_je_1)\bigr]^{p+1}
 =
 o\left(
 \bigl[\bar\rho_j^{\frac{n-2}{2}}u(\bar\rho_je_1)\bigr]^2
 \right)
 \quad\text{as }j\to\infty.
\]
{Substituting these limits in \eqref{eq:pohozaev}, we obtain
\[
 \cP(\bar\rho_j,u)
 =
 -\frac{(n-2)^2}{8}|\Sph^{n-1}|
 \bigl[\bar\rho_j^{\frac{n-2}{2}}u(\bar\rho_je_1)\bigr]^2
 \bigl(1+o(1)\bigr).
\]
Here \(o(1)\to0\) as \(j\to\infty\). Combining this asymptotic with
\(\bar\rho_j^{\frac{n-2}{2}}u(\bar\rho_je_1)/a_j\to1\) gives
\eqref{eq:mu-minimum-depth}.}
\end{proof}

{
After rescaling by the peak radius \(\rho_j\), the adjacent minimum radii
become the endpoints of an annulus.  For \(j\ge2\), define
}
\begin{equation}\label{eq:annular-radii}
 \sigma_j=\frac{\bar\rho_j}{\rho_j},
 \qquad
 R_j=\frac{\bar\rho_{j-1}}{\rho_j}.
\end{equation}
Since \(\rho_jR_j=\bar\rho_{j-1}\) is the largest original radius
represented in \(A_{\sigma_j,R_j}\), define
\begin{equation}\label{eq:metric-perturbation-scale}
 \eta_j=(\rho_jR_j)^2=\bar\rho_{j-1}^2.
\end{equation}
By \eqref{eq:peak-minimum-separation}, \(\sigma_j\to0, R_j\to\infty\) as \(j\to\infty\).

{The endpoint relations from
Lemma  \ref{lem:peak-minimum-radii} are}
\begin{equation}\label{eq:level-crossing-radii}
 \begin{gathered}
 \rho_{j+1}<r_j^-<\bar\rho_j<r_j^+<\rho_j,\\
 (r_j^-)^{\frac{n-2}{2}}\bar u(r_j^-)
 =(r_j^+)^{\frac{n-2}{2}}\bar u(r_j^+)=\eps_0,\\
 r^{\frac{n-2}{2}}\bar u(r)<\eps_0
 \qquad r_j^-<r<r_j^+.
 \end{gathered}
\end{equation}
The order of these radii is
\[
 \underset{\text{peak}}{\rho_j}
 >\underset{\text{crossing}}{r_j^+}
 >\underset{\text{minimum}}{\bar\rho_j}
 >\underset{\text{crossing}}{r_j^-}
 >\underset{\text{next peak}}{\rho_{j+1}}.
\]

\begin{proposition}\label{prop:radial-scales}
{
There exist \(c\in(0,1]\), \(C\ge1\), \(C_{\mathrm{sc}}>1\), and
\(j_0\ge2\). The constants \(c\), \(C\), and \(C_{\mathrm{sc}}\) depend only on
\(n\), \(C_0\), \(r_0\), \(\eps_0\), and the fixed normalized metric,
while \(j_0\) may also depend on the fixed solution. For every \(j\ge j_0\),
\begin{equation}\label{eq:crossing-peak-comparison}
 C_{\mathrm{sc}}^{-1}\rho_j\le r_j^+<\rho_j,
 \qquad
 \rho_{j+1}<r_j^-\le C_{\mathrm{sc}}\rho_{j+1}.
\end{equation}
\begin{equation}\label{eq:minimum-geometric-mean}
 c\rho_{j+1}\rho_j
 \le\bar\rho_j^2
 \le C\rho_{j+1}\rho_j,
\end{equation}
\begin{equation}\label{eq:minimum-depth-ratio}
 c\left(\frac{\rho_{j+1}}{\rho_j}\right)^{\frac{n-2}{2}}
 \le a_j^2
 \le C\left(\frac{\rho_{j+1}}{\rho_j}\right)^{\frac{n-2}{2}},
\end{equation}
and
\begin{equation}\label{eq:mu-peak-ratio}
 C^{-1}\left(\frac{\rho_{j+1}}{\rho_j}\right)^{\frac{n-2}{2}}
 \le\mu_j
 \le
 C\left(\frac{\rho_{j+1}}{\rho_j}\right)^{\frac{n-2}{2}}.
\end{equation}
Moreover,
\begin{equation}\label{eq:capacity-scales}
 C^{-1}\sigma_j^{n-2}
 \le\mu_j\le C\sigma_j^{n-2},
 \qquad
 C^{-1}R_j^{2-n}
 \le\mu_{j-1}\le CR_j^{2-n},
\end{equation}
\begin{equation}\label{eq:metric-scale-summability}
 \eta_j\le C\rho_j,
 \qquad
 \sum_{j\ge j_0}\eta_j<\infty,
\end{equation}
and
\begin{equation}\label{eq:peak-halving}
 \rho_{j+1}\le\frac12\rho_j.
\end{equation}
}
\end{proposition}

\begin{proof}
\emph{Step 1. Level crossing radii.}
By compactness of the parameter set in
\eqref{eq:bubble-parameter-set}, there is \(C_1\ge1\) such that, for every
limit in Lemma  \ref{lem:peak-bubble-compactness},
\[
 s^{\frac{n-2}{2}}
 \fint_{\Sph^{n-1}}U_{\lambda,b}(s\theta)\,\dd\theta
 \le C_1\min\{s^{\frac{n-2}{2}},s^{-\frac{n-2}{2}}\}.
\]
Choose \(C_{\mathrm{sc}}>1\) such that \(C_1C_{\mathrm{sc}}^{-\frac{n-2}{2}}
 <\frac{\eps_0}{2}\).
If \(r_{j_k}^+<C_{\mathrm{sc}}^{-1}\rho_{j_k}\) along a subsequence
\(j_k\to\infty\), then
\(C_{\mathrm{sc}}^{-1}\rho_{j_k}\) lies in the superlevel component containing
\(\rho_{j_k}\).
After rescaling by \(\rho_{j_k}\) and taking a bubble limit, we obtain
\[
 \begin{aligned}
 \eps_0
 &\le\lim_{k\to\infty}
 (C_{\mathrm{sc}}^{-1}\rho_{j_k})^{\frac{n-2}{2}}
 \bar u(C_{\mathrm{sc}}^{-1}\rho_{j_k})\\
 &=C_{\mathrm{sc}}^{-\frac{n-2}{2}}
 \fint_{\Sph^{n-1}}U_{\lambda,b}(C_{\mathrm{sc}}^{-1}\theta)\,\dd\theta
 \le C_1 C_{\mathrm{sc}}^{-\frac{n-2}{2}}<\frac{\eps_0}{2}.
 \end{aligned}
\]
{If \(r_{j_k}^->C_{\mathrm{sc}}\rho_{j_k+1}\) along a subsequence,
then \(C_{\mathrm{sc}}\rho_{j_k+1}\) lies in the superlevel component containing
\(\rho_{j_k+1}\).  Rescaling by \(\rho_{j_k+1}\) and passing to a bubble
limit gives}
\[
 \eps_0
 \le
 C_{\mathrm{sc}}^{\frac{n-2}{2}}
 \fint_{\Sph^{n-1}}U_{\lambda,b}(C_{\mathrm{sc}}\theta)\,\dd\theta
 \le C_1C_{\mathrm{sc}}^{-\frac{n-2}{2}}
 <\frac{\eps_0}{2},
\]
{which is impossible.}
This proves \eqref{eq:crossing-peak-comparison}.
Apply Lemma  \ref{lem:multiplicative-interval} to
\((r_j^-,r_j^+)\), with \(r_*=\bar\rho_j\).
{Using \eqref{eq:multiplicative-midpoint} and
\eqref{eq:crossing-peak-comparison}, we obtain
\eqref{eq:minimum-geometric-mean} and
\eqref{eq:minimum-depth-ratio}.}

\emph{Step 2. Endpoint comparisons and summability.}
{Combining \eqref{eq:mu-minimum-depth} with
\eqref{eq:minimum-depth-ratio}, we obtain
\eqref{eq:mu-peak-ratio}.}

By \eqref{eq:peak-minimum-separation} and
\eqref{eq:annular-radii},
\[
 \sigma_j\to0,
 \qquad
 R_j\to\infty
 \quad\text{as }j\to\infty.
\]
{Dividing \eqref{eq:minimum-geometric-mean} by \(\rho_j^2\),
we obtain}
\[
 c\frac{\rho_{j+1}}{\rho_j}
 \le\sigma_j^2
 \le C\frac{\rho_{j+1}}{\rho_j}.
\]
{Replacing \(j\) by \(j-1\) and dividing by \(\rho_j^2\), we
obtain}
\[
 c\frac{\rho_{j-1}}{\rho_j}
 \le R_j^2
 \le C\frac{\rho_{j-1}}{\rho_j}.
\]
{Substituting the preceding estimates from above and below for
\(\sigma_j^2\) and \(R_j^2\) into \eqref{eq:mu-peak-ratio}, we obtain
\eqref{eq:capacity-scales}.}

{Finally, by \eqref{eq:minimum-geometric-mean} with \(j-1\)
in place of \(j\), we have}
\[
 \eta_j=\bar\rho_{j-1}^2
 \le C\rho_j\rho_{j-1}
 \le C\rho_j.
\]
By \eqref{eq:peak-minimum-separation}, after increasing \(j_0\),
\eqref{eq:peak-halving} holds.
{Summing the resulting geometric series, we obtain
\eqref{eq:metric-scale-summability}.}
\end{proof}

{On the rescaled annulus, \eqref{eq:normal-orders} and
\eqref{eq:metric-perturbation-scale} give}
\[
 |h(\rho_jy)|
 \le C\rho_j^2|y|^2
 \le C(\rho_jR_j)^2
 =C\eta_j,
 \qquad y\in A_{\sigma_j,R_j}.
\]
{Since
\(\eta_j=(\rho_jR_j)^2\), it follows from
\eqref{eq:metric-scale-summability} that}
\begin{equation}\label{eq:outer-radius-rho-bound}
 R_j\le C\rho_j^{-1/2}.
\end{equation}

\section{Rescaled annuli and the exact Pohozaev increment}
\label{sec:euclidean-annuli}

{
Section  \ref{sec:pohozaev-dichotomy} fixed the adjacent minimum radii and
their capacity estimates.  We now rescale the region between them, express
the Pohozaev change as one symmetric bilinear form, and prove the bubble
compactness and comparison on the entire annulus needed for the later modulation.
}

For each \(j\), rescale
\(A_{\bar\rho_j,\bar\rho_{j-1}}\) by \(\rho_j\). Define the coordinates
\(y\), the scalar function \(v_j\), and the metric \(g_j\), for
\(0<|y|<r_0/\rho_j\), by
\begin{equation}\label{eq:annular-rescaling}
 y=\frac{x}{\rho_j},\qquad
 v_j(y)=\rho_j^{\frac{n-2}{2}}u(\rho_jy),\qquad
 (g_j)_{kl}(y)=g_{kl}(\rho_jy).
\end{equation}
Thus \(g_j=\rho_j^{-2}(y\mapsto\rho_jy)^*g\), expressed in the
\(y\) coordinates, and
\[
 \rho_j A_{\sigma_j,R_j}=A_{\bar\rho_j,\bar\rho_{j-1}}\subset B_1
\]
for all sufficiently large \(j\).  All derivatives, measures, and normals
in this section are Euclidean unless another metric is displayed.
The parameters \((\lambda_j,b_j)\) are selected later by
Proposition  \ref{prop:final-modulation}.

On \(A_{\sigma_j,R_j}\), with
\(L_{g_j}=\Delta_{g_j}-c(n)\operatorname{Scal}_{g_j}\), the conformal normal conditions and a direct scaling give
\begin{equation}\label{eq:scaled-gauge-equation}
 \det g_j=1,\qquad
 \sum_{l=1}^ng_j^{kl}(y)y^l=y^k,\qquad
 \operatorname{Scal}_{g_j}(y)=\rho_j^2\operatorname{Scal}_g(\rho_jy),\qquad
 -L_{g_j}v_j=n(n-2)v_j^p.
\end{equation}
Indeed, \(h(x)x=0\) implies \(g^{-1}(x)x=x\), and
\[
 L_{g_j}\bigl(\rho_j^{\frac{n-2}{2}}u(\rho_j\,\cdot)\bigr)
 =\rho_j^{\frac{n+2}{2}}(L_gu)(\rho_j\,\cdot)
\]
because \(\frac{n+2}{2}=\frac{n-2}{2}p\). All spheres in
\(A_{\sigma_j,R_j}\) are centered at \(y=0\).

Use the dilation derivative \(D_{\mathrm{di}}\) defined in
\eqref{eq:dilation-operator}.
Define the following symmetric bilinear form for the metric error operator
\(\Delta-L_{g_j}\).
For \(f_1,f_2\in C^2(\overline{A_{\sigma_j,R_j}})\), define
 {
\begin{equation}\label{eq:symmetric-dilation-form}
  B_j(f_1,f_2)=\frac12\int_{A_{\sigma_j,R_j}}\bigl[
 (D_{\mathrm{di}}f_1)(\Delta-L_{g_j})f_2
 {}+(D_{\mathrm{di}}f_2)(\Delta-L_{g_j})f_1\bigr]\,\dd y.
\end{equation}
}
Set
 {
\begin{equation}\label{eq:Qj-definition}
 Q_j= B_j(v_j,v_j).
\end{equation}
}

\par
\begin{proposition}\label{prop:annular-increment}
For every sufficiently large \(j\),
\begin{equation}\label{eq:exact-Qj-increment}
 Q_j=\mu_j-\mu_{j-1}.
\end{equation}
\end{proposition}

\par
\begin{proof}
{Under critical scaling,}
\[
 \cP(r,v_j)=\cP(\rho_jr,u)
 \qquad \sigma_j\le r\le R_j.
\]
{Applying Lemma  \ref{lem:exact-pohozaev} on
\(A_{\sigma_j,R_j}\), we obtain}
 {
\[
 Q_j
 =\cP(R_j,v_j)-\cP(\sigma_j,v_j)
 =\cP(\bar\rho_{j-1},u)-\cP(\bar\rho_j,u)
 =\mu_j-\mu_{j-1}.
\]
}
The Pohozaev increment is the outer boundary contribution minus the inner boundary
contribution: the outward
Euclidean normal of \(A_{\sigma_j,R_j}\) is \(y/R_j\) on
\(\partial B_{R_j}\) and \(-y/\sigma_j\) on
\(\partial B_{\sigma_j}\).
\end{proof}

\par
Since \(\det g_j=1\),
\begin{equation}\label{eq:scaled-metric-error}
 (\Delta-L_{g_j})v_j
 =-\sum_{k,l=1}^n
 \partial_k\bigl((g_j^{kl}-\delta^{kl})\partial_lv_j\bigr)
 +c(n)\operatorname{Scal}_{g_j}v_j.
\end{equation}
This is the coordinate formula following from \(\det g_j=1\).
The subsequent Pohozaev calculations use
\eqref{eq:scaled-metric-error} directly.

 We next compare \(v_j\) with the compact family of bubbles. The first lemma gives convergence on each fixed compact annulus.

\par
\begin{lemma}\label{lem:distance-to-bubble-family}
For every compact \(D\subset\R^n\setminus\{0\}\),
\[
 \inf_{(\lambda,b)\in\mathcal K}
 \norm{v_j-U_{\lambda,b}}_{C^2(D)}\to0
 \quad\text{as }j\to\infty.
\]
\end{lemma}

\par
\begin{proof}
If the assertion failed, there would be a compact
\(D\subset\R^n\setminus\{0\}\), a constant \(c_0>0\), and
\(j_k\to\infty\) such that
\[
 \inf_{(\lambda,b)\in\mathcal K}
 \norm{v_{j_k}-U_{\lambda,b}}_{C^2(D)}\ge c_0.
\]
By Lemma  \ref{lem:peak-bubble-compactness}, after taking a further
subsequence,
\[
 v_{j_k}\to U_{\lambda,b}
 \quad\text{in }C^2_{\rm loc}(\R^n\setminus\{0\})
 \quad\text{as }k\to\infty
\]
for some \((\lambda,b)\in\mathcal K\), a contradiction.
\end{proof}

\par
To choose the parameters by integrals over \(A_{\sigma_j,R_j}\), we first
compare \(v_j\) with the entire compact family of bubbles.
\begin{lemma}\label{lem:annular-bubble-comparison}
There is \(C\ge1\) such that, for all sufficiently large \(j\) and every
\((\lambda,b)\in\mathcal K\),
\[
 C^{-1}U_{\lambda,b}(y)\le v_j(y)
 \le CU_{\lambda,b}(y),\qquad y\in A_{\sigma_j,R_j}.
\]
\end{lemma}

\par
\begin{proof}
Write
\[
 \overline v_j(r)=\fint_{\Sph^{n-1}}v_j(r\theta)\,\dd\theta
 =\rho_j^{\frac{n-2}{2}}\bar u(\rho_jr).
\]
The radii \(r_j^\pm\) satisfy
\eqref{eq:level-crossing-radii}.
{Using \eqref{eq:crossing-peak-comparison} and
\eqref{eq:minimum-geometric-mean}, we obtain}
\[
 C^{-1}\rho_j\le r_j^+\le\rho_j,
 \qquad
 C^{-1}\sigma_j^2\le\frac{r_j^-}{\rho_j}\le C\sigma_j^2.
\]
{Applying the estimate from above and below in
\eqref{eq:multiplicative-two-sided} between
\(\bar\rho_j\) and \(r_j^+\), and then setting \(s=\rho_jr\), we obtain}
\begin{equation}\label{eq:inner-annular-average}
 C^{-1}\bigl(1+\sigma_j^{n-2}r^{2-n}\bigr)
 \le\overline v_j(r)\le
 C\bigl(1+\sigma_j^{n-2}r^{2-n}\bigr),
 \qquad \sigma_j\le r\le\frac{r_j^+}{\rho_j}.
\end{equation}

For the adjacent gap, use the already defined radii
\(r_{j-1}^-<\bar\rho_{j-1}<r_{j-1}^+\).
{It follows from \eqref{eq:crossing-peak-comparison} and
\eqref{eq:minimum-geometric-mean}, with \(j-1\) in place of \(j\), that}
\[
 \rho_j\le r_{j-1}^-\le C\rho_j,
 \qquad
 C^{-1}R_j^2\le\frac{r_{j-1}^+}{\rho_j}\le CR_j^2.
\]
{Applying the estimate from above and below in
\eqref{eq:multiplicative-two-sided} between
\(r_{j-1}^-\) and \(\bar\rho_{j-1}\), we find}
\begin{equation}\label{eq:outer-annular-average}
 C^{-1}\bigl(r^{2-n}+R_j^{2-n}\bigr)
 \le\overline v_j(r)\le
 C\bigl(r^{2-n}+R_j^{2-n}\bigr),
 \qquad \frac{r_{j-1}^-}{\rho_j}\le r\le R_j.
\end{equation}
On the remaining interval,
\[
 \frac{r_j^+}{\rho_j}\le r\le\frac{r_{j-1}^-}{\rho_j},
 \qquad C^{-1}\le r\le C.
\]
{From the superlevel condition and
\eqref{eq:critical-upper}, we have}
\[
 \eps_0r^{-(n-2)/2}\le\overline v_j(r)
 \le Cr^{-(n-2)/2},\qquad c\le\overline v_j(r)\le C.
\]
{Combining these bounds with
\eqref{eq:inner-annular-average} and
\eqref{eq:outer-annular-average}, we obtain}
 {
\begin{equation}\label{eq:annular-average-comparison}
 C^{-1}\bigl[(1+r)^{2-n}
 +\sigma_j^{n-2}r^{2-n}+R_j^{2-n}\bigr]
 \le\overline v_j(r)
 \le C\bigl[(1+r)^{2-n}
 +\sigma_j^{n-2}r^{2-n}+R_j^{2-n}\bigr].
\end{equation}
}
By \eqref{eq:capacity-scales}, the bracket is comparable with
\((1+r)^{2-n}\).
{Applying Lemma  \ref{lem:radial-harnack}, we convert this
average estimate into a pointwise estimate.  By compactness of
\(\mathcal K\),}
\[
 C^{-1}(1+|y|)^{2-n}
 \le U_{\lambda,b}(y)
 \le C(1+|y|)^{2-n}.
\]
Together with \eqref{eq:annular-average-comparison}, this proves the lemma.
\end{proof}

Proposition  \ref{prop:annular-increment} gives the exact Pohozaev
increment. Lemmas  \ref{lem:distance-to-bubble-family} and
\ref{lem:annular-bubble-comparison} give the compactness and uniform
comparison used to construct \(\Phi_j\) and choose its parameters in
Section  \ref{sec:reference-lower-bound}.

\section{Round sphere analysis and conformal transfer}
\label{sec:euclidean-response}

{
For every integer \(m\) with \(2\le m<(n-2)/2\), every
\(H\in\mathcal V_m\), and every \((\lambda,b)\in\mathcal K\),
this section constructs the unique solution
\(Z_{\lambda,b}[H]\in\dot H^1(\R^n)\) of
\eqref{eq:response-equation} satisfying
\eqref{eq:response-orthogonality}, and proves the uniform positivity
of \(\mathcal Q_{\lambda,b}\) on \(\mathcal V_m\). The construction
compactifies each polynomial metric jet on the round sphere and transfers
the normalized spherical solution to \(\R^n\). The second variation gives
the positive quadratic form, which a cutoff argument identifies with the
Euclidean coefficient in the annular estimate.
}

\subsection{Metric jets and compactification}

Recall that \(h=\log g\) satisfies \eqref{eq:normal-gauge} and
\eqref{eq:normal-orders}. The bubble \(U_{\lambda,b}\) and the compact
parameter set \(\mathcal K\) are specified in \eqref{eq:standard-bubble} and
\eqref{eq:bubble-parameter-set}. Lemma
\ref{lem:distance-to-bubble-family} explains why uniform estimates on this
parameter set are sufficient for the rescaled solutions \(v_j\).

\par

The positivity estimate in Proposition
\ref{prop:euclidean-strict-positivity} is applied separately to each
homogeneous coefficient of \(h=\log g\). We extract these coefficients
at the fixed singular point. For
\(2\le m\le d\), define the polynomial matrix
\(H^{(m)}\) on \(\R^n\) by
\begin{equation}\label{eq:metric-jets}
 H^{(m)}_{ij}(y)
 =\sum_{|\alpha|=m}\frac{\partial^\alpha h_{ij}(0)}{\alpha!}y^\alpha.
\end{equation}
These coefficients are fixed independently of \(j\).
The trace and radial identities in \eqref{eq:normal-gauge} imply that
\(H^{(m)}\) is symmetric and homogeneous of degree \(m\) and satisfies
\begin{equation}\label{eq:metric-jet-conditions}
 \sum_{i=1}^nH^{(m)}_{ii}=0,\qquad
 \sum_{j=1}^nH^{(m)}_{ij}(y)y^j=0
 \qquad 1\le i\le n.
\end{equation}
Thus \(H^{(m)}\in\mathcal V_m\), where the space and its coefficient
norm were defined in \eqref{eq:jet-gauge} and
\eqref{eq:coefficient-norm}, respectively.

To analyze \(H\in\mathcal V_m\), we identify the conformal metric determined by
\(U_{\lambda,b}\) with the standard round metric and represent \(H\) as a
covariant tensor on that sphere. For \((\lambda,b)\in\mathcal K\), the
bubble metric is
\begin{equation}\label{eq:bubble-metric}
 \gamma_{\lambda,b}
 =4U_{\lambda,b}^{4/(n-2)}
 \sum_{i=1}^n\dd y^i\otimes\dd y^i
 =\left(\frac{2\lambda}
 {\lambda^2+|y-b|^2}\right)^2
 \sum_{i=1}^n\dd y^i\otimes\dd y^i.
\end{equation}
This is the round metric on \(\Sph^n\) in the stereographic coordinates
determined by \((\lambda,b)\).
It satisfies
\[
 \operatorname{Ric}_{\gamma_{\lambda,b}}
 =(n-1)\gamma_{\lambda,b},\qquad
 \operatorname{Scal}_{\gamma_{\lambda,b}}=n(n-1).
\]
For \((\lambda,b)\in\mathcal K\) and a symmetric \(C^2\) matrix \(H\) with zero trace
on \(\R^n\), define the
covariant symmetric tensor \(k_{\lambda,b}[H]\) of rank \(2\) on \(\R^n\) by
\[
 k_{\lambda,b}[H]
 =\left(\frac{2\lambda}
 {\lambda^2+|y-b|^2}\right)^2
 \sum_{i,j=1}^nH_{ij}(y)\,\dd y^i\otimes\dd y^j.
\]

The next two lemmas establish the functional setting for
\(k_{\lambda,b}[H]\) and determine when it belongs to
\(H^1(\Sph^n)\). The inversion chart describes its behavior at
the point corresponding to \(y=\infty\), and Lemma
\ref{lem:euclidean-point-capacity} permits weak identities to be extended
across that point.

For fixed
\((\lambda,b)\in\mathcal K\), set
\begin{equation}\label{eq:inversion-chart}
 z=\frac{y-b}{|y-b|^2},\qquad
 y=b+\frac{z}{|z|^2}.
\end{equation}
In the \(z\) coordinates,
\begin{equation}\label{eq:inverted-bubble-metric}
 \gamma_{\lambda,b}
 =\frac{4\lambda^2}{(1+\lambda^2|z|^2)^2}
 \sum_{i=1}^n\dd z^i\otimes\dd z^i.
\end{equation}
The metric \(\gamma_{\lambda,b}\) extends smoothly from \(\R^n\) to
\(\R^n\cup\{\infty\}\simeq\Sph^n\). The \(y\) coordinates cover
\(\Sph^n\setminus\{\infty\}\), while \(z=0\) represents \(y=\infty\).
The ball \(B_{32}\) is centered at \(y=0\), and \(B_{1/8}\) in the
\(z\) coordinates is centered at \(z=0\). These two coordinate domains cover
\(\Sph^n\) uniformly for \((\lambda,b)\in\mathcal K\).

For a function, vector field, or covariant tensor \(T\), define
\[
 \begin{aligned}
 \norm{T}_{H^1_{\gamma_{\lambda,b}}}^2
 &=\int_{\R^n}\bigl(
 |T|_{\gamma_{\lambda,b}}^2+
 |\nabla^{\gamma_{\lambda,b}}T|_{\gamma_{\lambda,b}}^2
 \bigr)\,\dd V_{\gamma_{\lambda,b}},\\
 \norm{T}_{H^2_{\gamma_{\lambda,b}}}^2
 &=\int_{\R^n}\bigl(
 |T|_{\gamma_{\lambda,b}}^2+
 |\nabla^{\gamma_{\lambda,b}}T|_{\gamma_{\lambda,b}}^2+
 |(\nabla^{\gamma_{\lambda,b}})^2T|_{\gamma_{\lambda,b}}^2
 \bigr)\,\dd V_{\gamma_{\lambda,b}}.
 \end{aligned}
\]
Choose a fixed partition of unity \(\zeta_y+\zeta_z=1\) subordinate to these
coordinate domains.  If \(T_y\) and \(T_z\) are the coordinate components of
the same section, then, for \(s=1,2\),
\[
 \norm{T}_{H^s_{\gamma_{\lambda,b}}}^2
 \asymp
 \norm{\zeta_yT_y}_{H^s(B_{32})}^2
 +\norm{\zeta_zT_z}_{H^s(B_{1/8})}^2,
\]
with constants uniform for \((\lambda,b)\in\mathcal K\).
The space \(H^{-1}_{\gamma_{\lambda,b}}\) is the dual of
\(H^1_{\gamma_{\lambda,b}}\).

\begin{samepage}
\begin{lemma}\label{lem:euclidean-point-capacity}
Smooth functions, vector fields, and symmetric covariant tensors of rank \(2\)
compactly supported in the \(y\) coordinates are dense in the corresponding
\(H^s_{\gamma_{\lambda,b}}\) spaces, \(s=1,2\).
\end{lemma}

\begin{proof}
After smooth approximation on \(\Sph^n\), multiply by a rescaled radial
cutoff that is zero for \(|z|\le\varepsilon\) and one for
\(|z|\ge2\varepsilon\).
For each fixed smooth field, the \(H^s_{\gamma_{\lambda,b}}\) error is
\(O(\varepsilon^{n/2-s})\to0\) as \(\varepsilon\downarrow0\), since
\(n>2s\) for \(s=1,2\).
\end{proof}
\end{samepage}

The next lemma determines when \(k_{\lambda,b}[H]\) belongs to
\(H^1_{\gamma_{\lambda,b}}\) and estimates its tail.
For the fixed metric, we will apply these estimates with \(H=H^{(m)}\)
from \eqref{eq:metric-jets}.
For \(L\ge2\), fix a radial function
\(\chi_L\in C_c^\infty(\R^n)\), with \(0\le\chi_L\le1\), satisfying
\begin{equation}\label{eq:euclidean-cutoff}
 \chi_L=1\ \text{on }B_L,\qquad
 \chi_L=0\ \text{on }\R^n\setminus B_{2L},\qquad
 |\nabla\chi_L|\le CL^{-1},
\qquad
|\nabla^2\chi_L|\le CL^{-2}.
\end{equation}
Parameter derivatives \(\partial_{(\lambda,b)}^\beta\) are taken at fixed
\(y\), with \(H\) held fixed.

\begin{lemma}\label{lem:metric-jet-threshold}
If \(m\ge2\), \(H\ne0\) is a symmetric homogeneous polynomial matrix of
degree \(m\) satisfying \eqref{eq:metric-jet-conditions}, and
\((\lambda,b)\in\mathcal K\), then
\(k_{\lambda,b}[H]\in H^1_{\gamma_{\lambda,b}}\) if and only if
\(m<(n-2)/2\).
If \(m<(n-2)/2\), then, uniformly for \(|\beta|\le2\),
\[
 \norm{\partial_{(\lambda,b)}^\beta k_{\lambda,b}[H]}_{H^1_{\gamma_{\lambda,b}}}
 \le C|H|,
\]
and
\begin{equation}\label{eq:metric-jet-tail}
 \norm{\partial_{(\lambda,b)}^\beta
 k_{\lambda,b}[(1-\chi_L)H]}_{H^1_{\gamma_{\lambda,b}}}
 \le CL^{m-\frac{n-2}{2}}|H|.
\end{equation}
The constant depends only on \(n,m,\mathcal K\), and the derivative bounds
in \eqref{eq:euclidean-cutoff}.
\end{lemma}

\begin{proof}
Write \(\gamma=\gamma_{\lambda,b}\), \(k=k_{\lambda,b}[H]\), and
\(y=r\theta\), with \(r>0\) and \(\theta\in\Sph^{n-1}\).
By \eqref{eq:bubble-metric} and the definition of \(k\),
\[
 |k(r\theta)|_\gamma^2
 =r^{2m}\sum_{i,j=1}^nH_{ij}(\theta)^2,\qquad
 |\partial_r|_\gamma\asymp r^{-2},\qquad
 \dd V_\gamma\asymp r^{-n-1}\,\dd r\,\dd\theta,
 \quad r\ge4.
\]
{By metric compatibility and the Cauchy--Schwarz inequality,}
\[
 \frac{m}{r}|k|_\gamma^2
 =\frac12\partial_r|k|_\gamma^2
 =\langle\nabla^\gamma_{\partial_r}k,k\rangle_\gamma
 \le Cr^{-2}|\nabla^\gamma k|_\gamma|k|_\gamma.
\]
The spherical \(L^2\) norm of \(H\) is equivalent to the coefficient norm
\eqref{eq:coefficient-norm}. Consequently,
\[
 \int_{\{|y|>4\}}|\nabla^\gamma k|_\gamma^2\,\dd V_\gamma
 \ge c|H|^2\int_4^\infty r^{2m+1-n}\,\dd r.
\]
This integral diverges when \(2m\ge n-2\), proving necessity.

For the upper estimates, assume \(r\ge4\) and differentiate the formula for \(k\) at fixed
\(y,H\). {By polynomial growth, its components in the \(y\) coordinates satisfy}
\[
 |\partial_y^\alpha\partial_{(\lambda,b)}^\beta k_{ij}|
 \le C|H|r^{m-4-|\alpha|},
 \qquad |\alpha|\le1,\quad |\beta|\le2.
\]
The connection coefficients are \(O(r^{-1})\), and the inverse metric
has components \(O(r^4)\). Thus
\[
 |\partial_{(\lambda,b)}^\beta k|_\gamma
 +r^{-1}|\nabla^\gamma(\partial_{(\lambda,b)}^\beta k)|_\gamma
 \le C|H|r^m,\qquad r\ge4,\quad|\beta|\le2.
\]
On \(B_4\), the component bounds are uniform. Hence
\[
 \|\partial_{(\lambda,b)}^\beta k\|_{H^1_\gamma}^2
 \le C|H|^2
 \left(1+\int_4^\infty r^{2m+1-n}\,\dd r\right),
\]
which is finite for \(2m<n-2\).

Assume \(2m<n-2\). {For \(L\ge4\), it follows from
\eqref{eq:euclidean-cutoff} that}
\(|\nabla^\gamma\chi_L|_\gamma\le CL\) on \(A_{L,2L}\).
{Since \(\chi_L\) is independent of \(\lambda,b\), applying the
product rule, we obtain}
\[
 \begin{aligned}
 \|\partial_{(\lambda,b)}^\beta
 k_{\lambda,b}[(1-\chi_L)H]\|_{H^1_\gamma}^2
 &\le C|H|^2
 \left(\int_L^\infty r^{2m+1-n}\,\dd r+L^{2m+2-n}\right)\\
 &\le C|H|^2L^{2m+2-n},\qquad |\beta|\le2.
 \end{aligned}
\]
The smooth cutoff tensors are therefore Cauchy in \(H^1_\gamma\),
with limit \(\partial_{(\lambda,b)}^\beta k\).
This proves sufficiency and the first estimate in the lemma.
Taking square roots yields \eqref{eq:metric-jet-tail};
enlarging \(C\) covers \(2\le L<4\).
\end{proof}

{The range \(2\le m<(n-2)/2\) used in
Proposition  \ref{prop:euclidean-response} is precisely the range obtained
in Lemma  \ref{lem:metric-jet-threshold}.}

\subsection{The spherical linearized equation}

To solve the correction equation \eqref{eq:response-equation}, we first identify the kernel of \(-\Delta_{\gamma_{\lambda,b}}-n\). This kernel is spanned by the dilation and translation derivatives of \(U_{\lambda,b}\), each divided by \(U_{\lambda,b}\).  For
\((\lambda,b)\in\mathcal K\), we define
\[
 \begin{aligned}
 J_{\lambda,b,0}
 &=\lambda\partial_\lambda U_{\lambda,b}
 =\frac{n-2}{2}U_{\lambda,b}
   \frac{|y-b|^2-\lambda^2}{|y-b|^2+\lambda^2},\\
 J_{\lambda,b,l}
 &=\partial_{b^l}U_{\lambda,b}
 =(n-2)U_{\lambda,b}
   \frac{y^l-b^l}{|y-b|^2+\lambda^2},
 \qquad 1\le l\le n.
 \end{aligned}
\]
The corresponding spherical functions and their span are defined by
\[
 Y_{\lambda,b,l}=\frac{J_{\lambda,b,l}}{U_{\lambda,b}},
 \qquad
 \mathcal J_{\lambda,b}
 =\operatorname{span}\{
 Y_{\lambda,b,0},\ldots,Y_{\lambda,b,n}\}.
\]
{Calculating directly from \eqref{eq:bubble-metric}, we find}
\begin{equation}\label{eq:jacobi-hessian}
 -\Delta_{\gamma_{\lambda,b}}Y_{\lambda,b,l}
 =nY_{\lambda,b,l}\qquad \text{and} \qquad
 \nabla_{\gamma_{\lambda,b}}^2Y_{\lambda,b,l}
 =-Y_{\lambda,b,l}\gamma_{\lambda,b}.
\end{equation}

The parameter dependence can be placed in the tensor while the round
metric remains fixed. Let \(g_{S}\) be the unit round metric,
let \(N=(0,\ldots,0,1)\in\Sph^n\), and define stereographic projection by
\[
 \pi(\xi)=\frac{(\xi_1,\ldots,\xi_n)}{1-\xi_{n+1}},
 \qquad \xi\in\Sph^n\setminus\{N\}.
\]
Then
\[
 \pi^*\!\left(\frac4{(1+|x|^2)^2}\sum_{i=1}^n\dd x^i\otimes\dd x^i\right)
 =g_{S}.
\]
Define the isometry
\begin{equation}\label{eq:fixed-round-identification}
 \Xi_{\lambda,b}(\xi)=b+\lambda\pi(\xi),\qquad
 \Xi_{\lambda,b}^*\gamma_{\lambda,b}=g_{S}.
\end{equation}
Its extension sends \(N\) to \(y=\infty\).
The inversion chart and its metric are given by
\eqref{eq:inversion-chart}  and \eqref{eq:inverted-bubble-metric}.
{The first spherical harmonics are spanned by
\(\xi_1,\ldots,\xi_{n+1}\), where
\(\xi=(\xi_1,\ldots,\xi_{n+1})\in\Sph^n\).}
{For the explicit Jacobi fields, we have}
\begin{equation}\label{eq:fixed-round-jacobi}
 \Xi_{\lambda,b}^*Y_{\lambda,b,0}
 =\frac{n-2}{2}\xi_{n+1},\qquad
 \Xi_{\lambda,b}^*Y_{\lambda,b,l}
 =\frac{n-2}{2\lambda}\xi_l,\qquad 1\le l\le n.
\end{equation}

The isometry \eqref{eq:fixed-round-identification} reduces the kernel
and inverse estimates to the spherical harmonic decomposition on the
unit round sphere.

\begin{lemma}\label{lem:euclidean-scalar-fredholm}
For every \((\lambda,b)\in\mathcal K\), the kernel of
\[
 -\Delta_{\gamma_{\lambda,b}}-n:
 H^1_{\gamma_{\lambda,b}}\to
 H^{-1}_{\gamma_{\lambda,b}}
\]
is \(\mathcal J_{\lambda,b}\).  If
\(\phi\perp\mathcal J_{\lambda,b}\) in \(L^2_{\gamma_{\lambda,b}}\), then
\begin{equation}\label{eq:projected-scalar-inverse}
 \norm{\phi}_{H^1_{\gamma_{\lambda,b}}}
 \le C\norm{(-\Delta_{\gamma_{\lambda,b}}-n)\phi}
 _{H^{-1}_{\gamma_{\lambda,b}}},
\end{equation}
with \(C\) uniform on \(\mathcal K\). If
\(F\in H^{-1}_{\gamma_{\lambda,b}}\) satisfies
\[
 \langle F,Y_{\lambda,b,l}\rangle=0,
 \qquad 0\le l\le n,
\]
then there is a unique
\(\phi\perp\mathcal J_{\lambda,b}\) satisfying
\((-\Delta_{\gamma_{\lambda,b}}-n)\phi=F\), and
\eqref{eq:projected-scalar-inverse} holds.
\end{lemma}

\begin{proof}
{For smooth functions, pullback by
\eqref{eq:fixed-round-identification} preserves the scalar Sobolev norms and
satisfies}
\[
 (-\Delta_{g_{S}}-n)(\phi\circ\Xi_{\lambda,b})
 =\bigl[(-\Delta_{\gamma_{\lambda,b}}-n)\phi\bigr]
   \circ\Xi_{\lambda,b}.
\]
The weak identity follows by duality.
Using spherical harmonics, we decompose
\(L^2(\Sph^n)\) orthogonally, with shifted eigenvalues \(\ell(\ell+n-1)-n=(\ell-1)(\ell+n)\), \(\ell=0,1,2,\ldots\).
The eigenvalue in degree zero is \(-n\), and the only zero occurs at
\(\ell=1\). {Thus, from \eqref{eq:fixed-round-jacobi}, we obtain}
 {
\[
 \Xi_{\lambda,b}^*\ker(-\Delta_{\gamma_{\lambda,b}}-n)
 =\ker(-\Delta_{g_{S}}-n)
 =\operatorname{span}\{\xi_1,\ldots,\xi_{n+1}\}
 =\Xi_{\lambda,b}^*\mathcal J_{\lambda,b}.
\]
}

For \(\ell\ne1\),
\[
 |\ell(\ell+n-1)-n|
 \ge c(n)\bigl(1+\ell(\ell+n-1)\bigr).
\]
On the fixed sphere, retain \(F,\phi\) for the pullback objects and
denote their homogeneous components of degree \(\ell\) by
\(F_\ell,\phi_\ell\).
For \(F\in H^{-1}\), these components are defined by duality.
The equation and orthogonality conditions require
\[
 F_1=\phi_1=0,\qquad
 \phi_\ell=\frac{F_\ell}{(\ell-1)(\ell+n)},\qquad \ell\ne1.
\]
{By the spectral bound,}
\[
 \begin{aligned}
 \sum_{\ell\ne1}\bigl(1+\ell(\ell+n-1)\bigr)\|\phi_\ell\|_2^2
 &\le C(n)\sum_{\ell\ne1}
 \frac{\|F_\ell\|_2^2}{1+\ell(\ell+n-1)}\\
 &=C(n)\|F\|_{H^{-1}(g_{S})}^2.
 \end{aligned}
\]
Hence \(\sum_{\ell\ne1}\phi_\ell\) converges in \(H^1\) to the unique
orthogonal solution. {By the isometry, we obtain
\eqref{eq:projected-scalar-inverse} and solvability for}
\(\gamma_{\lambda,b}\), with the same constant.
\end{proof}

For each fixed \((\lambda,b)\in\mathcal K\), write
\(\gamma=\gamma_{\lambda,b}\). We apply the preceding inverse to the
scalar curvature variation produced by a symmetric covariant
tensor \(k\) of rank \(2\). For smooth \(k\) on \(\Sph^n\), our divergence
convention is
\[
 (\diver_\gamma k)_j
 =\sum_{i,a=1}^n\gamma^{ia}\nabla^\gamma_i k_{aj}.
\]
The scalar curvature linearization is
\[
 D\!\operatorname{Scal}_\gamma[k]
 =-\Delta_\gamma\tr_\gamma k+\diver_\gamma^2k
 -(n-1)\tr_\gamma k,\qquad
 \diver_\gamma^2k
 =\sum_{i,j=1}^n\gamma^{ij}
   \nabla^\gamma_i(\diver_\gamma k)_j.
\]
{For a smooth scalar test function \(\varphi\), integrating the
second derivative terms once, we obtain}
\[
 |\langle D\!\operatorname{Scal}_\gamma[k],\varphi\rangle|
 \le C\|k\|_{H^1_\gamma}\|\varphi\|_{H^1_\gamma}.
\]
Thus \(D\!\operatorname{Scal}_\gamma\) extends to a bounded map from \(H^1_\gamma\)
symmetric tensors of rank \(2\) to \(H^{-1}_\gamma\).
{For \(Y\in\mathcal J_{\lambda,b}\), integrating by parts and
using \eqref{eq:jacobi-hessian}, we obtain}
 {
\begin{equation}\label{eq:spherical-curvature-compatibility}
 \langle D\!\operatorname{Scal}_\gamma[k],Y\rangle
 =\int_{\Sph^n}\!
 \left\langle k,\nabla_\gamma^2Y
       -(\Delta_\gamma Y)\gamma-(n-1)Y\gamma\right\rangle_\gamma
 \,\dd V_\gamma
 =0.
\end{equation}
}
Smooth approximation extends this identity to every \(H^1_\gamma\)
symmetric tensor of rank \(2\).

{By Lemma  \ref{lem:euclidean-scalar-fredholm}, every such
\(k\) determines a unique function \(\psi[k]\in H^1_\gamma\) satisfying}
\begin{equation}\label{eq:spherical-scalar-equation}
 (-\Delta_\gamma-n)\psi[k]=-c(n)D\!\operatorname{Scal}_\gamma[k],
 \qquad \psi[k]\perp\mathcal J_{\lambda,b}
 \quad\text{in }L^2_\gamma.
\end{equation}
This is a weak equation on the whole sphere, and
\begin{equation}\label{eq:spherical-scalar-estimate}
 \|\psi[k]\|_{H^1_\gamma}\le C\|k\|_{H^1_\gamma}.
\end{equation}
The map \(k\mapsto\psi[k]\) is linear. The constants are uniform on
\(\mathcal K\), since the metrics are isometric to \(g_{S}\).

\subsection{The Euclidean correction}

We now write the correction in Euclidean coordinates.
For a symmetric \(C^2\) matrix \(H\) with zero trace on \(\R^n\) and
\((\lambda,b)\in\mathcal K\), the metric variation produces the source
\(F_{\lambda,b}[H]\) defined by
\begin{equation}\label{eq:metric-source}
 F_{\lambda,b}[H]
 =\sum_{i,j=1}^n
 \partial_i\bigl(H_{ij}\partial_jU_{\lambda,b}\bigr)
 +c(n)\left(\sum_{i,j=1}^n\partial_i\partial_jH_{ij}\right)
 U_{\lambda,b}.
\end{equation}

{Set \(g_\eps=\exp(\eps H)\).}
For a smooth scalar function \(w\), set
\(u_\eps=U_{\lambda,b}+\eps w\). On each fixed compact subset of
\(\R^n\), this function is positive for small \(|\eps|\).
Differentiation at zero gives
\[
 \left.\partial_\eps\right|_{\eps=0}
 \left[L_{g_\eps}u_\eps+n(n-2)u_\eps^p\right]
 =-F_{\lambda,b}[H]
 +\left[\Delta+n(n+2)U_{\lambda,b}^{p-1}\right]w.
\]
{For \(H\in\mathcal V_m\) and \(b=0\), we have
\(H\nabla U_{\lambda,0}=0\) because \(H(y)y=0\); hence the first summand in}
\eqref{eq:metric-source} vanishes. For \(m\ge4\), {this is the
polynomial correction equation used by Khuri--Marques--Schoen
\cite[Proposition  4.1, Lemma  4.2, and (4.3)--(4.5)]{KMS}.} For \(b\ne0\), both summands in
\eqref{eq:metric-source} are present.

{By \eqref{eq:bubble-metric},
\(\gamma_{\lambda,b}
=(2^{(n-2)/2}U_{\lambda,b})^{4/(n-2)}\delta\).}
Since \(\tr H=0\), \(\det g_\eps=1\).  The covariant metric direction is
\(H\), and
\[
 \left.\partial_\eps\right|_{\eps=0}
 \bigl((2^{(n-2)/2}U_{\lambda,b})^{4/(n-2)}g_\eps\bigr)
 =k_{\lambda,b}[H].
\]

At the Euclidean metric,
\[
 \left.\partial_\eps\right|_{\eps=0}
 \Delta_{g_\eps}(2^{(n-2)/2}U_{\lambda,b})
 =-\sum_{i,j=1}^n\partial_i
 \bigl(H_{ij}\partial_j(2^{(n-2)/2}U_{\lambda,b})\bigr),
 \qquad
 \left.\partial_\eps\right|_{\eps=0}\operatorname{Scal}_{g_\eps}
 =\sum_{i,j=1}^n\partial_i\partial_jH_{ij}.
\]
{By conformal covariance,}
\[
 L_{(2^{(n-2)/2}U_{\lambda,b})^{4/(n-2)}g_\eps}1
 =(2^{(n-2)/2}U_{\lambda,b})^{-p}
 L_{g_\eps}(2^{(n-2)/2}U_{\lambda,b}).
\]
Differentiation at \(\eps=0\), together with
\((2^{(n-2)/2})^{1-p}=1/4\), yields

\begin{equation}\label{eq:curvature-source-relation}
 c(n)D\!\operatorname{Scal}_{\gamma_{\lambda,b}}[k_{\lambda,b}[H]]
 =\frac14U_{\lambda,b}^{-p}F_{\lambda,b}[H].
\end{equation}

For a compactly supported symmetric \(C^2\) matrix \(T\) with zero trace
on \(\R^n\), the tensor \(k_{\lambda,b}[T]\) vanishes near \(y=\infty\).
{Since
\(\dd V_{\gamma_{\lambda,b}}=2^nU_{\lambda,b}^{p+1}\dd y\), combining
\eqref{eq:spherical-curvature-compatibility} with
\eqref{eq:curvature-source-relation}, we obtain}
\[
 0=c(n)\langle D\!\operatorname{Scal}_{\gamma_{\lambda,b}}[k_{\lambda,b}[T]],
 Y_{\lambda,b,l}\rangle
 =2^{n-2}\int_{\R^n}F_{\lambda,b}[T]J_{\lambda,b,l}\,\dd y.
\]
Consequently,
\begin{equation}\label{eq:cutoff-response-compatibility}
 \int_{\R^n}F_{\lambda,b}[T]J_{\lambda,b,l}\,\dd y=0,
 \qquad 0\le l\le n.
\end{equation}
This applies to \(T=\chi_LH\) for every \(H\in\mathcal V_m\), \(m\ge2\).
{For \(2\le m<(n-2)/2\), it also follows from
Lemma  \ref{lem:metric-jet-threshold} that
\(k_{\lambda,b}[H]\in H^1_{\gamma_{\lambda,b}}\).}
{Using the cutoff bounds and \eqref{eq:metric-source}, we have}
\[
 |F_{\lambda,b}[\chi_LH](y)J_{\lambda,b,l}(y)|
 \le C|H|(1+|y|)^{m+2-2n},
 \qquad 0\le l\le n.
\]
The right side is integrable because \(m<n-2\).
{Applying dominated convergence in
\eqref{eq:cutoff-response-compatibility}, we obtain}
\[
 \int_{\R^n}F_{\lambda,b}[H]J_{\lambda,b,l}\,\dd y=0,
 \qquad 0\le l\le n,\quad 2\le m<(n-2)/2.
\]

We now transfer the spherical solution to the homogeneous Sobolev
space \(\dot H^1(\R^n)\) defined in
Section  \ref{sec:common-preliminaries}; its dual is denoted by
\(\dot H^{-1}(\R^n)\).

For \(\phi\in H^1_{\gamma_{\lambda,b}}\), {set
\(w=2^{(n-2)/2}U_{\lambda,b}\phi\).} Then
\begin{equation}\label{eq:conformal-energy}
 \int_{\R^n}\left(
 |\nabla^{\gamma_{\lambda,b}}\phi|^2
 +\frac{n(n-2)}4\phi^2
 \right)\dd V_{\gamma_{\lambda,b}}
 =\int_{\R^n}|\nabla w|^2\,\dd y,
\end{equation}

\begin{equation}\label{eq:scalar-conjugacy}
 (-\Delta_{\gamma_{\lambda,b}}-n)\phi
 =(2^{\frac{n-2}{2}}U_{\lambda,b})^{-p}
 \left[-\Delta w
 -n(n+2)U_{\lambda,b}^{4/(n-2)}w\right].
\end{equation}

{By conformal covariance, these identities first hold for
smooth functions compactly supported in the \(y\) coordinates. Applying
Lemma  \ref{lem:euclidean-point-capacity} and density, we obtain the
isomorphism between \(H^1_{\gamma_{\lambda,b}}\) and \(\dot H^1(\R^n)\);}
\eqref{eq:scalar-conjugacy} holds in the corresponding weak sense.
The integral with \(\dd V_{\gamma_{\lambda,b}}\) is the spherical integral
written in the \(y\) chart.

Fix \(2\le m<(n-2)/2\), \(H\in\mathcal V_m\), and
\((\lambda,b)\in\mathcal K\). Define
\begin{equation}\label{eq:response-from-sphere}
 Z_{\lambda,b}[H]
 =U_{\lambda,b}\psi[k_{\lambda,b}[H]].
\end{equation}
{Changing variables by
\(Y_{\lambda,b,l}=J_{\lambda,b,l}/U_{\lambda,b}\) and
\(\dd V_{\gamma_{\lambda,b}}=2^nU_{\lambda,b}^{p+1}\dd y\), we obtain}
 {
\begin{equation}\label{eq:response-moment-transfer}
 \int_{\Sph^n}\psi[k_{\lambda,b}[H]]Y_{\lambda,b,l}
 \,\dd V_{\gamma_{\lambda,b}}
 =2^n\int_{\R^n}U_{\lambda,b}^{p-1}
 Z_{\lambda,b}[H]J_{\lambda,b,l}\,\dd y.
\end{equation}
}
{Taking
\(\phi=2^{-(n-2)/2}\psi[k_{\lambda,b}[H]]\) in
\eqref{eq:scalar-conjugacy}, we find \(w=Z_{\lambda,b}[H]\).
Combining these identities with \eqref{eq:spherical-scalar-equation} and
\eqref{eq:curvature-source-relation}, we obtain the}
Euclidean equation and orthogonality conditions
\begin{equation}\label{eq:response-equation}
 \Delta Z_{\lambda,b}[H]
 +n(n+2)U_{\lambda,b}^{4/(n-2)}Z_{\lambda,b}[H]
 =F_{\lambda,b}[H],
\end{equation}
\begin{equation}\label{eq:response-orthogonality}
 \int_{\R^n}U_{\lambda,b}^{4/(n-2)}
 Z_{\lambda,b}[H]J_{\lambda,b,l}\,\dd y=0,
 \qquad 0\le l\le n.
\end{equation}

\par
\begin{proposition}\label{prop:euclidean-response}
The problem \eqref{eq:response-equation}  and  \eqref{eq:response-orthogonality}
has a unique solution in \(\dot H^1(\R^n)\).
For \(|\alpha|\le2\) and \(|\beta|\le2\),
\begin{equation}\label{eq:response-decay}
 \left|
 \partial_{(\lambda,b)}^\beta\partial_y^\alpha
 Z_{\lambda,b}[H](y)\right|
 \le C_{\alpha,\beta}|H|
 (1+|y|)^{m+2-n-|\alpha|}.
\end{equation}
Here \(C_{\alpha,\beta}\) depends only on
\(n,m,\alpha,\beta\), and \(\mathcal K\).
\end{proposition}

\par
\begin{proof}
\emph{Step 1. Existence, uniqueness, and energy.}
Write \(w=Z_{\lambda,b}[H]\). Existence follows from
\eqref{eq:response-from-sphere}.
If \(w_1,w_2\in\dot H^1(\R^n)\) solve
\eqref{eq:response-equation}  and  \eqref{eq:response-orthogonality}, then
\((w_1-w_2)/U_{\lambda,b}\in H^1_{\gamma_{\lambda,b}}\) by
\eqref{eq:conformal-energy}.
{Combining \eqref{eq:scalar-conjugacy} and
\eqref{eq:response-moment-transfer}, and then applying
Lemma  \ref{lem:euclidean-scalar-fredholm}, we obtain}
 {
\[
 \frac{w_1-w_2}{U_{\lambda,b}}
 \in\ker(-\Delta_{\gamma_{\lambda,b}}-n)
       \cap\mathcal J_{\lambda,b}^{\perp}
 =\mathcal J_{\lambda,b}\cap\mathcal J_{\lambda,b}^{\perp}
 =\{0\}.
\]
}
By \eqref{eq:spherical-scalar-estimate} and
Lemma  \ref{lem:metric-jet-threshold},
\[
 \|\psi[k_{\lambda,b}[H]]\|_{H^1_{\gamma_{\lambda,b}}}
 \le C\|k_{\lambda,b}[H]\|_{H^1_{\gamma_{\lambda,b}}}
 \le C|H|.
\]
Taking
\(\phi=2^{-(n-2)/2}\psi[k_{\lambda,b}[H]]\) in
\eqref{eq:conformal-energy}, we obtain \(\|\nabla w\|_{L^2(\R^n)}\le C|H|\).

\emph{Step 2. Spatial decay.}
Direct differentiation of
\eqref{eq:metric-source} gives
\begin{equation}\label{eq:response-source-pointwise}
 |\partial_y^\alpha F_{\lambda,b}[H](y)|
 \le C_\alpha|H|(1+|y|)^{m-n-|\alpha|},
 \qquad |\alpha|\le2.
\end{equation}
Set \(L_{\lambda,b}=\Delta+n(n+2)U_{\lambda,b}^{p-1}\).
Choose \(R_0=R_0(n,m,\mathcal K)\ge8\) large enough that, for
\(r=|y|\ge R_0\),
\begin{equation}\label{eq:response-barrier}
 \begin{aligned}
 -L_{\lambda,b}r^{m+2-n}
 &=m(n-2-m)r^{m-n}
 -n(n+2)U_{\lambda,b}^{p-1}r^{m+2-n}\\
 &=m(n-2-m)r^{m-n}+O(r^{m-n-2})\\
 &\ge\frac12m(n-2-m)r^{m-n}.
 \end{aligned}
\end{equation}
After increasing \(R_0\),
\[
 -L_{\lambda,b}r^{-\frac{n-2}{4}}
 =\frac{3(n-2)^2}{16}r^{-\frac{n+6}{4}}
 +O(r^{-\frac{n+14}{4}})
 \ge c r^{-\frac{n+6}{4}}.
\]
{Using the bound in \(\dot H^1\),
\eqref{eq:response-source-pointwise}, and a scaled local estimate, we
obtain}
 {
\[
 R^{\frac{n-2}{4}}\sup_{A_{R,2R}}|w|
 \le{}C R^{-\frac{n-2}{4}}
 \norm{w}_{L^{2n/(n-2)}(A_{R/2,4R})}
 {}+C|H|R^{m+2-n+\frac{n-2}{4}}
 \to0.
\]
}
{By Kato's inequality,}
\[
 -L_{\lambda,b}|w|\le|F_{\lambda,b}[H]|
 \quad\text{in }\mathcal D'(\R^n\setminus B_{R_0}).
\]
For \(\varepsilon>0\), set
\[
 \Psi_\varepsilon(r)
 =A|H|r^{m+2-n}+\varepsilon r^{-\frac{n-2}{4}}.
\]
By \eqref{eq:response-barrier} and
\eqref{eq:response-source-pointwise}, choose \(A\), independently of
\((\lambda,b)\in\mathcal K\) and \(H\), so that
\[
 A|H|R_0^{m+2-n}\ge\sup_{\partial B_{R_0}}|w|,
 \qquad
 -A|H|L_{\lambda,b}(r^{m+2-n})\ge|F_{\lambda,b}[H]|
 \quad\text{for }r\ge R_0.
\]
Then, for every fixed \(\varepsilon>0\) and all sufficiently large \(R\),
\[
 \Psi_\varepsilon\ge|w|
 \quad\text{on }\partial B_{R_0}\cup\partial B_R,
 \qquad
 -L_{\lambda,b}\Psi_\varepsilon\ge|F_{\lambda,b}[H]|
 \quad\text{on }A_{R_0,R}.
\]
{Using the positive supersolution \(r^{-(n-2)/4}\), we apply
the maximum principle on \(A_{R_0,R}\) and obtain
\(|w|\le\Psi_\varepsilon\).}
Letting
\(R\to\infty\) and then \(\varepsilon\downarrow0\) yields
\begin{equation}\label{eq:response-spatial-decay}
 |w(y)|\le C|H||y|^{m+2-n},
 \qquad |y|\ge R_0.
\end{equation}
{The spatial derivative bounds in \eqref{eq:response-decay}
follow from scaled interior estimates on \(A_{r/2,2r}\).}

\emph{Step 3. Parameter derivatives.}
For this part of the proof, use the fixed coordinates and polynomial
matrix
\[
 x=\frac{y-b}{\lambda},\qquad
 T_{\lambda,b}(x)=H(b+\lambda x),
\]
and set
\[
 W_{\lambda,b}(x)
 =\lambda^{\frac{n-2}{2}}Z_{\lambda,b}[H](b+\lambda x).
\]
The matrix \(T_{\lambda,b}\) is trace free and has degree at most \(m\).
Its coefficients and their first two parameter derivatives are bounded
by \(C|H|\) on a fixed compact neighborhood of \(\mathcal K\).
The upper bound calculation in Lemma  \ref{lem:metric-jet-threshold}
uses only polynomial growth. {Repeating that calculation, we
obtain}
\[
 \bigl\|\pi^*k_{1,0}[\partial_{(\lambda,b)}^\beta T_{\lambda,b}]
 \bigr\|_{H^1(g_{S})}
 \le C_\beta|H|,\qquad |\beta|\le2.
\]
Indeed, in the unit Euclidean chart this tensor has components
\[
 \frac4{(1+|x|^2)^2}
 \partial_{(\lambda,b)}^\beta T_{\lambda,b,ij}(x),
\]
and its gradient energy at infinity is bounded by
\(C|H|^2\int_1^\infty r^{2m+1-n}\,\dd r\).

Changing variables in the equation and the moments yields
 {
\[
 \begin{aligned}
 \bigl(\Delta_x+n(n+2)U_{1,0}^{p-1}\bigr)W_{\lambda,b}
 &=F_{1,0}[T_{\lambda,b}],\\
 \int_{\R^n}U_{1,0}^{p-1}W_{\lambda,b}J_{1,0,l}\,\dd x&=0,
 \qquad 0\le l\le n.
 \end{aligned}
\]
}
Here \(F_{1,0}[T]\) is the expression in \eqref{eq:metric-source},
with derivatives taken in \(x\).
The compatibility identity
\eqref{eq:spherical-curvature-compatibility} applies to every polynomial with zero trace of degree at most \(m<(n-2)/2\), because its
compactified tensor lies in \(H^1\). The operator, its kernel, and
these moments are now independent of \((\lambda,b)\).

The construction \eqref{eq:spherical-scalar-equation} defines a fixed
linear solution operator on this polynomial family. Its boundedness
follows from \eqref{eq:spherical-scalar-estimate} and
\eqref{eq:conformal-energy}  through  \eqref{eq:scalar-conjugacy}.
Since \(T_{\lambda,b}\) depends \(C^2\) on the parameters, so does
\(W_{\lambda,b}\) in \(\dot H^1\), with
\[
 \|\partial_{(\lambda,b)}^\beta W_{\lambda,b}\|_{\dot H^1}
 \le C_\beta|H|,\qquad |\beta|\le2.
\]
Differentiating the fixed equation and its moments gives, for
\(|\beta|\le2\),
 {
\[
 \begin{aligned}
 \bigl(\Delta_x+n(n+2)U_{1,0}^{p-1}\bigr)
 \partial_{(\lambda,b)}^\beta W_{\lambda,b}
 &=F_{1,0}[\partial_{(\lambda,b)}^\beta T_{\lambda,b}],\\
 \int_{\R^n}U_{1,0}^{p-1}
 \partial_{(\lambda,b)}^\beta W_{\lambda,b}J_{1,0,l}\,\dd x&=0,
 \qquad 0\le l\le n.
 \end{aligned}
\]
}
Differentiating \eqref{eq:metric-source} gives
\[
 \left|\partial_x^\alpha
 F_{1,0}[\partial_{(\lambda,b)}^\beta T_{\lambda,b}](x)\right|
 \le C_{\alpha,\beta}|H|(1+|x|)^{m-n-|\alpha|},
 \qquad |\alpha|\le3,\quad |\beta|\le2.
\]
{Together with the preceding energy estimate, this source estimate}
allows the comparison argument for \eqref{eq:response-spatial-decay}
to be applied to each parameter derivative.
{Applying scaled interior estimates, we obtain}
\[
 |\partial_x^\alpha\partial_{(\lambda,b)}^\beta W_{\lambda,b}(x)|
 \le C_{\alpha,\beta}|H|(1+|x|)^{m+2-n-|\alpha|},
 \qquad |\alpha|\le4,\quad |\beta|\le2.
\]

Finally,
\(Z_{\lambda,b}[H](y)=\lambda^{-(n-2)/2}W_{\lambda,b}(x)\).
At fixed \(y\), the chain rule reads
\[
 \begin{aligned}
 \partial_\lambda Z_{\lambda,b}[H](y)
 &=\lambda^{-\frac{n-2}{2}}
 \left[\partial_\lambda W_{\lambda,b}
 -\lambda^{-1}\left(\frac{n-2}{2}W_{\lambda,b}
                   +x\cdot\nabla_xW_{\lambda,b}\right)\right](x),\\
 \partial_{b^i}Z_{\lambda,b}[H](y)
 &=\lambda^{-\frac{n-2}{2}}
 \left[\partial_{b^i}W_{\lambda,b}
       -\lambda^{-1}\partial_{x_i}W_{\lambda,b}\right](x).
 \end{aligned}
\]
For \(|\alpha|,|\beta|\le2\), the chain rule uses at most four spatial
derivatives of \(W_{\lambda,b}\). Each factor \(x\) accompanies one
spatial derivative, so
\[
 (1+|x|)^\nu
(1+|x|)^{m+2-n-|\alpha|-\nu}
=(1+|x|)^{m+2-n-|\alpha|},
\qquad 0\le\nu\le|\beta|\le2.
\]
Since \(1+|x|\asymp1+|y|\) uniformly on \(\mathcal K\), this proves
\eqref{eq:response-decay}.
\end{proof}

The function \(Z_{\lambda,b}[H]\) is defined on \(\R^n\);
no boundary condition is imposed on \(A_{\sigma_j,R_j}\).
For a centered bubble, the correction \(Z_{\lambda,0}[H]\) agrees with the specified
polynomial correction in the overlapping range:
 {
\[
 \psi_{m,\lambda}(H)=Z_{\lambda,0}[H],
 \qquad H\in\mathcal V_m,\quad 4\le m<(n-2)/2.
\]
}
Indeed, \(H(y)y=0\) reduces \eqref{eq:response-equation} to
the scaled form of \eqref{eq:upper-polynomial-correction}. By the decay in \eqref{eq:corrector-pointwise-estimate},
{
\[
 \int_{\{|y|>1\}}|\nabla\psi_{m,\lambda}(H)|^2\,\dd y
 \le C|H|^2\int_1^\infty r^{2m+1-n}\,\dd r<\infty.
\]
}
The spherical harmonics of degrees zero and one are absent in
 \(\psi_{m,\lambda}(H)\); since the centered Jacobi fields have these
angular degrees, their weighted moments vanish.
{The identity follows from uniqueness in
Proposition  \ref{prop:euclidean-response}.} The
higher degree polynomial corrections used in the upper estimate remain
separate from the finite energy construction above.

\subsection{The second variation and positivity}

{The spherical solution \(\psi[k]\) removes the scalar curvature
variation before we evaluate the second variation.}
Fix \((\lambda,b)\in\mathcal K\), write \(\gamma=\gamma_{\lambda,b}\),
and, for a symmetric covariant tensor \(k\) of rank \(2\) in
\(H^1_\gamma\), set
\begin{equation}\label{eq:scalar-corrected-metric}
 \bar k=k+\frac4{n-2}\psi[k]\gamma.
\end{equation}
{For a scalar function \(f\), the scalar curvature linearization
in the direction \(f\gamma\) is}
\[
 D\!\operatorname{Scal}_\gamma[f\gamma]=(n-1)(-\Delta_\gamma-n)f.
\]
Thus \eqref{eq:spherical-scalar-equation} and
\(4(n-1)c(n)/(n-2)=1\) imply
\[
 D\!\operatorname{Scal}_\gamma[\bar k]
 =D\!\operatorname{Scal}_\gamma[k]
 +\frac{4(n-1)}{n-2}(-\Delta_\gamma-n)\psi[k]=0.
\]

We remove the divergence of \(\bar k\) by solving an equation for a vector
field \(X\). Write \(X_j=\sum_{k=1}^n\gamma_{jk}X^k\). The Lie derivative is
\[
 (\mathcal L_X\gamma)_{ij}=\nabla^\gamma_iX_j+\nabla^\gamma_jX_i.
\]
Killing vector fields are those for which \(\mathcal L_X\gamma=0\).
Consider the equation
\begin{equation}\label{eq:divergence-correction-equation}
 \sum_{i=1}^n(\nabla^\gamma)^i
   (\nabla^\gamma_iX_j+\nabla^\gamma_jX_i)
 =\sum_{i=1}^n(\nabla^\gamma)^i\bar k_{ij},
 \qquad 1\le j\le n,
\end{equation}
and, for a solution \(X\), define
\begin{equation}\label{eq:corrected-divergence-free-tensor}
\widetilde k_{ij}
=\bar k_{ij}-\nabla^\gamma_iX_j-\nabla^\gamma_jX_i.
\end{equation}

\begin{lemma}\label{lem:divergence-correction}
Equation  \eqref{eq:divergence-correction-equation} has a unique
\(H^2_\gamma\) solution orthogonal in \(L^2_\gamma\) to the Killing vector
fields. The tensor \(\widetilde k\) satisfies
\[
 \operatorname{div}_\gamma \widetilde k=0,\qquad \operatorname{tr}_\gamma \widetilde k=0,
\]
and
\begin{equation}\label{eq:divergence-correction-bound}
 \norm{X}_{H^2_\gamma}+\norm{\widetilde k}_{H^1_\gamma}
 \le C\norm{k}_{H^1_\gamma},
\end{equation}
with \(C\) uniform for \((\lambda,b)\in\mathcal K\).
\end{lemma}

\begin{proof}
Write \(\gamma=\gamma_{\lambda,b}\) and let \(\bar k\) be given by
\eqref{eq:scalar-corrected-metric}. For every Killing vector field \(K\),
integration by parts gives
\[
 \int_{\Sph^n}\sum_{j=1}^n
 \left(\sum_{i=1}^n(\nabla^\gamma)^i\bar k_{ij}\right)K^j\,\dd V_\gamma
 =-\frac12\int_{\Sph^n}\sum_{i,j=1}^n
 \bar k^{ij}(\nabla^\gamma_iK_j+\nabla^\gamma_jK_i)\,\dd V_\gamma=0.
\]
The negative of the operator on the left of
\eqref{eq:divergence-correction-equation} has principal matrix \(|\xi|^2\delta_{ij}+\xi_i\xi_j\)
in a frame orthonormal with respect to \(\gamma\). The operator is strongly elliptic, and its kernel is
the space of Killing vector fields, because
\[
 -\int_{\Sph^n}\sum_{j=1}^nX^j
 \sum_{i=1}^n(\nabla^\gamma)^i
 (\nabla^\gamma_iX_j+\nabla^\gamma_jX_i)\,\dd V_\gamma
 =\frac12\int_{\Sph^n}|\mathcal L_X\gamma|_\gamma^2\,\dd V_\gamma.
\]
{By orthogonality of the source to the Killing fields and the
Fredholm alternative, there is a unique solution orthogonal to this
kernel.  By the elliptic estimate on its orthogonal complement,}
\[
 \norm{X}_{H^2_\gamma}
 \le C\norm{\operatorname{div}_\gamma\bar k}_{L^2_\gamma}
 \le C\norm{\bar k}_{H^1_\gamma}
 \le C\norm{k}_{H^1_\gamma},
\]
where the last inequality follows from
\eqref{eq:spherical-scalar-estimate}. The metrics \(\gamma_{\lambda,b}\)
are isometric to the same round metric, so the constants are uniform.

{It follows from \eqref{eq:divergence-correction-equation}
that \(\operatorname{div}_\gamma \widetilde k=0\).} Also
\(D\!\operatorname{Scal}_\gamma[\mathcal L_X\gamma]=0\), since \(\operatorname{Scal}_\gamma\) is constant.
Thus \(D\!\operatorname{Scal}_\gamma[\widetilde k]=0\), and {from the scalar curvature
linearization we obtain}
\[
 (-\Delta_\gamma-(n-1))\operatorname{tr}_\gamma \widetilde k=0.
\]
{By the spherical spectrum,}
\[
 n-1\notin\{\ell(\ell+n-1):\ell=0,1,2,\ldots\}
 \quad\to\quad \operatorname{tr}_\gamma \widetilde k=0.
\]
The estimate for \(\widetilde k\) follows from
\eqref{eq:corrected-divergence-free-tensor}.
All identities extend from smooth tensors to \(H^1_\gamma\) tensors
by density.
\end{proof}
In particular, \(\widetilde k\) depends linearly and continuously on \(k\).
With the scalar function fixed by \eqref{eq:spherical-scalar-equation},
we compute the metric second variation along \(\bar k\).
For a smooth metric
\(\widehat\gamma\) on \(\Sph^n\), define
\[
 \mathcal E(\widehat\gamma)
 =\int_{\Sph^n}\operatorname{Scal}_{\widehat\gamma}\,\dd V_{\widehat\gamma}
 -(n-1)(n-2)\operatorname{Vol}_{\widehat\gamma}(\Sph^n).
\]
The round metric is a critical point: \(D\mathcal E_\gamma=0\).
For a smooth symmetric tensor \(k\) of rank \(2\), define
 {
\begin{equation}\label{eq:metric-second-variation-form}
  Q_{\lambda,b}[k]
 :=-\frac{c(n)}4D^2\mathcal E_\gamma[\bar k,\bar k].
\end{equation}
}
The following identity uses Berger's second variation
formula \cite[formula  (4.2)]{BergerVariation}; we use the detailed form
recorded by Schoen \cite[Section  1]{SchoenVariational}.

\begin{proposition}\label{prop:euclidean-quadratic-form-kernel}
For every smooth symmetric tensor \(k\) of rank \(2\) on \(\Sph^n\),
 {
\begin{equation}\label{eq:general-euclidean-metric-form}
  Q_{\lambda,b}[k]
 =\frac{c(n)}8\int_{\Sph^n}
 \left(
 |\nabla^{\gamma_{\lambda,b}}\widetilde k|_{\gamma_{\lambda,b}}^2
 +2|\widetilde k|_{\gamma_{\lambda,b}}^2
 \right)\,\dd V_{\gamma_{\lambda,b}}.
\end{equation}
}
{The right side defines a continuous extension of
\(Q_{\lambda,b}\) to \(H^1_{\gamma_{\lambda,b}}\) symmetric
tensors of rank \(2\).} The form is
nonnegative and vanishes exactly when
\[
 k=\mathcal L_X\gamma_{\lambda,b}+\varphi\gamma_{\lambda,b}
\]
for some \(X\in H^2_{\gamma_{\lambda,b}}\) and
\(\varphi\in H^1_{\gamma_{\lambda,b}}\).
\end{proposition}

\begin{proof}
Write \(\gamma=\gamma_{\lambda,b}\). Suppose first that \(k\) is smooth.
{By elliptic regularity, \(\psi[k]\), \(X\), and \(\widetilde k\) are
smooth.}
 For every smooth symmetric tensor \(\ell\) of rank \(2\), diffeomorphism invariance
of \(\mathcal E\) and \(D\mathcal E_\gamma=0\) imply
 {
\[
 D^2\mathcal E_\gamma[\mathcal L_X\gamma,\ell]=0.
\]
}
Since \(\bar k=\widetilde k+\mathcal L_X\gamma\),
\[
 D^2\mathcal E_\gamma[\bar k,\bar k]
 =D^2\mathcal E_\gamma[\widetilde k,\widetilde k].
\]
{For a transverse and traceless tensor \(h\) on the unit round sphere,
that is, \(\operatorname{div}_\gamma h=0\) and
\(\operatorname{tr}_\gamma h=0\), the cited second variation formula
yields}
\[
 D^2\mathcal E_\gamma[h,h]
 =-\frac12\int_{\Sph^n}
 \left(|\nabla^\gamma h|_\gamma^2+2|h|_\gamma^2\right)\,\dd V_\gamma.
\]
{In a frame orthonormal with respect to \(\gamma\), the curvature term is
determined by}
\[
 \sum_{k,l=1}^n
 (\delta_{ij}\delta_{kl}-\delta_{il}\delta_{kj})h_{kl}=-h_{ij}.
\]
Applying these identities with \(h=\widetilde k\) and substituting in
\eqref{eq:metric-second-variation-form} proves
\eqref{eq:general-euclidean-metric-form}.

The integral vanishes exactly when \(\widetilde k=0\). In that case
\[
 k=\mathcal L_X\gamma-\frac4{n-2}\psi[k]\gamma.
\]
Conversely, suppose \(k=\mathcal L_Y\gamma+\varphi\gamma\).
Then \(\bar k=\mathcal L_Y\gamma+f\gamma\), where
\(f=\varphi+4\psi[k]/(n-2)\). Since \(D\!\operatorname{Scal}_\gamma[\bar k]=0\),
\[
 (-\Delta_\gamma-n)f=0,\qquad
 \nabla_\gamma^2 f=-f\gamma,\qquad
 f\gamma=-\frac12\mathcal L_{\nabla^\gamma f}\gamma.
\]
Thus
\[
 \bar k=\mathcal L_{Y-\frac12\nabla^\gamma f}\gamma.
\]
{By uniqueness in Lemma  \ref{lem:divergence-correction},
\(\widetilde k=0\). If \(k_\nu\to k\) in \(H^1_\gamma\), then
\eqref{eq:divergence-correction-bound} gives}
\[
 \|\widetilde k_\nu-\widetilde k\|_{H^1_\gamma}
 \le C\|k_\nu-k\|_{H^1_\gamma}\to0.
\]
The energy formula therefore extends continuously to \(H^1_\gamma\).
The kernel characterization extends as well, since
\((-\Delta_\gamma-n)f=0\) makes \(f\) smooth by elliptic regularity.
\end{proof}

For \(2\le m<(n-2)/2\) and \(H\in\mathcal V_m\), let
\(\widetilde k_{\lambda,b}[H]\) be the tensor obtained from
\(k_{\lambda,b}[H]\) by the construction above with background metric
\(\gamma_{\lambda,b}\).
By \eqref{eq:response-from-sphere} and
Lemma  \ref{lem:divergence-correction},
\[
 \begin{gathered}
 k_{\lambda,b}[H]
 +\frac4{n-2}\frac{Z_{\lambda,b}[H]}{U_{\lambda,b}}\gamma_{\lambda,b}
 =\widetilde k_{\lambda,b}[H]+\mathcal L_X\gamma_{\lambda,b},\\
 \tr_{\gamma_{\lambda,b}}\widetilde k_{\lambda,b}[H]=0,\qquad
 \diver_{\gamma_{\lambda,b}}\widetilde k_{\lambda,b}[H]=0.
 \end{gathered}
\]
The quadratic form on the polynomial coefficients is
\begin{equation}\label{eq:metric-quadratic-form}
 \mathcal Q_{\lambda,b}(H,H)
 =\frac{c(n)}8\int_{\Sph^n}
 \left(
 |\nabla^{\gamma_{\lambda,b}}\widetilde k_{\lambda,b}[H]|_{\gamma_{\lambda,b}}^2
 +2|\widetilde k_{\lambda,b}[H]|_{\gamma_{\lambda,b}}^2
 \right)\dd V_{\gamma_{\lambda,b}}.
\end{equation}
Thus
 \(\mathcal Q_{\lambda,b}(H,H)
= Q_{\lambda,b}[k_{\lambda,b}[H]]\), where square brackets
denote the general tensor form in
Proposition  \ref{prop:euclidean-quadratic-form-kernel}.

To prove strict positivity on \(\mathcal V_m\), we must show
\[
 \mathcal Q_{\lambda,b}(H,H)=0\quad\to\quad H=0.
\]
\begin{proposition}\label{prop:euclidean-strict-positivity}
For every integer \(m\) with \(2\le m<\frac{n-2}{2}\), there is
\(c_m=c_m(n,m,\mathcal K)>0\) such that
\begin{equation}\label{eq:metric-quadratic-lower-bound}
 \mathcal Q_{\lambda,b}(H,H)
 \ge c_m|H|^2
\end{equation}
for every \(H\in\mathcal V_m\) and every
\((\lambda,b)\in\mathcal K\).
\end{proposition}

\begin{proof}
{
We first record the polynomial consequence that excludes the kernel in
Proposition  \ref{prop:euclidean-quadratic-form-kernel}.

\emph{Claim.} Suppose that \(m\ge2\), \(H\in\mathcal V_m\), and that
\(X_{m+1}:\R^n\to\R^n\) and \(\varphi:\R^n\to\R\) are homogeneous
polynomials of degrees \(m+1\) and \(m\), respectively. If
\[
 H_{ij}=\partial_iX_{m+1}^j+\partial_jX_{m+1}^i
 +\varphi\delta_{ij},
\]
then \(H=0\).

To prove the claim, set \(S=y\cdot X_{m+1}\). From
\eqref{eq:metric-jet-conditions}, we obtain
\[
 2\diver X_{m+1}+n\varphi=0,\qquad
 \nabla S+mX_{m+1}-\frac2n(\diver X_{m+1})y=0.
\]
Taking the scalar product with \(y\) and then taking the divergence, we find
\[
 \diver X_{m+1}=\frac{n(m+1)}{|y|^2}S,\qquad
 |y|^2\Delta S+(m+1)\bigl[m(n-2)-2n\bigr]S=0.
\]
Decompose the polynomial \(S\) of degree \(m+2\) into homogeneous harmonic
polynomials:
\[
 S=\sum_{a=0}^{\lfloor(m+2)/2\rfloor}
 |y|^{2a}P_{m+2-2a},\qquad
 \Delta P_{m+2-2a}=0.
\]
On the summand indexed by \(a\), \(|y|^2\Delta\) has eigenvalue \(2a\bigl(2(m+1)+n-2a\bigr)\).
If \(m\ge3\), the coefficient of \(S\) in its equation is positive. If
\(m=2\), that coefficient is \(-12\); the total coefficient is nonzero for
\(a=0\), while for \(a\ge1\) it is
\[
 2a(6+n-2a)-12>0,
 \qquad n\ge7.
\]
Thus \(S=0\). {The formulas for \(\diver X_{m+1}\) and the
system of first order then imply}
\(\diver X_{m+1}=0\), \(X_{m+1}=0\), \(\varphi=0\), and \(H=0\). This
proves the claim.
}

By \eqref{eq:metric-quadratic-form}, we have\(\mathcal Q_{\lambda,b}(H,H)\ge0\) . {Suppose that
\(\mathcal Q_{\lambda,b}(H,H)=0\). By
Proposition  \ref{prop:euclidean-quadratic-form-kernel},}
\[
 k_{\lambda,b}[H]
 =\mathcal L_X\gamma_{\lambda,b}+\varphi\gamma_{\lambda,b}.
\]
In the Euclidean coordinates, divide this identity by
\((2\lambda/(\lambda^2+|y-b|^2))^2\) and take its part with zero trace.
{The components \(X^i\) then satisfy}
{
\begin{equation}\label{eq:euclidean-trace-free-equation}
 H_{ij}=\partial_iX^j+\partial_jX^i
 -\frac2n\left(\sum_{a=1}^n\partial_aX^a\right)\delta_{ij}
 \quad\text{near }y=0.
\end{equation}
}
Taking the Euclidean divergence gives
\[
 \Delta X^j+\left(1-\frac2n\right)
 \partial_j\left(\sum_{i=1}^n\partial_iX^i\right)
 =\sum_{i=1}^n\partial_iH_{ij},\qquad 1\le j\le n.
\]
The negative of the operator on the left is strongly elliptic, since
its principal matrix satisfies, for \(\xi,v\in\R^n\),
\[
 |\xi|^2|v|^2+\left(1-\frac2n\right)(\xi\cdot v)^2
 \ge |\xi|^2|v|^2.
\]
{By interior elliptic regularity, \(X\) is smooth near the
origin. Set}
\[
 X_{m+1}^j(y)=\sum_{|\alpha|=m+1}
 \frac{\partial^\alpha X^j(0)}{\alpha!}y^\alpha,
 \qquad 1\le j\le n.
\]
{Taking degree \(m\) in
\eqref{eq:euclidean-trace-free-equation}, we obtain}
{
\[
 H_{ij}=\partial_iX_{m+1}^j+\partial_jX_{m+1}^i
 -\frac2n\left(\sum_{a=1}^n\partial_aX_{m+1}^a\right)\delta_{ij}.
\]
}
{The claim implies that \(H=0\).}

Thus \(H\mapsto \widetilde k_{\lambda,b}[H]\) is injective on \(\mathcal V_m\).
We compare parameters using the fixed isometry
\eqref{eq:fixed-round-identification}.
{Using the integrable upper bound from
Lemma  \ref{lem:metric-jet-threshold} for the polynomial
\(H(b+\lambda x)\), we obtain}
\[
 \Xi_{\lambda_\nu,b_\nu}^*k_{\lambda_\nu,b_\nu}[H_\nu]
 \to
 \Xi_{\lambda,b}^*k_{\lambda,b}[H]
 \quad\text{in }H^1(g_{S})
\]
whenever \((\lambda_\nu,b_\nu,H_\nu)\to(\lambda,b,H)\) in
\(\mathcal K\times\mathcal V_m\).
Pullback preserves the scalar and vector equations and their
orthogonality conditions. By linearity and
\eqref{eq:divergence-correction-bound},
\[
 \begin{aligned}
 &\|\Xi_{\lambda_\nu,b_\nu}^*\widetilde k_{\lambda_\nu,b_\nu}[H_\nu]
       -\Xi_{\lambda,b}^*\widetilde k_{\lambda,b}[H]\|_{H^1(g_{S})}\\
 {}\le{}&C\|\Xi_{\lambda_\nu,b_\nu}^*k_{\lambda_\nu,b_\nu}[H_\nu]
       -\Xi_{\lambda,b}^*k_{\lambda,b}[H]\|_{H^1(g_{S})}
 \to0.
 \end{aligned}
\]
It then follows from
\eqref{eq:metric-quadratic-form} that \(\mathcal Q_{\lambda_\nu,b_\nu}(H_\nu,H_\nu)
 \to \mathcal Q_{\lambda,b}(H,H)\).
By compactness and the kernel calculation, let
\[
 c_m:=
 \min_{\substack{(\lambda,b)\in\mathcal K\\
 H\in\mathcal V_m,\ |H|=1}}
 \mathcal Q_{\lambda,b}(H,H)>0.
\]
{By quadratic homogeneity, we obtain
\eqref{eq:metric-quadratic-lower-bound}.}
\end{proof}

\subsection{Identification of the Euclidean coefficient}

We approximate the homogeneous direction \(H\) by \(\chi_LH\). Fix \(2\le m<(n-2)/2\) and
\(H\in\mathcal V_m\). For \(L\ge2\) and \((\lambda,b)\in\mathcal K\),
define
\begin{equation}\label{eq:cutoff-response-from-sphere}
 Z_{\lambda,b,L}[H]
 =U_{\lambda,b}\psi[k_{\lambda,b}[\chi_LH]].
\end{equation}
{By the conformal and moment identities,
\(Z_{\lambda,b,L}[H]\) satisfies}
{
\[
 \begin{gathered}
 \Delta Z_{\lambda,b,L}[H]
 +n(n+2)U_{\lambda,b}^{p-1}Z_{\lambda,b,L}[H]
 =F_{\lambda,b}[\chi_LH],\\
 \int_{\R^n}U_{\lambda,b}^{p-1}
 Z_{\lambda,b,L}[H]J_{\lambda,b,l}\,\dd y=0,
 \qquad0\le l\le n.
 \end{gathered}
\]
}

\begin{lemma}\label{lem:cutoff-response}
{This problem has a unique solution in \(\dot H^1(\R^n)\).
The cutoff corrections satisfy}
\[
 \norm{Z_{\lambda,b,L}[H]-Z_{\lambda,b}[H]}_{\dot H^1}
 \le CL^{m-\frac{n-2}{2}}|H|.
\]
For every fixed compact \(D\subset\R^n\),
\[
 Z_{\lambda,b,L}[H]\to Z_{\lambda,b}[H]
 \quad\text{in }C^2(D)
 \quad\text{as }L\to\infty,
\]
uniformly for \((\lambda,b)\in\mathcal K\) and for \(H\) satisfying
\eqref{eq:metric-jet-conditions} with \(|H|\le1\). {For each fixed
\(L\), the quotient \(Z_{\lambda,b,L}[H]/U_{\lambda,b}\) extends smoothly
to \(\Sph^n\) at \(z=0\).}
Moreover,
{
\[
 |Z_{\lambda,b,L}[H](y)|
 +|y||\nabla Z_{\lambda,b,L}[H](y)|
 \le C|H|
 \begin{cases}
  1,&|y|\le1,\\
  |y|^{m+2-n},&|y|\ge1.
 \end{cases}
\]
}
The constant depends only on \(n,m,\mathcal K\), and the derivative bounds
in \eqref{eq:euclidean-cutoff}; in particular, it is independent of
\(L,\lambda,b\), and \(H\).
\end{lemma}

\begin{proof}
{The construction \eqref{eq:cutoff-response-from-sphere} and
\eqref{eq:curvature-source-relation}  through  \eqref{eq:scalar-conjugacy} give the
equation and membership in \(\dot H^1\). The moments follow from
\eqref{eq:response-moment-transfer}; uniqueness follows from Step  1 of
Proposition  \ref{prop:euclidean-response}.}
By linearity of \(k\mapsto\psi[k]\),
\eqref{eq:spherical-scalar-estimate},
and \eqref{eq:metric-jet-tail},
 {
\[
 \|Z_{\lambda,b,L}[H]-Z_{\lambda,b}[H]\|_{\dot H^1}
 \le C\|k_{\lambda,b}[(1-\chi_L)H]\|_{H^1_{\gamma_{\lambda,b}}}
 \le CL^{m-\frac{n-2}{2}}|H|.
\]
}
The source in the spherical equation for
\(\psi[k_{\lambda,b}[\chi_LH]]\) is smooth and vanishes near \(z=0\).
This equation already holds weakly on the whole sphere by
\eqref{eq:spherical-scalar-equation}. {By elliptic regularity,}
\[
 \frac{Z_{\lambda,b,L}[H]}{U_{\lambda,b}}
 =\psi[k_{\lambda,b}[\chi_LH]]\in C^\infty(\Sph^n).
\]
For a fixed compact \(D\subset\R^n\), the two Euclidean sources agree
on a neighborhood of \(D\) for large \(L\).
{Using interior estimates and the preceding \(\dot H^1\)
estimate, we obtain}
 {
\[
 \|Z_{\lambda,b,L}[H]-Z_{\lambda,b}[H]\|_{C^2(D)}
 \le C_D\|Z_{\lambda,b,L}[H]-Z_{\lambda,b}[H]\|_{\dot H^1}
 \le C_D L^{m-\frac{n-2}{2}}|H|\to0.
\]
}
The constants are uniform for \((\lambda,b)\in\mathcal K\).

{For the pointwise estimate, using
\eqref{eq:metric-source} and the cutoff bounds, we have, uniformly in
\(L\),}
\[
 |F_{\lambda,b}[\chi_LH](y)|
 \le C|H|(1+|y|)^{m-n}.
\]
{By the uniform \(\dot H^1\) bound and local interior estimates,
we also have}
\[
 \sup_{\partial B_{R_0}}|Z_{\lambda,b,L}[H]|
 \le C|H|
\]
for a fixed sufficiently large \(R_0=R_0(n,m,\mathcal K)\).
The barrier calculation in Proposition  \ref{prop:euclidean-response}
therefore allows \(A\) and \(R_0\), independent of \(L\), to be chosen such that
\[
 \begin{gathered}
 A|H|R_0^{m+2-n}\ge
 \sup_{\partial B_{R_0}}|Z_{\lambda,b,L}[H]|,\\
 \bigl[-\Delta-n(n+2)U_{\lambda,b}^{p-1}\bigr]
 \bigl(A|H||y|^{m+2-n}\bigr)
 \ge |F_{\lambda,b}[\chi_LH](y)|,\qquad |y|\ge R_0.
 \end{gathered}
\]
The function \(|y|^{-(n-2)/4}\) is a positive supersolution there.
For each fixed \(L\), the source vanishes on \(A_{R/2,2R}\) when \(R>4L\).
{By the scaled local estimate and the Sobolev inequality,}
\[
 R^{\frac{n-2}{2}}\sup_{\partial B_R}|Z_{\lambda,b,L}[H]|
 \le C\|Z_{\lambda,b,L}[H]\|_{L^{2n/(n-2)}(A_{R/2,2R})}
 \to0.
\]
Thus, for fixed \(L,\varepsilon>0\), the function
\(A|H||y|^{m+2-n}+\varepsilon|y|^{-(n-2)/4}\) dominates
\(|Z_{\lambda,b,L}[H]|\) on both boundary spheres for large \(R\).
{Using Kato's inequality and the comparison principle on \(A_{R_0,R}\), we
obtain}
\[
 |Z_{\lambda,b,L}[H](y)|
 \le A|H||y|^{m+2-n}
 +\varepsilon|y|^{-(n-2)/4}
 \quad\text{on }A_{R_0,R}.
\]
Letting \(R\to\infty\) and then \(\varepsilon\downarrow0\) proves the
zeroth order bound with a constant independent of \(L\).
{The stated gradient bound follows from scaled interior
estimates.}
\end{proof}

{To identify the Euclidean coefficient with
\(\mathcal Q_{\lambda,b}(H,H)\), we apply the second variation to one
smooth cutoff path. Fix \(2\le m<(n-2)/2\), \(H\in\mathcal V_m\), and
\((\lambda,b)\in\mathcal K\). For \(L\ge2\), set}
{
\[
 g_{s,L}=\exp(s\chi_LH),\qquad
 V_{s,L}=U_{\lambda,b}+sZ_{\lambda,b,L}[H],
\]
and
\[
 s_L=\left[2\left(1+
 \left\|\frac{Z_{\lambda,b,L}[H]}{U_{\lambda,b}}\right\|_{L^\infty(\R^n)}
 \right)\right]^{-1}.
\]
Then \(V_{s,L}\ge U_{\lambda,b}/2>0\) whenever \(|s|<s_L\).
For \(0<\rho\le1\), set \(s=\eps\rho^m\) and define
\begin{equation}\label{eq:whole-space-scaling-family}
 \begin{aligned}
 g_{\eps,\rho,L}&:=g_{\eps\rho^m,L}
 =\exp(\eps\rho^m\chi_LH),\\
 V_{\eps,\rho,L}&:=V_{\eps\rho^m,L}
 =U_{\lambda,b}+\eps\rho^mZ_{\lambda,b,L}[H].
 \end{aligned}
\end{equation}
The derivatives in \(\rho\) keep \(L,H,\lambda,b,\eps\) fixed. Therefore,
\[
 \rho\partial_\rho=ms\partial_s,
 \qquad
 \rho\partial_\rho V_{\eps,\rho,L}=msZ_{\lambda,b,L}[H].
\]
}

\begin{proposition}\label{prop:whole-space-coefficient}
For \(|s|<s_L\), the metric \(4V_{s,L}^{4/(n-2)}g_{s,L}\) extends
smoothly to \(\Sph^n\). Moreover,
\begin{equation}\label{eq:whole-space-coefficient}
 \begin{aligned}
 &-\frac14\lim_{L\to\infty}
 \left.\partial_\eps^2\right|_{\eps=0}
 \int_{\R^n}
 \left[
 \sum_{k,l=1}^n(\rho\partial_\rho g_{\eps,\rho,L}^{kl})
 \partial_kV_{\eps,\rho,L}\partial_lV_{\eps,\rho,L}
 +\rho\partial_\rho(c(n)\operatorname{Scal}_{g_{\eps,\rho,L}})
 V_{\eps,\rho,L}^2
 \right]\,\dd y\\
 {}={}&2^{3-n}m\rho^{2m}\mathcal Q_{\lambda,b}(H,H).
 \end{aligned}
\end{equation}
The variation is taken at fixed \(L\), followed by \(L\to\infty\).
The limit is uniform for \((\lambda,b)\in\mathcal K\), with \(m,H,\rho\)
fixed.
\end{proposition}

\begin{proof}
\emph{Step 1. The smooth metric on the sphere.}
{For the path \((g_{s,L},V_{s,L})\), set} \(\widehat\gamma_{s,L}=4V_{s,L}^{4/(n-2)}g_{s,L}\).
Outside \(B_{2L}\), its expression is
\[
 \widehat\gamma_{s,L}
 =\left(1+s\frac{Z_{\lambda,b,L}[H]}{U_{\lambda,b}}\right)^{4/(n-2)}
 \gamma_{\lambda,b}.
\]
{By Lemma  \ref{lem:cutoff-response}, the quotient extends smoothly
across \(z=0\). From the definition of \(s_L\), the conformal factor is
positive for \(|s|<s_L\); hence the metric extends smoothly to the sphere.}

\emph{Step 2. The energy coefficient.}
{By conformal covariance and \(\det g_{s,L}=1\),}
 {
\begin{equation}\label{eq:whole-space-energy-expansion}
 \begin{aligned}
 &-\frac12\int_{\R^n}
 \bigl(|\nabla V_{s,L}|_{g_{s,L}}^2+c(n)\operatorname{Scal}_{g_{s,L}}V_{s,L}^2\bigr)\,\dd y
 +\frac{n(n-2)}{p+1}\int_{\R^n}V_{s,L}^{p+1}\,\dd y\\
 {}={}&-\frac{c(n)}{2^{n-1}}\mathcal E(\widehat\gamma_{s,L})=-\frac{c(n)}{2^{n-1}}\mathcal E(\gamma_{\lambda,b})
 +2^{2-n}s^2 Q_{\lambda,b}[k_{\lambda,b}[\chi_LH]]
 +O_L(s^3).
 \end{aligned}
\end{equation}
}
Indeed, the tangent of \(\widehat\gamma_{s,L}\) at zero is
\[
 k_{\lambda,b}[\chi_LH]
 +\frac4{n-2}\frac{Z_{\lambda,b,L}[H]}{U_{\lambda,b}}\gamma_{\lambda,b}.
\]
{The expansion therefore follows from
\eqref{eq:metric-second-variation-form} and
\(D\mathcal E_{\gamma_{\lambda,b}}=0\).}
Here \(O_L(s^3)\) is taken as \(s\to0\), with \(L\) fixed.

\emph{Step 3. Differentiation of the metric coefficients.}
{From the correction equation, we have}
\[
 -L_{g_{s,L}}V_{s,L}-n(n-2)V_{s,L}^p=O_L(s^2)
 \quad\text{in }\dot H^{-1}(\R^n).
\]
By the preceding differentiation identities, the contribution from
differentiating the scalar function in
\eqref{eq:whole-space-energy-expansion} is the negative of the pairing.
It satisfies
\[
 \begin{aligned}
 &\left|\left\langle
 -L_{g_{s,L}}V_{s,L}-n(n-2)V_{s,L}^p,\,
 msZ_{\lambda,b,L}[H]\right\rangle\right|\\
 {}\le{}&
 \|-L_{g_{s,L}}V_{s,L}-n(n-2)V_{s,L}^p\|_{\dot H^{-1}}
 |ms|\,\|Z_{\lambda,b,L}[H]\|_{\dot H^1}
 \le C_L|s|^3.
 \end{aligned}
\]
{Fix \(L\) and \(s\). By the smooth extension in
Lemma  \ref{lem:cutoff-response},}
\[
 \int_{\partial B_R}
 |\partial_rV_{s,L}|\,|Z_{\lambda,b,L}[H]|\,\dd S
 \le C_LR^{2-n}\to0
 \qquad\text{as }R\to\infty.
\]
Differentiating the energy therefore yields
 {
\[
 \begin{aligned}
 &-\frac12\int_{\R^n}
 \left[
 \sum_{k,l=1}^n(\rho\partial_\rho g_{\eps,\rho,L}^{kl})
 \partial_kV_{\eps,\rho,L}\partial_lV_{\eps,\rho,L}
 +\rho\partial_\rho(c(n)\operatorname{Scal}_{g_{\eps,\rho,L}})V_{\eps,\rho,L}^2
 \right]\,\dd y\\
 {}={}&2^{3-n}m s^2
  Q_{\lambda,b}[k_{\lambda,b}[\chi_LH]]+O_L(s^3).
 \end{aligned}
\]
}
Taking \(\frac{1}{2}\partial_\eps^2|_{\eps=0}\) identifies the coefficient at
fixed \(L\).

\emph{Step 4. Removal of the cutoff.}
Equations  \eqref{eq:metric-jet-tail} and
\eqref{eq:divergence-correction-bound} imply
 {
\[
  Q_{\lambda,b}[k_{\lambda,b}[\chi_LH]]
 \to Q_{\lambda,b}[k_{\lambda,b}[H]]
 =\mathcal Q_{\lambda,b}(H,H).
\]
}
The convergence is uniform on \(\mathcal K\). This proves
\eqref{eq:whole-space-coefficient}.
\end{proof}

{The correction in \eqref{eq:corrected-profile} is furnished by
Proposition  \ref{prop:euclidean-response}. The sign and value of the
Euclidean coefficient used in
Proposition  \ref{prop:finite-annulus-self-pairing} follow from Propositions
\ref{prop:euclidean-strict-positivity} and
\ref{prop:whole-space-coefficient}.}

\section{Pohozaev estimate for the corrected profile}
\label{sec:reference-lower-bound}

{
The purpose of this section is to prove a lower bound for the Pohozaev quantity
\(\widetilde Q_j\). Using the correction constructed in
Section  \ref{sec:euclidean-response}, we form the profile \(\Phi_j\) and select
the final bubble parameters by annular orthogonality. We expand
\(\widetilde Q_j=B_j(\Phi_j,\Phi_j)\) to second order, estimate the
integral remainder, and isolate the positive leading term when the first
nonzero degree is below \((n-2)/2\). The resulting lower bound is combined
with Proposition  \ref{prop:exact-reference-comparison} in
Section  \ref{sec:recurrence}.
}
{Throughout this section, \(C\) denotes a positive constant that
may depend on \(n\), the compact parameter set \(\mathcal K\), and the fixed
metric jets, but not on \(j\).}

\subsection{The corrected profile and final modulation}

The annulus \(A_{\sigma_j,R_j}\), the scale \(\eta_j\), and the function
\(v_j\) are given by \eqref{eq:annular-radii},
\eqref{eq:metric-perturbation-scale}, and \eqref{eq:annular-rescaling}.
We use the coefficients \(H^{(m)}\) from \eqref{eq:metric-jets} and the
parameter set \(\mathcal K\) from \eqref{eq:bubble-parameter-set}.

Only the degrees \(m<(n-2)/2\) admit the corrections used below. Accordingly,
for \((\lambda,b)\in\mathcal K\), define
\(\Phi_{j,\lambda,b}\) on \(\R^n\) by
\begin{equation}\label{eq:corrected-profile}
 \Phi_{j,\lambda,b}
 =U_{\lambda,b}
 +\sum_{\substack{m\in\mathbb N\\2\le m<\frac{n-2}{2}}}
  \rho_j^mZ_{\lambda,b}[H^{(m)}].
\end{equation}
{
All estimates below are uniform for
\((\lambda,b)\in\mathcal K\).}
\begin{lemma}\label{lem:corrected-profile-estimate}
There are \(C>0\) and \(j_1\) such that, if \(j\ge j_1\),
\((\lambda,b)\in\mathcal K\), \(y\in A_{\sigma_j,R_j}\), and
\(|\gamma|,|\beta|\le2\), then
\begin{equation}\label{eq:corrected-profile-estimate}
 \left|
 \partial_y^\gamma
 (\Phi_{j,\lambda,b}-U_{\lambda,b})(y)
 \right|
 \le C\eta_jU_{\lambda,b}(y)(1+|y|)^{-|\gamma|},
\end{equation}
and
\begin{equation}\label{eq:corrected-profile-parameter-derivatives}
 \left|
 \partial_{(\lambda,b)}^\beta
 (\Phi_{j,\lambda,b}-U_{\lambda,b})(y)
 \right|
 \le C\eta_jU_{\lambda,b}(y).
\end{equation}
Moreover,
\[
 \frac12U_{\lambda,b}
 \le\Phi_{j,\lambda,b}
 \le\frac32U_{\lambda,b}
 \quad\text{on }A_{\sigma_j,R_j}.
\]
\end{lemma}

\begin{proof}
Uniformly for \((\lambda,b)\in\mathcal K\),
\[
 C^{-1}(1+|y|)^{2-n}
 \le U_{\lambda,b}(y)
 \le C(1+|y|)^{2-n}.
\]
By Proposition  \ref{prop:euclidean-response},
\[
 \left|
 \partial_y^\gamma
 (\Phi_{j,\lambda,b}-U_{\lambda,b})(y)
 \right|
 \le C\sum_{2\le m<\frac{n-2}{2}}
 \rho_j^m|H^{(m)}|
 (1+|y|)^{m+2-n-|\gamma|}.
\]
For \(y\in A_{\sigma_j,R_j}\),
\[
 \rho_j(1+|y|)
 \le\rho_j+\rho_jR_j
 =\rho_j+\bar\rho_{j-1}
 \le2\bar\rho_{j-1}.
\]
Since \(m\ge2\) and
\(\eta_j=\bar\rho_{j-1}^2\), the last two displays imply
\eqref{eq:corrected-profile-estimate}. {The parameter estimates in
Proposition  \ref{prop:euclidean-response} give the same sum without the factor
\((1+|y|)^{-|\gamma|}\); hence the same argument proves
\eqref{eq:corrected-profile-parameter-derivatives}. Since \(\eta_j\to0\),
the positivity conclusion follows from \eqref{eq:corrected-profile-estimate}
with \(\gamma=0\), after increasing \(j_1\).}
\end{proof}

Define the modulation map by
\begin{equation}\label{eq:modulation-map}
 \mathcal M_j(\lambda,b)
 =\int_{A_{\sigma_j,R_j}}
 U_{\lambda,b}^{p-1}(v_j-\Phi_{j,\lambda,b})
 \nabla_{(\lambda,b)}U_{\lambda,b}\,\dd y.
\end{equation}
Here
\(\nabla_{(\lambda,b)}U=(\partial_\lambda U,\partial_{b^1}U,\ldots,
\partial_{b^n}U)^{\mathsf T}\).
The derivative of the modulation map at the approximate parameters is
compared with the weighted Gram matrix
\[
 \mathbf G_{\lambda,b}
 =\int_{\R^n}U_{\lambda,b}^{p-1}
 \nabla_{(\lambda,b)}U_{\lambda,b}
 \otimes\nabla_{(\lambda,b)}U_{\lambda,b}\,\dd y.
\]
The bubble derivatives are linearly independent. Their weighted Gram
matrix is therefore positive definite, and {by continuity on
\(\mathcal K\),}
\begin{equation}\label{eq:modulation-gram-positivity}
 \inf_{(\lambda,b)\in\mathcal K}
 \lambda_{\min}(\mathbf G_{\lambda,b})>0.
\end{equation}

The integrands defining \(\mathcal M_j\) and its first two parameter
derivatives are bounded by \(CU_{\lambda,b}^{p+1}\). {Indeed, by
Lemma  \ref{lem:annular-bubble-comparison},
\eqref{eq:corrected-profile-parameter-derivatives}, and the explicit bubble
formula, each parameter derivative of a bubble factor is bounded by a
constant times that factor.} Uniformly on \(\mathcal K\),
\begin{equation}\label{eq:modulation-tail-bound}
 \int_{B_r}U_{\lambda,b}^{p+1}\,\dd y+
 \int_{\R^n\setminus B_R}U_{\lambda,b}^{p+1}\,\dd y
 \le C(r^n+R^{-n}),\qquad 0<r<1<R.
\end{equation}
This estimate controls differentiation under the integral and passage
from fixed annuli to \(\R^n\).

\begin{proposition}\label{prop:final-modulation}
There is \(j_2\ge j_1\) such that, for every \(j\ge j_2\), there are
\((\lambda_j,b_j)\in\mathcal K\) satisfying
\begin{equation}\label{eq:final-orthogonality}
 \int_{A_{\sigma_j,R_j}}U_{\lambda_j,b_j}^{p-1}
 (v_j-\Phi_{j,\lambda_j,b_j})J_{\lambda_j,b_j,l}\,\dd y=0,
 \qquad 0\le l\le n.
\end{equation}
For every fixed \(0<r<R<\infty\), these parameters satisfy
\begin{equation}\label{eq:final-parameter-approximation}
 \norm{v_j-U_{\lambda_j,b_j}}_{C^2(\overline{A_{r,R}})}
 \to0
 \quad\text{as }j\to\infty.
\end{equation}
\end{proposition}

\begin{proof}
\emph{Step 1. Approximate bubble parameters.}
Choose \((\lambda_j^0,b_j^0)\in\mathcal K\) minimizing
\[
 \norm{v_j-U_{\lambda,b}}_{C^2(\overline{A_{1/3,3}})}
 \quad\text{over }(\lambda,b)\in\mathcal K.
\]
{By Lemma  \ref{lem:distance-to-bubble-family}, this minimum tends
to zero.}
Every subsequence has a further subsequence such that
\[
 (\lambda_{j_k}^0,b_{j_k}^0)\to(\lambda,b)\in\mathcal K,
 \qquad
 v_{j_k}\to U
 \quad\text{in }C^2_{\rm loc}(\R^n\setminus\{0\}),
\]
{By Lemma  \ref{lem:peak-bubble-compactness}, \(U\) is a bubble.}
{By the minimizing property,}
\[
 \|U-U_{\lambda,b}\|_{C^2(\overline{A_{1/3,3}})}
 =\lim_{k\to\infty}
 \|v_{j_k}-U_{\lambda_{j_k}^0,b_{j_k}^0}\|_{C^2(\overline{A_{1/3,3}})}
 =0.
\]
{By the explicit bubble formula
\eqref{eq:standard-bubble}, \(U=U_{\lambda,b}\) on \(\R^n\).}
Thus every parameter limit satisfies \eqref{eq:peak-limit-parameter-bounds}.
{After increasing \(j_1\),}
\[
 \frac34\lambda_*\le\lambda_j^0\le\frac32,\qquad
 |b_j^0|\le\frac{1+3b_*}{4}.
\]
{Repeating the subsequence argument, we obtain, for every fixed
\(0<r<R<\infty\),}
\[
 \|v_j-U_{\lambda_j^0,b_j^0}\|_{C^2(\overline{A_{r,R}})}
 \to0.
\]

{
These inequalities and \eqref{eq:bubble-parameter-set} show that the closed
parameter ball centered at \((\lambda_j^0,b_j^0)\) of radius
\[
 \frac18\min\{\lambda_*,1-b_*,1\}
\]
is contained in \(\mathcal K\).
}

\emph{Step 2. The modulation derivative.}
{For \(\tau\in\{\lambda,b^1,\ldots,b^n\}\), differentiating
\eqref{eq:modulation-map}, we obtain}
\[
 \begin{aligned}
 \partial_\tau\mathcal M_j(\lambda,b)
 &=-\int_{A_{\sigma_j,R_j}}
 U_{\lambda,b}^{p-1}(\partial_\tau\Phi_{j,\lambda,b})
 \nabla_{(\lambda,b)}U_{\lambda,b}\,\dd y\\
 &\quad+\int_{A_{\sigma_j,R_j}}
 (v_j-\Phi_{j,\lambda,b})
 \partial_\tau\bigl(U_{\lambda,b}^{p-1}
 \nabla_{(\lambda,b)}U_{\lambda,b}\bigr)\,\dd y.
 \end{aligned}
\]
{At \((\lambda,b)=(\lambda_j^0,b_j^0)\), it follows from
\eqref{eq:corrected-profile-parameter-derivatives} that}
 {
\[
 \left|\int_{A_{\sigma_j,R_j}}U_{\lambda,b}^{p-1}
 (\partial_\tau\Phi_{j,\lambda,b}-\partial_\tau U_{\lambda,b})
 \nabla_{(\lambda,b)}U_{\lambda,b}\,\dd y\right|
 \le C\eta_j\int_{\R^n}U_{\lambda,b}^{p+1}\,\dd y
 \le C\eta_j.
\]
}
Step  1 controls the second integral on fixed annuli.
{Using \eqref{eq:modulation-tail-bound} on the complementary
regions, we obtain}
 {
\[
 \limsup_{j\to\infty}\bigl(
 |\mathcal M_j(\lambda_j^0,b_j^0)|
 +\|D\mathcal M_j(\lambda_j^0,b_j^0)
       +\mathbf G_{\lambda_j^0,b_j^0}\|\bigr)
 \le C(r^n+R^{-n}),\qquad 0<r<1<R.
\]
}
Letting \(r\downarrow0\) and \(R\to\infty\) yields
\[
 |\mathcal M_j(\lambda_j^0,b_j^0)|\to0,
\]
and
\[
 D\mathcal M_j(\lambda_j^0,b_j^0)
 =-\mathbf G_{\lambda_j^0,b_j^0}+o(1).
\]
{By \eqref{eq:modulation-gram-positivity} and the integrable
bounds for the second derivatives, in the Euclidean operator norms,}
\[
 \begin{gathered}
 \left\|[D\mathcal M_j(\lambda_j^0,b_j^0)]^{-1}\right\|\le C,\quad 
 \sup_{\substack{(\lambda,b)\in\mathcal K\\
 |(\lambda,b)-(\lambda_j^0,b_j^0)|
 \le\frac18\min\{\lambda_*,1-b_*,1\}}}
 \|D^2\mathcal M_j(\lambda,b)\|\le C.
 \end{gathered}
\]

\emph{Step 3. Selection of the parameters.}
The map
\[
 (\lambda,b)\longmapsto
 (\lambda,b)-[D\mathcal M_j(\lambda_j^0,b_j^0)]^{-1}
 \mathcal M_j(\lambda,b)
\]
is a contraction on the ball of radius
\(2C|\mathcal M_j(\lambda_j^0,b_j^0)|\) about \((\lambda_j^0,b_j^0)\),
for all sufficiently large \(j\). Its fixed point satisfies
\[
 \mathcal M_j(\lambda_j,b_j)=0,\qquad
 |(\lambda_j,b_j)-(\lambda_j^0,b_j^0)|
 \le2C|\mathcal M_j(\lambda_j^0,b_j^0)|\to0.
\]
Since \(J_{\lambda,b,0}=\lambda\partial_\lambda U_{\lambda,b}\) and
\(J_{\lambda,b,l}=\partial_{b^l}U_{\lambda,b}\), this is
\eqref{eq:final-orthogonality}.
{Combining the parameter convergence with Step  1, we obtain
\eqref{eq:final-parameter-approximation}.}
\end{proof}

For the rest of the manuscript, \((\lambda_j,b_j)\) denotes this fixed
choice.  Define the abbreviations
\[
 U_j:=U_{\lambda_j,b_j},
 \qquad
 \Phi_j:=\Phi_{j,\lambda_j,b_j}.
\]
These abbreviations suppress the selected parameters. Using
 \( B_j\) from \eqref{eq:symmetric-dilation-form}, define
 {
\begin{equation}\label{eq:reference-pohozaev-quantity}
 \widetilde Q_j
 = B_j(\Phi_j,\Phi_j).
\end{equation}
}

\subsection{Expansion of the Pohozaev integral for the corrected profile}

We now estimate \(\widetilde Q_j\).
Unless another domain is displayed, integrals in the following estimates
are over \(A_{\sigma_j,R_j}\) in the Euclidean \(y\) coordinates.
We write \(r=|y|\) in the radial estimates. Every little \(o\) term is
taken as \(j\to\infty\), with the background metric fixed, uniformly for
\((\lambda_j,b_j)\in\mathcal K\).

To separate \(\widetilde Q_j\) into its linear, quadratic, and cubic Taylor terms, for \(0\le\tau\le1\) define a metric \(g_{j,\tau}\) and a scalar function
\(V_{j,\tau}\) on \(A_{\sigma_j,R_j}\) by
\begin{equation}\label{eq:reference-path}
 g_{j,\tau}(y)=\exp\{\tau h(\rho_jy)\},\qquad
 V_{j,\tau}=U_j
 +\tau\bigl(\Phi_j-U_j\bigr).
\end{equation}
By the conformal normal conditions,
\[
 \det g_{j,\tau}=1,\qquad
 \sum_{l=1}^ng_{j,\tau}^{kl}(y)y_l=y^k
 \qquad 1\le k\le n.
\]
In particular, \(g_{j,0}^{kl}=\delta^{kl}\), \(g_{j,1}=g_j\), and
\(V_{j,1}=\Phi_j\). {Define
\(\mathfrak q_j:[0,1]\to\R\) by}
\begin{equation}\label{eq:reference-pohozaev-path}
 \mathfrak q_j(\tau)=\int_{A_{\sigma_j,R_j}}
 \bigl(D_{\mathrm{di}}V_{j,\tau}\bigr)
 (\Delta-L_{g_{j,\tau}})V_{j,\tau}\,\dd y.
\end{equation}
{From \eqref{eq:reference-path} and
\eqref{eq:reference-pohozaev-quantity}, we have}
\[
 \mathfrak q_j(1)=\widetilde Q_j.
\]
Since \(g_{j,0}\) is Euclidean and \(V_{j,0}=U_j\),
\(\mathfrak q_j(0)=0\). {Applying Taylor's formula, we obtain}
\begin{equation}\label{eq:reference-pohozaev-taylor-formula}
 \widetilde Q_j=\mathfrak q_j'(0)+\frac12\mathfrak q_j''(0)
 +\frac12\int_0^1(1-\tau)^2\mathfrak q_j^{(3)}(\tau)\,\dd\tau.
\end{equation}
{We estimate the quadratic term first, using
Lemma  \ref{lem:annular-bilinear-density} for the common radial bounds and
Proposition  \ref{prop:finite-annulus-self-pairing} for the positive diagonal
coefficient. The linear and cubic terms are estimated in
Lemmas  \ref{lem:reference-linear-term} and
\ref{lem:reference-third-variation}.}

The next lemma writes the Pohozaev integral in terms of the scale
derivatives of the metric coefficients. Fix a \(C^1\) family \(g_\rho\)
of \(C^3\) metrics on \(A_{\sigma,R}\), with
\[
 \det g_\rho=1,\qquad
 \sum_{l=1}^ng_\rho^{kl}(y)y_l=y^k.
\]
For \(v\in C^2(\overline{A_{\sigma,R}})\) and \(\sigma\le r\le R\),
denote the contribution from a sphere by
\begin{equation}\label{eq:scaling-boundary-term}
 \mathcal B_{g_\rho}(r,v)
 =\frac r2\int_{\partial B_r}
 \left[
 \sum_{k,l=1}^n(g_\rho^{kl}-\delta^{kl})\partial_kv\partial_lv
 +c(n)\operatorname{Scal}_{g_\rho}v^2
 \right]\,\dd S.
\end{equation}

\begin{lemma}\label{lem:euclidean-scaling-identity}
Assume that
\[
 \rho\partial_\rho g_\rho^{kl}=y\cdot\nabla_y g_\rho^{kl},
 \qquad
 \rho\partial_\rho \operatorname{Scal}_{g_\rho}
 =2\operatorname{Scal}_{g_\rho}+y\cdot\nabla_y\operatorname{Scal}_{g_\rho}.
\]
At each fixed \(\rho\), every \(v\in C^2(\overline{A_{\sigma,R}})\)
satisfies
\begin{equation}\label{eq:euclidean-scaling-identity}
 \begin{aligned}
 &\int_{A_{\sigma,R}}
 \left(\frac{n-2}{2}v+y\cdot\nabla_yv\right)
 (\Delta-L_{g_\rho})v\,\dd y\\
 {}={}&-\frac12\int_{A_{\sigma,R}}
 \left[
 \sum_{k,l=1}^n(\rho\partial_\rho g_\rho^{kl})\partial_kv\partial_lv
 +\rho\partial_\rho(c(n)\operatorname{Scal}_{g_\rho})v^2
 \right]\,\dd y\\
 &\quad+\frac12\int_{\partial A_{\sigma,R}}(y\cdot\nu)
 \left[
 \sum_{k,l=1}^n(g_\rho^{kl}-\delta^{kl})\partial_kv\partial_lv
 +c(n)\operatorname{Scal}_{g_\rho}v^2
 \right]\,\dd S.
 \end{aligned}
\end{equation}
Here \(\nu\) is the outward Euclidean unit normal. The last integral equals
\(\mathcal B_{g_\rho}(R,v)-\mathcal B_{g_\rho}(\sigma,v)\).
\end{lemma}

\begin{proof}
Since \(\det g_\rho=1\),
\[
 (\Delta-L_{g_\rho})v
 =-\sum_{k,l=1}^n
 \partial_k\bigl((g_\rho^{kl}-\delta^{kl})\partial_lv\bigr)
 +c(n)\operatorname{Scal}_{g_\rho}v.
\]
{Integrating the principal part against
\((n-2)v/2+y\cdot\nabla_yv\) and using symmetry, we obtain the volume
term}
\[
 -\frac12\int_{A_{\sigma,R}}\sum_{k,l=1}^n
 y\cdot\nabla_y(g_\rho^{kl}-\delta^{kl})\partial_kv\partial_lv\,\dd y.
\]
The scalar term contributes
\[
 -c(n)\int_{A_{\sigma,R}}
 \left(\operatorname{Scal}_{g_\rho}+\frac12y\cdot\nabla_y\operatorname{Scal}_{g_\rho}\right)v^2\,\dd y.
\]
{Using the two scale identities, we recover the volume
integral in \eqref{eq:euclidean-scaling-identity}.} On either boundary sphere,
\(\sum_k(g_\rho^{kl}-\delta^{kl})\nu_k=0\). The remaining boundary
contribution is the last integral in that formula.
\end{proof}

To identify the quadratic term in
\eqref{eq:reference-pohozaev-taylor-formula}, we retain the homogeneous metric
directions and their scalar corrections in the following variation. The
annulus \(A_{\sigma_j,R_j}\), the scale \(\rho_j\),
and the bubble parameters \((\lambda_j,b_j)\) are fixed. For
\(H=\sum_{m=2}^{d}H^{(m)}\), where every
\(H^{(m)}\) is a symmetric homogeneous polynomial matrix of degree \(m\)
satisfying \eqref{eq:metric-jet-conditions}, set
\begin{equation}\label{eq:finite-annulus-quadratic-path}
 \begin{aligned}
 g_{\eps,H}
 &=\exp\left\{\eps
 \sum_{m=2}^{d}\rho_j^mH^{(m)}\right\},\\
 v_{\eps,H}
 &=U_j
 +\eps\sum_{2\le m<\frac{n-2}{2}}\rho_j^m
 Z_{\lambda_j,b_j}[H^{(m)}].
 \end{aligned}
\end{equation}
The quadratic Taylor coefficient defines the quadratic form
\begin{equation}\label{eq:finite-annulus-quadratic-form}
 I_j^{\mathrm{ann}}(H,H)=\frac12
 \left.\partial_\eps^2\right|_{\eps=0}
 \int_{A_{\sigma_j,R_j}}
 \left(\frac{n-2}{2}v_{\eps,H}
 +y\cdot\nabla_yv_{\eps,H}\right)
 (\Delta-L_{g_{\eps,H}})v_{\eps,H}\,\dd y.
\end{equation}
For
\(\widetilde H=\sum_{m=2}^{d}\widetilde H^{(m)}\), with every
\(\widetilde H^{(m)}\) a symmetric homogeneous polynomial matrix of degree \(m\)
satisfying \eqref{eq:metric-jet-conditions}, define
\begin{equation}\label{eq:finite-annulus-polarization}
 I_j^{\mathrm{ann}}(H,\widetilde H)=\frac12\bigl[
 I_j^{\mathrm{ann}}(H+\widetilde H,H+\widetilde H)
 -I_j^{\mathrm{ann}}(H,H)-I_j^{\mathrm{ann}}(\widetilde H,\widetilde H)
 \bigr].
\end{equation}

For integers
\(2\le\alpha,\beta\le d\), the
\(H^{(\alpha)}\) direction in the polarization
\(I_j^{\mathrm{ann}}(H^{(\alpha)},H^{(\beta)})\) from
\eqref{eq:finite-annulus-polarization} includes
\(\rho_j^\alpha Z_{\lambda_j,b_j}[H^{(\alpha)}]\) exactly when
\(\alpha<(n-2)/2\); the \(H^{(\beta)}\) direction is interpreted in the
same way. For the estimate below, we expand this polarization to second
order, integrate the principal divergence terms once by parts, and then
integrate over \(\Sph^{n-1}\).

{We also apply this expansion to a symmetric tensor \(T\) and a
scalar function \(z\) on \(A_{\sigma_j,R_j}\). For a fixed
\(\alpha\ge2\), suppose that}
\begin{equation}\label{eq:extended-bilinear-direction-bounds}
 \begin{gathered}
 T=T^{\mathsf T},\qquad
 \sum_{k=1}^nT_{kk}=0,\qquad
 \sum_{l=1}^nT_{kl}(y)y_l=0,\\
 \sum_{\ell=0}^2r^\ell|\nabla_y^\ell T(y)|
 \le C\rho_j^\alpha r^\alpha,\\
 \sum_{\ell=0}^2r^\ell|\nabla_y^\ell z(y)|
 \le C\rho_j^\alpha
 \begin{cases}
  1,&0<r\le1,\\
  r^{\alpha+2-n},&1\le r\le R_j,
 \end{cases}
 \qquad r=|y|.
 \end{gathered}
\end{equation}
{The constant in these bounds is independent of \(j\), and \(z\)
may vanish.}

For two such pairs \((T_1,z_1)\) and \((T_2,z_2)\), use the mixed
coefficient of the family
\[
 g_{s,\tau}=\exp(sT_1+\tau T_2),\qquad
 v_{s,\tau}=U_j+sz_1+\tau z_2.
\]
{The corresponding mixed coefficient is one half of
\(\partial_s\partial_\tau|_{s=\tau=0}\) of the integral in
\eqref{eq:finite-annulus-quadratic-form}.}

\begin{lemma}\label{lem:annular-bilinear-density}
{Fix integers \(\alpha,\beta\ge2\). For \(i=1,2\), let
\((T_i,z_i)\) satisfy \eqref{eq:extended-bilinear-direction-bounds} with
exponent \(\alpha\) when \(i=1\) and exponent \(\beta\) when \(i=2\).
There is \(C>0\), independent of \(j\), such that, for all sufficiently
large \(j\), the mixed coefficient defined above can be written as a radial integral over
\([\sigma_j,R_j]\) and two boundary terms. After the divergence terms are
integrated once by parts, the absolute value of the radial density is bounded
by}
\begin{equation}\label{eq:annular-bilinear-density}
 C\rho_j^{\alpha+\beta}
 \begin{cases}
  r^{n-3}\,\dd r,&0<r\le1,\\
  r^{\alpha+\beta+1-n}\,\dd r,&1\le r\le R_j,
 \end{cases}
\end{equation}
{and the boundary terms over \(\partial B_{\sigma_j}\) and
\(\partial B_{R_j}\) are bounded, respectively, by}
\begin{equation}\label{eq:annular-bilinear-boundary-density}
 C\rho_j^{\alpha+\beta}\sigma_j^{n-2},\qquad
 C\rho_j^{\alpha+\beta}R_j^{\alpha+\beta+2-n},
\end{equation}
{The constant may depend on \(n\), \(\alpha\), \(\beta\), the
compact parameter set \(\mathcal K\), and the constants in
\eqref{eq:extended-bilinear-direction-bounds}.}
\end{lemma}

\begin{proof}
Use the family \(g_{s,\tau}=\exp(sT_1+\tau T_2)\) specified above.
At \(s=\tau=0\),
\[
 \partial_sg_{s,\tau}^{-1}=-T_1,\qquad
 \partial_\tau g_{s,\tau}^{-1}=-T_2,\qquad
 \partial_s\partial_\tau g_{s,\tau}^{-1}
 =\frac12(T_1T_2+T_2T_1).
\]
The linear curvature coefficients are
\(\sum_{k,l=1}^n\partial_k\partial_l(T_i)_{kl}\), for \(i=1,2\).
The mixed curvature coefficient satisfies
\[
 \left|\left.\partial_s\partial_\tau \operatorname{Scal}_{g_{s,\tau}}
 \right|_{s=\tau=0}\right|
 \le C\bigl(
 |T_1||\nabla_y^2T_2|+|T_2||\nabla_y^2T_1|
 +|\nabla_yT_1||\nabla_yT_2|\bigr).
\]
These formulas and \eqref{eq:extended-bilinear-direction-bounds}
determine all terms in the mixed coefficient.
{The complete expansion is recorded in
\eqref{eq:exact-mixed-coefficient}.}

Integrate each divergence term once. For \(r\ge1\), the terms with two
metric factors have radial bound
\[
 C\rho_j^{\alpha+\beta}
 \bigl(r^{\alpha+\beta}r^{1-n}r^{1-n}
 +r^{\alpha+\beta-2}r^{2-n}r^{2-n}\bigr)r^{n-1}\,\dd r
 =C\rho_j^{\alpha+\beta}r^{\alpha+\beta+1-n}\,\dd r.
\]
The terms with \(z_1\) and \(T_2\) have radial bound
\[
 C\rho_j^{\alpha+\beta}
 \bigl(r^\beta r^{\alpha+1-n}r^{1-n}
 +r^{\beta-2}r^{\alpha+2-n}r^{2-n}\bigr)r^{n-1}\,\dd r
 =C\rho_j^{\alpha+\beta}r^{\alpha+\beta+1-n}\,\dd r.
\]
Interchanging the two pairs covers the terms with \(T_1\) and \(z_2\).
{For \(r\le1\), by the bounds in
\eqref{eq:extended-bilinear-direction-bounds}, the common integrand is
bounded by \(C\rho_j^{\alpha+\beta}r^{-2}\), and hence the density is}
\(C\rho_j^{\alpha+\beta}r^{n-3}\,\dd r\).

On a sphere of radius \(r\ge1\), the corresponding boundary products
are bounded by
\[
 C\rho_j^{\alpha+\beta}
 r^{\alpha+\beta}r^{1-n}r^{2-n}r^{n-1}
 =C\rho_j^{\alpha+\beta}r^{\alpha+\beta+2-n}.
\]
For \(r\le1\), they are bounded by
\(C\rho_j^{\alpha+\beta}r^{n-2}\).
This proves both the volume and boundary estimates.
\end{proof}

{Integrating these densities and including the two boundary
estimates, we obtain}
\begin{equation}\label{eq:integrated-bilinear-density}
 |I_j^{\mathrm{ann}}(H^{(\alpha)},H^{(\beta)})|
 \le C\rho_j^{\alpha+\beta}
 \begin{cases}
  1,&\alpha+\beta<n-2,\\
  1+\log R_j,&\alpha+\beta=n-2,\\
  1+R_j^{\alpha+\beta+2-n},&\alpha+\beta>n-2.
 \end{cases}
\end{equation}
The same bound applies to the mixed coefficient of any two pairs
satisfying \eqref{eq:extended-bilinear-direction-bounds}.
The comparison with the coefficient on Euclidean space will determine
the sign of the diagonal term \(I_j^{\mathrm{ann}}(H^{(m)},H^{(m)})\) for \(2\le m<(n-2)/2\).

Fix an integer \(m\) with \(2\le m<(n-2)/2\), a symmetric homogeneous
polynomial matrix \(H\) of degree \(m\) satisfying
\eqref{eq:metric-jet-conditions}, an index \(j\), and
\(L>2R_j\). In \eqref{eq:whole-space-scaling-family}, take
\((\lambda,b)=(\lambda_j,b_j)\). Since \(\chi_L=1\) on
\(A_{\sigma_j,R_j}\), {we apply
Lemma  \ref{lem:euclidean-scaling-identity} to the decomposition}
\[
 \R^n=\overline{B_{\sigma_j}}\cup A_{\sigma_j,R_j}
 \cup(\R^n\setminus B_{R_j})
\]
{and obtain the following four terms.} For fixed \(j,L,H\), take
\(|\eps|\rho_j^m<s_L\); the derivatives with respect to \(\rho\) are evaluated in a
neighborhood of \(\rho_j\) where \(V_{\eps,\rho,L}>0\).
\begin{subequations}\label{eq:finite-to-whole-space-terms}
\begin{align}
I_{j,L}^{(1)}(\eps)
&:=-\frac12\int_{\R^n}\Biggl[
\sum_{k,l=1}^n
\left.(\rho\partial_\rho g_{\eps,\rho,L}^{kl})\right|_{\rho=\rho_j}
\partial_kV_{\eps,\rho_j,L}\,
\partial_lV_{\eps,\rho_j,L}
\notag\\
&\qquad\qquad
+\left.\rho\partial_\rho
\bigl(c(n)\operatorname{Scal}_{g_{\eps,\rho,L}}\bigr)
\right|_{\rho=\rho_j}
V_{\eps,\rho_j,L}^2
\Biggr]\,\dd y,
\\
I_{j,L}^{(2)}(\eps)
&:=\frac12
\int_{B_{\sigma_j}\cup(\R^n\setminus B_{R_j})}\Biggl[
\sum_{k,l=1}^n
\left.(\rho\partial_\rho g_{\eps,\rho,L}^{kl})
\right|_{\rho=\rho_j}
\partial_kV_{\eps,\rho_j,L}\,
\partial_lV_{\eps,\rho_j,L}
\notag\\
&\qquad\qquad
+\left.\rho\partial_\rho
\bigl(c(n)\operatorname{Scal}_{g_{\eps,\rho,L}}\bigr)
\right|_{\rho=\rho_j}
V_{\eps,\rho_j,L}^2
\Biggr]\,\dd y,
\label{eq:complementary-volume-term}
\\
I_{j,L}^{(3)}(\eps)
&:=\mathcal B_{g_{\eps,\rho_j,L}}
\bigl(R_j,V_{\eps,\rho_j,L}\bigr),
\qquad
I_{j,L}^{(4)}(\eps)
:=\mathcal B_{g_{\eps,\rho_j,L}}
\bigl(\sigma_j,V_{\eps,\rho_j,L}\bigr).
\label{eq:boundary-terms}
\end{align}
\end{subequations}
The notation \(I_{j,L}^{(a)}\) is local to this decomposition and is distinct
from the quadratic form \(I_j^{\mathrm{ann}}\) in
\eqref{eq:finite-annulus-quadratic-form}.
{Here \(I_{j,L}^{(1)}\) is the volume term over the whole space,
\(I_{j,L}^{(2)}\) removes its inner and outer complements, and
\(I_{j,L}^{(3)}\) and \(I_{j,L}^{(4)}\) in \eqref{eq:boundary-terms} are the outer and inner boundary
terms. Thus \(I_{j,L}^{(1)}+I_{j,L}^{(2)}\) is the volume term over
\(A_{\sigma_j,R_j}\) in \eqref{eq:euclidean-scaling-identity}.}
{Applying Lemma  \ref{lem:euclidean-scaling-identity}, we
therefore obtain}
 {
\[
 \int_{A_{\sigma_j,R_j}}
 (D_{\mathrm{di}}V_{\eps,\rho_j,L})
 (\Delta-L_{g_{\eps,\rho_j,L}})V_{\eps,\rho_j,L}\,\dd y
 =I_{j,L}^{(1)}(\eps)+I_{j,L}^{(2)}(\eps)
 +I_{j,L}^{(3)}(\eps)-I_{j,L}^{(4)}(\eps).
\]
}
The annulus \(A_{L,2L}\) is contained in the integration region of
\(I_{j,L}^{(2)}(\eps)\).

\subsection{A homogeneous term and the inner and outer regions}

The next lemma uses an integer \(m\) with \(2\le m<(n-2)/2\), a symmetric
homogeneous polynomial matrix \(H\) of degree \(m\) satisfying
\eqref{eq:metric-jet-conditions}, the choices
\((\lambda,b)=(\lambda_j,b_j)\) and \(\rho=\rho_j\) in
\eqref{eq:whole-space-scaling-family}, and \(L>2R_j\).

Proposition  \ref{prop:whole-space-coefficient} identifies the quadratic
coefficient of \(I_{j,L}^{(1)}(\eps)\). The next lemma bounds those of
\(I_{j,L}^{(2)}(\eps)\), \(I_{j,L}^{(3)}(\eps)\), and
\(I_{j,L}^{(4)}(\eps)\), including the integral over \(A_{L,2L}\).

In the following estimates, the coefficient of \(\eps^2\) means
\(\frac{1}{2}\partial_\eps^2|_{\eps=0}\).
\begin{lemma}\label{lem:finite-annulus-coefficient-bounds}
There is \(C=C(n,m,\mathcal K)>0\) such that, for all sufficiently large
\(j\), the absolute value of the coefficient of \(\eps^2\) in the
integrand defining \(I_{j,L}^{(2)}(\eps)\) in
\eqref{eq:complementary-volume-term} is bounded by
\begin{equation}\label{eq:complement-density}
 C\rho_j^{2m}|H|^2
 \begin{cases}
  r^{-2},&0<r\le1,\\
  r^{2m+2-2n},&r\ge1.
 \end{cases}
\end{equation}
The absolute value of the coefficient of \(\eps^2\) in
\(\mathcal B_{g_{\eps,\rho,L}}(r,V_{\eps,\rho,L})\) is bounded by
\begin{equation}\label{eq:boundary-density}
 C\rho_j^{2m}|H|^2
 \begin{cases}
  r^{n-2},&0<r\le1,\\
  r^{2m+2-n},&r\ge1.
 \end{cases}
\end{equation}
At \(r=R_j\) and \(r=\sigma_j\), respectively,
\eqref{eq:boundary-density} bounds the coefficients of
\(I_{j,L}^{(3)}(\eps)\) and \(I_{j,L}^{(4)}(\eps)\).
The coefficient of \(\eps^2\) in the part of \(I_{j,L}^{(2)}(\eps)\) over
\(A_{L,2L}\), where the cutoff varies, is
\(O(\rho_j^{2m}L^{2m+2-n}|H|^2)\). The constants are
uniform for \((\lambda_j,b_j)\in\mathcal K\) and \(L>2R_j\).
\end{lemma}

\begin{proof}
Fix \(j\), take \((\lambda,b)=(\lambda_j,b_j)\), and keep \(\rho\) variable
while differentiating; all resulting estimates are evaluated at
\(\rho=\rho_j\). {For \(\ell=1,2\), differentiating the
matrix exponential and using the coordinate formula for scalar curvature,
we obtain}
\begin{equation}\label{eq:metric-coefficient-bounds}
 \begin{aligned}
 \frac1{\ell!}\left|
 \left.\partial_\eps^\ell\right|_{\eps=0}
 \rho\partial_\rho g_{\eps,\rho,L}^{kl}\right|
 &\le C\rho^{\ell m}|H|^{\ell}r^{\ell m},\\
 \frac1{\ell!}\left|
 \left.\partial_\eps^\ell\right|_{\eps=0}
 \rho\partial_\rho\bigl(c(n)\operatorname{Scal}_{g_{\eps,\rho,L}}\bigr)\right|
 &\le C\rho^{\ell m}|H|^{\ell}r^{\ell m-2}.
 \end{aligned}
\end{equation}
On \(A_{L,2L}\), the bounds in \eqref{eq:metric-coefficient-bounds} remain
valid with \(r\) replaced by a quantity comparable to \(L\), because
\[
|\nabla_y^a(\chi_LH)|
\le C|H|L^{m-a},\qquad a=0,1,2.
\]

Write \(U=U_j\) and
\(Z_L=Z_{\lambda_j,b_j,L}[H]\). Uniformly in \(L\),
\[
 |U|+r|\nabla_yU|
 \le C\min\{1,r^{2-n}\},
\]
{while, by Lemma  \ref{lem:cutoff-response}, we have}
\[
 |Z_L|+r|\nabla_yZ_L|
 \le C|H|
 \begin{cases}
  1,&r\le1,\\
  r^{m+2-n},&r\ge1.
 \end{cases}
\]
{Since both \(\rho\partial_\rho g_{\eps,\rho,L}^{kl}\) and
\(\rho\partial_\rho \operatorname{Scal}_{g_{\eps,\rho,L}}\) vanish at \(\eps=0\), the
coefficient of \(\eps^2\) in the volume integrand is}
\[
 \begin{aligned}
 &\sum_{k,l=1}^n\bigg[
 \frac14\left.\partial_\eps^2\right|_{\eps=0}
 \bigl(\rho\partial_\rho g_{\eps,\rho,L}^{kl}\bigr)
 \partial_kU\partial_lU
 +\left.\partial_\eps\right|_{\eps=0}
 \bigl(\rho\partial_\rho g_{\eps,\rho,L}^{kl}\bigr)
 \partial_kU\,\partial_l(\rho^mZ_L)
 \bigg]\\
 &\quad+
 \frac14\left.\partial_\eps^2\right|_{\eps=0}
 \bigl[\rho\partial_\rho(c(n)\operatorname{Scal}_{g_{\eps,\rho,L}})\bigr]U^2
 +\left.\partial_\eps\right|_{\eps=0}
 \bigl[\rho\partial_\rho(c(n)\operatorname{Scal}_{g_{\eps,\rho,L}})\bigr]
 U(\rho^mZ_L).
 \end{aligned}
\]
{Substituting \eqref{eq:metric-coefficient-bounds} and the two
displayed bounds for \(U\) and \(Z_L\), and then evaluating at
\(\rho=\rho_j\), we obtain \eqref{eq:complement-density}.}

{Applying the same coefficient bounds to
\eqref{eq:scaling-boundary-term}, and then multiplying by \(r\) and the
area of \(\partial B_r\), we obtain
\eqref{eq:boundary-density}.  Finally, integrating the outer line of
\eqref{eq:complement-density} over \(A_{L,2L}\), we obtain}
\[
 C\rho_j^{2m}L^{2m+2-n}|H|^2.
\]
This proves all three assertions.
\end{proof}

\begin{proposition}\label{prop:finite-annulus-self-pairing}
For each integer \(m\) with \(2\le m<\frac{n-2}{2}\), there is
\(C=C(n,m,\mathcal K)>0\) such that, for all sufficiently large \(j\),
 {
\begin{equation}\label{eq:finite-annulus-self-pairing}
 \bigl|I_j^{\mathrm{ann}}(H^{(m)},H^{(m)})
 -2^{3-n}m\rho_j^{2m}
 \mathcal Q_{\lambda_j,b_j}(H^{(m)},H^{(m)})\bigr|
 \le C\rho_j^{2m}
 \bigl(\sigma_j^{n-2}+R_j^{2m+2-n}\bigr)
 |H^{(m)}|^2.
\end{equation}
}
\end{proposition}

\begin{proof}
Take \(H=H^{(m)}\) in \eqref{eq:finite-to-whole-space-terms}.
For \(H^{(m)}=0\), both sides vanish. By quadratic homogeneity,
assume first that \(|H^{(m)}|=1\). By Lemma
\ref{lem:finite-annulus-coefficient-bounds}, the coefficient of \(\eps^2\)
of the integrand in \(I_{j,L}^{(2)}(\eps)\) is bounded by
\[
 C\rho_j^{2m}
 \begin{cases}
  r^{-2},&0<r\le1,\\
  r^{2m+2-2n},&r\ge1,
 \end{cases}
\]
and the already integrated coefficient on \(\partial B_r\) is bounded by
\[
 C\rho_j^{2m}
 \begin{cases}
  r^{n-2},&0<r\le1,\\
  r^{2m+2-n},&r\ge1.
 \end{cases}
\]
{Since \(m<(n-2)/2\), the contribution of \(A_{L,2L}\) to the
quadratic coefficient of \(I_{j,L}^{(2)}\) satisfies}
\[
 C\rho_j^{2m}L^{2m+2-n}\to0
 \quad\text{as }L\to\infty.
\]
{By Proposition  \ref{prop:whole-space-coefficient},}
\[
 \lim_{L\to\infty}\frac12\partial_\eps^2 I_{j,L}^{(1)}(0)
 =2^{3-n}m\rho_j^{2m}
   \mathcal Q_{\lambda_j,b_j}(H^{(m)},H^{(m)}).
\]
{For the remaining terms, using the radial bounds, we obtain}
\[
 \begin{aligned}
 &\limsup_{L\to\infty}\left|\frac12\partial_\eps^2
  \bigl(I_{j,L}^{(2)}+I_{j,L}^{(3)}-I_{j,L}^{(4)}\bigr)(0)\right|\\
 {}\le{}&C\rho_j^{2m}\left(
 \int_0^{\sigma_j}r^{n-3}\,\dd r
 +\int_{R_j}^{\infty}r^{2m+1-n}\,\dd r
 +\sigma_j^{n-2}+R_j^{2m+2-n}\right)\\
 {}\le{}&C\rho_j^{2m}
 \bigl(\sigma_j^{n-2}+R_j^{2m+2-n}\bigr).
 \end{aligned}
\]
By Lemma  \ref{lem:cutoff-response}, the scalar corrections converge in
\(C^2\) on \(\overline{A_{\sigma_j,R_j}}\).
{It follows from the explicit quadratic coefficient in
\eqref{eq:finite-annulus-quadratic-form} that}
\[
 \lim_{L\to\infty}\frac12\partial_\eps^2
 \bigl(I_{j,L}^{(1)}+I_{j,L}^{(2)}+I_{j,L}^{(3)}-I_{j,L}^{(4)}\bigr)(0)
 =I_j^{\mathrm{ann}}(H^{(m)},H^{(m)}).
\]
Combining these limits and restoring \(|H^{(m)}|^2\) by homogeneity
proves \eqref{eq:finite-annulus-self-pairing}.
\end{proof}

By Propositions  \ref{prop:finite-annulus-self-pairing} and
\ref{prop:euclidean-strict-positivity}, for every fixed \(2\le m<\frac{n-2}{2}\),
\begin{equation}\label{eq:finite-annulus-positive-term}
 I_j^{\mathrm{ann}}(H^{(m)},H^{(m)})
 \ge c_m\rho_j^{2m}|H^{(m)}|^2
\end{equation}
{for all sufficiently large \(j\), where \(c_m>0\) is independent
of \(j\).}

\subsection{\texorpdfstring{Interactions between terms of different degrees, the case \(m=(n-2)/2\), and the Taylor remainder}{Interactions between terms of different degrees, the case m=(n-2)/2, and the Taylor remainder}}

\begin{lemma}\label{lem:critical-self-pairing}
If \(n\) is even, there is \(C>0\) such that, for all sufficiently large
\(j\),
\begin{equation}\label{eq:critical-self-pairing}
 |I_j^{\mathrm{ann}}(H^{(\frac{n-2}{2})},H^{(\frac{n-2}{2})})|
 \le C\rho_j^{n-2}(1+\log R_j).
\end{equation}
\end{lemma}

\begin{proof}
For \(m=(n-2)/2\), the path in
\eqref{eq:finite-annulus-quadratic-path} has only the metric direction
\(\rho_j^{(n-2)/2}H^{((n-2)/2)}\); its scalar direction is zero.
{Using \eqref{eq:integrated-bilinear-density} with
\(\alpha=\beta=m\), we obtain}
 {
\[
 |I_j^{\mathrm{ann}}(H^{(m)},H^{(m)})|
 \le C\rho_j^{n-2}
 \left(\int_{\sigma_j}^1r^{n-3}\,\dd r
 +\int_1^{R_j}r^{-1}\,\dd r\right)
 \le C\rho_j^{n-2}(1+\log R_j).
\]
}
This proves \eqref{eq:critical-self-pairing}.
\end{proof}

The remainder estimate depends on the first nonzero homogeneous coefficient
of the metric expansion. Define
\begin{equation}\label{eq:first-nonzero-degree}
 m_*=\min\left\{
 m\in\left\{2,\ldots,d\right\}:
 H^{(m)}\ne0\right\}.
\end{equation}
If \(H^{(2)}=\cdots=H^{(d)}=0\), set \(m_*=\infty\). In this case,
Taylor's theorem gives \(h(x)=O(|x|^{d+1})\). More generally,
\[
 \left|
 \partial^\beta\left(
 h(x)-\sum_{m=2}^{d}H^{(m)}(x)
 \right)\right|
 \le C|x|^{d+1-|\beta|}
 \qquad |\beta|\le3.
\]
{Scaling this estimate, we obtain, for every
multiindex \(\beta\) with \(|\beta|\le3\), every
\(y\in A_{\sigma_j,R_j}\), and all sufficiently large \(j\),}
\begin{equation}\label{eq:scaled-metric-taylor-remainder}
\left|\partial_y^\beta\left(
 h(\rho_jy)
 -\sum_{m=2}^{d}
 \rho_j^mH^{(m)}(y)\right)\right|+\left|\partial_y^\beta
 \left.\rho\partial_\rho\left(
 h(\rho y)
 -\sum_{m=2}^{d}
 \rho^mH^{(m)}(y)\right)\right|_{\rho=\rho_j}\right|\le
 C\rho_j^{d+1}
 |y|^{d+1-|\beta|}.
\end{equation}
The next lemma combines the integrated density bound with the Taylor
remainder estimate.

\begin{lemma}\label{lem:mixed-and-taylor-remainder}
There is \(C>0\) such that, for all sufficiently large \(j\), if
\(m_1\ne m_2\) and
\(2\le m_1,m_2\le d\), then
\begin{equation}\label{eq:mixed-degree-bound}
 |I_j^{\mathrm{ann}}(H^{(m_1)},H^{(m_2)})|
 \le C\rho_j^{m_1+m_2}.
\end{equation}
Moreover,
\begin{equation}\label{eq:quadratic-taylor-remainder}
 \frac12\mathfrak q_j''(0)
 -\sum_{m_1,m_2=2}^{d}
 I_j^{\mathrm{ann}}(H^{(m_1)},H^{(m_2)})
 =\begin{cases}
   o(\rho_j^{2m_*}),&m_*<\frac{n-2}{2},\\
   o(\rho_j^{n-2}),&
   n\text{ is even and }m_*=\frac{n-2}{2},
   \text{ or }m_*=\infty.
  \end{cases}
\end{equation}
The little \(o\) terms in \eqref{eq:quadratic-taylor-remainder} are
taken as \(j\to\infty\), uniformly for
\((\lambda_j,b_j)\in\mathcal K\).
\end{lemma}

\begin{proof}
For distinct \(m_1,m_2\le d\),
\[
 m_1+m_2\le2d-1<n-2.
\]
{The bound \eqref{eq:mixed-degree-bound} follows directly
from \eqref{eq:integrated-bilinear-density}.}

{Set \(\ell=d+1\).}
The two directions containing the Taylor remainder are
\[
 \begin{aligned}
 T_j(y)
 &=
 h(\rho_jy)-\sum_{m=2}^{\ell-1}\rho_j^mH^{(m)}(y),\\
 \widehat T_j(y)
 &=
 \left.\rho\partial_\rho\left(
 h(\rho y)-\sum_{m=2}^{\ell-1}\rho^mH^{(m)}(y)
 \right)\right|_{\rho=\rho_j}.
 \end{aligned}
\]
They are symmetric, trace free, and satisfy the radial gauge condition.
{By \eqref{eq:scaled-metric-taylor-remainder},}
\[
 \sum_{a=0}^2r^a
\left(
 |\nabla_y^aT_j(y)|
 +|\nabla_y^a\widehat T_j(y)|
\right)
 \le C\rho_j^{\ell}r^{\ell}.
\]
Thus \((T_j,0)\) and \((\widehat T_j,0)\) satisfy
\eqref{eq:extended-bilinear-direction-bounds} with order \(\ell\).

{Applying the extended conclusion of Lemma
\ref{lem:annular-bilinear-density} to two remainder directions, we obtain}
\[
 C\rho_j^{2\ell}
 \left(1+R_j^{2\ell-(n-2)}\right)
 \le
 C\rho_j^{\ell+\frac{n-2}{2}}
 =o(\rho_j^{n-2})
 \quad\text{as }j\to\infty,
\]
where \eqref{eq:outer-radius-rho-bound} was used.  {Pair
\(\rho_j^mH^{(m)}\), together with its correction
\(Z_{\lambda_j,b_j}[H^{(m)}]\) when
\(m<(n-2)/2\), with either remainder direction. We obtain}
\begin{equation}\label{eq:mixed-taylor-remainder-bound}
 C\rho_j^{m+\ell}
 \left(1+\int_1^{R_j}r^{m+\ell+1-n}\,\dd r\right).
\end{equation}

Suppose \(m_*<(n-2)/2\). If \(m+\ell\le n-2\), the integral in
\eqref{eq:mixed-taylor-remainder-bound} is bounded or logarithmic.
Since \(m+\ell>2m_*\), its contribution is \(o(\rho_j^{2m_*})\).
{If \(m+\ell>n-2\), by
\eqref{eq:outer-radius-rho-bound},}
\[
 \rho_j^{m+\ell}R_j^{m+\ell+2-n}
 \le
 C\rho_j^{\frac{m+\ell+n-2}{2}}.
\]
The smallest possible exponent occurs at \(m=m_*\), and
\[
 \frac{m_*+\ell+n-2}{2}-2m_*
 =
 \frac{\ell+n-2-3m_*}{2}>0
\]
because \(m_*\) is an integer strictly smaller than \((n-2)/2\) and
\(\ell=d+1\).  Hence this case is also
\(o(\rho_j^{2m_*})\).

If \(n\) is even and \(m_*=(n-2)/2\), the only mixed remainder term has
\(m=(n-2)/2\), and \eqref{eq:mixed-taylor-remainder-bound} is bounded by
\[
 C\rho_j^{n-1}R_j
 =O(\rho_j^{n-\frac32})
 =o(\rho_j^{n-2}).
\]
If \(m_*=\infty\), every Taylor tensor through degree \(\ell-1\) vanishes, so
only the pairing of two remainders occurs.
{Decomposing \(\frac{1}{2}\mathfrak q_j''(0)\) into pairings
of two homogeneous terms, of a homogeneous term with a remainder, and of
two remainders, we obtain
\eqref{eq:quadratic-taylor-remainder}.}
\end{proof}

Since the coefficients \(H^{(m)}\) of \(h=\log g\) are fixed,
{using \eqref{eq:mixed-degree-bound},
Lemma  \ref{lem:critical-self-pairing}, and
Lemma  \ref{lem:mixed-and-taylor-remainder}, we obtain, when
\(m_*<(n-2)/2\),}
\begin{equation}\label{eq:leading-quadratic-coefficient}
 \frac12\mathfrak q_j''(0)
 =I_j^{\mathrm{ann}}(H^{(m_*)},H^{(m_*)})+o(\rho_j^{2m_*})
 \quad\text{as }j\to\infty.
\end{equation}

\subsection{The linear term and the Taylor remainder}

\begin{lemma}\label{lem:reference-linear-term}
For all sufficiently large \(j\),
\[
 |\mathfrak q_j'(0)|
 \le C\eta_j(\mu_j+\mu_{j-1}).
\]
\end{lemma}

\begin{proof}
Since \(g_{j,0}^{kl}=\delta^{kl}\) and \(\operatorname{Scal}_{g_{j,0}}=0\), differentiation of
\(V_{j,\tau}\) makes no contribution at \(\tau=0\).
{It follows from \eqref{eq:metric-source} that}
\[
 \mathfrak q_j'(0)
 =\int_{A_{\sigma_j,R_j}}
 \left(D_{\mathrm{di}}U_j\right)
 F_{\lambda_j,b_j}[h(\rho_j\,\cdot)]\,\dd y.
\]
Choose a smooth radial cutoff \(\chi_j\) equal to one on \(B_{R_j}\), zero
outside \(B_{2R_j}\), and satisfying
\(|\nabla_y^a\chi_j|\le CR_j^{-a}\) for \(a=1,2\).
For all sufficiently large \(j\), \(\sigma_j<1<R_j\) and
\(2\rho_jR_j<r_0\). Set
\[
 T(y)=\chi_j(y)h(\rho_jy)\quad\text{on }B_{2R_j},
\]
and extend \(T\) by zero outside \(B_{2R_j}\).
Then \(T\) is a compactly supported \(C^2\) symmetric matrix, and
\(\tr T=0\) by \eqref{eq:normal-gauge}.
{By the definitions of the Jacobi fields,}
\[
 D_{\mathrm{di}}U_j
 =-J_{\lambda_j,b_j,0}
  -\sum_{\ell=1}^n b_j^\ell J_{\lambda_j,b_j,\ell}.
\]
{Thus, from \eqref{eq:cutoff-response-compatibility},}
\[
 \int_{\R^n}
 \left(D_{\mathrm{di}}U_j\right)
 F_{\lambda_j,b_j}[T]\,\dd y=0.
\]
{Since \(T=h(\rho_j\,\cdot)\) on \(B_{R_j}\), splitting
the preceding identity at \(\sigma_j\) and \(R_j\), we obtain}
 {
\[
 \mathfrak q_j'(0)
 ={}-\int_{B_{\sigma_j}}
 \left(D_{\mathrm{di}}U_j\right)
 F_{\lambda_j,b_j}[T]\,\dd y
 -\int_{A_{R_j,2R_j}}
 \left(D_{\mathrm{di}}U_j\right)
 F_{\lambda_j,b_j}[T]\,\dd y.
\]
}
{On \(B_{\sigma_j}\), by \eqref{eq:normal-orders},}
\(|\nabla_y^a[h(\rho_j\,\cdot)](y)|
\le C\rho_j^2|y|^{2-a}\) for \(a=0,1,2\).
The local bounds for \(U_j\) and \eqref{eq:metric-source} imply
\[
 |F_{\lambda_j,b_j}[T]|\le C\rho_j^2,
 \qquad
 \left|D_{\mathrm{di}}U_j\right|\le C.
\]
{On \(A_{R_j,2R_j}\), using the cutoff bounds and
\(\eta_j=\rho_j^2R_j^2\), we have}
\[
 |T|+R_j|\nabla_yT|+R_j^2|\nabla_y^2T|\le C\eta_j.
\]
Combining this with
\(|\nabla_y^aU_j|\le CR_j^{2-n-a}\) for \(a=0,1,2\), we obtain
\[
 |F_{\lambda_j,b_j}[T]|\le C\eta_jR_j^{-n},
 \qquad
 \left|D_{\mathrm{di}}U_j\right|
 \le CR_j^{2-n}.
\]
All constants are uniform for \((\lambda_j,b_j)\in\mathcal K\).
Since \(\sigma_j<R_j\), we have
\(\rho_j^2\sigma_j^n\le\eta_j\sigma_j^{n-2}\).
Applying these bounds to the preceding decomposition and using
\eqref{eq:capacity-scales}, we obtain
\[
 |\mathfrak q_j'(0)|
 \le C\rho_j^2\sigma_j^n+C\eta_jR_j^{2-n}
 \le C\eta_j(\mu_j+\mu_{j-1}).
\]
\end{proof}

It remains to estimate the integral remainder in
\eqref{eq:reference-pohozaev-taylor-formula}. The next lemma bounds
\(\mathfrak q_j^{(3)}(\tau)\) uniformly for \(0\le\tau\le1\).

\begin{lemma}\label{lem:reference-third-variation}
For all sufficiently large \(j\), uniformly for \(0\le\tau\le1\),
\begin{equation}\label{eq:reference-third-variation}
 |\mathfrak q_j^{(3)}(\tau)|
 \le C\eta_j(\mu_j+\mu_{j-1})
 +\begin{cases}
  C\eta_j\rho_j^{2m_*},&m_*<\frac{n-2}{2},\\
  C\eta_j\rho_j^{n-2}(1+\log R_j),
  &n\text{ is even and }m_*=\frac{n-2}{2},\\
  C\eta_j\rho_j^{n-2},&m_*=\infty.
\end{cases}
\end{equation}
\end{lemma}

The proof is given in Appendix  \ref{app:third-variation}.

\begin{proposition}\label{prop:reference-pohozaev-lower-bound}
There are \(c,C>0\) such that, for all sufficiently large \(j\),
\begin{equation}\label{eq:reference-pohozaev-lower-bound}
 \widetilde Q_j\ge
 \begin{cases}
  c\rho_j^{2m_*}-o(\rho_j^{2m_*})
  -C\eta_j(\mu_j+\mu_{j-1}),&m_*<\frac{n-2}{2},\\
  -C\rho_j^{n-2}(1+|\log\rho_j|)
  -C\eta_j(\mu_j+\mu_{j-1}),
  &n\text{ is even and }m_*=\frac{n-2}{2},\text{ or }m_*=\infty.
 \end{cases}
\end{equation}
The little \(o\) term in \eqref{eq:reference-pohozaev-lower-bound} is taken
as \(j\to\infty\), uniformly for
\((\lambda_j,b_j)\in\mathcal K\).
{The constants may depend on the fixed background metric. In the
first case, \(c\) also depends on the nonzero coefficient
\(H^{(m_*)}\).}
\end{proposition}

\begin{proof}
We apply the decomposition \eqref{eq:reference-pohozaev-taylor-formula}.
{If \(m_*<\frac{n-2}{2}\), combining
\eqref{eq:finite-annulus-positive-term} with
\eqref{eq:leading-quadratic-coefficient} gives}
\[
 \frac12\mathfrak q_j''(0)
 \ge c\rho_j^{2m_*}-o(\rho_j^{2m_*})
 \quad\text{as }j\to\infty.
\]
{If \(n\) is even and \(m_*=(n-2)/2\), then
\eqref{eq:critical-self-pairing}, \eqref{eq:mixed-degree-bound},
\eqref{eq:quadratic-taylor-remainder}, and
\eqref{eq:outer-radius-rho-bound} imply}
\[
 \left|\frac12\mathfrak q_j''(0)\right|
 \le C\rho_j^{n-2}(1+\log R_j).
\]
{If \(m_*=\infty\), the same bound follows directly from
\eqref{eq:quadratic-taylor-remainder}.}
{By Lemmas  \ref{lem:reference-linear-term} and
\ref{lem:reference-third-variation}, the linear and cubic terms satisfy}
\[
 \begin{aligned}
 |\mathfrak q_j'(0)|
 &\le C\eta_j(\mu_j+\mu_{j-1}),\\
 \left|\frac12\int_0^1(1-\tau)^2
 \mathfrak q_j^{(3)}(\tau)\,\dd\tau\right|
 &\le C\eta_j(\mu_j+\mu_{j-1})\\
 &\quad+
 \begin{cases}
  o(\rho_j^{2m_*}),
  &m_*<\frac{n-2}{2},\\
  C\rho_j^{n-2}(1+\log R_j),
  &n\text{ is even and }m_*=\frac{n-2}{2},\text{ or }m_*=\infty.
  \end{cases}
 \end{aligned}
\]
These bounds hold as \(j\to\infty\), uniformly for
\((\lambda_j,b_j)\in\mathcal K\).
{Finally, \eqref{eq:outer-radius-rho-bound} gives
\(1+\log R_j\le C(1+|\log\rho_j|)\).}
{Substituting in
\eqref{eq:reference-pohozaev-taylor-formula}, we obtain
\eqref{eq:reference-pohozaev-lower-bound}.}
\end{proof}

\section{Sharp comparison on rescaled annuli}
\label{sec:exact-reference-comparison}

With the parameters \((\lambda_j,b_j)\) fixed by
Proposition  \ref{prop:final-modulation}, let
\(\Phi_j=\Phi_{j,\lambda_j,b_j}\) and
\(\widetilde Q_j=B_j(\Phi_j,\Phi_j)\). We compare \(\widetilde Q_j\) with
\(Q_j=B_j(v_j,v_j)\) and set \(e_j=v_j-\Phi_j\). Its boundary values need
not vanish.
We construct solutions \(e_j^-\) and \(e_j^+\) carrying its inner and outer
boundary values, respectively, and then set
\(e_j^0=e_j-e_j^--e_j^+\in H_0^1(A_{\sigma_j,R_j})\). 
The distribution \(\mathcal R_j^\Phi\) and the norm
\(\|\cdot\|_{\dot H^{-1}(A_{\sigma_j,R_j})}\) are defined in
\eqref{eq:reference-equation-error} and
\eqref{eq:homogeneous-dual-norm}, respectively.
 The principal
estimates proved below are
\[
 \begin{aligned}
 \|\nabla_y e_j^0\|_{L^2(A_{\sigma_j,R_j})}
 &\le C\bigl(
 \|\mathcal R_j^\Phi\|_{\dot H^{-1}(A_{\sigma_j,R_j})}
 +\mu_j+\mu_{j-1}\bigr),\\
 |B_j(\Phi_j,e_j^-)|+|B_j(\Phi_j,e_j^+)|
 &\le C\eta_j(\mu_j+\mu_{j-1}).
 \end{aligned}
\]
Together with Lemma  \ref{lem:reference-equation-error-estimate}
and the sharper bound for \(B_j(\Phi_j,e_j^0)\), these estimates give
Proposition  \ref{prop:exact-reference-comparison}.
 All spheres remain
centered at the puncture, and the parameters are the fixed choice from
Proposition  \ref{prop:final-modulation}.

\subsection{\texorpdfstring{The distribution
\(\mathcal R_j^\Phi\) and its dual norm}{The distribution
R_j^Φ and its dual norm}}

By  \eqref{eq:Qj-definition} and
\eqref{eq:reference-pohozaev-quantity}, we have
 
\begin{equation}\label{eq:exact-dilation-polarization}
 Q_j-\widetilde Q_j
 =2 B_j(\Phi_j,e_j)+ B_j(e_j,e_j).
\end{equation}

The two terms in \eqref{eq:exact-dilation-polarization} are estimated
separately below.

All integrals and norms below are over \(A_{\sigma_j,R_j}\) unless
another domain is displayed. Constants are uniform for the parameters
in \(\mathcal K\). Every little \(o\) term is taken as \(j\to\infty\),
with the metric and solution fixed.

For a bounded annulus \(E\subset\R^n\), let \(H_0^1(E)\) be the closure
of \(C_c^\infty(E)\) in \(H^1(E)\). The Poincar\'e inequality makes
\(\|\nabla_y(\cdot)\|_{L^2(E)}\) an equivalent norm on this space for
each fixed \(E\). We equip \(H_0^1(E)\) with the Dirichlet norm and
denote the resulting dual space by \(\dot H^{-1}(E)\); its norm is
\begin{equation}\label{eq:homogeneous-dual-norm}
 \norm{f}_{\dot H^{-1}(E)}
 =\sup_{\substack{\varphi\in H_0^1(E)\\
 \norm{\nabla_y\varphi}_{L^2(E)}=1}}
 |\langle f,\varphi\rangle|.
\end{equation}
Subtracting the equations for \(v_j\) and \(\Phi_j\) gives a linear equation
for their difference. To express its nonlinear term, define
\(\vartheta_j:A_{\sigma_j,R_j}\to(0,\infty)\) by
\[
 \vartheta_j(y)=p\int_0^1
 \bigl[\Phi_j(y)
 +s(v_j(y)-\Phi_j(y))\bigr]^{p-1}\,\dd s.
\]
By construction,
\(v_j^p-\Phi_j^p=\vartheta_j(v_j-\Phi_j)\).
Define the differential expression
 
\[
  L_j=-L_{g_j}-n(n-2)\vartheta_j.
\]

For \(f\in H^1(A_{\sigma_j,R_j})\), this expression defines a distribution
by integration by parts. On \(H_0^1(A_{\sigma_j,R_j})\), it is understood
with the Dirichlet norm and the dual norm in
\eqref{eq:homogeneous-dual-norm}.
By Lemmas  \ref{lem:corrected-profile-estimate} and
\ref{lem:annular-bubble-comparison},
\begin{equation}\label{eq:secant-coefficient-bounds}
 cU_j^{p-1}\le\vartheta_j
 \le CU_j^{p-1}\qquad\text{on }A_{\sigma_j,R_j}.
\end{equation}
Define the scalar distribution
\(\mathcal R_j^\Phi\in\dot H^{-1}(A_{\sigma_j,R_j})\) by
\begin{equation}\label{eq:reference-equation-error}
 \mathcal R_j^\Phi=-L_{g_j}\Phi_j
 -n(n-2)\Phi_j^p.
\end{equation}
The equations for \(v_j\) and \(\Phi_j\) give
 
\[
  L_je_j=-\mathcal R_j^\Phi
 \quad\text{in }A_{\sigma_j,R_j}.
\]

\par
We first record the dual norm estimate used for \(\mathcal R_j^\Phi\)
and for the functional
 \(\varphi\mapsto B_j(\Phi_j,\varphi)\).
For a fixed \(s\ge2\), consider a vector field \(A\) and a scalar function
\(f\) on \(A_{\sigma_j,R_j}\) satisfying
\begin{equation}\label{eq:source-field-bounds}
 |A(y)|+|y||f(y)|
 \le C_*\rho_j^s
 \begin{cases}
  1,&0<|y|\le1,\\
  |y|^{s+1-n},&1\le|y|\le R_j.
 \end{cases}
\end{equation}
Here \(C_*\) is independent of \(j\). The exponent \(s\) records the power
of \(\rho_j\) in this estimate.

\begin{lemma}\label{lem:degree-source-bound}
Under \eqref{eq:source-field-bounds}, there is \(C=C(n,s,C_*)\) such that,
for all sufficiently large \(j\),
\begin{equation}\label{eq:degree-source-bound}
 \norm{\operatorname{div}_y A+f}_{\dot H^{-1}(A_{\sigma_j,R_j})}
 \le C\rho_j^s
 \begin{cases}
  1,&s<(n-2)/2,\\
  (1+\log R_j)^{1/2},&s=(n-2)/2,\\
  R_j^{s-(n-2)/2},&s>(n-2)/2.
 \end{cases}
\end{equation}
The constant is independent of \(\sigma_j\), \(R_j\), and \(j\).
\end{lemma}

\begin{proof}
Extension of \(\varphi\in H_0^1(A_{\sigma_j,R_j})\) by zero and the Hardy
inequality give
\[
 \begin{aligned}
 |\langle\operatorname{div}_y A+f,\varphi\rangle|
 &\le
 \left(\norm{A}_{L^2}
 +\frac2{n-2}\norm{|y|f}_{L^2}\right)
 \norm{\nabla_y\varphi}_{L^2},\\
 \norm{A}_{L^2}^2+\norm{|y|f}_{L^2}^2
 &\le C\rho_j^{2s}
 \left(1+\int_1^{R_j}r^{2s+1-n}\,\dd r\right).
 \end{aligned}
\]
All norms in this display are over \(A_{\sigma_j,R_j}\).
The radial integral is
\[
 \int_1^{R_j}r^{2s+1-n}\,\dd r
 =\begin{cases}
 \dfrac{R_j^{2s+2-n}-1}{2s+2-n},&2s\ne n-2,\\[4pt]
 \log R_j,&2s=n-2.
 \end{cases}
\]
Taking square roots and the supremum over
\(\norm{\nabla_y\varphi}_{L^2}=1\) proves
\eqref{eq:degree-source-bound}.
\end{proof}

The correction equation cancels each term in \(\mathcal R_j^\Phi\) that is linear in \(H^{(m)}\), \(2\le m<(n-2)/2\).  We estimate the uncancelled linear terms and the
nonlinear remainder separately. For these estimates, set
\[
 m_0=\min\{m_*,d+1\},\qquad \min\{\infty,d+1\}=d+1.
\]
Thus \(h(x)=O(|x|^{m_0})\), including when all coefficients through
degree \(d\) vanish.

\begin{lemma}\label{lem:reference-equation-error-estimate}
For all sufficiently large \(j\),
\begin{equation}\label{eq:reference-equation-error-estimate}
 \norm{\mathcal R_j^\Phi}_{\dot H^{-1}(A_{\sigma_j,R_j})}
 \le\begin{cases}
  o(\rho_j^{m_*}),&m_*<\frac{n-2}{2},\\
  C\rho_j^{\frac{n-2}{2}}(1+\log R_j)^{1/2},&n\text{ is even and }m_*=\frac{n-2}{2},\\
  o(\rho_j^{\frac{n-2}{2}}),&m_*=\infty.
 \end{cases}
\end{equation}
The little \(o\) terms in \eqref{eq:reference-equation-error-estimate} are
taken as \(j\to\infty\), uniformly for
\((\lambda_j,b_j)\in\mathcal K\).
\end{lemma}

\begin{proof}
We first decompose \(\mathcal R_j^\Phi\).
For a symmetric \(C^2\) matrix \(T\) with zero trace and a scalar function \(f\),
write, within this proof,
\[
 \dot L[T]f
 =-\sum_{k,l=1}^n\partial_k(T_{kl}\partial_lf)
 -c(n)\left(\sum_{k,l=1}^n\partial_k\partial_lT_{kl}\right)f.
\]
This is \(\left.\partial_tL_{\exp(tT)}f\right|_{t=0}\).
The correction equation gives
 
\[
 -\Delta(\Phi_j-U_j)
 -n(n+2)U_j^{p-1}(\Phi_j-U_j)
 =\dot L\left[
 \sum_{\substack{m\in\mathbb N\\2\le m<(n-2)/2}}
 \rho_j^mH^{(m)}\right]U_j.
\]

The first term below contains the linear metric terms that are not removed
by the corrections in \eqref{eq:corrected-profile}. The remaining three
terms are at least quadratic in the metric expansion and its correction.
Define
\[
 \begin{aligned}
 \mathcal R_j^{[1]}
 &:=-\dot L\left[h(\rho_j\,\cdot)
 -\sum_{\substack{m\in\mathbb N\\2\le m<(n-2)/2}}
 \rho_j^mH^{(m)}\right]U_j,\\
 \mathcal R_j^{[2]}
 &:=\{\Delta-L_{g_j}+\dot L[h(\rho_j\,\cdot)]\}U_j,\\
 \mathcal R_j^{[3]}
 &:=(\Delta-L_{g_j})(\Phi_j-U_j),\\
 \mathcal R_j^{[4]}
 &:=-n(n-2)\bigl[\Phi_j^p-U_j^p-pU_j^{p-1}(\Phi_j-U_j)\bigr].
 \end{aligned}
\]
Substitution in \eqref{eq:reference-equation-error} yields
\[
 \mathcal R_j^\Phi
 =\mathcal R_j^{[1]}+\mathcal R_j^{[2]}
 +\mathcal R_j^{[3]}+\mathcal R_j^{[4]}.
\]

We estimate the first term degree by degree.
For a symmetric matrix \(T\) with zero trace and
\(\sum_{a=0}^2r^a|\nabla_y^aT|
\le C\rho_j^s r^s\), \(s\ge2\), write
\[
 \begin{gathered}
 A^k=-\sum_{l=1}^nT_{kl}\partial_lU_j,\qquad
 f=-c(n)\left(\sum_{k,l=1}^n\partial_k\partial_lT_{kl}\right)U_j,\qquad 
 \dot L[T]U_j=\operatorname{div}_y A+f.
 \end{gathered}
\]
The bubble bounds imply
\[
 \begin{aligned}
 |A|+r|f|&\le C\rho_j^s(r^s+r^{s-1})\le C\rho_j^s,
 &&0<r\le1,\\
 |A|+r|f|&\le C\rho_j^s r^{s+1-n},&&r\ge1.
 \end{aligned}
\]
These are \eqref{eq:source-field-bounds}, so
Lemma  \ref{lem:degree-source-bound} applies.

Set \(\ell=d+1\).
The Taylor remainder in \eqref{eq:scaled-metric-taylor-remainder}
contributes at most
\[
 C\rho_j^\ell R_j^{\ell-\frac{n-2}{2}}
 \le C\rho_j^{\frac{\ell+(n-2)/2}{2}}
 =o(\rho_j^{\frac{n-2}{2}}).
\]
The second inequality follows from
\eqref{eq:outer-radius-rho-bound}.
When \(n\) is even, the omitted degree \((n-2)/2\) contributes at most
\(C\rho_j^{(n-2)/2}(1+\log R_j)^{1/2}\).
The tensor \(H^{((n-2)/2)}\) vanishes when \(m_*=\infty\). If
\(m_*<(n-2)/2\), then
\[
 \rho_j^{\frac{n-2}{2}-m_*}(1+\log R_j)^{1/2}
 \le C\rho_j^{\frac{n-2}{2}-m_*}(1+|\log\rho_j|)^{1/2}
 \to0.
\]
Thus \(\mathcal R_j^{[1]}\) satisfies the three bounds in
\eqref{eq:reference-equation-error-estimate}.

It remains to estimate the three nonlinear terms together.
Put \(w=\Phi_j-U_j\), and use \(m_0=\min\{m_*,d+1\}\) as above.
Equations  \eqref{eq:scaled-metric-taylor-remainder},
\eqref{eq:response-decay}, and \eqref{eq:corrected-profile} give
\[
 \sum_{a=0}^2r^a
|\nabla_y^a[h(\rho_j\,\cdot)]|
 \le C\rho_j^{m_0} r^{m_0},\qquad
 |w|+(1+r)|\nabla_yw|
 \le C\rho_j^{m_0}(1+r)^{m_0+2-n}.
\]
For \(m_*=(n-2)/2\) or \(m_*=\infty\), \(w=0\).
The matrix exponential and the scalar curvature formula give
\[
 \begin{aligned}
 |g_j^{-1}-I+h(\rho_jy)|&\le C\rho_j^{2m_0}r^{2m_0},\\
 \left|\operatorname{Scal}_{g_j}-\sum_{k,l=1}^n
 \partial_k\partial_l[h_{kl}(\rho_j\,\cdot)]\right|
 &\le C\rho_j^{2m_0}r^{2m_0-2},\\
 |g_j^{-1}-I|+r^2|\operatorname{Scal}_{g_j}|&\le C\rho_j^{m_0} r^{m_0}.
 \end{aligned}
\]
Here \(I\) is the identity matrix. In particular,
 
\[
 \mathcal R_j^{[2]}
 ={}-\operatorname{div}_y\bigl((g_j^{-1}-I+h(\rho_jy))\nabla_yU_j\bigr)
 +c(n)\left(\operatorname{Scal}_{g_j}-\sum_{k,l=1}^n
 \partial_k\partial_l[h_{kl}(\rho_j\,\cdot)]\right)U_j.
\]

The curvature remainder is bounded by
\[
 C\rho_j^2\bigl(|h||\nabla_x^2h|+|\nabla_xh|^2\bigr)(\rho_jy),
\]
and \(\partial_y^\alpha[h(\rho_j\,\cdot)]
=\rho_j^{|\alpha|}(\partial_x^\alpha h)(\rho_jy)\).

For \(1\le k\le n\), define
\[
 \begin{aligned}
 A_2^k&=-\sum_{l=1}^n
 (g_j^{kl}-\delta^{kl}+h_{kl}(\rho_jy))\partial_lU_j,\\
 f_2&=c(n)\left(\operatorname{Scal}_{g_j}-\sum_{k,l=1}^n
 \partial_k\partial_l[h_{kl}(\rho_j\,\cdot)]\right)U_j,\\
 A_3^k&=-\sum_{l=1}^n(g_j^{kl}-\delta^{kl})\partial_lw,
 \qquad f_3=c(n)\operatorname{Scal}_{g_j}w.
 \end{aligned}
\]
Then \(\mathcal R_j^{[a]}=\operatorname{div}_yA_a+f_a\), \(a=2,3\), and
\[
 |A_a|+r|f_a|\le C\rho_j^{2m_0}
 \begin{cases}
  1,&r\le1,\\
  r^{2m_0+1-n},&1\le r\le R_j.
 \end{cases}
\]
The radial gauge gives \(A_a\cdot\nu=0\) on every centered sphere.
Taylor's formula on the interval between \(U_j\) and \(\Phi_j\), where
\(\frac{1}{2}U_j\le\Phi_j\le\frac{3}{2}U_j\), gives
\[
 |\mathcal R_j^{[4]}|
 \le CU_j^{p-2}|w|^2
 \le C\rho_j^{2m_0}(1+r)^{2m_0-n-2}.
\]
Thus all three terms satisfy \eqref{eq:source-field-bounds} with \(s=2m_0\).
Lemma  \ref{lem:degree-source-bound} and the bound
\(R_j\le C\rho_j^{-1/2}\) from
\eqref{eq:outer-radius-rho-bound} imply
\[
 \begin{aligned}
 \sum_{a=2}^4\|\mathcal R_j^{[a]}\|_{\dot H^{-1}}
 &\le C\rho_j^{2m_0}(1+|\log\rho_j|)^{1/2}
       +C\rho_j^{m_0+\frac{n-2}{4}}\\
 &=o\!\left(\rho_j^{\min\{m_0,(n-2)/2\}}\right).
 \end{aligned}
\]
Both exponents in the upper bound exceed \(\min\{m_0,(n-2)/2\}\).
Combining this estimate with the bound for
\(\mathcal R_j^{[1]}\) proves the lemma.
\end{proof}

\subsection{A uniform estimate modulo the Jacobi fields}

For the fixed solution \(v_j\), the coefficient \(\vartheta_j\) is
specified by the integral formula above. The following a priori estimate
controls a function in \(H_0^1(A_{\sigma_j,R_j})\) by \(L_jf\) and its
weighted moments against the Jacobi fields.

The weak form of \(L_j\) is also used below to construct the two boundary
corrections. For an open set \(E\subset A_{\sigma_j,R_j}\) and
\(\phi,\psi\in H^1(E)\), define
 
\[
  S_{j,E}(\phi,\psi)=\int_E\left[
 \sum_{k,l=1}^ng_j^{kl}\partial_k\phi\partial_l\psi
 +\bigl(c(n)\operatorname{Scal}_{g_j}-n(n-2)\vartheta_j\bigr)\phi\psi
 \right]\,\dd y.
\]
This is the symmetric weak form of  \( L_j\).

By \eqref{eq:normal-orders} and \eqref{eq:outer-radius-rho-bound},
\begin{equation}\label{eq:scaled-metric-coefficient-bounds}
 \max_{1\le k,l\le n}|(g_j)_{kl}-\delta_{kl}|\le C\eta_j,
 \qquad |\operatorname{Scal}_{g_j}(y)|\le C\rho_j^4|y|^2,
 \qquad y\in A_{\sigma_j,R_j}.
\end{equation}
Fix \(C_1>0\) so that, for all sufficiently large \(j\),
\begin{equation}\label{eq:secant-potential-regions}
 \begin{aligned}
 |c(n)\operatorname{Scal}_{g_j}-n(n-2)\vartheta_j|&\le C_1,
 &&\sigma_j\le|y|\le1,\\
 |c(n)\operatorname{Scal}_{g_j}-n(n-2)\vartheta_j|&\le C_1|y|^{-4},
 &&1\le|y|\le R_j.
 \end{aligned}
\end{equation}
On \(1\le|y|\le R_j\), use \(U_j^{p-1}\le C|y|^{-4}\) and
\(\rho_j^4|y|^6\le\rho_j^4R_j^6\le C\rho_j\).

For every compact \(D\subset\R^n\setminus\{0\}\),
\[
 \norm{v_j-U_j}_{C^2(D)}\to0
 \quad\text{as }j\to\infty.
\]
This follows from \eqref{eq:final-parameter-approximation}.
Let \(j_k\to\infty\) be any subsequence such that
\[
 (\lambda_{j_k},b_{j_k})
 \to(\lambda_\infty,b_\infty)\in\mathcal K
 \quad\text{as }k\to\infty.
\]
Then
\begin{equation}\label{eq:secant-local-limit}
 \vartheta_{j_k}\to pU_{\lambda_\infty,b_\infty}^{p-1}
 \quad\text{locally uniformly on }\R^n\setminus\{0\}
 \quad\text{as }k\to\infty.
\end{equation}
For a measurable set \(E\subset\R^n\), \(\mathbf1_E\) denotes its
characteristic function. Moreover, after extending \(\vartheta_{j_k}\) by zero,
\eqref{eq:secant-coefficient-bounds} gives, for every
\(0<r<1<R\),
\[
 \begin{aligned}
 \limsup_{k\to\infty}\;&\phantom{{}+{}}
 \left\|
 \mathbf1_{A_{\sigma_{j_k},R_{j_k}}\cap
 (B_r\cup(\R^n\setminus B_R))}\vartheta_{j_k}
 \right\|_{L^{n/2}(\R^n)}+
 \left\|
 pU_{\lambda_\infty,b_\infty}^{p-1}
 \right\|_{L^{n/2}(B_r\cup(\R^n\setminus B_R))}
 \le C(r^2+R^{-2}).
 \end{aligned}
\]
Together with \eqref{eq:secant-local-limit}, this proves
\[
 \norm{\mathbf1_{A_{\sigma_{j_k},R_{j_k}}}\vartheta_{j_k}
 -pU_{\lambda_\infty,b_\infty}^{p-1}}
 _{L^{n/2}(\R^n)}\to0
 \quad\text{as }k\to\infty.
\]
Similarly,
\begin{equation}\label{eq:scalar-curvature-Ln2-small}
 \norm{\mathbf1_{A_{\sigma_j,R_j}}\operatorname{Scal}_{g_j}}_{L^{n/2}(\R^n)}
 \le C\rho_j^4R_j^4=C\eta_j^2.
\end{equation}

\begin{proposition}\label{prop:constrained-inverse}
There is \(C>0\) such that, for all sufficiently large \(j\), every
\(f\in H_0^1(A_{\sigma_j,R_j})\) satisfies
 
\[
 \norm{\nabla_yf}_{L^2(A_{\sigma_j,R_j})}
 \le C\left[
 \norm{ L_jf}_{\dot H^{-1}(A_{\sigma_j,R_j})}
 +\sum_{l=0}^n
 \left|\int_{A_{\sigma_j,R_j}}
 U_j^{p-1}fJ_{\lambda_j,b_j,l}\,\dd y\right|
 \right].
\]

\end{proposition}

\begin{proof}
If the estimate fails, there are \(j_k\to\infty\) and
\(f_{j_k}\in H_0^1(A_{\sigma_{j_k},R_{j_k}})\) such that
 
\[
 \begin{gathered}
 \norm{\nabla_yf_{j_k}}_2=1,\qquad
 \norm{ L_{j_k}f_{j_k}}_{\dot H^{-1}}\to0,\\
 \int_{A_{\sigma_{j_k},R_{j_k}}}
 U_{j_k}^{p-1}f_{j_k}J_{\lambda_{j_k},b_{j_k},l}\,\dd y
 \to0,\qquad 0\le l\le n.
 \end{gathered}
\]

Extend \(f_{j_k}\) by zero to \(\R^n\). After taking a subsequence,
\begin{equation}\label{eq:inverse-weak-compactness}
 \begin{gathered}
 (\lambda_{j_k},b_{j_k})\to
 (\lambda_\infty,b_\infty)\in\mathcal K,\\
 f_{j_k}\rightharpoonup f\quad\text{in }\dot H^1(\R^n),\qquad
 f_{j_k}\to f\quad\text{in }L^2_{\rm loc}(\R^n).
 \end{gathered}
\end{equation}
The coefficient limits imply
\begin{equation}\label{eq:inverse-limit-equation}
 -\Delta f-n(n+2)U_{\lambda_\infty,b_\infty}^{p-1}f=0
 \quad\text{in }\R^n\setminus\{0\}.
\end{equation}
For \(\varphi\in C_c^\infty(\R^n)\), choose a smooth cutoff \(\chi_r\)
which is zero on \(B_r\), one outside \(B_{2r}\), and satisfies
\(|\nabla\chi_r|\le C/r\). Then
\[
 \|\nabla((1-\chi_r)\varphi)\|_2^2
 \le C_\varphi(r^n+r^{n-2})\to0.
\]
Since \(f\in\dot H^1\) and
\(U_{\lambda_\infty,b_\infty}^{p-1}\in L^{n/2}(\R^n)\),
\[
 \begin{aligned}
 &\left\langle
 (-\Delta-n(n+2)U_{\lambda_\infty,b_\infty}^{p-1})f,\varphi
 \right\rangle\\
 &\quad=\lim_{r\downarrow0}\left\langle
 (-\Delta-n(n+2)U_{\lambda_\infty,b_\infty}^{p-1})f,\chi_r\varphi
 \right\rangle=0.
 \end{aligned}
\]
Thus \eqref{eq:inverse-limit-equation} holds on \(\R^n\).
Lemma  \ref{lem:euclidean-scalar-fredholm} with
\eqref{eq:scalar-conjugacy} gives
\[
 f=\sum_{l=0}^na_lJ_{\lambda_\infty,b_\infty,l}.
\]

We next pass the moment conditions to the limit.
The Sobolev inequality and \(|J_{\lambda,b,l}|\le CU_{\lambda,b}\)
give, uniformly for \((\lambda,b)\in\mathcal K\),
\[
 \left|\int_{B_r\cup(\R^n\setminus B_R)}
 U_{\lambda,b}^{p-1}wJ_{\lambda,b,l}\,\dd y\right|
 \le C\|\nabla w\|_2
 \left(r^{(n+2)/2}+R^{-(n+2)/2}\right)
\]
for \(w\in\dot H^1(\R^n)\) and \(0<r<1<R\).
Use this estimate with the strong local convergence in
\eqref{eq:inverse-weak-compactness} to obtain
\[
 \int_{\R^n}U_{\lambda_\infty,b_\infty}^{p-1}
 fJ_{\lambda_\infty,b_\infty,l}\,\dd y=0,
 \qquad 0\le l\le n.
\]
Multiplying these identities by the coefficients of \(f\) gives
 
\[
 \int_{\R^n}U_{\lambda_\infty,b_\infty}^{p-1}f^2\,\dd y
 =\sum_{l=0}^na_l
\int_{\R^n}U_{\lambda_\infty,b_\infty}^{p-1}
fJ_{\lambda_\infty,b_\infty,l}\,\dd y
 =0.
\]
Hence \(f=0\).

It remains to recover the Dirichlet energy and exclude loss at the two
ends of the annulus.
For fixed \(0<r<1<R\), local convergence gives
\[
 \int_{A_{r,R}}\vartheta_{j_k}f_{j_k}^2\,\dd y\to0.
\]
The pointwise bound \eqref{eq:secant-coefficient-bounds} and the Sobolev
inequality control the complementary regions:
\[
 0\le\limsup_{k\to\infty}
 \int_{A_{\sigma_{j_k},R_{j_k}}}\vartheta_{j_k}f_{j_k}^2\,\dd y
 \le C(r^2+R^{-2}).
\]
Letting \(r\downarrow0\) and \(R\to\infty\) makes the potential integral
tend to zero. The scalar curvature pairing tends to zero by
\eqref{eq:scalar-curvature-Ln2-small}. Testing
 \( L_{j_k}f_{j_k}\) against \(f_{j_k}\), and using
\eqref{eq:scaled-metric-coefficient-bounds}, now gives
 
\[
 o(1)= S_{j_k,A_{\sigma_{j_k},R_{j_k}}}(f_{j_k},f_{j_k})
 =1+o(1).
\]

This contradiction proves the estimate.
\end{proof}

\par
\subsection{Boundary corrections and the remainder with zero boundary values}

Lemma  \ref{lem:annular-boundary-traces} bounds the traces prescribed in the
two Dirichlet problems.

\par

\par

\par
\begin{lemma}\label{lem:annular-boundary-traces}
For every \(\alpha\in(0,1)\), there is \(C_\alpha>0\) such that, for all
sufficiently large \(j\),
\begin{equation}\label{eq:annular-boundary-traces}
 \begin{aligned}
 \norm{(v_j-\Phi_j)(\sigma_j\,\cdot)}_{C^{1,\alpha}(\Sph^{n-1})}
 &\le C_\alpha,\\
 \norm{R_j^{n-2}(v_j-\Phi_j)(R_j\,\cdot)}_{C^{1,\alpha}(\Sph^{n-1})}
 &\le C_\alpha.
 \end{aligned}
\end{equation}
\end{lemma}

\par
\begin{proof}
Fix \(\alpha'\in(\alpha,1)\). At the inner boundary,
\[
 v_j(\sigma_j\xi)=\sigma_j^{-\frac{n-2}{2}}
 \bigl[\bar\rho_j^{\frac{n-2}{2}}u(\bar\rho_je_1)\bigr]
 \frac{u(\bar\rho_j\xi)}{u(\bar\rho_je_1)},
 \qquad \xi\in A_{1/2,2}.
\]
By \eqref{eq:minimum-harmonic-limit}, the quotient is uniformly bounded in
\(C^{2,\alpha'}(A_{1/2,2})\). Equations
\eqref{eq:mu-minimum-depth} and \eqref{eq:capacity-scales} give
\(C^{-1}\sigma_j^{\frac{n-2}{2}}\le\bar\rho_j^{\frac{n-2}{2}}u(\bar\rho_je_1)\le C\sigma_j^{\frac{n-2}{2}}\).
Thus
\[
 \norm{v_j(\sigma_j\,\cdot)}_{C^{2,\alpha'}(A_{1/2,2})}\le C.
\]
Applying \eqref{eq:minimum-harmonic-limit},
\eqref{eq:mu-minimum-depth}, and \eqref{eq:capacity-scales} with \(j-1\)
in place of \(j\) gives
\[
 \norm{R_j^{n-2}v_j(R_j\,\cdot)}_{C^{2,\alpha'}(A_{1/2,2})}\le C.
\]
The formula \eqref{eq:standard-bubble} and compactness of \(\mathcal K\)
give
\[
 \norm{U_j(\sigma_j\,\cdot)}_{C^{2,\alpha'}(A_{1/2,2})}
 +\norm{R_j^{n-2}U_j(R_j\,\cdot)}_{C^{2,\alpha'}(A_{1/2,2})}
 \le C.
\]
By Lemma  \ref{lem:corrected-profile-estimate}, for
\(|\beta|\le2\) and \(\xi\in\Sph^{n-1}\),
 
\[
 \begin{aligned}
 \sigma_j^{|\beta|}
 |\partial_y^\beta(\Phi_j-U_j)(\sigma_j\xi)|
 &\le C\eta_j,\\
 R_j^{n-2+|\beta|}
 |\partial_y^\beta(\Phi_j-U_j)(R_j\xi)|
 &\le C\eta_j.
 \end{aligned}
\]

 Restricting the resulting \(C^{1,\alpha}\) bounds to \(\Sph^{n-1}\)
proves \eqref{eq:annular-boundary-traces}.
\end{proof}

\par

Choose fixed auxiliary radii \(r_{\rm c}\in(0,1/8)\) and
\(R_{\rm c}>8\). They determine inner and outer annuli on which the
term of order zero is small relative to the Dirichlet energy. The first
and third inequalities below give coercivity on these annuli; the second
and fourth are used in the radial barriers. Let \(C_{\rm P}\) be the
dimensional constant in the Poincar\'e inequality on a ball, and choose
the radii so that
\begin{equation}\label{eq:boundary-radius-choice}
 \begin{gathered}
 C_{\rm P}C_1r_{\rm c}^2\le\frac14,
 \qquad
 \frac{2(C_1+1)}{n-4}r_{\rm c}^2\le\frac{n-4}{2},\\
 \frac{4C_1}{(n-2)^2R_{\rm c}^2}\le\frac14,
 \qquad
 C_1\left(1+\frac{C_1}{n-4}\right)R_{\rm c}^{-2}
 \le\frac{2(n-4)(1+C_1/(n-4))-C_1}{2}.
 \end{gathered}
\end{equation}

\par
\par
After increasing the initial index, assume
\begin{equation}\label{eq:boundary-region-separation}
 2\sigma_j<r_{\rm c}<\frac18,
 \qquad 8<R_{\rm c}<\frac12R_j.
\end{equation}

\par
We will write \(e_j=e_j^-+e_j^0+e_j^+\), with \(e_j^-\) and \(e_j^+\)
carrying the inner and outer boundary values and \(e_j^0\) having zero
trace. For every sufficiently large \(j\), the boundary functions are
specified by the problems
 
\begin{equation}\label{eq:boundary-extension-problems}
 \begin{cases}
  L_je_j^-=0,\\
 e_j^-=v_j-\Phi_j\quad\text{on }\partial B_{\sigma_j},\\
 e_j^-=0\quad\text{on }\partial B_{r_{\rm c}},
 \end{cases}
 \qquad
 \begin{cases}
  L_je_j^+=0,\\
 e_j^+=0\quad\text{on }\partial B_{R_{\rm c}},\\
 e_j^+=v_j-\Phi_j\quad\text{on }\partial B_{R_j},
 \end{cases}
\end{equation}

with unknowns \(e_j^-\) on \(A_{\sigma_j,r_{\rm c}}\) and \(e_j^+\) on
\(A_{R_{\rm c},R_j}\).

\begin{lemma}\label{lem:boundary-extensions}
The two problems in \eqref{eq:boundary-extension-problems} have unique weak
solutions.
For each fixed \(j\),
\[
 e_j^-\in C^2(\overline{A_{\sigma_j,r_{\rm c}}}),
 \qquad
 e_j^+\in C^2(\overline{A_{R_{\rm c},R_j}}).
\]
After extension by zero to \(A_{\sigma_j,R_j}\),
\begin{equation}\label{eq:boundary-extension-energy}
 \begin{aligned}
 \int_{A_{\sigma_j,R_j}}\bigl(|\nabla_ye_j^-|^2
 +U_j^{p-1}(e_j^-)^2\bigr)\,\dd y
 &\le C\sigma_j^{n-2},\\
 \int_{A_{\sigma_j,R_j}}\bigl(|\nabla_ye_j^+|^2
 +U_j^{p-1}(e_j^+)^2\bigr)\,\dd y
 &\le CR_j^{2-n}.
 \end{aligned}
\end{equation}
\end{lemma}

\begin{proof}
We first prove, for a fixed \(c_*>0\),
 
\begin{equation}\label{eq:boundary-region-coercivity}
 \begin{aligned}
  S_{j,A_{\sigma_j,r_{\rm c}}}(\phi,\phi)
 &\ge c_*\int_{A_{\sigma_j,r_{\rm c}}}|\nabla_y\phi|^2\,\dd y,
 &&\phi\in H_0^1(A_{\sigma_j,r_{\rm c}}),\\
  S_{j,A_{R_{\rm c},R_j}}(\phi,\phi)
 &\ge c_*\int_{A_{R_{\rm c},R_j}}|\nabla_y\phi|^2\,\dd y,
 &&\phi\in H_0^1(A_{R_{\rm c},R_j}).
 \end{aligned}
\end{equation}

\emph{Inner region.}
For the inner region, extend \(\phi\) by zero across
\(\partial B_{\sigma_j}\) and apply the Poincar\'e inequality on
\(B_{r_{\rm c}}\). The first inequality in
\eqref{eq:boundary-radius-choice} gives
\[
 \left|\int_{A_{\sigma_j,r_{\rm c}}}
 \bigl(c(n)\operatorname{Scal}_{g_j}-n(n-2)\vartheta_j\bigr)\phi^2\,\dd y\right|
 \le\frac14
 \int_{A_{\sigma_j,r_{\rm c}}}|\nabla_y\phi|^2\,\dd y.
\]

\emph{Outer region.}
Extend \(\phi\) by zero to \(\R^n\), apply the Hardy inequality, and use
\(|y|^{-4}\le R_{\rm c}^{-2}|y|^{-2}\). The third inequality in
\eqref{eq:boundary-radius-choice} gives
\[
 \left|\int_{A_{R_{\rm c},R_j}}
 \bigl(c(n)\operatorname{Scal}_{g_j}-n(n-2)\vartheta_j\bigr)\phi^2\,\dd y\right|
 \le\frac14
 \int_{A_{R_{\rm c},R_j}}|\nabla_y\phi|^2\,\dd y.
\]

\emph{Principal part.}
For all sufficiently large \(j\),
\eqref{eq:scaled-metric-coefficient-bounds} gives
\[
 \begin{aligned}
 \left|\int_{A_{\sigma_j,r_{\rm c}}}\sum_{k,l=1}^n
 (g_j^{kl}-\delta^{kl})\partial_k\phi\partial_l\phi\,\dd y\right|
 &\le\frac14\int_{A_{\sigma_j,r_{\rm c}}}|\nabla_y\phi|^2\,\dd y,\\
 \left|\int_{A_{R_{\rm c},R_j}}\sum_{k,l=1}^n
 (g_j^{kl}-\delta^{kl})\partial_k\phi\partial_l\phi\,\dd y\right|
 &\le\frac14\int_{A_{R_{\rm c},R_j}}|\nabla_y\phi|^2\,\dd y.
 \end{aligned}
\]
The three estimates imply \eqref{eq:boundary-region-coercivity}.
The same estimates give the maximum principle on each annulus by testing
the negative part of a weak supersolution with nonnegative boundary trace.

Choose \(\chi_{\rm c}\in C^\infty([1,\infty))\) such that
\[
 0\le\chi_{\rm c}\le1,\qquad \chi_{\rm c}(1)=1,\qquad
 \chi_{\rm c}(s)=0,\qquad s\ge2.
\]
On the two collars in \eqref{eq:boundary-region-separation}, define
\[
 \widetilde e_j^-(r\theta)=
 \chi_{\rm c}\left(\frac r{\sigma_j}\right)
 (v_j-\Phi_j)(\sigma_j\theta),
 \qquad \sigma_j\le r\le2\sigma_j
\]
and
\[
 \widetilde e_j^+(r\theta)=
 \chi_{\rm c}\left(\frac{R_j}{r}\right)
 (v_j-\Phi_j)(R_j\theta),
 \qquad R_j/2\le r\le R_j.
\]
Extend \(\widetilde e_j^-\) by zero to \(A_{2\sigma_j,r_{\rm c}}\) and
\(\widetilde e_j^+\)
by zero to \(A_{R_{\rm c},R_j/2}\).  The trace estimates
\eqref{eq:annular-boundary-traces} imply
\[
 |\widetilde e_j^-|\le C,\qquad
 |\nabla_y\widetilde e_j^-|\le C\sigma_j^{-1}
 \quad\text{on }A_{\sigma_j,2\sigma_j},
\]
and
\[
 |\widetilde e_j^+|\le CR_j^{2-n},\qquad
 |\nabla_y\widetilde e_j^+|\le CR_j^{1-n}
 \quad\text{on }A_{R_j/2,R_j}.
\]
Since \(U_j^{p-1}\le C\) on the inner collar and
\(U_j^{p-1}\le Cr^{-4}\) on the outer collar,
\begin{equation}\label{eq:explicit-boundary-collar-energy}
 \begin{aligned}
 \int_{A_{\sigma_j,r_{\rm c}}}
 \bigl(|\nabla_y\widetilde e_j^-|^2
      +U_j^{p-1}(\widetilde e_j^-)^2\bigr)\,\dd y
 &\le C\sigma_j^{n-2},\\
 \int_{A_{R_{\rm c},R_j}}
 \bigl(|\nabla_y\widetilde e_j^+|^2
      +U_j^{p-1}(\widetilde e_j^+)^2\bigr)\,\dd y
 &\le CR_j^{2-n}.
 \end{aligned}
\end{equation}

The coercivity estimate  \eqref{eq:boundary-region-coercivity} and the Lax--Milgram
theorem give unique functions with zero trace,
\(w_j^-\in H_0^1(A_{\sigma_j,r_{\rm c}})\) and
\(w_j^+\in H_0^1(A_{R_{\rm c},R_j})\) such that
 
\[
  S_{j,A_{\sigma_j,r_{\rm c}}}(w_j^-,\varphi)
 =- S_{j,A_{\sigma_j,r_{\rm c}}}(\widetilde e_j^-,\varphi),
 \qquad \varphi\in H_0^1(A_{\sigma_j,r_{\rm c}}),
\]
and
\[
  S_{j,A_{R_{\rm c},R_j}}(w_j^+,\varphi)
 =- S_{j,A_{R_{\rm c},R_j}}(\widetilde e_j^+,\varphi),
 \qquad \varphi\in H_0^1(A_{R_{\rm c},R_j}).
\]

Test with \(w_j^-\) and \(w_j^+\), respectively.
Coercivity and the Poincar\'e and Hardy estimates give
\[
 \begin{aligned}
 \|\nabla_yw_j^-\|_2^2
 &\le C\left(
 \int_{A_{\sigma_j,r_{\rm c}}}
 \bigl(|\nabla_y\widetilde e_j^-|^2
      +U_j^{p-1}(\widetilde e_j^-)^2\bigr)\,\dd y
 \right)^{1/2}\|\nabla_yw_j^-\|_2,\\
 \|\nabla_yw_j^+\|_2^2
 &\le C\left(
 \int_{A_{R_{\rm c},R_j}}
 \bigl(|\nabla_y\widetilde e_j^+|^2
      +U_j^{p-1}(\widetilde e_j^+)^2\bigr)\,\dd y
 \right)^{1/2}\|\nabla_yw_j^+\|_2.
 \end{aligned}
\]
The \(L^2\) norms are over the corresponding annuli.
Equation  \eqref{eq:explicit-boundary-collar-energy} therefore yields
\[
 \int_{A_{\sigma_j,r_{\rm c}}}|\nabla_yw_j^-|^2\,\dd y
 \le C\sigma_j^{n-2},\qquad
 \int_{A_{R_{\rm c},R_j}}|\nabla_yw_j^+|^2\,\dd y
 \le CR_j^{2-n}.
\]
Set \(e_j^\pm=\widetilde e_j^\pm+w_j^\pm\).
These functions solve \eqref{eq:boundary-extension-problems} weakly.
For fixed \(j\), elliptic boundary regularity gives
\[
 e_j^-\in C^2(\overline{A_{\sigma_j,r_{\rm c}}}),\qquad
 e_j^+\in C^2(\overline{A_{R_{\rm c},R_j}}).
\]
The gradient estimates, with the Poincar\'e and Hardy inequalities,
bound the weighted \(L^2\) terms for \(w_j^\pm\).
Combining these bounds with \eqref{eq:explicit-boundary-collar-energy}
proves \eqref{eq:boundary-extension-energy}. Coercivity gives uniqueness.
\end{proof}

\par
The energy bounds in \eqref{eq:boundary-extension-energy} alone would
produce square roots of the endpoint quantities in a direct estimate of
the linear Pohozaev term. The following pointwise estimates give the
sharper order \(\mu_j+\mu_{j-1}\) for the conormals on the fixed auxiliary
spheres and for the Jacobi moments of the boundary corrections.

On the auxiliary spheres, use the normals
\begin{equation}\label{eq:boundary-auxiliary-normals}
 \nu=\frac{y}{r_{\rm c}}\quad\text{on }\partial B_{r_{\rm c}},
 \qquad
 \nu=-\frac{y}{R_{\rm c}}\quad\text{on }\partial B_{R_{\rm c}}.
\end{equation}
For \(r\in\{r_{\rm c},R_{\rm c}\}\), let
\(H^{1/2}(\partial B_r)\) be the Sobolev trace space and let
\(H^{-1/2}(\partial B_r)\) be its dual. Their duality pairing extends the
Euclidean surface integral.

\begin{lemma}\label{lem:boundary-extension-decay}
The functions in Lemma  \ref{lem:boundary-extensions} satisfy
\begin{equation}\label{eq:boundary-extension-pointwise}
 \begin{aligned}
 |e_j^-(y)|+|y||\nabla_ye_j^-(y)|
 &\le C\sigma_j^{n-2}|y|^{2-n},
 &&\sigma_j\le|y|\le r_{\rm c},\\
 |e_j^+(y)|+|y||\nabla_ye_j^+(y)|
 &\le CR_j^{2-n},
 &&R_{\rm c}\le|y|\le R_j.
 \end{aligned}
\end{equation}
For the normals in \eqref{eq:boundary-auxiliary-normals},
\begin{equation}\label{eq:boundary-extension-conormal}
 \begin{aligned}
 \norm{\sum_{k,l=1}^ng_j^{kl}\nu_k\partial_le_j^-}_{H^{-1/2}(\partial B_{r_{\rm c}})}
 &\le C\sigma_j^{n-2},\\
 \norm{\sum_{k,l=1}^ng_j^{kl}\nu_k\partial_le_j^+}_{H^{-1/2}(\partial B_{R_{\rm c}})}
 &\le CR_j^{2-n}.
 \end{aligned}
\end{equation}
\end{lemma}

\begin{proof}
If \(f=f(|y|)\), the radial gauge in \eqref{eq:scaled-gauge-equation} and
\(\det g_j=1\) imply
\begin{equation}\label{eq:radial-laplacian-identity}
 \Delta_{g_j}f=\Delta f.
\end{equation}
A direct calculation gives
\[
 -\Delta_{g_j}\left(
 r^{2-n}\exp\left(\frac{C_1+1}{n-4}r^2\right)\right)
 =\frac{2(C_1+1)}{n-4}
 \left(n-4-\frac{2(C_1+1)}{n-4}r^2\right)
 r^{2-n}\exp\left(\frac{C_1+1}{n-4}r^2\right).
\]
The inequality
\(2(C_1+1)r_{\rm c}^2/(n-4)\le(n-4)/2\) in
\eqref{eq:boundary-radius-choice}, together with
\eqref{eq:secant-potential-regions}, gives
 
\[
  L_j\left[
 \sigma_j^{n-2}r^{2-n}
 \exp\left(\frac{C_1+1}{n-4}r^2\right)\right]\ge0
 \quad\text{on }A_{\sigma_j,r_{\rm c}}.
\]
The boundary values and \eqref{eq:annular-boundary-traces} give
\[
 |e_j^-(y)|\le
 C\sigma_j^{n-2}r^{2-n}
 \exp\left(\frac{C_1+1}{n-4}r^2\right),
 \qquad r\in\{\sigma_j,r_{\rm c}\}.
\]
Apply the maximum principle proved with
\eqref{eq:boundary-region-coercivity} to both signs of \(e_j^-\).
Since the exponential factor is uniformly
bounded for \(r\le r_{\rm c}\), we obtain
\begin{equation}\label{eq:inner-extension-C0}
 |e_j^-(y)|\le C\sigma_j^{n-2}|y|^{2-n}.
\end{equation}

Equation \eqref{eq:radial-laplacian-identity} gives
\(-\Delta_{g_j}r^{-2}=2(n-4)r^{-4}\).  The last inequality in
\eqref{eq:boundary-radius-choice} gives
 
\[
  L_j\left\{R_j^{2-n}
 \left[1+\left(1+\frac{C_1}{n-4}\right)r^{-2}\right]\right\}
 \ge0\quad\text{on }A_{R_{\rm c},R_j}.
\]
At the two endpoints,
\[
 |e_j^+(y)|\le CR_j^{2-n}
 \left[1+\left(1+\frac{C_1}{n-4}\right)r^{-2}\right],
 \qquad r\in\{R_{\rm c},R_j\}.
\]
The maximum principle and \(r\ge R_{\rm c}\) yield
\begin{equation}\label{eq:outer-extension-C0}
 |e_j^+(y)|\le CR_j^{2-n}.
\end{equation}

If \(A_{r/2,4r}\) lies in the region where \(e_j^\pm\) solves its
equation, scaling and \eqref{eq:secant-potential-regions} give the
interior estimate
\[
 r\|\nabla_ye_j^\pm\|_{L^\infty(A_{r,2r})}
 \le C\|e_j^\pm\|_{L^\infty(A_{r/2,4r})}.
\]
Substituting \eqref{eq:inner-extension-C0} and
\eqref{eq:outer-extension-C0} yields
\[
 |y||\nabla_ye_j^-(y)|
 \le C\sigma_j^{n-2}|y|^{2-n},
 \qquad
 |y||\nabla_ye_j^+(y)|
 \le CR_j^{2-n}.
\]
For sufficiently large \(j\), the scaled boundary gradient estimates and
\eqref{eq:annular-boundary-traces} with \(\alpha=1/2\) give
\[
 \begin{aligned}
 \|\nabla_ye_j^-\|_{L^\infty(A_{\sigma_j,2\sigma_j})}
 &\le C\sigma_j^{-1}\bigl(
 \|e_j^-\|_{L^\infty(A_{\sigma_j,4\sigma_j})}
 +\|e_j^-(\sigma_j\,\cdot)\|_{C^{1,1/2}(\Sph^{n-1})}\bigr)\\
 &\le C\sigma_j^{-1},\\
 \|\nabla_ye_j^+\|_{L^\infty(A_{R_j/2,R_j})}
 &\le CR_j^{-1}\bigl(
 \|e_j^+\|_{L^\infty(A_{R_j/4,R_j})}
 +\|e_j^+(R_j\,\cdot)\|_{C^{1,1/2}(\Sph^{n-1})}\bigr)\\
 &\le CR_j^{1-n}.
 \end{aligned}
\]
The zero boundary traces on the fixed auxiliary spheres allow the same
boundary estimate there. Together these estimates prove
\eqref{eq:boundary-extension-pointwise}.

On the two fixed auxiliary spheres, the pointwise gradient bounds and
uniform ellipticity give
\[
 \left|\sum_{k,l=1}^ng_j^{kl}\nu_k\partial_le_j^-\right|
 \le C\sigma_j^{n-2},\qquad
 \left|\sum_{k,l=1}^ng_j^{kl}\nu_k\partial_le_j^+\right|
 \le CR_j^{2-n}.
\]
Boundary regularity identifies the two displayed conormal expressions with
the weak conormals. Their \(L^\infty\) bounds on the fixed spheres imply
\eqref{eq:boundary-extension-conormal}.
\end{proof}

\par
\begin{proposition}\label{prop:zero-boundary-remainder}
Extend \(e_j^-\) by zero across \(\partial B_{r_{\rm c}}\), extend
\(e_j^+\) by zero across \(\partial B_{R_{\rm c}}\), and set
\begin{equation}\label{eq:zero-boundary-remainder}
 e_j^0=e_j-e_j^--e_j^+,\qquad
 e_j=e_j^-+e_j^0+e_j^+.
\end{equation}
Then \(e_j^0\in H_0^1(A_{\sigma_j,R_j})\) and
\begin{equation}\label{eq:zero-boundary-final-size}
 \|\nabla_ye_j^0\|_{L^2(A_{\sigma_j,R_j})}
 \le C\left(
 \|\mathcal R_j^\Phi\|_{\dot H^{-1}(A_{\sigma_j,R_j})}
 +\mu_j+\mu_{j-1}\right).
\end{equation}
Moreover,
\begin{equation}\label{eq:total-error-energy}
 \int_{A_{\sigma_j,R_j}}
 \left(|\nabla_ye_j|^2+|y|^{-2}e_j^2\right)\,\dd y
 \le C\left[
 \norm{\nabla_ye_j^0}_{L^2(A_{\sigma_j,R_j})}^2
 +\mu_j+\mu_{j-1}\right].
\end{equation}
\end{proposition}

\begin{proof}
The zero traces of \(e_j^-\) and \(e_j^+\) on their auxiliary spheres
make their extensions \(H^1(A_{\sigma_j,R_j})\) functions. They need
not belong to \(H_0^1(A_{\sigma_j,R_j})\), since their traces on the
physical boundaries are generally nonzero. The boundary conditions in
\eqref{eq:boundary-extension-problems} give
\(e_j^0\in H_0^1(A_{\sigma_j,R_j})\).

The modulation condition applies to \(e_j\), not to \(e_j^0\). Thus
\[
 \int_{A_{\sigma_j,R_j}}U_j^{p-1}e_j^0J_{\lambda_j,b_j,l}\,\dd y
 =-\int_{A_{\sigma_j,R_j}}
 U_j^{p-1}(e_j^-+e_j^+)J_{\lambda_j,b_j,l}\,\dd y,
 \qquad 0\le l\le n.
\]
By \eqref{eq:boundary-extension-pointwise} and
\(|J_{\lambda_j,b_j,l}|\le CU_j\), the inner and outer integrals on the
right are bounded by
\[
 C\sigma_j^{n-2}\int_{\sigma_j}^{r_{\rm c}}r\,\dd r,
 \qquad
 CR_j^{2-n}\int_{R_{\rm c}}^{R_j}r^{-3}\,\dd r,
\]
respectively. Hence
\begin{equation}\label{eq:zero-boundary-moment-bound}
 \sum_{l=0}^n
 \left|\int_{A_{\sigma_j,R_j}}
 U_j^{p-1}e_j^0J_{\lambda_j,b_j,l}\,\dd y\right|
 \le C(\mu_j+\mu_{j-1}).
\end{equation}

For \(r\in\{r_{\rm c},R_{\rm c}\}\), choose a fixed annular
neighborhood of \(\partial B_r\) contained in
\(A_{\sigma_j,R_j}\) for all sufficiently large \(j\). The local trace
theorem there, followed by the Sobolev inequality for the zero extension
of \(\varphi\), gives
\begin{equation}\label{eq:uniform-fixed-sphere-trace}
 \|\operatorname{Tr}_{\partial B_r}\varphi\|_{H^{1/2}(\partial B_r)}
 \le C\|\nabla_y\varphi\|_{L^2(A_{\sigma_j,R_j})},
 \qquad \varphi\in H_0^1(A_{\sigma_j,R_j}),
\end{equation}
where \(C\) is independent of \(j\).

In the following weak identity,
\(\langle\mathcal R_j^\Phi,\varphi\rangle\) is the
\(\dot H^{-1}\)--\(H_0^1\) duality pairing associated with
\eqref{eq:homogeneous-dual-norm}. The surface terms use the
\(H^{-1/2}(\partial B_r)\)--\(H^{1/2}(\partial B_r)\) pairing and the
normals in \eqref{eq:boundary-auxiliary-normals}. Splitting the two
boundary regions gives, for every
\(\varphi\in H_0^1(A_{\sigma_j,R_j})\),
\[
 \begin{aligned}
  S_{j,A_{\sigma_j,R_j}}(e_j^0,\varphi)
 ={}&-\langle\mathcal R_j^\Phi,\varphi\rangle
 -\left\langle\sum_{k,l=1}^ng_j^{kl}\nu_k\partial_le_j^-,
  \operatorname{Tr}_{\partial B_{r_{\rm c}}}\varphi\right\rangle\\
 &-\left\langle\sum_{k,l=1}^ng_j^{kl}\nu_k\partial_le_j^+,
  \operatorname{Tr}_{\partial B_{R_{\rm c}}}\varphi\right\rangle.
 \end{aligned}
\]
Equations  \eqref{eq:boundary-extension-conormal},
\eqref{eq:uniform-fixed-sphere-trace}, and
\eqref{eq:zero-boundary-moment-bound}, together with
Proposition  \ref{prop:constrained-inverse}, imply
\[
 \norm{\nabla_ye_j^0}_{L^2(A_{\sigma_j,R_j})}
 \le C\bigl[
 \norm{\mathcal R_j^\Phi}_{\dot H^{-1}(A_{\sigma_j,R_j})}
 +\sigma_j^{n-2}+R_j^{2-n}\bigr].
\]
Equation  \eqref{eq:capacity-scales} now gives
\eqref{eq:zero-boundary-final-size}.

Lemmas  \ref{lem:boundary-extensions} and
\ref{lem:boundary-extension-decay} give
\[
 \begin{aligned}
 \int_{A_{\sigma_j,R_j}}\bigl(|\nabla_ye_j^-|^2+|y|^{-2}(e_j^-)^2\bigr)\,\dd y
 &\le C\sigma_j^{n-2},\\
 \int_{A_{\sigma_j,R_j}}\bigl(|\nabla_ye_j^+|^2+|y|^{-2}(e_j^+)^2\bigr)\,\dd y
 &\le CR_j^{2-n}.
 \end{aligned}
\]
By \eqref{eq:zero-boundary-remainder},
\[
 \|\nabla_ye_j\|_2^2+\|e_j/|y|\|_2^2
 \le3\sum_{\nu\in\{-,0,+\}}
 \bigl(\|\nabla_ye_j^\nu\|_2^2+\|e_j^\nu/|y|\|_2^2\bigr),
\]
where all norms are over \(A_{\sigma_j,R_j}\). The Hardy inequality
for the zero extension of \(e_j^0\), followed by
\eqref{eq:capacity-scales}, proves \eqref{eq:total-error-energy}.
\end{proof}

\subsection{Estimates for the linear Pohozaev term}

For the Pohozaev form we use the metric error operator
\(\Delta-L_{g_j}\). Its principal coefficient satisfies, on every
centered sphere,
\begin{equation}\label{eq:metric-error-radial-conormal}
 \sum_{k=1}^n(g_j^{kl}-\delta^{kl})\nu_k=0,
 \qquad 1\le l\le n.
\end{equation}
Consequently, the zero extensions of \(e_j^-\) and \(e_j^+\) produce no
surface distribution for this operator.
The weak gradient of each extension \(e_j^\pm\) is its piecewise gradient,
because its auxiliary trace is zero. For the expression of second order
\[
 (\Delta-L_{g_j})f
 =-\sum_{k,l=1}^n\partial_k
 ((g_j^{kl}-\delta^{kl})\partial_lf)+c(n)\operatorname{Scal}_{g_j}f,
\]
the possible surface distribution has coefficient
\(\sum_{k,l}(g_j^{kl}-\delta^{kl})\nu_k\partial_le_j^\pm=0\).
Thus \((\Delta-L_{g_j})e_j^\pm\) is the extension by zero of the
expression on its original boundary annulus.

For smooth functions on \(\overline{A_{\sigma_j,R_j}}\), define the boundary form
 
\begin{equation}\label{eq:dilation-boundary-form}
 \begin{aligned}
  B_j^\partial(f_1,f_2)
 ={}&\frac{R_j}{2}\int_{\partial B_{R_j}}
 \left[
 \sum_{k,l=1}^n(g_j^{kl}-\delta^{kl})\partial_kf_1\partial_lf_2
 +c(n)\operatorname{Scal}_{g_j}f_1f_2\right]\,\dd S\\
 &-\frac{\sigma_j}{2}\int_{\partial B_{\sigma_j}}
 \left[
 \sum_{k,l=1}^n(g_j^{kl}-\delta^{kl})\partial_kf_1\partial_lf_2
 +c(n)\operatorname{Scal}_{g_j}f_1f_2\right]\,\dd S.
 \end{aligned}
\end{equation}

\par
\begin{lemma}\label{lem:weak-dilation-identity}
For \(f_1,f_2\in C^2(\overline{A_{\sigma_j,R_j}})\),
 
\begin{equation}\label{eq:weak-dilation-identity}
 \begin{aligned}
  B_j(f_1,f_2)
 ={}&-\frac12\int_{A_{\sigma_j,R_j}}\sum_{k,l=1}^n
 y\cdot\nabla_y(g_j^{kl}-\delta^{kl})
 \partial_kf_1\partial_lf_2\,\dd y\\
 &-c(n)\int_{A_{\sigma_j,R_j}}
 \left(\operatorname{Scal}_{g_j}+\frac12y\cdot\nabla_y\operatorname{Scal}_{g_j}\right)f_1f_2\,\dd y
 + B_j^\partial(f_1,f_2).
 \end{aligned}
\end{equation}

\end{lemma}

\par
\begin{proof}
For \(i\in\{1,2\}\),
\begin{equation}\label{eq:dilation-derivative-identity}
\partial_k(D_{\mathrm{di}}f_i)
=\frac n2\partial_kf_i+\sum_{l=1}^ny^l\partial_{lk}f_i.
\end{equation}
The principal conormal term vanishes by
\eqref{eq:metric-error-radial-conormal}.
After one integration by parts, symmetry and
\eqref{eq:dilation-derivative-identity} give
\[
 \begin{aligned}
 &\frac12\int_{A_{\sigma_j,R_j}}
 \sum_{k,l=1}^n(g_j^{kl}-\delta^{kl})
 \left[
 \begin{aligned}
&\partial_k\left(D_{\mathrm{di}}f_1\right)\partial_lf_2+\partial_k\left(D_{\mathrm{di}}f_2\right)\partial_lf_1
 \end{aligned}
 \right]\,\dd y\\
 \quad=&\frac n2\int_{A_{\sigma_j,R_j}}
 \sum_{k,l=1}^n(g_j^{kl}-\delta^{kl})\partial_kf_1\partial_lf_2\,\dd y
+\frac12\int_{A_{\sigma_j,R_j}}
 y\cdot\nabla_y\left[
 \sum_{k,l=1}^n(g_j^{kl}-\delta^{kl})\partial_kf_1\partial_lf_2
 \right]\,\dd y\\
 \qquad&-\frac12\int_{A_{\sigma_j,R_j}}
 \sum_{k,l=1}^n y\cdot\nabla_y(g_j^{kl}-\delta^{kl})
 \partial_kf_1\partial_lf_2\,\dd y\\
 \quad=&-\frac12\int_{A_{\sigma_j,R_j}}
 \sum_{k,l=1}^n y\cdot\nabla_y(g_j^{kl}-\delta^{kl})
 \partial_kf_1\partial_lf_2\,\dd y+\frac12\int_{\partial A_{\sigma_j,R_j}}
 (y\cdot\nu)\sum_{k,l=1}^n
 (g_j^{kl}-\delta^{kl})\partial_kf_1\partial_lf_2\,\dd S.
 \end{aligned}
\]
For the scalar part, the product rule is
\[
 \operatorname{div}_y(y\operatorname{Scal}_{g_j}f_1f_2)
 =(n\operatorname{Scal}_{g_j}+y\cdot\nabla_y\operatorname{Scal}_{g_j})f_1f_2
 +\operatorname{Scal}_{g_j}\,y\cdot\nabla_y(f_1f_2).
\]
Consequently,
\[
 \begin{aligned}
 &\frac{c(n)}2\int_{A_{\sigma_j,R_j}}\operatorname{Scal}_{g_j}
 \bigl[(n-2)f_1f_2+y\cdot\nabla_y(f_1f_2)\bigr]\,\dd y\\
 \quad=&-c(n)\int_{A_{\sigma_j,R_j}}
\left(\operatorname{Scal}_{g_j}+\frac12y\cdot\nabla_y\operatorname{Scal}_{g_j}\right)f_1f_2\,\dd y+\frac{c(n)}2\int_{\partial A_{\sigma_j,R_j}}
 (y\cdot\nu)\operatorname{Scal}_{g_j}f_1f_2\,\dd S.
 \end{aligned}
\]
On \(\partial B_{R_j}\), \(y\cdot\nu=R_j\); on
\(\partial B_{\sigma_j}\), \(y\cdot\nu=-\sigma_j\).
The two boundary integrals sum to
 \( B_j^\partial(f_1,f_2)\) by \eqref{eq:dilation-boundary-form},
which proves \eqref{eq:weak-dilation-identity}.
\end{proof}

  On \(A_{\sigma_j,R_j}\), Taylor's formula and \eqref{eq:normal-orders} imply
\begin{equation}\label{eq:dilation-coefficient-bounds}
 \begin{aligned}
 |y\cdot\nabla_y(g_j^{kl}-\delta^{kl})|
 &\le C\rho_j^2|y|^2\le C\eta_j,\\
 \left|c(n)\operatorname{Scal}_{g_j}+\frac{c(n)}2y\cdot\nabla_y\operatorname{Scal}_{g_j}\right|
 &\le C\rho_j^4|y|^2\le C\eta_j^2|y|^{-2}.
 \end{aligned}
\end{equation}
Consequently,
\begin{equation}\label{eq:weak-dilation-volume-bound}
 \begin{aligned}
 | B_j(f_1,f_2)- B_j^\partial(f_1,f_2)|
 \le& C\eta_j\bigl(
 \norm{\nabla_yf_1}_{L^2(A_{\sigma_j,R_j})}
 \norm{\nabla_yf_2}_{L^2(A_{\sigma_j,R_j})}\\
 &+\norm{|y|^{-1}f_1}_{L^2(A_{\sigma_j,R_j})}
 \norm{|y|^{-1}f_2}_{L^2(A_{\sigma_j,R_j})}\bigr).
 \end{aligned}
\end{equation}
For smooth \(f_1\) with zero boundary trace,
\(\nabla_yf_1=(\partial_\nu f_1)\nu\) on each boundary component.
By equation  \eqref{eq:metric-error-radial-conormal},
\[
 \sum_{k,l=1}^n(g_j^{kl}-\delta^{kl})\partial_kf_1\partial_lf_2
 =(\partial_\nu f_1)\sum_{k,l=1}^n
 (g_j^{kl}-\delta^{kl})\nu_k\partial_lf_2=0.
\]
The scalar boundary term also vanishes because \(f_1=0\).
Extension by zero, the Hardy inequality, and density therefore give, for
\(f_1\in H_0^1(A_{\sigma_j,R_j})\) and \(f_2\in\dot H^1(\R^n)\),
 
\begin{equation}\label{eq:weak-dilation-zero-trace-bound}
 | B_j(f_1,f_2)|
 \le C\eta_j\norm{\nabla_yf_1}_{L^2(A_{\sigma_j,R_j})}
 \norm{\nabla_yf_2}_{L^2(\R^n)}.
\end{equation}
Equations  \eqref{eq:weak-dilation-volume-bound} and
\eqref{eq:weak-dilation-zero-trace-bound} extend
 \( B_j\) continuously to \(H_0^1(A_{\sigma_j,R_j})\times\dot H^1(\R^n)\).

When one argument is a boundary correction, we apply
\eqref{eq:weak-dilation-identity} on its original annulus and extend the
volume integrands by zero. There is no contribution from the auxiliary
sphere: the trace and tangential derivatives of the correction vanish
there, while the radial conormal coefficient in
\eqref{eq:metric-error-radial-conormal} is zero.

By Proposition  \ref{prop:euclidean-response},
\[
 \norm{\nabla_y\Phi_j}_{L^2(\R^n)}\le C.
\]
We estimate the linear term using
\(e_j=e_j^-+e_j^0+e_j^+\), and the quadratic term using
\eqref{eq:total-error-energy}. In the linear term with \(e_j^0\),
the physical boundary integral vanishes by
\eqref{eq:metric-error-radial-conormal}.
We write \(I\) for the \(n\times n\) identity matrix in the coefficient
estimates used to prove the next lemma.
\begin{lemma}\label{lem:reference-dilation-functional}
For all sufficiently large \(j\) and every
\(\varphi\in H_0^1(A_{\sigma_j,R_j})\),
 
\begin{equation}\label{eq:reference-dilation-functional}
 | B_j(\Phi_j,\varphi)|
 \le\norm{\nabla_y\varphi}_{L^2(A_{\sigma_j,R_j})}
 \begin{cases}
  C\rho_j^{m_*},&m_*<\frac{n-2}{2},\\
  C\rho_j^{\frac{n-2}{2}}(1+\log R_j)^{1/2},&n\text{ is even and }m_*=\frac{n-2}{2},\\
  o(\rho_j^{\frac{n-2}{2}}),&m_*=\infty.
 \end{cases}
\end{equation}

The little \(o\) term in \eqref{eq:reference-dilation-functional} is
taken as \(j\to\infty\), uniformly for
\((\lambda_j,b_j)\in\mathcal K\).
\end{lemma}

\begin{proof}
Since \(\varphi\in H_0^1(A_{\sigma_j,R_j})\), the boundary term in
\eqref{eq:weak-dilation-identity} vanishes by
\eqref{eq:metric-error-radial-conormal}. Define
\[
 \begin{aligned}
 A^l&=\frac12\sum_{k=1}^n
 y\cdot\nabla_y(g_j^{kl}-\delta^{kl})\,\partial_k\Phi_j,
 \qquad 1\le l\le n,\\
 f&=-c(n)\left(\operatorname{Scal}_{g_j}+\frac12y\cdot\nabla_y\operatorname{Scal}_{g_j}\right)\Phi_j.
 \end{aligned}
\]
Then
\[
  B_j(\Phi_j,\varphi)
 =\langle\operatorname{div}_y A+f,\varphi\rangle.
\]

Use \(r=|y|\) and the degree \(m_0=\min\{m_*,d+1\}\) fixed above.
The Taylor bounds through three spatial derivatives in
\eqref{eq:scaled-metric-taylor-remainder} give
\[
 \sum_{a=0}^3r^a|\nabla_y^a[h(\rho_j\,\cdot)](y)|
 \le C\rho_j^{m_0}r^{m_0}.
\]
Differentiating the matrix exponential once and the scalar curvature
formula once therefore gives
\[
 |y\cdot\nabla_y(g_j^{-1}-I)|
 +r^2\bigl(|\operatorname{Scal}_{g_j}|+|y\cdot\nabla_y\operatorname{Scal}_{g_j}|\bigr)
 \le C\rho_j^{m_0}r^{m_0}.
\]
Lemma  \ref{lem:corrected-profile-estimate} and the bubble formula imply
\[
 |\Phi_j|+(1+r)|\nabla_y\Phi_j|\le C(1+r)^{2-n}.
\]
Consequently \(A,f\) satisfy \eqref{eq:source-field-bounds} with \(s=m_0\).
Lemma  \ref{lem:degree-source-bound} gives the first two cases of
\eqref{eq:reference-dilation-functional}. If \(m_*=\infty\), then
\(m_0>(n-2)/2\), and \eqref{eq:outer-radius-rho-bound} gives
\[
 \rho_j^{m_0}R_j^{m_0-(n-2)/2}
 \le C\rho_j^{(m_0+(n-2)/2)/2}
 =o(\rho_j^{(n-2)/2}).
\]
This proves the remaining case.
\end{proof}

\par
It remains to estimate the terms containing the two boundary extensions
\(e_j^-\) and \(e_j^+\).

\begin{lemma}\label{lem:boundary-dilation-terms}
For all sufficiently large \(j\),
 
\begin{equation}\label{eq:boundary-dilation-bound}
 | B_j(\Phi_j,e_j^-)|
 +| B_j(\Phi_j,e_j^+)|
 \le C\eta_j(\mu_j+\mu_{j-1}).
\end{equation}

\end{lemma}

\begin{proof}
We first estimate the volume integrals. With \(r=|y|\),
Lemma  \ref{lem:corrected-profile-estimate} gives
\[
 \begin{aligned}
 |\Phi_j|+r|\nabla_y\Phi_j|&\le C,
 &&\sigma_j\le r\le r_{\rm c},\\
 |\Phi_j|+r|\nabla_y\Phi_j|&\le Cr^{2-n},
 &&R_{\rm c}\le r\le R_j.
 \end{aligned}
\]
Insert these estimates, \eqref{eq:dilation-coefficient-bounds}, and
\eqref{eq:boundary-extension-pointwise} in
\eqref{eq:weak-dilation-identity}. On the inner annulus, the principal
and scalar volume terms are bounded by
\[
 C\rho_j^2\sigma_j^{n-2}\int_{\sigma_j}^{r_{\rm c}}r\,\dd r,
 \qquad
 C\rho_j^4\sigma_j^{n-2}\int_{\sigma_j}^{r_{\rm c}}r^3\,\dd r.
\]
On the outer annulus their sum is bounded by
\[
 CR_j^{2-n}\int_{R_{\rm c}}^{R_j}
 (\rho_j^2r+\rho_j^4r^3)\,\dd r
 \le C(\eta_j+\eta_j^2)R_j^{2-n}.
\]
Consequently,
 
\[
 \begin{aligned}
 | B_j(\Phi_j,e_j^-)
 - B_j^\partial(\Phi_j,e_j^-)|
 &\le C\rho_j^2\sigma_j^{n-2},\\
 | B_j(\Phi_j,e_j^+)
 - B_j^\partial(\Phi_j,e_j^+)|
 &\le C\eta_jR_j^{2-n}.
 \end{aligned}
\]

We next estimate the physical boundary integrals. Only tangential
derivatives occur in the principal boundary term,
by \eqref{eq:metric-error-radial-conormal}. The common boundary values
of \(e_j\) and \(e_j^\pm\) therefore give
 
\[
 \begin{aligned}
  B_j^\partial(\Phi_j,e_j^-)
 &=-\frac{\sigma_j}{2}\int_{\partial B_{\sigma_j}}
 \left[
 \sum_{k,l=1}^n(g_j^{kl}-\delta^{kl})
 \partial_k\Phi_j\partial_le_j+c(n)\operatorname{Scal}_{g_j}\Phi_je_j
 \right]\,\dd S,\\
  B_j^\partial(\Phi_j,e_j^+)
 &=\frac{R_j}{2}\int_{\partial B_{R_j}}
 \left[
 \sum_{k,l=1}^n(g_j^{kl}-\delta^{kl})
 \partial_k\Phi_j\partial_le_j+c(n)\operatorname{Scal}_{g_j}\Phi_je_j
 \right]\,\dd S.
 \end{aligned}
\]

The trace bounds in Lemma  \ref{lem:annular-boundary-traces} imply
\begin{equation}\label{eq:linear-boundary-extension-surface}
 \begin{aligned}
 | B_j^\partial(\Phi_j,e_j^-)|
 &\le C(\rho_j^2\sigma_j^n+\rho_j^4\sigma_j^{n+2}),\\
 | B_j^\partial(\Phi_j,e_j^+)|
 &\le C(\rho_j^2R_j^{4-n}+\rho_j^4R_j^{6-n}).
 \end{aligned}
\end{equation}
There is no integral on an auxiliary sphere, where \(e_j^\pm=0\):
its tangential derivatives vanish, and the radial conormal coefficient
is zero.
Since \(\rho_j^2\le\eta_j\), \(\rho_j^2R_j^2=\eta_j\), and
\eqref{eq:capacity-scales} holds, the volume and boundary estimates give
\eqref{eq:boundary-dilation-bound}.
\end{proof}

\par
\subsection{The full comparison}

\begin{proposition}\label{prop:exact-reference-comparison}
For all sufficiently large \(j\),
\begin{equation}\label{eq:exact-reference-comparison}
 Q_j\ge\widetilde Q_j-C\eta_j(\mu_j+\mu_{j-1})
 -\begin{cases}
  o(\rho_j^{2m_*}),&m_*<\frac{n-2}{2},\\
  C\rho_j^{n-2}(1+\log R_j),&n\text{ is even and }m_*=\frac{n-2}{2},\text{ or }m_*=\infty.
 \end{cases}
\end{equation}
The little \(o\) term in \eqref{eq:exact-reference-comparison} is taken as
\(j\to\infty\), uniformly for
\((\lambda_j,b_j)\in\mathcal K\).
\end{proposition}

\par
\begin{proof}
We first estimate the quadratic term in the full error. Applying
\eqref{eq:weak-dilation-identity} with \(f_1=f_2=e_j\) and using the
coefficient estimates gives
\[
 | B_j(e_j,e_j)- B_j^\partial(e_j,e_j)|
 \le C\eta_j\int_{A_{\sigma_j,R_j}}
 \left(|\nabla_ye_j|^2+|y|^{-2}e_j^2\right)\,\dd y.
\]

The boundary calculation uses the same tangential trace estimates as in
\eqref{eq:linear-boundary-extension-surface}:
\[
 | B_j^\partial(e_j,e_j)|
 \le C\left(
 \rho_j^2\sigma_j^n+\rho_j^4\sigma_j^{n+2}
 +\rho_j^2R_j^{4-n}+\rho_j^4R_j^{6-n}\right)
 \le C\eta_j(\mu_j+\mu_{j-1}).
\]
Together with \eqref{eq:total-error-energy}, this yields
\[
 | B_j(e_j,e_j)|
 \le C\eta_j\left[
 \norm{\nabla_ye_j^0}_{L^2(A_{\sigma_j,R_j})}^2
 +\mu_j+\mu_{j-1}\right].
\]

For the linear term, the decomposition of \(e_j\), the polarization
identity, and Lemma  \ref{lem:boundary-dilation-terms} give
\begin{equation}\label{eq:complete-comparison-bound}
 |Q_j-\widetilde Q_j|
 \le 2| B_j(\Phi_j,e_j^0)|
 +C\eta_j\left[
 \norm{\nabla_ye_j^0}_{L^2(A_{\sigma_j,R_j})}^2
 +\mu_j+\mu_{j-1}\right].
\end{equation}

Set
\[
 \begin{gathered}
 D_j=\|\mathcal R_j^\Phi\|_{\dot H^{-1}(A_{\sigma_j,R_j})},
 \qquad M_j=\mu_j+\mu_{j-1},\\
 \Lambda_j=\sup_{\substack{\varphi\in H_0^1(A_{\sigma_j,R_j})\\
 \|\nabla_y\varphi\|_{L^2(A_{\sigma_j,R_j})}=1}}
 |B_j(\Phi_j,\varphi)|.
 \end{gathered}
\]
Proposition  \ref{prop:zero-boundary-remainder} gives
\[
 \|\nabla_ye_j^0\|_{L^2(A_{\sigma_j,R_j})}
 \le C(D_j+M_j).
\]
Lemma  \ref{lem:reference-dilation-functional}
gives the sharper estimate for \(\Lambda_j\) in each case, as well as
\(\Lambda_j\le C\eta_j\) in all three cases. If \(n\) is even and \(m_*=(n-2)/2\), the bound \(\Lambda_j\le C\eta_j\) follows from
\[
 \rho_j^{(n-2)/2}(1+\log R_j)^{1/2}
 \le C\rho_j^2\le C\eta_j,
\]
using \(n\ge7\) and \eqref{eq:outer-radius-rho-bound}.
Since \(M_j\to0\), \eqref{eq:complete-comparison-bound} gives the
common estimate
\begin{equation}\label{eq:uniform-comparison-substitution}
 \begin{aligned}
 |Q_j-\widetilde Q_j|
 &\le2\Lambda_j
 \|\nabla_ye_j^0\|_{L^2(A_{\sigma_j,R_j})}+C\eta_j\left[
 \|\nabla_ye_j^0\|_{L^2(A_{\sigma_j,R_j})}^2+M_j
 \right]\\
 &\le C\bigl(\Lambda_jD_j+\eta_jD_j^2+\eta_jM_j\bigr).
 \end{aligned}
\end{equation}

For the terms containing \(M_j\), we use
\(\Lambda_j\le C\eta_j\). For \(\Lambda_jD_j\), we use the bounds in
Lemmas  \ref{lem:reference-equation-error-estimate} and
\ref{lem:reference-dilation-functional} for each value of \(m_*\).

If \(m_*<(n-2)/2\), then
\(D_j=o(\rho_j^{m_*})\) and \(\Lambda_j=O(\rho_j^{m_*})\).
By  \eqref{eq:uniform-comparison-substitution},
\[
 |Q_j-\widetilde Q_j|
 \le C\eta_jM_j+o(\rho_j^{2m_*}).
\]
If \(n\) is even and \(m_*=(n-2)/2\), both \(D_j\) and \(\Lambda_j\)
are bounded by
\(C\rho_j^{(n-2)/2}(1+\log R_j)^{1/2}\). Hence
\[
 |Q_j-\widetilde Q_j|
 \le C\eta_jM_j+C\rho_j^{n-2}(1+\log R_j).
\]
Finally, if \(m_*=\infty\), then
\(D_j=o(\rho_j^{(n-2)/2})\) and
\(\Lambda_j=o(\rho_j^{(n-2)/2})\), hence
\[
 |Q_j-\widetilde Q_j|
 \le C\eta_jM_j+o(\rho_j^{n-2}).
\]
All little \(o\) terms are uniform on \(\mathcal K\).
These bounds for the absolute values imply the lower estimate
\eqref{eq:exact-reference-comparison} required in Section  \ref{sec:recurrence}.
\end{proof}

\par
\section{Discrete scale iteration and Fowler asymptotics}\label{sec:recurrence}
\label{sec:fowler}

The exact increment from Section  \ref{sec:euclidean-annuli}, the lower
bound for \(\widetilde Q_j\) from Section  \ref{sec:reference-lower-bound},
and the comparison in Section  \ref{sec:exact-reference-comparison} yield
a backward recurrence for \(\mu_j\). Its tail estimate excludes a
nonremovable singularity with \(\cP_\infty=0\). We then return to the
general hypotheses of
Theorem  \ref{thm:lower-bound}, prove the lower bound, and derive the Fowler
asymptotics.

\subsection{The recurrence and the matching lower bound}
Recall that the radii \(\rho_j\), the numbers \(\mu_j\), the scales
\(\eta_j\), and the first nonzero metric degree \(m_*\) are specified in
Lemma  \ref{lem:peak-minimum-radii}, \eqref{eq:mu-definition},
\eqref{eq:metric-perturbation-scale}, and \eqref{eq:first-nonzero-degree}. The exact
increment identity is \eqref{eq:exact-Qj-increment}. Every little \(o\) term
below is taken as \(j\to\infty\), with the background metric fixed,
uniformly for \((\lambda_j,b_j)\in\mathcal K\).

For the recurrence and its tail estimate \eqref{eq:mu-tail-bound}, assume that
\eqref{eq:contradiction-case} holds. If
\(m_*<\frac{n-2}{2}\), Propositions
\ref{prop:reference-pohozaev-lower-bound} and
\ref{prop:exact-reference-comparison} give
\[
 Q_j\ge c\rho_j^{2m_*}-o(\rho_j^{2m_*})
 -C\eta_j(\mu_j+\mu_{j-1}).
\]
For all sufficiently large \(j\),
\[
 c\rho_j^{2m_*}-o(\rho_j^{2m_*})
 \ge\frac c2\rho_j^{2m_*}\ge0.
\]
If \(n\) is even and \(m_*=\frac{n-2}{2}\), or if \(m_*=\infty\),
\eqref{eq:reference-pohozaev-lower-bound},
\eqref{eq:exact-reference-comparison}, and the bound
\[
 1+\log R_j\le C(1+|\log\rho_j|)
\]
from \eqref{eq:outer-radius-rho-bound} bound the last error term of \eqref{eq:exact-reference-comparison} by $C\rho_j^{n-2}(1+|\log\rho_j|)$.
Thus all three cases yield
\begin{equation}\label{eq:final-annular-increment}
 Q_j\ge-C\eta_j(\mu_j+\mu_{j-1})
 -C\rho_j^{n-2}(1+|\log\rho_j|)
\end{equation}
for all sufficiently large \(j\).

\par
\begin{lemma}\label{lem:backward-recurrence}
There is \(C\ge1\) such that, for all sufficiently large \(j\),
\begin{equation}\label{eq:one-step-backward-recurrence}
 \mu_{j-1}\le(1+C\eta_j)\mu_j
 +C\rho_j^{n-2}(1+|\log\rho_j|),
\end{equation}
and
\begin{equation}\label{eq:mu-tail-bound}
 \mu_j\le C\rho_{j+1}^{n-2}
 (1+|\log\rho_{j+1}|).
\end{equation}
\end{lemma}
Since \(\bar\rho_j\) decreases as \(j\) increases,
\eqref{eq:one-step-backward-recurrence} propagates control from the smaller
radius \(\bar\rho_j\) to the larger scale \(\bar\rho_{j-1}\).  Its
iteration therefore runs backward along the radial chain.

\par
\begin{proof}
By Proposition  \ref{prop:annular-increment} and
\eqref{eq:final-annular-increment},
\[
 (1-C\eta_j)\mu_{j-1}
 \le(1+C\eta_j)\mu_j
 +C\rho_j^{n-2}(1+|\log\rho_j|).
\]
Increase the initial index so that \(C\eta_j\le1/2\).
Since \(\eta_j\to0\),
\eqref{eq:one-step-backward-recurrence} follows
after division by \(1-C\eta_j\) and an enlargement of \(C\).

Applying \eqref{eq:one-step-backward-recurrence} successively with indices
\(j+1,\ldots,J\), we have
\begin{equation}\label{eq:iterated-backward-recurrence}
 \mu_j\le{}\mu_J\prod_{k=j+1}^J(1+C\eta_k)
 +C\sum_{i=j+1}^J
 \rho_i^{n-2}(1+|\log\rho_i|)
 \prod_{k=j+1}^{i-1}(1+C\eta_k).
\end{equation}
By \eqref{eq:metric-scale-summability},
\[
 \prod_{k=j+1}^J(1+C\eta_k)
 \le\exp\left(C\sum_{k=j+1}^J\eta_k\right)\le C.
\]
Moreover, by \eqref{eq:capacity-scales},
\[
 \mu_J\to0,\qquad \sigma_J\to0
 \quad\text{as }J\to\infty.
\]
Fix \(j\), and then let \(J\to\infty\) in
\eqref{eq:iterated-backward-recurrence}. This yields
\begin{equation}\label{eq:mu-tail-series}
 \mu_j\le C\sum_{i\ge j+1}
 \rho_i^{n-2}(1+|\log\rho_i|).
\end{equation}

Equation  \eqref{eq:peak-halving} gives
\(\rho_{k+1}\le\rho_k/2\) for all \(k\ge j_0\). The function
\[
 f(s)=s^{n-2}(1+|\log s|)
\]
is increasing on \((0,1)\), since \(f'(s)=s^{n-3}\bigl[(n-2)(1-\log s)-1\bigr]>0\).
For \(i=j+1+\ell\), we therefore have
\[
 \rho_i^{n-2}(1+|\log\rho_i|)
 \le f(2^{-\ell}\rho_{j+1})
 =2^{-\ell(n-2)}\rho_{j+1}^{n-2}
 \bigl(1+|\log\rho_{j+1}|+\ell\log2\bigr).
\]
Both series in \(2^{-\ell(n-2)}\) and
\(\ell2^{-\ell(n-2)}\) converge.  Substitution in
\eqref{eq:mu-tail-series} proves
\eqref{eq:mu-tail-bound}.
\end{proof}

\par
\begin{proposition}\label{prop:pohozaev-removability}
Let \(n\ge7\), and let \(g,u\) be the normalized metric and positive
solution of \eqref{eq:yamabe} fixed in
Section  \ref{sec:pohozaev-dichotomy}. Assume the critical upper bound
\eqref{eq:critical-upper}. Then $P_\infty=0$ if and only if the singularity of $u_0$
at the origin is removable.
\end{proposition}

\par
\begin{proof}
If $u_0$ extends smoothly across the origin, the normalized
solution $u$ has a positive $C^2$ extension there.
It follows directly from the definition of $P(r,u)$ that
$P_\infty=0$.

Conversely, suppose that \(\cP_\infty=0\) and the origin is nonremovable.
Then \eqref{eq:contradiction-case} holds. We use the peak and minimum
sequences constructed under that assumption in
Lemma  \ref{lem:peak-minimum-radii}. By
\eqref{eq:mu-peak-ratio} and
Lemma  \ref{lem:backward-recurrence},
\[
 C^{-1}\left(\frac{\rho_{j+1}}{\rho_j}\right)^{\frac{n-2}{2}}
 \le\mu_j
 \le C\rho_{j+1}^{n-2}(1+|\log\rho_{j+1}|).
\]
Dividing by
\(\bigl(\rho_{j+1}/\rho_j\bigr)^{(n-2)/2}\), we obtain
\[
 1\le C(\rho_j\rho_{j+1})^{\frac{n-2}{2}}
 (1+|\log\rho_{j+1}|).
\]
For large \(j\), \(\rho_j<1\), so the right side is at most
\[
 C\rho_{j+1}^{\frac{n-2}{2}}(1+|\log\rho_{j+1}|),
\]
which tends to zero as \(j\to\infty\). This contradiction proves the
proposition.
\end{proof}

\par
\begin{proof}[Proof of Theorem  \ref{thm:lower-bound}]
 Suppose that \(0\) is a nonremovable singularity. 
 We first work with a solution $u$ in a normalized metric \eqref{eq:normal-gauge}. By Proposition~\ref{prop:pohozaev-removability}, \(\cP_\infty<0\), and Lemma~\ref{lem:pohozaev-limit} then yields
\[
\liminf_{r\downarrow0}r^{(n-2)/2}\bar u(r)>0.
\]
After decreasing \(r_1\), there is \(c>0\) such that
\(r^{\frac{n-2}{2}}\bar u(r)\ge c\) whenever \(0<r<r_1\).  By
\eqref{eq:annular-harnack},
\[
 u(r\theta)\ge C_{\rm H}^{-1}\bar u(r)
 \ge cr^{-\frac{n-2}{2}},
 \qquad 0<r<r_1,\ \theta\in\Sph^{n-1}.
\]
Undoing the fixed conformal change, coordinate change, and rescaling from
Section  \ref{sec:pohozaev-dichotomy} changes only \(c\) and \(r_1\). This proves
\eqref{eq:lower-bound} in the original metric.
\end{proof}

\par
\subsection{Fowler asymptotics}

In this subsection, we return to the original smooth metric
and solution $(g_0,u_0)$, writing them as $(g,u)$. We now complete the proof of Theorem~\ref{thm:main}. 

With \(t=-\log r\), write a Fowler solution as
\begin{equation}\label{eq:fowler-transformation}
 t=-\log r,\qquad u_F(r)=r^{-\frac{n-2}{2}}v_F(t).
\end{equation}
Here \(v_F\) is either a positive constant solution or a positive
nonconstant periodic solution of \eqref{eq:fowler-ode}.

\begin{proposition}\label{prop:fowler-asymptotics}
If a solution $u$ of \eqref{eq:yamabe} satisfies both \eqref{eq:critical-upper} and \eqref{eq:lower-bound}, then in any geodesic normal coordinate $x$ of $g$, there exist a Fowler
solution \(u_F\) and a constant \(\alpha>0\) such that, 
\begin{equation}\label{eq:fowler-asymptotic}
 u(x)=u_F(|x|)\bigl(1+O(|x|^\alpha)\bigr)
 \quad\text{as }x\to0.
\end{equation}
Furthermore, there exist $a\in \mathbb{R}^n$ and
some $\beta>1$ such that
\begin{equation}\label{eq:fowler-asymptotic-order1}
 u(x)=u_F(|x|)
+
(a\cdot x)\Bigl[|x|\,u_F'(|x|)+(n-2)u_F(|x|)\Bigr]
+
O(|x|^\beta)u_F(|x|)
 \quad\text{as }x\to0.
\end{equation}
\end{proposition}

\par
\begin{proof}
As commented earlier, \eqref{eq:fowler-asymptotic} has been established in Marques~\cite[Theorem 8]{f-mar}, or Taliaferro--Zhang~\cite[Theorem 1]{TZ} supplemented by \mbox{\cite[Theorem  3]{HanLiTeixeira}}. \eqref{eq:fowler-asymptotic-order1} is an equivalent formulation of 
 Marques~\cite[Theorem 10]{f-mar}.
Marques~\cite[Theorem 8]{f-mar} adapts the proof of \cite[Proposition 5]{KMPS}, while
Marques~\cite[Theorem 10]{f-mar} is proved by a reduction to Korevaar-Mazzeo-Pacard-Schoen~\cite[Corollary 1]{KMPS}. 

For the reader's convenience, we sketch below a proof of \eqref{eq:fowler-asymptotic} following \cite{TZ}. In any geodesic normal coordinate $x$ of $g$, 
\eqref{eq:critical-upper} and \eqref{eq:lower-bound} imply that there are
\(c,C>0\) and \(r_1\in(0,r_0)\) such that
\begin{equation}\label{eq:fowler-two-sided-bound}
 c|x|^{-\frac{n-2}{2}}\le u(x)\le C|x|^{-\frac{n-2}{2}},
 \quad 0<|x|<r_1.
\end{equation}
Choose \(r_2\in(0,\frac12\min\{r_0,r_1\})\) sufficiently small. Applying Lemma~\ref{lem:radial-harnack} with \(k=2\), we obtain
\begin{equation}\label{eq:fowler-derivative-bounds}
 |\nabla_x^ku(x)|\le C_k|x|^{-\frac{n-2}{2}-k},
 \qquad 1\le k\le2,\quad 0<|x|<r_2.
\end{equation}
In geodesic normal coordinates,
\[
(\Delta-L_g)u
=-\sum_{i,j=1}^n\partial_i
\bigl((g^{ij}-\delta^{ij})\partial_j u\bigr)-\frac12\sum_{i,j=1}^n
g^{ij}(\partial_i\log\det g)\partial_j u
+c(n)\operatorname{Scal}_g\,u.
\]
Taylor's theorem gives
$|g^{ij}-\delta^{ij}|\le C|x|^2$ and
$|\nabla_xg^{ij}|+|\nabla_x\log\det g|\le C|x|$.
The scalar curvature is bounded near the origin.
Together with \eqref{eq:critical-upper} and
\eqref{eq:fowler-derivative-bounds}, these estimates give
\begin{equation}\label{eq:fowler-metric-error}
|(\Delta-L_g)u(x)|
\le C\bigl(|x|^2|\nabla_x^2u|
          +|x||\nabla_xu|+u\bigr)
\le C|x|^{-\frac{n-2}{2}}.
\end{equation}

Set 
\[
 \widetilde u
 :=\bigl(n(n-2)\bigr)^{1/(p-1)}u.
\]
It follows from equations  \eqref{eq:fowler-two-sided-bound} and
\eqref{eq:fowler-metric-error} that
\begin{equation}\label{eq:normalized-euclidean-ratio}
 \frac{-\Delta\widetilde u(x)}{\widetilde u(x)^p}
 =1-\frac{(\Delta-L_g)u(x)}{n(n-2)u(x)^p}
 =1+O(|x|^2).
\end{equation}
Fix \(\alpha\in(0,1)\) and set
\(w(x)=r_2^{(n-2)/2}\widetilde u(r_2 x)\). Then \(w\in C^2(B_1\setminus\{0\})\) satisfying
\[c|x|^{-(n-2)/2}\le w(x)\le C|x|^{-(n-2)/2}\]
and 
\[
\frac{-\Delta w(x)}{w(x)^p}
 =\left(\frac{-\Delta\widetilde u}{\widetilde u^p}\right)(r_2x)
 =1+O(|x|^\alpha).
\]

By \eqref{eq:normalized-euclidean-ratio}, the theorem of Taliaferro--Zhang
\cite[Theorem  1]{TZ} applies and gives a fixed positive radial profile $w_F(x)$ on \(\R^n\setminus\{0\}\) of \(-\Delta w_F = w_F^p\)
with relative error \(O(|x|^\alpha)\). For the convergence step in
their equation  (2.25), see the discussion in
Han--Xiong--Zhang  \cite[proof of Theorem  1.4]{HanXiongZhang} and its
reference to Han--Li--Teixeira  \cite[Theorem  3]{HanLiTeixeira}.
Undoing the dilation gives a positive radial solution
\(\widetilde u_F(x) = r_2^{-\frac{n-2}{2}}w_F(r_2^{-1}x)\) of
\(
 -\Delta\widetilde u_F=\widetilde u_F^p
\) on \(\R^n\setminus\{0\}\)
such that
\[
 \widetilde u(x)=\widetilde u_F(|x|)
 \bigl(1+O(|x|^\alpha)\bigr)
 \quad\text{as }x\to0.
\]
Set
\[
 u_F=\bigl(n(n-2)\bigr)^{-1/(p-1)}\widetilde u_F.
\]
Then \(-\Delta u_F=n(n-2)u_F^p\), and
\eqref{eq:fowler-asymptotic} follows. By
\eqref{eq:fowler-two-sided-bound},
\[
 0<c\le v_F(t)\le C
 \quad\text{for all sufficiently large }t.
\]
Thus the singularity of \(u_F\) is nonremovable. The radial
classification recalled in the Introduction shows that \(u_F\) is a
Fowler solution.
\end{proof}

\par


\begin{proof}[Proof of Theorem  \ref{thm:main}]
 As discussed in the proof of Theorem  \ref{thm:lower-bound}, it suffices to work in a geodesic normal coordinates of $g$ now. 
By Theorem  \ref{thm:upper-bound}, the critical upper bound \eqref{eq:critical-upper} holds. If the
singularity is nonremovable, Theorem  \ref{thm:lower-bound} and
Proposition  \ref{prop:fowler-asymptotics} give the lower bound \eqref{eq:lower-bound}, 
\eqref{eq:main-fowler-asymptotic} and
\eqref{eq:main-fowler-asymptotic-order1}.
If the singularity is removable, elliptic regularity for the original
equation with the smooth metric \(g_0\) gives a smooth extension.
\end{proof}


\appendix

\refstepcounter{section}\label{app:quadratic-estimate}
\section*{Appendix A. The Pohozaev quadratic form}
\addcontentsline{toc}{section}{Appendix A. The Pohozaev quadratic form}

{
The purpose of this appendix is to prove the signed estimate for \(\int_{B_{R_k}}Z_kE_k\,\dd y\)
used in Section  3.  We first recall the correctors of
Khuri--Marques--Schoen and their quadratic form on Euclidean space, and then
establish the required positivity.  Two
comparison arguments pass from the Euclidean form to a finite ball and then
to the actual metric.  The final subsection combines these results to prove
Lemma  \ref{lem:signed-residual}.
}

\subsection*{Correctors and the form on Euclidean space}

For $H^{(m)}\in\mathcal V_m$ and every
$a\in\{1,\ldots,n\}$, integration by parts on the sphere gives
{
\begin{equation}\label{eq:appendix-moments}
 \int_{\Sph^{n-1}}\sum_{i,j=1}^n
 \partial_i\partial_jH_{ij}^{(m)}\dd\theta=0,
 \qquad
 \int_{\Sph^{n-1}}y^a\sum_{i,j=1}^n
 \partial_i\partial_jH_{ij}^{(m)}\dd\theta=0.
\end{equation}
}
For a homogeneous polynomial \(P_\ell\) of degree \(\ell\ge0\), define
{
\begin{equation*}
 \mathcal F(P_\ell)
 =\operatorname{span}\Bigl\{
 |y|^{2j}\Delta^kP_\ell:
 0\le k\le\bigl\lfloor\ell/2\bigr\rfloor,
 \ 0\le j\le k+2\Bigr\}.
\end{equation*}
}
The following result is due to
Khuri--Marques--Schoen \cite[Proposition  4.1]{KMS}.

 \begin{proposition}\label{prop:kms-corrector}
Let \(P_\ell\) be a homogeneous polynomial of degree \(\ell<n-4\),
orthogonal on \(\Sph^{n-1}\) to \(1,y^1,\ldots,y^n\). There is a unique
\(\Gamma(P_\ell)\in\mathcal F(P_\ell)\) such that
{
\begin{equation*}
 \psi(P_\ell)=\Gamma(P_\ell)(1+|y|^2)^{-n/2}
\end{equation*}
}
 satisfies
{
\begin{equation*}
 \Delta\psi(P_\ell)+n(n+2)U^{4/(n-2)}\psi(P_\ell)=P_\ell U
 \qquad\text{in }\R^n.
\end{equation*}
}
Moreover, \(\deg\Gamma(P_\ell)\le\ell+4\),
\(\Gamma(P_\ell)(0)=0\), and \(\nabla\Gamma(P_\ell)(0)=0\).
\end{proposition}

For a symmetric polynomial matrix \(H\), write
{
\[
 (\delta H)_j=\sum_{i=1}^n\partial_iH_{ij},\qquad
 \delta^2H=\sum_{i,j=1}^n\partial_i\partial_jH_{ij}.
\]
}
For \(H^{(m)}\in\mathcal V_m\), \(2\le m\le n-4\),
\eqref{eq:appendix-moments} shows that
\(c(n)\delta^2H^{(m)}\) satisfies the orthogonality conditions in
Proposition  \ref{prop:kms-corrector}. Define
{
\[
 Z(H^{(m)})
 =\Gamma\bigl(c(n)\delta^2H^{(m)}\bigr)(1+|y|^2)^{-n/2}.
\]
}
Then
{
\[
 \bigl(\Delta+n(n+2)U^{4/(n-2)}\bigr) Z(H^{(m)})
 =c(n)\delta^2H^{(m)}U.
\]
}
For \(m\ge4\), this is \(\psi_{m,1}(H^{(m)})\) from
\eqref{eq:upper-polynomial-correction}. For \(m=2,3\),
\eqref{eq:appendix-moments} gives \(\delta^2H^{(m)}=0\); we extend the
notation by setting
\[
 \psi_{m,\lambda}(H^{(m)})=Z(H^{(m)})=0.
\]
For a polynomial \(\Gamma(P_\ell)\), let
\(\Gamma^{(j)}(P_\ell)\) denote its homogeneous part of degree \(j\).
Set
{
\begin{equation*}
 B(H^{(s)},H^{(t)})
 =\int_{\Sph^{n-1}}
 \Biggl\{-\frac12\sum_{i=1}^n
                   (\delta H^{(s)})_i(\delta H^{(t)})_i
        +\frac14\sum_{i,j,\ell=1}^n
                   \partial_\ell H_{ij}^{(s)}
                   \partial_\ell H_{ij}^{(t)}\Biggr\}\dd\theta
\end{equation*}
}
 and, for finite $R$,
{
\begin{equation}\label{eq:truncated-radial-coefficient}
 c_r(R)=-\int_0^R
 \frac{(1-\rho^2)\rho^{r+n-3}}
      {(1+\rho^2)^{n-1}}\dd\rho.
\end{equation}
}
 When $r<n-2$, we also write $c_r=c_r(\infty)$.
Khuri--Marques--Schoen define the Pohozaev quadratic form
$I_\varepsilon^{(n)}$ on $\mathcal V_{\le d}$ in
\cite[Appendix  A]{KMS} by
{
\begin{equation}\label{eq:kms-pohozaev-form}
 I_\varepsilon^{(n)}
 =I_{1,\varepsilon}^{(n)}+I_{2,\varepsilon}^{(n)},
\end{equation}
}
{
If \(n\) is odd, then
\[
\begin{aligned}
 I_{1,\varepsilon}^{(n)}(H,H)
 &=\frac{n-2}{2}\sum_{s,t=2}^d
 \varepsilon^{s+t}c_{s+t}B(H^{(s)},H^{(t)}),\\
 I_{2,\varepsilon}^{(n)}(H,H)
 &=-\sum_{s,t=4}^d s\varepsilon^{s+t}
 \int_{\R^n}
 \delta^2H^{(s)}Z(H^{(t)})U\,\dd y.
\end{aligned}
\]
If \(n\) is even, then
\[
\begin{aligned}
 I_{1,\varepsilon}^{(n)}(H,H)
 &=\frac{n-2}{2}\sum_{s,t=2}^{d-1}
 \varepsilon^{s+t}c_{s+t}B(H^{(s)},H^{(t)})
 +\frac{n-2}{2}\varepsilon^{n-2}|\log\varepsilon|
 B(H^{(d)},H^{(d)}),\\
 I_{2,\varepsilon}^{(n)}(H,H)
 &=-\sum_{s,t=4}^{d-1}s\varepsilon^{s+t}
 \int_{\R^n}
 \delta^2H^{(s)}Z(H^{(t)})U\,\dd y\\
 &\quad-d\varepsilon^{n-2}|\log\varepsilon|
 \int_{\Sph^{n-1}}
 \delta^2H^{(d)}
 \Gamma^{(d+2)}\Bigl(
 c(n)\delta^2H^{(d)}\Bigr)\,\dd\theta.
\end{aligned}
\]
}

When $d=3$,  {by the moment identities,}
\(\delta^2H^{(d)}=0\), and the last term
in $I_{2,\varepsilon}^{(n)}$ is zero.
For \(0<\varepsilon<e^{-1}\), define the weighted norm
{
\begin{equation*}
 \|H\|_{\varepsilon,*}^2
 =\sum_{m=2}^d\varepsilon^{2m}|\log\varepsilon|^{\theta_m}
 |H^{(m)}|^2.
\end{equation*}
}
\subsection*{Positivity of the quadratic form}

 \begin{proposition}\label{prop:kms-algebraic-positivity}
If $7\le n\le24$, there is $b_n>0$ such that
{
\[
 I_\varepsilon^{(n)}(H,H)
 \ge b_n\|H\|_{\varepsilon,*}^2
\]
}
for every $H\in\mathcal V_{\le d}$ and every
$0<\varepsilon<e^{-1}$.
If $25\le n\le29$, the Pohozaev quadratic form is
indefinite on $\mathcal V_{\le d}$, but its restriction to
$\mathcal V_{\le6}$ is positive: there is $b_n^{(6)}>0$ such that
{
\[
 I_\varepsilon^{(n)}(H,H)
 \ge b_n^{(6)}\sum_{m=2}^6\varepsilon^{2m}|H^{(m)}|^2
\]
}
 for every $H\in\mathcal V_{\le 6}$ and every $\varepsilon>0$.
\end{proposition}

\begin{proof}
For $7\le n\le24$, Khuri--Marques--Schoen
\cite[Proposition  A.4]{KMS} prove
 {
\begin{equation*}
 I_\varepsilon^{(n)}(H,H)
 \ge \beta\sum_{m=2}^d
 \varepsilon^{2m}|\log\varepsilon|^{\theta_m}
 \int_{\Sph^{n-1}}\sum_{i,j=1}^n
 \bigl(H_{ij}^{(m)}\bigr)^2\dd\theta \ge b_n\|H\|_{\varepsilon,*}^2.
\end{equation*}
}
For $n\ge25$, the full form
$I_\varepsilon^{(n)}$ is indefinite by Khuri--Marques--Schoen
\cite[Proposition  A.8]{KMS}.

 Define
{
\[
 b_{2r}=\int_0^\infty
 \frac{s^{2r+n-3}}{(1+s^2)^{n-1}}\,\dd s,
 \qquad 0\le2r<n-2,
\]
}
Then $b_0>0$.

For $25\le n\le29$, we prove positivity on
$\mathcal V_{\le6}$ by examining the blocks in the orthogonal decomposition
used by Khuri--Marques--Schoen.  Normalize each block by $b_0$ and by the squared norm
of its underlying harmonic tensor or polynomial.  In Case  1, also divide
by $(n-2)/8$; in Case  2, also divide by $(n-2)/2$.  Put
{
\begin{equation}\label{eq:bc-definition}
 \widehat b_0=1,\qquad
 \widehat b_{2r}=\prod_{\ell=1}^r\frac{n+2\ell-4}{n-2\ell},
 \qquad
 \widehat c_{2r}=\frac{4r}{n-2r-2}\widehat b_{2r}.
\end{equation}
}
These constants satisfy
\[
 \widehat b_{2r}=\frac{b_{2r}}{b_0},\qquad
 \widehat c_{2r}=\frac{c_{2r}}{b_0}.
\]
The hats distinguish these normalized constants from the radial integrals
$c_r$ in \eqref{eq:truncated-radial-coefficient}.  Since only indices at
most $14$ occur, all
denominators in \eqref{eq:bc-definition} are positive in this range.

By Khuri--Marques--Schoen
\cite[Lemma  A.6 and the proof of Proposition  A.4]{KMS}, the space
decomposes orthogonally into three families of blocks.  The Case  1 blocks are
{
\[
 M^{W,\mathrm{odd}}=(st\widehat c_{s+t})_{s,t\in\{3,5\}},
 \qquad
 M^{W,\mathrm{even}}=(st\widehat c_{s+t})_{s,t\in\{4,6\}}.
\]
}
The Case  2 blocks are
{
\[
 M^{D,a}_{rt}=\widehat c_{2a+2r+2t}
 \bigg(rt+\frac a4(n+2r+2t+2a-2)\bigg),
\]
}
where $2\le a\le6$ and
$0\le r,t\le\lfloor(6-a)/2\rfloor$.
For Case  3, let $2\le a\le4$ and
$1\le r,t\le\lfloor(6-a)/2\rfloor$, and set
 {
\[
 \alpha_a=\frac{n-2}{n-1}a(a-1)(n+a-1)(n+a-2),
\]
}
 {
\[
 A_{a,r+1}
 =(a-1)\bigg\{1-\frac{n-2}{2(n-1)}(n+a-1)\bigg\}
 -\frac{r+1}{2}(n+2r+2a-4).
\]
}
For  $k=a+2r$, define
{
\begin{align*}
 \Gamma(k,r,r+1)
 &=-\frac1{(2k-2r)(n-2r-2)},\\
 \Gamma(k,r,r)
 &=-\frac{2-2(r+1)(n+2k-2r)\Gamma(k,r,r+1)}
 {(2k-2r-2)(n-2r)},\\
 \Gamma(k,r,r-1)
 &=-\frac{1-2r(n+2k-2r-2)\Gamma(k,r,r)}
 {(2k-2r-4)(n-2r+2)},
\end{align*}
}
 and, for $0\le\ell\le r-2$,
{
\[
 \Gamma(k,r,\ell)
 =\frac{2(\ell+1)(n+2\ell+2k-4r)}
 {(2k+2\ell-4r-2)(n-2\ell)}\Gamma(k,r,\ell+1).
\]
}
Then
{
\begin{equation}\label{eq:case-three-appendix}
\begin{aligned}
 M^{H,a}_{rt}
 &=\frac{n-2}{2}\widehat c_{2a+2r+2t}\alpha_a
 \bigg\{A_{a,r+1}
 +\frac{(n+2a+2r+2t-2)(a+2r)}4\bigg\}\\
 &\quad-\frac{n-2}{4(n-1)}\alpha_a^2(a+r+t)
 \sum_{\ell=0}^{r+1}\Gamma(a+2r,r,\ell)
 \widehat b_{2a+2t+2\ell}.
\end{aligned}
\end{equation}
}
The  symmetry in $r,t$ follows from
Khuri--Marques--Schoen \cite[(A.11)]{KMS}.
The degree two $W$ component vanishes.  The block sizes are $2,2$ in
Case  1, $3,2,2,1,1$ in Case  2 for $a=2,\ldots,6$, and $2,1,1$ in
Case  3 for $a=2,3,4$.  By Khuri--Marques--Schoen
\cite[Lemma  A.6]{KMS}, this list exhausts the orthogonal decomposition
of $\mathcal V_{\le6}$, up to multiplicity.

For a symmetric block $M$, let $M^{[j]}$ be its leading $j\times j$
principal submatrix and put
{
\begin{equation*}
 p_j(M)=\frac{\det M^{[j]}}{\det M^{[j-1]}},
 \qquad \det M^{[0]}=1.
\end{equation*}
}
 Every block has order at most three,
so it suffices to record the following formulas obtained by completing squares.
For a symmetric matrix $M=(m_{ij})$ of order two, with $p_1=m_{11}\ne0$,
{
\[
 \sum_{i,j=1}^2m_{ij}x_ix_j
 =p_1\bigl(x_1+m_{12}x_2/p_1\bigr)^2+p_2x_2^2,
 \qquad p_2=m_{22}-m_{12}^2/p_1.
\]
}
 For a symmetric matrix of order three, define $p_1,p_2$ in the same way.
If $p_1p_2\ne0$, then
{
\[
\begin{aligned}
 \sum_{i,j=1}^3m_{ij}x_ix_j
 &=p_1\bigl(x_1+m_{12}x_2/p_1+m_{13}x_3/p_1\bigr)^2\\
 &\quad+p_2\bigl(x_2+(m_{23}-m_{12}m_{13}/p_1)x_3/p_2\bigr)^2
       +p_3x_3^2,
\end{aligned}
\]
}
 where
{
\[
 p_3=m_{33}-\frac{m_{13}^2}{p_1}
       -\frac{(m_{23}-m_{12}m_{13}/p_1)^2}{p_2}
     =\frac{\det M}{p_1p_2}.
\]
}
Thus the pivots $p_j(M)$ are the coefficients of the successive squares.

For brevity, write $W_{\mathrm{odd}}=M^{W,\mathrm{odd}}$,
$W_{\mathrm{even}}=M^{W,\mathrm{even}}$, $D_a=M^{D,a}$, and
$H_a=M^{H,a}$.  Direct elimination in the two Case  1 blocks gives
{
\begin{align*}
 p_1(W_{\mathrm{odd}})
 &=\frac{108n(n+2)}{(n-8)(n-6)(n-4)},\\
 p_2(W_{\mathrm{odd}})
 &=\frac{100n(66-n)(n-2)(n+2)(n+4)}
 {3(n-12)(n-10)^2(n-8)(n-6)(n-4)},\\
 p_1(W_{\mathrm{even}})
 &=\frac{256n(n+2)(n+4)}
 {(n-10)(n-8)(n-6)(n-4)},\\
 p_2(W_{\mathrm{even}})
 &=\frac{36n(102-n)(n-2)(n+2)(n+4)(n+6)}
 {(n-14)(n-12)^2(n-10)(n-8)(n-6)(n-4)}.
\end{align*}
}
 All four quantities are positive when $25\le n\le29$.

 In Case  2, the first pivot has the common expression
{
\begin{equation*}
 p_1(D_a)=\frac{a^2(n+2a-2)}{n-2a-2}
 \prod_{\ell=1}^a\frac{n+2\ell-4}{n-2\ell}>0,
 \qquad 2\le a\le6.
\end{equation*}
}
 The remaining pivots are
{
\begin{align*}
 p_2(D_2)
 &=-\frac{n(n+2)(n+4)(n^2-54n+152)}
 {(n-10)(n-8)^2(n-6)(n-4)},\\
 p_3(D_2)
 &=\frac{1536n(n-3)(n-2)(n+2)(n+4)(n+6)(n^2-34n+144)}
 {(n-14)(n-12)^2(n-10)^2(n-8)(n-6)(n-4)(n^2-54n+152)},\\
 p_2(D_3)
 &=-\frac{n(n+2)(n+4)(n+6)(n^2-86n+248)}
 {(n-12)(n-10)^2(n-8)(n-6)(n-4)},\\
 p_2(D_4)
 &=-\frac{n(n+2)(n+4)(n+6)(n+8)(n^2-126n+368)}
 {(n-14)(n-12)^2(n-10)(n-8)(n-6)(n-4)}.
\end{align*}
}
 On the interval $25\le n\le29$, the four quadratic factors in these
pivots satisfy
{
\begin{equation*}
\begin{gathered}
 n^2-54n+152\le-573,
 \qquad n^2-34n+144\le-1,\\
 n^2-86n+248\le-1277,
 \qquad n^2-126n+368\le-2157.
\end{gathered}
\end{equation*}
}
Consequently, all Case  2 pivots are positive.

Applying the same elimination to
\eqref{eq:case-three-appendix} gives, in Case  3,
{
\begin{align*}
 p_1(H_2)
 &=\frac{16n^2(n-3)(n-2)^2(n+1)(n+2)(7n-8)^2}
 {3(n-10)(n-8)(n-6)(n-4)^2(n-1)^3},\\
 p_2(H_2)
 &=-\frac{n^2(n-3)(n-2)^3(n+1)(n+2)(n+4)(17n-18)^2}
 {3(n-14)(n-12)^2(n-10)(n-8)(n-6)^2(n-4)^2(n-1)^3}\times(n^2-86n+408),\\
 p_1(H_3)
 &=\frac{15n(n-3)(n-2)^2(n+1)(n+2)^2(n+4)(11n-14)^2}
 {(n-12)(n-10)(n-8)(n-6)(n-4)^2(n-1)^3},\\
 p_1(H_4)
 &=\frac{576n(n-3)(n-2)^2(n+2)^2(n+3)(n+4)(n+6)(8n-11)^2}
 {5(n-14)(n-12)(n-10)(n-8)(n-6)(n-4)^2(n-1)^3}.
\end{align*}
}
Here
 {
\begin{equation*}
 n^2-86n+408\le-1117,
 \qquad 25\le n\le29,
\end{equation*}
}
so every Case  3 pivot is positive as well.  Sylvester's criterion and
the  exhaustive orthogonal decomposition {establish positivity on}
$\mathcal V_{\le6}$.  Conjugating by
$H^{(m)}\mapsto\varepsilon^mH^{(m)}$ and using the equivalence of the
spherical $L^2$ and coefficient norms proves the asserted estimate.
\end{proof}

\begin{remark}\label{rem:degree-six-exact-algebra}
The two assertions in Proposition  \ref{prop:kms-algebraic-positivity}
are established differently. For \(7\le n\le24\), we invoke
Proposition  A.4 of Khuri--Marques--Schoen, whose proof of the full
positivity statement for \(6\le n\le24\) uses a Maple calculation \cite{KMS}. The argument above establishes
the additional assertion for \(25\le n\le29\) on
\(\mathcal V_{\le6}\). On this restricted space, every block in the
orthogonal decomposition has order at most three. The displayed
completion of squares reduces positivity to the signs of finitely many
pivots, while the factorized formulas and polynomial inequalities above
establish every required sign. Thus the restricted assertion follows
from exact algebraic identities; no numerical approximation or computer
algebra calculation is needed to determine their signs. Computer algebra
can independently check these identities, but the proof does not depend
on such a check.
\end{remark}

\begin{remark}
The positivity range on $\mathcal V_{\le6}$ is sharp. When $n=30$,
{
\begin{equation*}
 n^2-34n+144=24,
 \qquad
 n^2-54n+152=-568.
\end{equation*}
}
 {From the formula for $p_3(D_2)$, we obtain}
{
\[
 p_3(D_2)=-\frac{102816}{50765}<0.
\]
}
Thus $I_\varepsilon^{(30)}$ is indefinite on $\mathcal V_{\le6}$.
\end{remark}

\subsection*{Comparison at finite radius}

Return to the family in \eqref{eq:truncated-pohozaev-form}. Recall that
{
\[
 Z_0=\frac{n-2}{2}\frac{1-|y|^2}{1+|y|^2}U.
\]
}
Every $H^{(m)}$ has  zero trace and satisfies
$\sum_{j=1}^nH_{ij}^{(m)}y^j=0$.  Hence
$\det g_{\varepsilon,t}=1$ and, for every radial function $f$,
{
\[
 \Delta_{g_{\varepsilon,t}}f
 =\sum_{i,j=1}^n\partial_i
 \bigl(g_{\varepsilon,t}^{ij}\partial_jf\bigr)
 =\sum_{i=1}^n\partial_i\biggl(\frac{f'(|y|)}{|y|}y^i\biggr)
 =\Delta f.
\]
}
The double divergence has degree $m-2$.  For $m=2,3$, the  moment
identities therefore imply
{
\[
 \delta^2H^{(m)}=0.
\]
}
Together with the corrector equations for the
remaining degrees, these identities imply
{
\[
 E_{\varepsilon,0}=0,
 \qquad
 \partial_tE_{\varepsilon,t}\big|_{t=0}=0.
\]
}
Consequently,
{
\begin{equation*}
 I_{\varepsilon,R}^{(n)}(H,H)
 =\frac{1}{2c(n)}\int_{B_R}
 Z_0\,\partial_t^2E_{\varepsilon,t}\big|_{t=0}\dd y.
\end{equation*}
}

The quadratic expansion in \cite[Lemma  A.2]{KMS} gives the second scalar
curvature coefficient. With
\(h_{\varepsilon,1}=\sum_{m=2}^d\varepsilon^mH^{(m)}\), it is
{
\begin{equation*}
\begin{aligned}
 \frac12\partial_t^2\operatorname{Scal}_{g_{\varepsilon,t}}\big|_{t=0}
 &=-\sum_{i,j,\ell=1}^n\partial_j
 \bigl(h_{\varepsilon,1,ij}\partial_\ell
 h_{\varepsilon,1,i\ell}\bigr)
 +\frac12\sum_{i,j,\ell=1}^n
 \partial_jh_{\varepsilon,1,ij}\,
 \partial_\ell h_{\varepsilon,1,i\ell}\\
 &\quad-\frac14\sum_{i,j,\ell=1}^n
 \partial_\ell h_{\varepsilon,1,ij}\,
 \partial_\ell h_{\varepsilon,1,ij}.
\end{aligned}
\end{equation*}
}
The first term has zero pairing with the radial function $Z_0U$;
both its boundary term and its interior first order term vanish because
$\sum_{j=1}^nh_{\varepsilon,1,ij}y^j=0$.  Polar coordinates then give
the first sum in \eqref{eq:truncated-pohozaev-expansion}.

 Set
 {
\[
 \Psi_\varepsilon=\sum_{m=4}^d
 \varepsilon^m Z(H^{(m)}).
\]
}
The mixed term involving the metric and corrector
variations in the Laplacian also vanishes:
{
\begin{equation*}
 \int_{B_R}Z_0\sum_{i,j=1}^n\partial_i
 \bigl(h_{\varepsilon,1,ij}\partial_j\Psi_\varepsilon\bigr)\dd y=0.
\end{equation*}
}
 Indeed, after integration by parts,
the boundary contribution contains
$\sum_{i=1}^nh_{\varepsilon,1,ij}y^i$, while the interior contribution
contains $\sum_{i=1}^nh_{\varepsilon,1,ij}\partial_iZ_0$; both vanish by
the gauge condition and the radiality of $Z_0$.  After separating the
scalar curvature contribution, the remaining part of the second
variation is
{
\begin{equation*}
\begin{aligned}
 \sum_{s,t=4}^d\varepsilon^{s+t}\biggl\{
 \int_{B_R}Z_0\,\delta^2H^{(s)}Z(H^{(t)})\,\dd y
 -\frac{2n(n+2)}{c(n)(n-2)}\int_{B_R}
       U^{4/(n-2)}\frac{Z_0}{U}Z(H^{(s)})Z(H^{(t)})\,\dd y\biggr\}.
\end{aligned}
\end{equation*}
}
To identify this expression with the corresponding
term in $I_{\varepsilon,R}^{(n)}$, differentiate the corrector equation under
dilations.  For each
$4\le s\le d$, this gives
{
\begin{equation*}
\begin{aligned}
 &\bigl(\Delta+n(n+2)U^{4/(n-2)}\bigr)
 D_{\mathrm{di}}Z(H^{(s)})\\
 {}={}&c(n)Z_0\,\delta^2H^{(s)}
 +s c(n)U\,\delta^2H^{(s)}
 -\frac{4n(n+2)}{n-2}U^{4/(n-2)}
       \frac{Z_0}{U}Z(H^{(s)}).
\end{aligned}
\end{equation*}
}
 Apply Green's second identity on $B_R$ to
$Z(H^{(t)})$ and
$\frac{n-2}{2}Z(H^{(s)})+y\mathbin{\cdot}\nabla Z(H^{(s)})$:
{
\begin{equation*}
\begin{aligned}
 &\int_{B_R}\biggl\{
 Z(H^{(t)})\Delta\bigl(D_{\mathrm{di}}Z(H^{(s)})\bigr)
 -D_{\mathrm{di}}Z(H^{(s)})
       \Delta Z(H^{(t)})\biggr\}\dd y\\
 {}={}&\int_{\partial B_R}\biggl\{
 Z(H^{(t)})\partial_r\bigl(D_{\mathrm{di}}Z(H^{(s)})\bigr)
 -\partial_r Z(H^{(t)})
       D_{\mathrm{di}}Z(H^{(s)})\biggr\}\dd S.
\end{aligned}
\end{equation*}
}
 A second integration by parts gives the bilinear Pohozaev identity
{
\begin{equation*}
\begin{aligned}
 &\int_{B_R}\biggl\{
 D_{\mathrm{di}}Z(H^{(s)})
 \bigl(\Delta+n(n+2)U^{4/(n-2)}\bigr) Z(H^{(t)})\\
 &\quad+D_{\mathrm{di}}Z(H^{(t)})
 \bigl(\Delta+n(n+2)U^{4/(n-2)}\bigr) Z(H^{(s)})\biggr\}\dd y\\
 {}={}&\int_{\partial B_R}\biggl\{
 D_{\mathrm{di}}Z(H^{(s)})\partial_r Z(H^{(t)})
 +D_{\mathrm{di}}Z(H^{(t)})\partial_r Z(H^{(s)})\\
 &\quad-R\bigl(\nabla Z(H^{(s)})\mathbin{\cdot}\nabla Z(H^{(t)})
       -n(n+2)U^{4/(n-2)}Z(H^{(s)})Z(H^{(t)})\bigr)\biggr\}\dd S\\
 &\quad-\frac{4n(n+2)}{n-2}\int_{B_R}
 U^{4/(n-2)}\frac{Z_0}{U}Z(H^{(s)})Z(H^{(t)})\dd y.
\end{aligned}
\end{equation*}
}
 After multiplication by the
coefficients in the second variation, the
$U^{4/(n-2)}(Z_0/U)Z(H^{(s)})Z(H^{(t)})$ terms cancel.  Summing in
$s,t$, symmetrizing, and using the corrector equation on $\partial B_R$
to replace the radial second derivative in
$\partial_r(D_{\mathrm{di}}Z(H^{(s)}))$, we obtain the second interior sum
below.  The boundary contributions form
$\mathcal R_{\varepsilon,R}(H,H)$ in
\eqref{eq:boundary-remainder-definition}.  Hence the second variation
on a finite ball is

{
\begin{equation}\label{eq:truncated-pohozaev-expansion}
\begin{aligned}
 I_{\varepsilon,R}^{(n)}(H,H)
 &=\frac{n-2}{2}\sum_{s,t=2}^d
 \varepsilon^{s+t}c_{s+t}(R)B(H^{(s)},H^{(t)})\\
 &\quad-\sum_{s,t=4}^d s\varepsilon^{s+t}
 \int_{B_R}\delta^2H^{(s)}Z(H^{(t)})U\,\dd y
 +\mathcal R_{\varepsilon,R}(H,H),
\end{aligned}
\end{equation}
}
 where
{
\begin{equation}\label{eq:boundary-remainder-definition}
\begin{aligned}
 \mathcal R_{\varepsilon,R}(H,H)
 &:=R\sum_{s,t=4}^d\varepsilon^{s+t}
 \int_{\partial B_R}\delta^2H^{(s)}U Z(H^{(t)})\,\dd S\\
 &\quad-c(n)^{-1}\sum_{s,t=4}^d\varepsilon^{s+t}
 \int_{\partial B_R}\partial_r Z(H^{(s)})
 D_{\mathrm{di}}Z(H^{(t)})\dd S\\
 &\quad+\frac{R}{2c(n)}\sum_{s,t=4}^d\varepsilon^{s+t}
 \int_{\partial B_R}
 \nabla Z(H^{(s)})\mathbin{\cdot}\nabla Z(H^{(t)})
 \,\dd S\\
 &\quad-\frac{Rn(n+2)}{2c(n)}\sum_{s,t=4}^d\varepsilon^{s+t}
 \int_{\partial B_R}U^{4/(n-2)}Z(H^{(s)})Z(H^{(t)})\,\dd S.
\end{aligned}
\end{equation}
}
We also use $I_\varepsilon^{(n)}(H,K)$ for the symmetric bilinear form
associated with $I_\varepsilon^{(n)}$.

{
\begin{lemma}\label{lem:finite-kms-bridge}
There exists \(C=C(n)>0\) such that, for every \(0<\varepsilon<1\),
\(R\ge2\), and \(H,K\in\mathcal V_{\le d}\), the following estimates
hold.  If \(n\) is odd, then
\[
 \big|I_{\varepsilon,R}^{(n)}(H,K)-I_\varepsilon^{(n)}(H,K)\big|
 \le CR^{-1}\|H\|_{\varepsilon,R}\|K\|_{\varepsilon,R}.
\]
If \(n\) is even, then
\[
\begin{aligned}
 &\big|I_{\varepsilon,R}^{(n)}(H,K)-I_\varepsilon^{(n)}(H,K)\big|\\
 {}\le{}&C\left(
 R^{-2}+(1+\log R)^{-1/2}
 +\frac{|\log R-|\log\varepsilon||}{1+\log R}\right)
 \|H\|_{\varepsilon,R}\|K\|_{\varepsilon,R}.
\end{aligned}
\]
Moreover, for every \(\delta>0\), there exists
\(\varepsilon_0=\varepsilon_0(n,\delta)>0\) such that, if
\(0<\varepsilon\le\varepsilon_0\) and \(R=\delta/\varepsilon\), then
\[
 \big|I_{\varepsilon,R}^{(n)}(H,K)-I_\varepsilon^{(n)}(H,K)\big|
 \le C(n,\delta)(1+\log R)^{-1/2}
 \|H\|_{\varepsilon,R}\|K\|_{\varepsilon,R}.
\]
\end{lemma}
}

 \begin{proof}
\begingroup
\setlength{\abovedisplayskip}{5pt plus 1pt minus 1pt}
\setlength{\belowdisplayskip}{5pt plus 1pt minus 1pt}
\setlength{\abovedisplayshortskip}{3pt plus 1pt}
\setlength{\belowdisplayshortskip}{3pt plus 1pt}
\setlength{\jot}{2pt}
We first prove the diagonal estimate.  Set
{
\begin{equation*}
 \mathcal J_{\varepsilon,R}(H)
 :=I_{\varepsilon,R}^{(n)}(H,H)-I_\varepsilon^{(n)}(H,H).
\end{equation*}
}
 We denote its associated symmetric bilinear form by
$\mathcal J_{\varepsilon,R}(H,K)$.
Suppose that $n$ is even.  Subtracting \eqref{eq:kms-pohozaev-form}
from \eqref{eq:truncated-pohozaev-expansion}, we obtain
{
\begin{equation*}
 \mathcal J_{\varepsilon,R}(H)
 =\mathcal J_{\varepsilon,R}^{(1)}(H)
 +\mathcal J_{\varepsilon,R}^{(2)}(H)
 +\mathcal J_{\varepsilon,R}^{(3)}(H)
 +\mathcal R_{\varepsilon,R}(H,H),
\end{equation*}
}
 where
{
\begin{align*}
 \mathcal J_{\varepsilon,R}^{(1)}(H)
 &:=\frac{n-2}{2}\sum_{s,t=2}^{d-1}\varepsilon^{s+t}
 \bigl\{c_{s+t}(R)-c_{s+t}\bigr\}B(H^{(s)},H^{(t)})+\sum_{s,t=4}^{d-1}s\varepsilon^{s+t}
 \int_{\mathbb R^n\setminus B_R}
 \delta^2H^{(s)}Z(H^{(t)})U\,\dd y,\\[2pt]
 \mathcal J_{\varepsilon,R}^{(2)}(H)
 &:=(n-2)\sum_{t=2}^{d-1}\varepsilon^{d+t}c_{d+t}(R)
 B(H^{(d)},H^{(t)})\\
 &\qquad-\sum_{t=4}^{d-1}\varepsilon^{d+t}\int_{B_R}\bigg(
 d\,\delta^2H^{(d)}Z(H^{(t)})
 +t\,\delta^2H^{(t)}Z(H^{(d)})
 \bigg)U\dd y,\\[2pt]
 \mathcal J_{\varepsilon,R}^{(3)}(H)
 &:=\frac{n-2}{2}\varepsilon^{n-2}
 (c_{n-2}(R)-|\log\varepsilon|)B(H^{(d)},H^{(d)})\\
 &\qquad-d\varepsilon^{n-2}\Big(
 \int_{B_R}\delta^2H^{(d)}Z(H^{(d)})U\,\dd y -|\log\varepsilon|\int_{\Sph^{n-1}}
 \delta^2H^{(d)}
 \Gamma^{(d+2)}\bigl(c(n)\delta^2H^{(d)}\bigr)\dd\theta\Big).
\end{align*}
}
 Write $r=|y|$.  By the homogeneity of $H^{(s)}$ and
Khuri--Marques--Schoen \cite[(4.4)]{KMS}, we have
{
\begin{equation}\label{eq:corrector-pointwise-estimate}
 |\delta^2H^{(s)}(y)|
 \le C|H^{(s)}|r^{s-2},\quad
 |Z(H^{(t)})(y)|+(1+r)|\nabla Z(H^{(t)})(y)|
 \le C|H^{(t)}|(1+r)^{t+2-n}.
\end{equation}
}
 On the finite dimensional spaces
$\mathcal V_s\times\mathcal V_t$,
{
\[
 |B(H^{(s)},H^{(t)})|\le C|H^{(s)}|\,|H^{(t)}|.
\]
}
 By \eqref{eq:truncated-radial-coefficient},
\eqref{eq:corrector-pointwise-estimate}, and the  Cauchy--Schwarz
inequality, we have
{
\begin{equation*}
\begin{aligned}
 \bigl|\mathcal J_{\varepsilon,R}^{(1)}(H)\bigr|
 &\le C\sum_{s,t=2}^{d-1}\varepsilon^{s+t}|H^{(s)}|\,|H^{(t)}|
 \bigg\{|c_{s+t}(R)-c_{s+t}|
 +\int_R^\infty r^{s+t+1-n}\dd r\bigg\}\\
 &\le C\sum_{s,t=2}^{d-1}\varepsilon^{s+t}|H^{(s)}|\,|H^{(t)}|
 \int_R^\infty r^{s+t+1-n}\dd r\\
 &\le C\sum_{s,t=2}^{d-1}
 \varepsilon^{s+t}R^{s+t+2-n}|H^{(s)}|\,|H^{(t)}|
 \le CR^{-2}\bigg(\sum_{m=2}^{d-1}
 \varepsilon^m|H^{(m)}|\bigg)^2
 \le CR^{-2}\|H\|_{\varepsilon,R}^2.
\end{aligned}
\end{equation*}
}
 For
$\mathcal J_{\varepsilon,R}^{(2)}$, estimate
\eqref{eq:corrector-pointwise-estimate} and $d+t<n-2$ give
{
\begin{equation*}\begin{aligned}
 \bigl|\mathcal J_{\varepsilon,R}^{(2)}(H)\bigr|
 &\le C\sum_{t=2}^{d-1}\varepsilon^{d+t}|H^{(d)}|\,|H^{(t)}|
 \bigg(1+\int_1^R r^{d+t-n+1}\dd r\bigg)\\
 &\le C\sum_{t=2}^{d-1}\varepsilon^{d+t}|H^{(d)}|\,|H^{(t)}|
 \le C(1+\log R)^{-1/2}\|H\|_{\varepsilon,R}^2,
\end{aligned}\end{equation*}
}
where the last inequality uses the factor \(1+\log R\) in the \(m=d=(n-2)/2\) component of the norm  $\|H\|_{\varepsilon,R}$
 in \eqref{eq:weighted-jet-norm}.

We finally estimate $\mathcal J_{\varepsilon,R}^{(3)}(H)$.  By
Proposition  \ref{prop:kms-corrector}, the corrector polynomial has
degree at most $d+2$ and the  same parity as its homogeneous source.
Hence
{
\begin{equation*}
 \Gamma\bigg(c(n)\delta^2H^{(d)}\bigg)
 =
 \Gamma^{(d+2)}\bigg(c(n)\delta^2H^{(d)}\bigg)
 +\sum_{j=0}^{d}\Gamma^{(j)}\bigg(c(n)\delta^2H^{(d)}\bigg).
\end{equation*}
}
Since $2d=n-2$, polar coordinates and finite dimensionality yield
{
\begin{equation*}
\begin{aligned}
 &\int_{B_R}\bigg(\delta^2H^{(d)}\bigg)Z(H^{(d)})U\dd y\\
 {}={}&\int_{\Sph^{n-1}}\bigg(\delta^2H^{(d)}\bigg)
 \Gamma^{(d+2)}\bigg(c(n)\delta^2H^{(d)}\bigg)\dd\theta
 \int_0^R\frac{r^{2n-3}}{(1+r^2)^{n-1}}\dd r\\
 &\quad+O\bigg(|H^{(d)}|^2\sum_{j=0}^{d}
 \int_0^\infty\frac{r^{d+j+n-3}}{(1+r^2)^{n-1}}\dd r\bigg)\\
 {}={}&\log R\int_{\Sph^{n-1}}\bigg(\delta^2H^{(d)}\bigg)
 \Gamma^{(d+2)}\bigg(c(n)\delta^2H^{(d)}\bigg)\dd\theta
 +O\bigl(|H^{(d)}|^2\bigr),
\end{aligned}
\end{equation*}
}
where we have used
{
\begin{equation*}
 \int_0^R\frac{r^{2n-3}}{(1+r^2)^{n-1}}\dd r
 =\int_0^1\frac{r^{2n-3}}{(1+r^2)^{n-1}}\dd r
 +\int_1^R\bigl\{r^{-1}+O(r^{-3})\bigr\}\dd r
 =\log R+O(1).
\end{equation*}
}
By \eqref{eq:truncated-radial-coefficient},
$c_{n-2}(R)=\log R+O(1)$.  Consequently,
{
\begin{equation*}
\begin{aligned}
 \bigl|\mathcal J_{\varepsilon,R}^{(3)}(H)\bigr|
 &\le C\varepsilon^{n-2}\Bigl\{
 |c_{n-2}(R)-|\log\varepsilon||
 +|\log R-|\log\varepsilon||+1\Bigr\}|H^{(d)}|^2\\
 &\le C\varepsilon^{n-2}
 \Bigl\{1+\bigl|\log R-|\log\varepsilon|\bigr|\Bigr\}|H^{(d)}|^2
 \le C\frac{1+|\log R-|\log\varepsilon||}{1+\log R}
 \|H\|_{\varepsilon,R}^2.
\end{aligned}
\end{equation*}
}
Using \eqref{eq:corrector-pointwise-estimate} in
\eqref{eq:boundary-remainder-definition}, we obtain
{
\begin{equation}\label{eq:boundary-remainder-estimate}
\begin{aligned}
 |\mathcal R_{\varepsilon,R}(H,H)|
 &\le C\sum_{s,t=4}^{d}
 \varepsilon^{s+t}R^{s+t+2-n}|H^{(s)}|\,|H^{(t)}|\\
 &\le CR^{-2}\bigg(\sum_{m=4}^{d-1}\varepsilon^m|H^{(m)}|\bigg)^2
   +CR^{-1}\varepsilon^d|H^{(d)}|
      \sum_{m=4}^{d-1}\varepsilon^m|H^{(m)}|
   +C\varepsilon^{n-2}|H^{(d)}|^2\\
 &\le C\bigl\{R^{-2}+(1+\log R)^{-1/2}\bigr\}
 \|H\|_{\varepsilon,R}^2.
\end{aligned}
\end{equation}
}
When $n$ is odd, direct subtraction of
\eqref{eq:kms-pohozaev-form} from
\eqref{eq:truncated-pohozaev-expansion} leaves only the terms satisfying $s+t<n-2$.  In this case, the upper index in
$\mathcal J_{\varepsilon,R}^{(1)}$ is $d=(n-3)/2$, and
$s+t+2-n\le-1$.  For even $n$, combine the estimates for
$\mathcal J_{\varepsilon,R}^{(1)}$,
$\mathcal J_{\varepsilon,R}^{(2)}$,
$\mathcal J_{\varepsilon,R}^{(3)}$, and
 $\mathcal R_{\varepsilon,R}$.  For odd $n$, combine the estimate for
$\mathcal J_{\varepsilon,R}^{(1)}$ with
\eqref{eq:boundary-remainder-estimate}.  In both cases,
{
\begin{equation*}
 |\mathcal J_{\varepsilon,R}(H)|
 \le C\|H\|_{\varepsilon,R}^2
 \begin{cases}
  R^{-1},&n\text{ is odd},\\
  R^{-2}+(1+\log R)^{-1/2}
  +\displaystyle\frac{|\log R-|\log\varepsilon||}{1+\log R},
  &n\text{ is even}.
 \end{cases}
\end{equation*}
}
The corresponding bilinear estimate follows from polarization.  Indeed,
if $|\mathcal J_{\varepsilon,R}(X)|\le A\|X\|_{\varepsilon,R}^2$, then
for every $a>0$,
{
\[
\begin{aligned}
 4|\mathcal J_{\varepsilon,R}(H,K)|
 &\le A\|aH+a^{-1}K\|_{\varepsilon,R}^2
      +A\|aH-a^{-1}K\|_{\varepsilon,R}^2\\
 &=2A\bigl(a^2\|H\|_{\varepsilon,R}^2
          +a^{-2}\|K\|_{\varepsilon,R}^2\bigr).
\end{aligned}
\]
}
 Optimization over $a>0$ gives
$|\mathcal J_{\varepsilon,R}(H,K)|
\le A\|H\|_{\varepsilon,R}\|K\|_{\varepsilon,R}$.
If $R=\delta/\varepsilon$, then
$|\log R-|\log\varepsilon||=|\log\delta|$.  Since $R\to\infty$ as
$\varepsilon\to0$, the stated simplified estimate follows.
\endgroup
\end{proof}

\subsection*{\texorpdfstring{Comparison with the integral
involving \(Z_kE_k\)}{Comparison with the integral involving Z_kE_k}}

 \begin{lemma}\label{lem:actual-finite-comparison}
For every $\vartheta>0$ there is
$\delta_0=\delta_0(n,C_{\mathrm g},\vartheta)>0$ such that,
if $R_k=\delta/\varepsilon_k$ with
$0<\delta\le\delta_0$, then the following estimate holds for all
sufficiently large $k$:
\[
 \Bigg|c(n)^{-1}\int_{B_{R_k}}Z_kE_k\dd y
 -I_{\varepsilon_k,R_k}^{(n)}(H_k,H_k)\Bigg|
 \le\vartheta\|H_k\|_{\varepsilon_k,R_k}^2
 +C_\vartheta\varepsilon_k^{\min\{n-2,14\}}
 (1+|\log\varepsilon_k|).
\]
Here $C_\vartheta=C(n,C_{\mathrm g},\vartheta)>0$.
\end{lemma}

\begin{proof}
Throughout the proof, write $\varepsilon=\varepsilon_k$ and $R=R_k$;
retain the index $k$ on the Taylor jets.  For $0\le t\le1$, set
{
\begin{equation*}
 h_t=\sum_{m=2}^{n-4}t\varepsilon^mH_k^{(m)},
 \qquad g_t=\exp h_t,
 \qquad
 \varphi_{\lambda,t}
 =U_\lambda+\sum_{m=4}^{n-4}
 t\varepsilon^m\psi_{m,\lambda}(H_k^{(m)}).
\end{equation*}
}
Put
{
\begin{equation*}
 Z_t=-\lambda\partial_\lambda
       \varphi_{\lambda,t}\big|_{\lambda=1},
 \qquad
 E_t=-L_{g_t}\varphi_{1,t}-n(n-2)\varphi_{1,t}^{p},
\end{equation*}
}
and set
{
\begin{equation*}
 F_R(t)=c(n)^{-1}\int_{B_R}Z_tE_t\dd y.
\end{equation*}
}
By the corrector equations and \eqref{eq:appendix-moments},
{
\begin{equation*}
 F_R(0)=F_R'(0)=0.
\end{equation*}
}
By \eqref{eq:truncated-pohozaev-form}, the terms
in $F_R''(0)/2$ involving only Taylor jets of degree at most
$d$ equal
$I_{\varepsilon,R}^{(n)}(H_k,H_k)$.  Taylor's formula therefore gives
the exact decomposition
{
\begin{equation}\label{eq:actual-formal-exact-decomposition}
\begin{aligned}
 c(n)^{-1}\int_{B_R}Z_kE_k\dd y-I_{\varepsilon,R}^{(n)}(H_k,H_k)
 &=\Biggl\{c(n)^{-1}\int_{B_R}Z_kE_k\dd y-F_R(1)\Biggr\}\\
 &\quad+\Biggl\{\frac12F_R''(0)
             -I_{\varepsilon,R}^{(n)}(H_k,H_k)\Biggr\}
 +\frac12\int_0^1(1-t)^2F_R'''(t)\dd t.
\end{aligned}
\end{equation}
}

We estimate each of the three terms on the
right side.  The rescaled logarithm of the
actual metric has the expansion
{
\begin{equation*}
 \log g_k(y)=\sum_{m=2}^{n-4}\varepsilon^mH_k^{(m)}(y)
 +\varepsilon^{n-3}H_k^{(n-3)}(y)+T_{k,\varepsilon}(y),
\end{equation*}
}
 where, for $0\le a\le2$,
{
\begin{equation*}
 |D^aT_{k,\varepsilon}(y)|
 \le C\varepsilon^{n-2}(1+|y|)^{n-2-a}.
\end{equation*}
}
 Since $\varphi_k=\varphi_{1,1}$ and $Z_k=Z_1$, the nonlinear terms in
the two errors agree.  The exact difference is therefore
{
\begin{equation*}
\begin{aligned}
 &c(n)^{-1}\int_{B_R}Z_kE_k\dd y-F_R(1)\\
 {}={}&\int_{B_R}Z_1\biggl\{
 -c(n)^{-1}\sum_{i,j=1}^n\partial_i
 \bigl((g_k^{ij}-g_1^{ij})\partial_j\varphi_{1,1}\bigr)
 +(\operatorname{Scal}_{g_k}-\operatorname{Scal}_{g_1})\varphi_{1,1}\biggr\}\dd y.
\end{aligned}
\end{equation*}
}
 We first isolate the term that is linear in $H_k^{(n-3)}$.  At the
Euclidean metric,
{
\begin{equation*}
\begin{aligned}
 D\Delta_\delta[H_k^{(n-3)}]U
 &=-\sum_{i,j=1}^nH_{k,ij}^{(n-3)}\partial_i\partial_jU
   -\sum_{i,j=1}^n\partial_iH_{k,ij}^{(n-3)}\partial_jU=0,\\
 D\!\operatorname{Scal}_\delta[H_k^{(n-3)}]
 &=\delta^2H_k^{(n-3)}.
\end{aligned}
\end{equation*}
}
The first equality follows from radiality, the zero trace condition, and
$\sum_{j=1}^nH_{k,ij}^{(n-3)}y^j=0$.  Since $Z_0U$ is radial,
{by \eqref{eq:appendix-moments},}
{
\begin{equation*}
 \int_{B_R}Z_0U\delta^2H_k^{(n-3)}\dd y=0.
\end{equation*}
}
Every remaining term either contains $T_{k,\varepsilon}$ or contains
at least two Taylor jets.  These terms come from the expansion of the
inverse metric in the Laplacian, the  $D^2h$ and $Dh\,Dh$ terms in scalar
curvature, and their pairings with the affine corrector.  More precisely,
if $h=\log g_1$ and $\tau=\log g_k-\log g_1$, then
the Taylor bounds give \(|h|+|\tau|\le C\delta^2\) on \(B_R\), and the
coordinate expansions yield
{
\begin{equation*}
\begin{aligned}
 (\Delta_{\exp(h+\tau)}-\Delta_{\exp h})f
 &=-\sum_{i,j=1}^n\partial_i(\tau_{ij}\partial_jf)
 +\sum_{i,j=1}^n\partial_i
 \bigl\{O(|h||\tau|+|\tau|^2)\partial_jf\bigr\},\\
 &\left|\operatorname{Scal}_{\exp(h+\tau)}
 -\operatorname{Scal}_{\exp h}
 -\sum_{i,j=1}^n\partial_i\partial_j\tau_{ij}\right|\\
 &\quad\le C\bigl((|h|+|\tau|)|D^2\tau|
          +|\tau||D^2h|+|Dh||D\tau|
          +|D\tau|^2+|\tau||Dh|^2\bigr).
\end{aligned}
\end{equation*}
}
With \(r=|y|\), we have
\[
 \begin{aligned}
 |D\tau|
 &\le C\varepsilon^{n-3}(1+r)^{n-4}
      +C\varepsilon^{n-2}(1+r)^{n-3},\\
 |Dh|&\le C\varepsilon^2(1+r),\qquad
 |Z_1\varphi_{1,1}|\le C(1+r)^{4-2n}.
 \end{aligned}
\]
Consequently,
\[
 \begin{aligned}
 &\int_{B_R}|Z_1\varphi_{1,1}|
 \bigl(|D\tau|^2+|\tau||Dh|^2\bigr)\,\dd y\\
 &\quad\le
 C\varepsilon^{2n-6}
 \left(1+\int_1^Rr^{n-5}\,\dd r\right)
 {}+C\varepsilon^{2n-4}
 \left(1+\int_1^Rr^{n-3}\,\dd r\right)\\
 &\qquad
 +C\varepsilon^{n+1}\left(1+\int_1^Rr^2\,\dd r\right)
 +C\varepsilon^{n+2}\left(1+\int_1^Rr^3\,\dd r\right)
 \le C\varepsilon^{n-2},
 \end{aligned}
\]
where \(R=\delta/\varepsilon\) and \(0<\delta<1\).
Combining this estimate with the corrector bounds gives
{
\begin{equation*}
\begin{aligned}
 C\biggl\{&\varepsilon^{n-2}
       \biggl(1+\int_1^Rr^{-1}\dd r\biggr)
 +\varepsilon^{2n-6}\biggl(1+\int_1^Rr^{n-5}\dd r\biggr)\\
 &+\sum_{m=2}^{n-4}\varepsilon^{m+n-3}|H_k^{(m)}|
       \biggl(1+\int_1^Rr^{m-2}\dd r\biggr)
 +\sum_{m=2}^{n-4}\varepsilon^{m+n-2}|H_k^{(m)}|
       \biggl(1+\int_1^Rr^{m-1}\dd r\biggr)\biggr\}.
\end{aligned}
\end{equation*}
}
The first term bounds the linear Taylor remainder, the terms containing
\(|DT_{k,\varepsilon}|^2\), and the terms containing
\(|\tau||Dh|^2\). The second term bounds the square of
\(H_k^{(n-3)}\), including its contribution to \(|D\tau|^2\).
The last two terms bound the pairings of \(H_k^{(n-3)}\) and
\(T_{k,\varepsilon}\) with the retained Taylor jets. Since
\(R=\delta/\varepsilon\)
and all Taylor coefficients are uniformly bounded,  the preceding
quantity is at most
$C\varepsilon^{n-2}(1+|\log\varepsilon|)$.  Consequently,
{
\begin{equation}\label{eq:actual-metric-tail}
 \big|c(n)^{-1}\int_{B_R}Z_kE_k\dd y-F_R(1)\big|
 \le C\varepsilon^{n-2}(1+|\log\varepsilon|).
\end{equation}
}

Extend the quadratic form in
\eqref{eq:truncated-pohozaev-form} to $\mathcal V_{\le n-4}$, and denote
the extended form and its polarization by
$\mathcal Q_{\varepsilon,R}$.
Set
{
\begin{equation*}
 \widehat H_{k,>d}
 =\widehat H_k-H_k
 =\sum_{m=d+1}^{n-4}H_k^{(m)}.
\end{equation*}
}
The restriction of $\mathcal Q_{\varepsilon,R}$ to
$\mathcal V_{\le d}$ is $I_{\varepsilon,R}^{(n)}$.  Hence
 {
\begin{equation*}
 \frac12F_R''(0)-I_{\varepsilon,R}^{(n)}(H_k,H_k)
 =2\mathcal Q_{\varepsilon,R}(H_k,\widehat H_{k,>d})
  +\mathcal Q_{\varepsilon,R}
  (\widehat H_{k,>d},\widehat H_{k,>d}).
\end{equation*}
}
We estimate the mixed coefficients of $\mathcal Q_{\varepsilon,R}$.
Fix $2\le s,t\le n-4$ and, for $a,b$ near zero, define
{
\begin{equation*}
 \begin{aligned}
 g_{a,b}
 &=\exp\bigl(a\varepsilon^sH_k^{(s)}
             +b\varepsilon^tH_k^{(t)}\bigr),\\
 \varphi_{\lambda,a,b}
 &=U_\lambda
   +a\varepsilon^s\psi_{s,\lambda}(H_k^{(s)})
   +b\varepsilon^t\psi_{t,\lambda}(H_k^{(t)}),\\
 E_{a,b}
 &=-L_{g_{a,b}}\varphi_{1,a,b}
   -n(n-2)\varphi_{1,a,b}^{p},\\
 Z_{a,b}
 &=-\lambda\partial_\lambda\varphi_{\lambda,a,b}
      \big|_{\lambda=1}.
 \end{aligned}
\end{equation*}
}
The corrector equations give
\[
 E_{0,0}=0,\qquad
 \left.\partial_aE_{a,b}\right|_{a=b=0}
 =\left.\partial_bE_{a,b}\right|_{a=b=0}=0.
\]
Consequently,
{
\begin{equation*}
 \begin{aligned}
 \mathcal Q_{\varepsilon,R}(H_k^{(s)},H_k^{(t)})
 &=\frac1{2c(n)}
   \left.\partial_a\partial_b
   \int_{B_R}Z_{a,b}E_{a,b}\,\dd y\right|_{a=b=0}\\
 &=\frac1{2c(n)}\int_{B_R}Z_0
   \left.\partial_a\partial_bE_{a,b}\right|_{a=b=0}\,\dd y.
 \end{aligned}
\end{equation*}
}
The mixed derivative consists of contractions of
\begin{equation*}
\begin{gathered}
 DH_k^{(s)}DH_k^{(t)}U,\qquad
 H_k^{(s)}D^2H_k^{(t)}U,\qquad
 \sum_{i,j=1}^n\partial_i
 \bigl(H_{k,ij}^{(s)}\partial_j Z(H_k^{(t)})\bigr),\\
 \delta^2H_k^{(s)}
 Z(H_k^{(t)}),\qquad
 U^{p-2}Z(H_k^{(s)})Z(H_k^{(t)}),
\end{gathered}
\end{equation*}
 and the terms obtained by interchanging $s$ and $t$.  {Using the coordinate
formulas for $L_{\exp h}$ and
\eqref{eq:corrector-pointwise-estimate}, we obtain}
{
\begin{equation*}
 \bigg|\frac{\partial^2}{\partial a\,\partial b}
 E_{a,b}(y)\bigg|_{a=b=0}
 \le C\varepsilon^{s+t}|H_k^{(s)}|\,|H_k^{(t)}|
       (1+|y|)^{s+t-n}.
\end{equation*}
}
Since
$|Z_0(y)|\le C(1+|y|)^{2-n}$, polar coordinates give
{
\begin{equation*}
\begin{aligned}
 |\mathcal Q_{\varepsilon,R}(H_k^{(s)},H_k^{(t)})|
 &\le C\varepsilon^{s+t}|H_k^{(s)}|\,|H_k^{(t)}|
 \int_0^Rr^{n-1}(1+r)^{s+t+2-2n}\dd r\\
 &\le C\varepsilon^{s+t}|H_k^{(s)}|\,|H_k^{(t)}|
 \Biggl\{1+\int_1^Rr^{s+t-n+1}\dd r\Biggr\}.
\end{aligned}
\end{equation*}
}
Substituting this bound into the decomposition
of $F_R''(0)/2$ yields
{
\begin{equation*}
\begin{aligned}
 \big|\tfrac12F_R''(0)-I_{\varepsilon,R}^{(n)}(H_k,H_k)\big|
 &\le C\sum_{\substack{2\le s,t\le n-4\\\max\{s,t\}>d}}
 \varepsilon^{s+t}|H_k^{(s)}|\,|H_k^{(t)}|
 \Biggl\{1+\int_1^Rr^{s+t-n+1}\dd r\Biggr\}.
\end{aligned}
\end{equation*}
}
Since $R=\delta/\varepsilon$,
{
\begin{equation}\label{eq:actual-radial-trichotomy}
 \varepsilon^{s+t}
 \Biggl\{1+\int_1^Rr^{s+t-n+1}\dd r\Biggr\}
 \le C
 \begin{cases}
  \varepsilon^{s+t},&s+t<n-2,\\
  \varepsilon^{n-2}(1+|\log\varepsilon|),&s+t=n-2,\\
  \varepsilon^{n-2}\delta^{s+t-n+2},&s+t>n-2.
 \end{cases}
\end{equation}
}
Young's inequality and the uniform bounds for the Taylor coefficients therefore
imply
{
\begin{equation*}
 \big|\tfrac12F_R''(0)-I_{\varepsilon,R}^{(n)}(H_k,H_k)\big|
 \le\frac{\vartheta}{2}
 \|H_k\|_{\varepsilon,R}^2
 +C_\vartheta\varepsilon^{\min\{n-2,14\}}
 (1+|\log\varepsilon|).
\end{equation*}
}

 It remains to estimate the integral remainder.  Put
{
\begin{equation*}
 \mathcal H_\varepsilon(r)
 =\sum_{m=2}^{n-4}\varepsilon^m|H_k^{(m)}|(1+r)^m
 +\{\varepsilon(1+r)\}^{n-3}.
\end{equation*}
}
On $B_R$, $\mathcal H_\varepsilon\le C\delta^2$.
After decreasing $\delta$, we apply
Khuri--Marques--Schoen \cite[(4.4)]{KMS} to obtain
\begin{equation*}
 \frac12U\le\varphi_{1,t}\le2U,
 \qquad
 |Z_t|\le C(1+r)^{2-n},
 \qquad
 |\partial_tZ_t|\le
 C\mathcal H_\varepsilon(r)(1+r)^{2-n},
\end{equation*}
uniformly for $0\le t\le1$.
 For $a=2,3$, Taylor's formula
for the matrix exponential and the coordinate formula for scalar
curvature give
{
\begin{equation*}
 |\partial_t^ag_t^{-1}|+(1+r)|D\partial_t^ag_t^{-1}|
 +(1+r)^2|\partial_t^a\operatorname{Scal}_{g_t}|
 \le C\mathcal H_\varepsilon(r)^a.
\end{equation*}
}
 The same corrector estimate gives
{
\begin{equation*}
 \sum_{m=4}^{n-4}\varepsilon^m
 \bigl\{|\psi_{m,\lambda}|+(1+r)|D\psi_{m,\lambda}|\bigr\}
 \le C\mathcal H_\varepsilon(r)(1+r)^{2-n}.
\end{equation*}
}
The same bound holds after one logarithmic scale derivative.  In
particular,
$|\partial_t\varphi_{1,t}|\le C\mathcal H_\varepsilon U$.
Together with
$\frac{1}{2}U\le\varphi_{1,t}\le2U$, this estimate bounds uniformly in
$t$ all derivatives produced from $\varphi_{1,t}^p$.  Substitution in
the coordinate formula for $E_t$ gives
{
\begin{equation*}
 |\partial_t^2E_t(y)|
 \le C\mathcal H_\varepsilon(r)^2(1+r)^{-n},
 \qquad
 |\partial_t^3E_t(y)|
 \le C\mathcal H_\varepsilon(r)^3(1+r)^{-n}.
\end{equation*}
}
The functions $\varphi_{\lambda,t}$ and $Z_t$ are affine in $t$.
Therefore $\partial_t^2Z_t=0$, and direct differentiation gives the
exact identity
{
\begin{equation*}
 F_R'''(t)=c(n)^{-1}\int_{B_R}
 \bigl\{3\partial_tZ_t\,\partial_t^2E_t
       +Z_t\partial_t^3E_t\bigr\}\dd y.
\end{equation*}
}
{By the preceding pointwise estimates,}
{
\begin{equation*}
 |F_R'''(t)|
 \le C\int_0^Rr^{n-1}(1+r)^{2-2n}
 \mathcal H_\varepsilon(r)^3\dd r.
\end{equation*}
}
Moreover,
{
\begin{equation*}
\begin{aligned}
 \mathcal H_\varepsilon(r)^2
 &\le C\sum_{s,t=2}^d\varepsilon^{s+t}
 |H_k^{(s)}|\,|H_k^{(t)}|(1+r)^{s+t}\\
 &\quad+C\sum_{\substack{2\le s,t\le n-4\\\max\{s,t\}>d}}
 \varepsilon^{s+t}|H_k^{(s)}|\,|H_k^{(t)}|(1+r)^{s+t}
 +C\varepsilon^{2n-6}(1+r)^{2n-6}.
\end{aligned}
\end{equation*}
}
Applying \eqref{eq:actual-radial-trichotomy} to these three sums and using
the uniform bounds for the Taylor coefficients, we obtain
{
\begin{align*}
 \int_0^Rr^{n-1}(1+r)^{2-2n}
 \mathcal H_\varepsilon(r)^2\dd r
 &\le C\sum_{s,t=2}^d\varepsilon^{s+t}
 |H_k^{(s)}|\,|H_k^{(t)}|
 \bigg(1+\int_1^Rr^{s+t-n+1}\dd r\bigg)\\
 &\quad+C\varepsilon^{\min\{n-2,14\}}
 (1+|\log\varepsilon|)\\
 &\le C\|H_k\|_{\varepsilon,R}^2
 +C\varepsilon^{\min\{n-2,14\}}
 (1+|\log\varepsilon|).
\end{align*}
}
Therefore
{
\begin{equation*}
 \frac12\Big|\int_0^1(1-t)^2F_R'''(t)\dd t\Big|
 \le C\delta^2\|H_k\|_{\varepsilon,R}^2
 +C\delta^2\varepsilon^{\min\{n-2,14\}}
 (1+|\log\varepsilon|).
\end{equation*}
}
Choose $\delta_0$ so that $C\delta_0^2\le\vartheta/2$.  From
\eqref{eq:actual-formal-exact-decomposition},
\eqref{eq:actual-metric-tail}, and the  preceding estimates, we obtain
{
\begin{equation*}
 \Bigg|c(n)^{-1}\int_{B_R}Z_kE_k\dd y
 -I_{\varepsilon,R}^{(n)}(H_k,H_k)\Bigg|
 \le\vartheta\|H_k\|_{\varepsilon,R}^2
 +C_\vartheta\varepsilon^{\min\{n-2,14\}}
 (1+|\log\varepsilon|).
\end{equation*}
}
\end{proof}

\subsection*{Deduction of the signed estimate}

 \begin{proof}[Proof of Lemma  \ref{lem:signed-residual}]
Since $R_k=\delta/\varepsilon_k$, for every
$H\in\mathcal V_{\le d}$ and all sufficiently large $k$,
{
\begin{equation}\label{eq:appendix-norm-comparison}
 \frac12\|H\|_{\varepsilon_k,R_k}^2
 \le \|H\|_{\varepsilon_k,*}^2
 \le 2\|H\|_{\varepsilon_k,R_k}^2.
\end{equation}
}

\emph{Case  1: $7\le n\le24$.}  
By
Proposition  \ref{prop:kms-algebraic-positivity},
Lemma  \ref{lem:finite-kms-bridge}, and
 \eqref{eq:appendix-norm-comparison}, we have
{
\begin{align*}
 I_{\varepsilon_k,R_k}^{(n)}(H_k,H_k)
 &\ge I_{\varepsilon_k}^{(n)}(H_k,H_k)
 -C(1+\log R_k)^{-1/2}
  \|H_k\|_{\varepsilon_k,R_k}^2\\
 &\ge b_n\|H_k\|_{\varepsilon_k,*}^2
 -C(1+\log R_k)^{-1/2}
  \|H_k\|_{\varepsilon_k,R_k}^2\\
 &\ge \Bigl(\frac{b_n}{2}
       -C(1+\log R_k)^{-1/2}\Bigr)
       \|H_k\|_{\varepsilon_k,R_k}^2\ge \frac{b_n}{4}
       \|H_k\|_{\varepsilon_k,R_k}^2.
\end{align*}
}
 By Lemma  \ref{lem:actual-finite-comparison} with
$\vartheta=b_n/8$,
{
\begin{align*}
 c(n)^{-1}\int_{B_{R_k}}Z_kE_k\dd y
 &\ge I_{\varepsilon_k,R_k}^{(n)}(H_k,H_k)
 -\frac{b_n}{8}\|H_k\|_{\varepsilon_k,R_k}^2
 -C\varepsilon_k^{\min\{n-2,14\}}
       (1+|\log\varepsilon_k|)\\
 &\ge \frac{b_n}{8}\|H_k\|_{\varepsilon_k,R_k}^2
 -o(\varepsilon_k^{(n-2)/2}).
\end{align*}
}
 Multiplying by $c(n)>0$, we obtain the desired estimate.

 \emph{Case  2: $25\le n\le29$.}  
Write $H_k=H_{k,\le6}+H_{k,>6}$.  By
Proposition  \ref{prop:kms-algebraic-positivity},
\eqref{eq:appendix-norm-comparison},
Lemma  \ref{lem:finite-kms-bridge}, and Young's inequality, we have
{
\begin{align*}
 I_{\varepsilon_k,R_k}^{(n)}(H_k,H_k)
 &=I_{\varepsilon_k,R_k}^{(n)}(H_{k,\le6},H_{k,\le6})
 +2I_{\varepsilon_k,R_k}^{(n)}(H_{k,\le6},H_{k,>6})
 +I_{\varepsilon_k,R_k}^{(n)}(H_{k,>6},H_{k,>6})\\
 &\ge I_{\varepsilon_k}^{(n)}(H_{k,\le6},H_{k,\le6})
 -C(1+\log R_k)^{-1/2}
       \|H_{k,\le6}\|_{\varepsilon_k,R_k}^2\\
 &\quad-2C_n\|H_{k,\le6}\|_{\varepsilon_k,R_k}
          \|H_{k,>6}\|_{\varepsilon_k,R_k}
 -C_n\|H_{k,>6}\|_{\varepsilon_k,R_k}^2\\
 &\ge \bigl(\tfrac12b_n^{(6)}-C(1+\log R_k)^{-1/2}\bigr)
       \|H_{k,\le6}\|_{\varepsilon_k,R_k}^2\\
 &\quad-2C_n\|H_{k,\le6}\|_{\varepsilon_k,R_k}
          \|H_{k,>6}\|_{\varepsilon_k,R_k}
 -C_n\|H_{k,>6}\|_{\varepsilon_k,R_k}^2\\
 &\ge \frac{b_n^{(6)}}4
       \|H_{k,\le6}\|_{\varepsilon_k,R_k}^2
 -C_n\|H_{k,>6}\|_{\varepsilon_k,R_k}^2.
\end{align*}
}
 {By the uniform bounds for the Taylor coefficients,}
{
\[
\begin{aligned}
 \|H_{k,>6}\|_{\varepsilon_k,R_k}^2
 =\sum_{m=7}^d
       \varepsilon_k^{2m}(1+\log R_k)^{\theta_m}
       |H_k^{(m)}|^2
 \le C\varepsilon_k^{14}(1+|\log\varepsilon_k|).
\end{aligned}
\]
}
 By these two bounds and Lemma  \ref{lem:actual-finite-comparison} with
$\vartheta=b_n^{(6)}/8$,
{
\begin{align*}
 c(n)^{-1}\int_{B_{R_k}}Z_kE_k\dd y
 &\ge I_{\varepsilon_k,R_k}^{(n)}(H_k,H_k)
 -\frac{b_n^{(6)}}8
       \|H_{k,\le6}\|_{\varepsilon_k,R_k}^2\\
 &\quad-\frac{b_n^{(6)}}8
       \|H_{k,>6}\|_{\varepsilon_k,R_k}^2
 -C\varepsilon_k^{14}(1+|\log\varepsilon_k|)\\
 &\ge \frac{b_n^{(6)}}8
       \|H_{k,\le6}\|_{\varepsilon_k,R_k}^2
 -C\varepsilon_k^{14}(1+|\log\varepsilon_k|)\\
 &\ge \frac{b_n^{(6)}}8
       \|H_{k,\le6}\|_{\varepsilon_k,R_k}^2
 -o(\varepsilon_k^{(n-2)/2}).
\end{align*}
}
Multiplying by $c(n)>0$, we obtain the desired estimate.
\end{proof}

\refstepcounter{section}\label{app:proofs}
\section*{Appendix B. Coefficient calculations for the annular estimates}
\addcontentsline{toc}{section}{Appendix B. Coefficient calculations for the annular estimates}

{
This appendix supplies the two calculations deferred from Section  7: the
mixed radial density estimate in Lemma  \ref{lem:annular-bilinear-density}
and the third variation estimate in
Lemma  \ref{lem:reference-third-variation}. All calculations use the
Euclidean \(y\) coordinates, with \(r=|y|\). The parameters
\((\lambda_j,b_j)\) are fixed as in
Proposition  \ref{prop:final-modulation}. Constants are uniform for
\((\lambda_j,b_j)\in\mathcal K\), and every little \(o\) term is taken as
\(j\to\infty\), with the background metric fixed.
}

 \subsection*{Mixed radial densities}

\begin{proof}[Proof of the mixed coefficient estimate in
Lemma  \ref{lem:annular-bilinear-density}]
\emph{Step 1. The exact mixed coefficient.}
{For the homogeneous directions, fix
\(2\le\alpha,\beta\le d\) and write}
\[
 z_\alpha=\begin{cases}
  \rho_j^\alpha Z_{\lambda_j,b_j}[H^{(\alpha)}],
  &\alpha<\frac{n-2}{2},\\
  0,&\alpha\ge \frac{n-2}{2},
 \end{cases}
 \qquad
 z_\beta=\begin{cases}
  \rho_j^\beta Z_{\lambda_j,b_j}[H^{(\beta)}],
  &\beta<\frac{n-2}{2},\\
  0,&\beta\ge \frac{n-2}{2}.
 \end{cases}
\]
Consider the family with two parameters
\[
 g_{s,\tau}^{\alpha,\beta}
 =\exp\{s\rho_j^\alpha H^{(\alpha)}
       +\tau\rho_j^\beta H^{(\beta)}\},\qquad
 v_{s,\tau}^{\alpha,\beta}=U_j+sz_\alpha+\tau z_\beta.
\]
{For the general assertion in the lemma, make the same
definitions with the two pairs of tensors and scalar functions satisfying
\eqref{eq:extended-bilinear-direction-bounds}. The differentiation formulas
below depend only on the bounds stated there.}
Define the operator coefficients by
\[
 \begin{aligned}
 E_\alpha
 &=\left.\partial_s(\Delta-L_{g_{s,\tau}^{\alpha,\beta}})
   \right|_{s=\tau=0},\\
 E_\beta
 &=\left.\partial_\tau(\Delta-L_{g_{s,\tau}^{\alpha,\beta}})
   \right|_{s=\tau=0},\\
 E_{\alpha\beta}
 &=\left.\partial_s\partial_\tau
   (\Delta-L_{g_{s,\tau}^{\alpha,\beta}})\right|_{s=\tau=0}.
 \end{aligned}
\]
Thus \(E_{\alpha\beta}\) is the coefficient of \(s\tau\), without a factor
of \(1/2\). Since the inverse metric is
\(\exp\{-s\rho_j^\alpha H^{(\alpha)}
-\tau\rho_j^\beta H^{(\beta)}\}\), the first coefficient is
\[
 E_\alpha f=\rho_j^\alpha\left[
 \sum_{k,l=1}^n\partial_k(H^{(\alpha)}_{kl}\partial_lf)
 +c(n)\left(\sum_{k,l=1}^n
 \partial_k\partial_lH^{(\alpha)}_{kl}\right)f\right],
\]
and \(E_\beta\) has the same formula with \(\alpha\) replaced by \(\beta\).
With \(U=U_j\), direct differentiation of
\eqref{eq:finite-annulus-quadratic-form} gives
 {
\begin{equation}\label{eq:exact-mixed-coefficient}
 \begin{aligned}
 2I_j^{\mathrm{ann}}(H^{(\alpha)},H^{(\beta)})
 =&\int_{A_{\sigma_j,R_j}}(D_{\mathrm{di}}U)
 E_{\alpha\beta}U\,\dd y+\int_{A_{\sigma_j,R_j}}(D_{\mathrm{di}}z_\alpha)
 E_\beta U\,\dd y\\
 &+\int_{A_{\sigma_j,R_j}}(D_{\mathrm{di}}z_\beta)
 E_\alpha U\,\dd y+\int_{A_{\sigma_j,R_j}}(D_{\mathrm{di}}U)
 (E_\alpha z_\beta+E_\beta z_\alpha)\,\dd y.
 \end{aligned}
\end{equation}
}
{For two general pairs of tensors and scalar functions, the
left side of
\eqref{eq:exact-mixed-coefficient} means twice the mixed coefficient defined
before Lemma  \ref{lem:annular-bilinear-density}.}
In divergence form,
\[
\begin{aligned}
 E_\alpha f
 &=-\sum_{k,l=1}^n\partial_k\left(
 \left.\partial_s(g_{s,\tau}^{\alpha,\beta})^{kl}\right|_{s=\tau=0}
 \partial_lf\right)
 +c(n)\left.\partial_s\operatorname{Scal}_{g_{s,\tau}^{\alpha,\beta}}\right|_{s=\tau=0}f,\\
 E_\beta f
 &=-\sum_{k,l=1}^n\partial_k\left(
 \left.\partial_\tau(g_{s,\tau}^{\alpha,\beta})^{kl}\right|_{s=\tau=0}
 \partial_lf\right)
 +c(n)\left.\partial_\tau \operatorname{Scal}_{g_{s,\tau}^{\alpha,\beta}}\right|_{s=\tau=0}f,\\
 E_{\alpha\beta}f
 &=-\sum_{k,l=1}^n\partial_k\left(
 \left.\partial_s\partial_\tau
 (g_{s,\tau}^{\alpha,\beta})^{kl}\right|_{s=\tau=0}
 \partial_lf\right)
 +c(n)\left.\partial_s\partial_\tau
 \operatorname{Scal}_{g_{s,\tau}^{\alpha,\beta}}\right|_{s=\tau=0}f.
\end{aligned}
\]
\emph{Step 2. Uniform coefficient and response bounds.}
{Applying the matrix exponential and the coordinate formula for
scalar curvature, we obtain
\[
\begin{aligned}
 \left|\left.\partial_s(g_{s,\tau}^{\alpha,\beta})^{-1}
 \right|_{s=\tau=0}\right|
 +r\left|\nabla_y\left.\partial_s(g_{s,\tau}^{\alpha,\beta})^{-1}
 \right|_{s=\tau=0}\right|
 &\le C\rho_j^\alpha r^\alpha,\\
 \left|\left.\partial_\tau(g_{s,\tau}^{\alpha,\beta})^{-1}
 \right|_{s=\tau=0}\right|
 +r\left|\nabla_y\left.\partial_\tau(g_{s,\tau}^{\alpha,\beta})^{-1}
 \right|_{s=\tau=0}\right|
 &\le C\rho_j^\beta r^\beta,\\
 \left|\left.\partial_s\partial_\tau(g_{s,\tau}^{\alpha,\beta})^{-1}
 \right|_{s=\tau=0}\right|
 +r\left|\nabla_y\left.
 \partial_s\partial_\tau(g_{s,\tau}^{\alpha,\beta})^{-1}
 \right|_{s=\tau=0}\right|
 &\le C\rho_j^{\alpha+\beta}r^{\alpha+\beta},\\
 \left|\left.\partial_s\operatorname{Scal}_{g_{s,\tau}^{\alpha,\beta}}
 \right|_{s=\tau=0}\right|
 &\le C\rho_j^\alpha r^{\alpha-2},\\
 \left|\left.\partial_\tau \operatorname{Scal}_{g_{s,\tau}^{\alpha,\beta}}
 \right|_{s=\tau=0}\right|
 &\le C\rho_j^\beta r^{\beta-2},\\
 \left|\left.\partial_s\partial_\tau \operatorname{Scal}_{g_{s,\tau}^{\alpha,\beta}}
 \right|_{s=\tau=0}\right|
 &\le C\rho_j^{\alpha+\beta}r^{\alpha+\beta-2}.
\end{aligned}
\]
}
{By Proposition  \ref{prop:euclidean-response}, for
\(\ell\in\{\alpha,\beta\}\) with \(\ell<\frac{n-2}{2}\),
\[
 |z_\ell|+|D_{\mathrm{di}}z_\ell|
 +r\bigl(|\nabla_yz_\ell|
 +|\nabla_y(D_{\mathrm{di}}z_\ell)|\bigr)
 \le C\rho_j^\ell
 \begin{cases}
  1,&r\le1,\\
  r^{\ell+2-n},&r\ge1
 \end{cases}.
\]
}
{The bubble \(U=U_j\) satisfies
\[
 |U|+|D_{\mathrm{di}}U|
 +r\bigl(|\nabla_yU|
 +|\nabla_y(D_{\mathrm{di}}U)|\bigr)
 \le C
 \begin{cases}
  1,&r\le1,\\
  r^{2-n},&r\ge1
\end{cases}.
\]
}

\emph{Step 3. Radial and boundary densities.}
Integrate every divergence term in
\eqref{eq:exact-mixed-coefficient} once. On \(r\ge1\), the first
integral has radial bound
\[
 C\rho_j^{\alpha+\beta}
 \bigl(r^{\alpha+\beta}r^{1-n}r^{1-n}
      +r^{\alpha+\beta-2}r^{2-n}r^{2-n}\bigr)
 r^{n-1}\,\dd r
 =C\rho_j^{\alpha+\beta}r^{\alpha+\beta+1-n}\,\dd r.
\]
{The preceding coefficient and response bounds give the same
outer power for the other four integrals.} After integration over
\(\Sph^{n-1}\), the sum of the absolute
values of the resulting radial densities is bounded by
\[
 C\rho_j^{\alpha+\beta}
 \begin{cases}
  r^{n-3}\,\dd r,&0<r\le1,\\
  r^{\alpha+\beta+1-n}\,\dd r,&1\le r\le R_j
 \end{cases}.
\]
{The boundary terms produced by this integration by parts are
bounded on \(\partial B_{\sigma_j}\) and \(\partial B_{R_j}\), respectively,
by}
\[
 C\rho_j^{\alpha+\beta}\sigma_j^{n-2},
 \qquad
 C\rho_j^{\alpha+\beta}R_j^{\alpha+\beta+2-n}.
\]
These are \eqref{eq:annular-bilinear-density} and
\eqref{eq:annular-bilinear-boundary-density}. {Because the
calculation uses only \eqref{eq:extended-bilinear-direction-bounds}, it also
proves the assertion for the two general pairs of tensors
and scalar functions.}
\end{proof}

 \subsection*{Third variation of the Pohozaev integral}\phantomsection\label{app:third-variation}

\begin{proof}[Proof of Lemma  \ref{lem:reference-third-variation}]
\emph{Step 1. Profile and metric bounds.}
In this proof, write
\[
 U=U_j,\qquad
 V=V_{j,\tau},\qquad
 \dot V=\partial_\tau V_{j,\tau}.
\]
Write \(E_\tau=\Delta-L_{g_{j,\tau}}\), use
\(D_{\mathrm{di}}\) from \eqref{eq:dilation-operator}, and set
\(m_0=\min\{m_*,d+1\}\).
If \(m_*=(n-2)/2\) or \(m_*=\infty\), then
\(\dot V=0\) by \eqref{eq:corrected-profile}.  {By
Proposition  \ref{prop:euclidean-response},
Lemma  \ref{lem:corrected-profile-estimate}, and
\eqref{eq:scaled-metric-taylor-remainder}, the following bounds hold uniformly for
\(0\le\tau\le1\), with \(r=|y|\):}
{
\begin{equation}\label{eq:third-order-profile-bounds}
 \begin{aligned}
 \sum_{k=0}^2r^k|\nabla_y^kV|
 &\le CU,\\
 \sum_{k=0}^2r^k|\nabla_y^k\dot V|
 &\le C\rho_j^{m_0}(1+r)^{m_0}U,\\
  \sum_{k=0}^2r^k|\nabla_y^k[h(\rho_j\,\cdot)](y)|
 &\le C\rho_j^{m_0}(1+r)^{m_0}.
 \end{aligned}
\end{equation}
}
{By \eqref{eq:annular-radii}, for \(r\le R_j\) and all
sufficiently large \(j\),}
\[
 \rho_j(1+r)\le \rho_j+\bar\rho_{j-1}
 \le2\bar\rho_{j-1}.
\]
{Since \(m_0\ge2\), it follows from
\eqref{eq:capacity-scales} that}
\[
 \rho_j^{m_0}(1+r)^{m_0}\le C\eta_j.
\]

{For \(\nu=1,2,3\), differentiating the matrix exponential and
using the coordinate formula for scalar curvature, we obtain}
\begin{equation}\label{eq:third-order-metric-bounds}
 \max_{1\le k,l\le n}
 \left(
 |\partial_\tau^\nu(g_{j,\tau}^{kl}-\delta^{kl})|
 +r|\nabla_y\partial_\tau^\nu g_{j,\tau}^{kl}|
 \right)
 +r^2|\partial_\tau^\nu \operatorname{Scal}_{g_{j,\tau}}|
 \le C\rho_j^{\nu m_0}(1+r)^{\nu m_0}.
\end{equation}
In this calculation every spatial derivative of \(h(\rho_jy)\) is a
derivative of the composite function:
\[
 \partial_y^\alpha[h(\rho_j\,\cdot)](y)
 =\rho_j^{|\alpha|}(\partial_x^\alpha h)(\rho_jy).
\]
{More explicitly, the scalar curvature formula gives, for
\(\nu=1,2,3\),}
\[
 |\partial_\tau^\nu \operatorname{Scal}_{g_{j,\tau}}(y)|
 \le C\rho_j^2
 \left(
 |h|^{\nu-1}|\nabla_x^2h|
 +|h|^{\max\{\nu-2,0\}}|\nabla_xh|^2
 \right)(\rho_jy).
\]
{The third estimate in
\eqref{eq:third-order-profile-bounds} and the preceding smallness bound imply
\eqref{eq:third-order-metric-bounds}. When \(\nu=1\), the
\(|\nabla_xh|^2\) term contains an additional factor bounded by
\(C\eta_j\). Since \(\det g_{j,\tau}=1\), applying
\eqref{eq:third-order-metric-bounds} to a \(C^2\) function \(f\) gives}
\begin{equation}\label{eq:third-order-operator-bound}
 |(\partial_\tau^\nu E_\tau)f|
 \le
 Cr^{-2}\rho_j^{\nu m_0}(1+r)^{\nu m_0}
 \sum_{k=0}^2r^k|\nabla_y^kf|,
 \qquad 1\le\nu\le3.
\end{equation}

\emph{Step 2. Differentiation of the Pohozaev integral.}
Differentiate \eqref{eq:reference-pohozaev-path} directly. The annulus is
independent of \(\tau\), and \(V\) is affine in \(\tau\). The complete Leibniz
expansion is therefore
\begin{equation}\label{eq:third-order-leibniz}
 \begin{aligned}
 \mathfrak q_j^{(3)}(\tau)
 =\int_{A_{\sigma_j,R_j}}\bigl\{
 &(D_{\mathrm{di}}V)(\partial_\tau^3E_\tau)V
 +3(D_{\mathrm{di}}\dot V)(\partial_\tau^2E_\tau)V\\
 &+3(D_{\mathrm{di}}V)(\partial_\tau^2E_\tau)\dot V
 +6(D_{\mathrm{di}}\dot V)(\partial_\tau E_\tau)\dot V
 \bigr\}\,\dd y.
\end{aligned}
\end{equation}
\emph{Step 3. Radial integration.}
Equations \eqref{eq:third-order-profile-bounds}--%
\eqref{eq:third-order-operator-bound} bound the absolute value of every
integrand in \eqref{eq:third-order-leibniz} by
\[
 Cr^{-2}\rho_j^{3m_0}(1+r)^{3m_0}U^2
 \le
 C\eta_jr^{-2}\rho_j^{2m_0}(1+r)^{2m_0}U^2.
\]
Using \(U\le C(1+r)^{2-n}\) and integrating in \(r\) gives
\begin{equation}\label{eq:third-order-radial-integral}
 |\mathfrak q_j^{(3)}(\tau)|
 \le
 C\eta_j\rho_j^{2m_0}
 \left(
 1+\int_1^{R_j}r^{2m_0+1-n}\,\dd r
 \right).
\end{equation}
{If \(m_0=m_*<(n-2)/2\), the integral in
\eqref{eq:third-order-radial-integral} is uniformly bounded. If
\(m_0=m_*=(n-2)/2\), it equals \(\log R_j\). If \(m_*=\infty\), then
\(m_0=d+1>(n-2)/2\), and \eqref{eq:outer-radius-rho-bound} gives}
\[
 \rho_j^{2m_0}\left(1+R_j^{2m_0+2-n}\right)
 \le C\rho_j^{m_0+\frac{n-2}{2}}
 =o(\rho_j^{n-2})
 \quad\text{as }j\to\infty.
\]
The little \(o\) term is uniform for
\((\lambda_j,b_j)\in\mathcal K\). {These three estimates prove
\eqref{eq:reference-third-variation}.}
\end{proof}

\subsection*{AI assistance and interactions}

This project began in mid July 2026 with an investigation of whether the main conclusion of \cite{HanXiongZhang} could fail without the assumption that the metric has flatness order at least \((n-2)/2\). We made progress gradually, first in dimension seven, then in dimensions up to 10,  then  in dimensions up to 24; and found by early August that the threshold dimension is 30:  the upper bound holds in dimensions \(7\le n \le 29\) and there are counter-examples to the upper bound in dimension \(n\ge 30\). By August 9, we produced an initial draft proof of Theorems  \ref{thm:main} and  \ref{thm:upper-bound}  and an exploratory draft of a possible counterexample in dimension \(30\). We then began an extensive verification, correction, addition, and revision process.

ln this process, the coauthors used GPT 5.6 Sol and Codex as interactive research-assistance tools for literature reading, exploratory proof searches, calculations, checking arguments, and manuscript preparation. The systems responded to specific questions and directions from the coauthors; they did not independently formulate or complete the theorems and did not reliably include complete or accurate citations and attributions. The coauthors determined the research questions and proof strategies, supplied and verified the citations and attributions, and selected, checked, and substantially revised every adopted argument. This process involved repeated exchanges among the coauthors and the systems, approximately sixty manuscript versions, and extensive human verification and rewriting. The coauthors are responsible for all mathematical claims, citations, and final content. 

\subsection*{Conflict of interest} 

The authors declare that they have no conflicts of interests. 


\bigskip
\begingroup
\small
\noindent\textsc{Zheng-Chao Han}\\
Department of Mathematics, Rutgers University,\\
Hill Center--Busch Campus, 110 Frelinghuysen Road,\\
Piscataway, NJ 08854, USA\\
\textit{Email address:}
\href{mailto:zchan@math.rutgers.edu}{\nolinkurl{zchan@math.rutgers.edu}}

\medskip
\noindent\textsc{Qinfeng Jiang}\\
School of Mathematical Sciences, Beijing Normal University,\\
Beijing 100875, China\\
\textit{Email address:}
\href{mailto:202531130031@mail.bnu.edu.cn}%
{\nolinkurl{202531130031@mail.bnu.edu.cn}}

\medskip
\noindent\textsc{Hua-Yang Wang}\\
School of Mathematical Sciences, Laboratory of Mathematics and Complex
Systems, MOE,\\
Beijing Normal University, Beijing 100875, China\\
\textit{Email address:}
\href{mailto:wanghuayang@amss.ac.cn}{\nolinkurl{wanghuayang@amss.ac.cn}}

\medskip
\noindent\textsc{Jingang Xiong}\\
School of Mathematical Sciences, Laboratory of Mathematics and Complex
Systems, MOE,\\
Beijing Normal University, Beijing 100875, China;\\
Center for Basic Mathematics, Institute for Advanced Study,\\
Beijing Normal University, Beijing 100875, China\\
\textit{Email address:}
\href{mailto:jx@bnu.edu.cn}{\nolinkurl{jx@bnu.edu.cn}}

\medskip
\noindent\textsc{Lei Zhang}\\
Department of Mathematics, University of Florida,\\
358 Little Hall, P.O. Box 118105,\\
Gainesville, FL 32611-8105, USA\\
\textit{Email address:}
\href{mailto:leizhang@ufl.edu}{\nolinkurl{leizhang@ufl.edu}}
\endgroup
\end{document}